\documentclass[10pt]{article}
\usepackage[text={7in,10in},centering]{geometry}
\usepackage{color}
\usepackage{array}
\usepackage{mathdots}
\usepackage{amsmath,amssymb,amsfonts}

\usepackage{verbatim,graphicx}
\usepackage{varioref}           

\usepackage{mathtools}
\usepackage[citecolor=red]{hyperref}
\usepackage{mathrsfs}
\usepackage{pb-diagram}
\usepackage{epstopdf}
\usepackage{enumerate}
\usepackage{extarrows}
\usepackage{amsthm}
\usepackage{bbm}
\usepackage[all]{xypic}
\usepackage{bigints}
\usepackage{bm}
\usepackage{epsfig}
\usepackage{appendix}
\usepackage{cleveref}
\usepackage{cite}
\makeatletter
\newcommand{\citecomment}[2][]{\citen{#2}#1\citevar}
\newcommand{\citeone}[1]{\citecomment{#1}}
\newcommand{\citetwo}[2][]{\citecomment[,~#1]{#2}}
\newcommand{\citevar}{\@ifnextchar\bgroup{;~\citeone}{\@ifnextchar[{;~\citetwo}{]}}}
\newcommand{\citefirst}{\@ifnextchar\bgroup{\citeone}{\@ifnextchar[{\citetwo}{]}}}
\newcommand{\cites}{[\citefirst}
\makeatother
\usepackage[hyphenbreaks]{breakurl}
\usepackage[utf8x]{inputenc}
\usepackage{latexsym}
\usepackage{longtable}
\usepackage{listings}
\usepackage{mathrsfs}
\usepackage{nomencl}
\usepackage{pdfpages}
\usepackage{pifont}
\usepackage{supertabular}
\usepackage{tikz}

\allowdisplaybreaks

\numberwithin{equation}{section}

 \def\S{{\cal S}}

\numberwithin{equation}{subsection}
\newtheorem{corollary}[equation]{Corollary}
\newtheorem{conjecture}[equation]{Conjecture}
\Crefname{conjecture}{Conjecture}{Conjectures}

\newtheorem{lemma}[equation]{Lemma}
\newtheorem{proposition}[equation]{Proposition}

\newtheorem{theorem}[equation]{Theorem}

\theoremstyle{definition}
\newtheorem{definition}[equation]{Definition}
\newtheorem{theorem*}{Theorem}
\newtheorem{example}[equation]{Example}
\newtheorem{remark}[equation]{Remark}

\definecolor{red}{rgb}{0.9,0,0}

\definecolor{blue}{rgb}{0,0,0.9}

\title{Distinction and quadratic base change for regular supercuspidal representations}
\author{Chuijia Wang\thanks{\textit{Keywords:} distinguished representations, base change, regular supercuspidal representations.}\thanks{\textit{Mathematics Subject Classification} 22E50, 11F70, 11S37}\thanks{Department of Mathematics, The University of Hong Kong, Pokfulam, Hong Kong, China}\thanks{wangchuijia@u.nus.edu}}
\date{}

\makeatletter
\def\thanks#1{\protected@xdef\@thanks{\@thanks\protect\footnotetext{#1}}}
\makeatother
\begin{document}

\maketitle
\begin{abstract}
In this article, we study Prasad's conjecture for regular supercuspidal representations based on the machinery developed by Hakim and Murnaghan to study distinguished representations, and the fundamental work of Kaletha on parameterization of regular supercuspidal representations. For regular supercuspidal representations, we give some new interpretations of the numerical quantities appearing in Prasad's formula, and reduce the proof to the case of tori. The proof of Prasad's conjecture then reduces to a comparison of various quadratic characters appearing naturally in the above process. We also have some new observations on these characters and study the relation between them in detail. For some particular examples, we show the coincidence of these characters, which gives a new purely local proof of Prasad's conjecture for regular supercuspidal representations of these groups. We also prove Prasad's conjecture for regular supercuspidal representations of $G(E)$, when $E/F$ is unramified and $G$ is a general tamely ramified quasi-split reductive group.
\end{abstract}

\tableofcontents

\section{Introduction}
Let $G$ be a connected reductive algebraic group over a non-archimedean local field $F$ with residual characteristic $p\neq2$. In the study of the representation theory of $G(F)$, one important problem is to capture the distinguished representations with respect to a pair $(H,\chi)$, where $H$ is a closed subgroup of $G$ and $\chi$ is a one-dimensional character of $H(F)$. More precisely, an irreducible smooth representation $\pi$ of $G(F)$ is said to be $(H,\chi)$-distinguished if $$\mathrm{Hom}_{H(F)}(\pi,\chi)\neq 0.$$
The distinction property of an irreducible representation could be described in two common ways. One is in terms of the symmetry of the representation itself, such as the self-dual or conjugate self-dual property. The other one is in terms of the functorial property of the Langlands parameter associated to the representation.

Sakellaridis and Venkatesh \cite{SV17} have set up a general framework to understand the relationship between distinguished representations and Langlands functoriality when the corresponding homogeneous space $G/H$ is a spherical variety. Moreover, Dipendra Prasad \cite{Pra15} has a more precise conjecture describing the relationship between the distinction property of an irreducible representation with respect to a specific quadratic character and the base change property of its Langlands parameter when $G/H$ is a Galois symmetric space.

The main focus of this article is to understand Prasad's conjecture on the Galois distinction problem. More precisely, let $E/F$ be a quadratic extension and let $\sigma$ be the non-trivial element in $\mathrm{Gal}(E/F)$. We are interested in the distinction problem for the pair $(G,\omega_{G(F),E})$, where $G$ is regarded as a closed subgroup of $\mathrm{Res}_{E/F}G_{E}$ and $\omega_{G(F),E}$ is the specific quadratic character of $G(F)$ associated to the quadratic extension $E/F$ defined in \cite{Pra15}. Let $\Pi(G(E))$ denote the set of equivalence classes of irreducible representations of $G(E)$ and let $\Phi(G(E))$ denote the set of $\widehat{G}$-conjugacy classes of Langlands parameters of $G(E)$. The local Langlands correspondence states that there exists a finite-to-one surjective map from $\Pi(G(E))$ to $\Phi(G(E))$, which is denoted by $\mathrm{LLC}$. Prasad \cite{Pra15} predicts that the Langlands parameters of $(G(F),\omega_{G(F),E})$-distinguished representations are exactly those which are base changes of Langlands parameters of $G^{\mathrm{op}}(F)$, where $G^{\mathrm{op}}$ is a quasi-split $F$-form of $G$ defined in \Cref{Gop} such that $G^{\mathrm{op}}(E)\cong G(E)$. One can refer to \Cref{section5} for a more detailed description of these notions. In other words, we have the following
\begin{equation}\label{conjecturalparameter}
\begin{aligned}
\bigcup_{\alpha\in H^1(\mathrm{Gal}(E/F),G(E))}\Pi_{(G_{\alpha}(F),\omega_{G_{\alpha}(F),E})}(G(E)) &\xhookrightarrow{\iota}\Pi(G(E)) \xrightarrow{\mathrm{LLC}}\Phi(G(E))\xleftarrow{\mathrm{BC}}\Phi(G^{\mathrm{op}}(F)),\\
\mathrm{Im}(\mathrm{LLC} \circ\iota)&=\mathrm{Im}(\mathrm{BC}).
\end{aligned}
\end{equation}
Moreover, he also conjectures that the following identity holds for an irreducible discrete series representation $\pi$ of $G(E)$ in a generic $L$-packet:
\begin{equation}\label{conjecturalidentity}
\begin{aligned}
\sum_{\alpha\in H^1(\mathrm{Gal}(E/F),G(E))}\dim\mathrm{Hom}_{G_{\alpha}(F)}(\pi,\omega_{G_{\alpha}(F),E})=\sum\limits_{\tilde{\phi}}m(\lambda_{\pi},\tilde{\phi}),
\end{aligned}
\end{equation}
where the sum of the right-hand side runs over all the Langlands parameters $\tilde{\phi}:W_F\rightarrow \prescript{L}{}G^{\mathrm{op}}$ such that $\tilde{\phi}|_{W_{E}}=\phi_\pi$, $\lambda_{\pi} \in \mathrm{Irr}(\pi_{0}(S_{\phi_{\pi}}))$ denotes the irreducible representation of the component group of $S_{\phi_{\pi}}:=C_{\widehat{G}}(\mathrm{Im}\phi_{\pi})$ associated to $\pi$, and $m(\lambda_{\pi},\tilde{\phi})$ is the multiplicity of the trivial representation in the restriction of $\lambda_{\pi}|_{\pi_{0}(S_{\tilde{\phi}})}$. Indeed, \eqref{conjecturalparameter} is a special case of \eqref{conjecturalidentity} when both sides of \eqref{conjecturalidentity} are non-zero.

This conjecture unifies a number of conjectures which aim to capture the distinction property of representations for Galois symmetric pairs and considerable progress has been made in attacking Prasad's conjecture for specific Galois pairs. For the pair $(\mathrm{GL}_n(E),\mathrm{GL}_n(F))$, this conjecture generalizes a conjecture of Jacquet, Rallis and Flicker \cite{Fli91}, which states that the $(\mathrm{GL}_n(F),(\omega_{E/F}\circ \mathrm{det})^{n-1})$-distinguished representations of $\mathrm{GL}_n(E)$ are precisely those arising as base change from $\mathrm{U}_{n,E/F}(F)$, and the conjecture is proved by Kable \cite{Kab04} for discrete series representations. For the pair $(\mathrm{GL}_n(E),\mathrm{U}_{n,E/F}(F))$, this conjecture generalizes a conjecture of Jacquet \cite{Jac05, Jac10} and is proved by Feigon, Lapid and Offen \cite{FLO12}. Anandavardhanan and Prasad prove the conjecture for $(\mathrm{SL}_n(E),\mathrm{SL}_n(F))$ in \cite{AP03, AP18}, where they give the first example of a non-supercuspidal Gelfand pair. More precisely, they find that the dimension of the space of invariant linear forms is strictly larger than one for some supercuspidal representations. Hengfei Lu \cite{Lu19, Lu20} proves this conjecture for certain classical groups of small rank and their similitude groups using the machinery of local theta correspondence. All of these works depend on concrete analysis of the terms appearing on both sides of Prasad's identity, and developing a general method which could work for all Galois symmetric pairs would be interesting. Recently, Rapha\"{e}l Beuzart-Plessis \cite{BP18} develops a relative trace formula approach to attack Prasad's conjecture for general groups. His method turns out to be powerful, and has already been used to prove Prasad's conjecture for Steinberg representations, which recovers early results proved by Broussous and Court\'{e}s \cite{Bro14, Cou15}using the geometry of Bruhat--Tits building, and by Matringe \cite{Mat17} through careful analysis of double cosets appearing in Mackey theory.

In this article, we try to give some evidence on the possibility of proving this conjecture for regular supercuspidal representations without constraints on the Galois symmetric pair using a different purely local method. Supercuspidal representations are fundamental objects in the study of representation theory of $p$-adic groups. They serve as building blocks of all irreducible smooth representations in the sense that every irreducible representation can be realized as a subrepresentation of a parabolic induction from a supercuspidal representation of a Levi subgroup. One of the reasons why we work on regular supercuspidal representations is that they are parametrized by simple data and their Langlands parameters have good functorial properties. These representations are special cases of tamely ramified supercuspidal representations satisfying certain regularity conditions. Explicit constructions of tame supercuspidal representations date back to the work of Howe \cite{How77}. Following his pioneering work, tremendous progress has been made in the explicit construction of tame supercuspidal representations for general reductive groups. Jiu-Kang Yu \cite{Yu01} first develops a general construction of tamely ramified supercuspidal representations using generic cuspidal $G$-data $(\overrightarrow{G},x,\overrightarrow{r},\rho,\overrightarrow{\phi})$, which generalizes the earlier work of Adler \cite{Adl98}. Ju-Lee Kim \cite{Kim07} later shows that the representations constructed by Yu exhaust all the irreducible supercuspidal representations when the characteristic of the residue is large enough. Recently, the exhaustion theorem is improved by Jessica Fintzen \cite{Fin21b} under a weaker constraint on $p$. Based on Yu's work, Tasho Kaletha \cite{Kal19b} classifies regular supercuspidal representations using simpler data $(S,\mu)$ and establishes the local Langlands correspondence for regular supercuspidal representations. As pointed out by Kaletha, one crucial property of a regular supercuspidal representation is that its Langlands parameter factors through the $L$-group of the elliptic torus $S$ which defines the representation
$$\prescript{L}{}{S}\stackrel{\prescript{L}{}j_\chi}{\longrightarrow} \prescript{L}{}{G}.$$
This property enables us to describe the base change of $G$ in terms of the base change of $S$ because of the following commutative diagram of functorial maps
$$\xymatrix{
  \prescript{L}{}{S} \ar[d]_{\prescript{L}{}{j}_\chi} \ar[r]^-{i}
                & \prescript{L}{}{(\mathrm{Res}_{E/F}S_E)}  \ar[d]_{\prescript{L}{}{j_{\chi_{\mathrm{BC}}}}}  \\
  \prescript{L}{}{G} \ar[r]^-{i}
                &   \prescript{L}{}{(\mathrm{Res}_{E/F}G_E)}. }$$

To state and prove Prasad's conjecture for regular supercuspidal representations, we need the Langlands--Vogan bijection for regular supercuspidal representations of $G(E)$. For these representations, the full Langlands--Vogan bijection, that is, the map $\Pi(G,E)\rightarrow \Phi(G,E)$ and the bijection between the set of characters of the component group and the set of rational equivalence classes of rational embeddings of tori, are constructed in a series of papers \cite{Kal19b, FKS23}. In fact, they also deal with more general supercuspidal representations, which are known as non-singular representations. Their work enables us to use the parametrization results of regular supercuspidal representations unconditionally in the study of the distinction problem.

The main ingredients to prove the conjecture are Hakim--Murnaghan's description of the Hom space by Mackey theory \cite{HM08} and Kaletha's parametrization of regular supercuspidal representations using tame regular elliptic pairs \cite{Kal19b}. We will use these ingredients in three steps. The first step is to describe the base change map in terms of regular supercuspidal $L$-packet data defined by Kaletha, and check the compatibility between our formulation and the original base change map defined by restricting the Langlands parameter to the Weil group of the quadratic extension field. Then the right hand side of the conjectural formula \eqref{conjecturalidentity} can be expressed as an identity involving a set of elliptic maximal tori over $F$ and a set of characters of the elliptic tori satisfying certain conditions. The second step is to provide a reinterpretation of Hakim--Murnaghan's formula on the dimension of the Hom space using the same datum that appears on the parameter side of the formula. The final step is to give a comparison between both sides, and reduce the proof to an identity comparing various characters of the elliptic maximal torus $S(F)$, which naturally appear in the distinction problem, the base change problem, and the construction of local Langlands correspondence. A priori, it seems that there is no obvious relation between these characters, whose constructions are of different nature. However, we can prove this conjectural identity for several examples, which gives a new purely local proof of Prasad's conjecture for regular supercuspidal representations of these groups. We also prove Prasad's conjecture for regular supercuspidal representations of $G(E)$ when $E/F$ is unramified, and $G$ is a general quasi-split reductive group.

\begin{theorem*}[\Cref{unramifiedprasad}]
Prasad's \Cref{conjecture1} holds for regular supercuspidal representations when $E/F$ is unramified.
\end{theorem*}

\begin{theorem*}[\Cref{prasadexample}]
Prasad's \Cref{conjecture1} holds for regular supercuspidal representations of $\mathrm{SL}_2$,$\mathrm{GL}_2$,$\mathrm{GL}_n$ ($n$ odd) and $\mathrm{U}_{n,E/F}$ ($n$ odd) for arbitrary quadratic extension $E/F$.
\end{theorem*}

This article is organized as follows. In \Cref{section3,section4}, we briefly summarize Kaletha's work on the construction of $L$-parameters and $L$-packets for regular supercuspidal representations, and study the base change map of Langlands parameters (\Cref{naivebasechange} and \Cref{basechangesame}) in terms of Kaletha's supercuspidal $L$-packet data. After that, we give a brief introduction to Prasad's conjecture in \Cref{section5}, where we also obtain a new factorization formula (\Cref{rootcharacter}) for the restriction of Prasad's quadratic character to an elliptic maximal torus. Then we review Hakim and Murnaghan's work on the computation of the dimension of the space of invariant linear forms and give a new interpretation of their formula (\Cref{HMformularein}) in \Cref{section6}. After all these preparation, the reduction to the case of tori (\Cref{LHS,RHS}) and the comparison of the two sides (\Cref{embeddingoverF} and \Cref{comparisonoflr}) are established in \Cref{section7} except for \Cref{productofcharacters}. In \Cref{section8}, we present a detailed study of these quadratic characters and prove \Cref{characterunramified,characterexample}, which finishes the proof of \Cref{unramifiedprasad,prasadexample}.

\textbf{Acknowledgement.}
This article consists of my Ph.D thesis at National University of Singapore. I want to express my sincere thanks to my supervisor Wee Teck Gan for his patient guidance, constant encouragement and valuable discussions over the years. His insight and suggestions play an important role in the whole project. I would like to thank Dipendra Prasad for proposing his conjecture, as well as continuous communications and encouragement during the period of preparing this article. I am also grateful to Tasho Kaletha, Jeffrey Hakim, Fiona Murnaghan, Jiu-Kang Yu, Laure Blasco, Jessica Fintzen, Wen-Wei Li, Chong Zhang and David Schwein for kindly answering my questions. Thanks are due to Tasho Kaletha, Jeffrey Hakim, Hung Yean Loke, Hengfei Lu, Jiandi Zou for providing useful comments and feedbacks on an early draft. Thanks are also due to Chen-Bo Zhu and Caihua Luo for encouragement. I also want to express my thanks to the anonymous referee for his patience and kind suggestions, which greatly improve the exposition of this article. This work was partially supported by a Singapore government MOE Tier 1 grant R-146-000-320-114 and Israel Science Foundation grant 737/20.

\section{Notations and conventions}
\subsection{Notations}\label{notation}
Let $F$ be a non-archimedean local field with residual characteristic $p\neq 2$. Fix a valuation $v_F:F\rightarrow \mathbb{Z}\cup \{\infty\}$ on $F$. Let $O_F$, $p_F$, $\varpi_F$ and $k_F$ denote the ring of integers, the prime ideal, a uniformizer and the residual field of $F$. Let $q_F$ denote the cardinality of $k_F$. Let $\overline{F}$ be an algebraic closure of $F$ and let $F^s$ be the separable closure of $F$ in $\overline{F}$.

For any finite Galois extension $L/F$, let $\mathrm{Gal}(L/F)$ be the finite Galois group. Let $\Gamma_F=\mathrm{Gal}(F^s/F)\cong \lim\limits_{\stackrel{\longleftarrow}{L\supset F, [L:F]<\infty}}\mathrm{Gal}(L/F)$
be the absolute Galois group with the profinite topology. Let $W_{L/F}$ be the relative Weil group of $L/F$, that is, an extension
$$1\rightarrow L^{\times}\rightarrow W_{L/F}\rightarrow \mathrm{Gal}(L/F)\rightarrow 1,$$
which corresponds to the fundamental class in $H^{2}(\mathrm{Gal}(L/F),L^{\times})$. Let $W_F\cong\lim\limits_{\stackrel{\longleftarrow}{L\supset F, [L:F]<\infty}}W_{L/F}$ be the Weil group with locally profinite topology. Let $I_F$ be the inertia group and let $P_F$ be the wild inertia group. The natural map from $W_F$ to $\Gamma_F$ fits into the commutative diagram
$$\xymatrix{
  1  \ar[r] & I_F \ar@{=}[d] \ar[r] & W_F \ar[d] \ar[r] & \mathbb{Z} \ar@{^{(}->}[d] \ar[r] & 1 \\
  1 \ar[r] & I_F \ar[r] & \Gamma_F \ar[r] & \widehat{\mathbb{Z}} \ar[r] & 1 .}$$
Let $WD_F:=W_F\times \mathrm{SL}_2(\mathbb{C})$ denote the Weil--Deligne group. Let $F^{\mathrm{ur}}$ be the maximal unramified extension of $F$ in $F^s$ and let $\mathrm{Fr}\in \mathrm{Gal}(F^{\mathrm{ur}}/F)$ be the lift of the Frobenius automorphism of $\overline{k_F}$, which is a topological generator of $\widehat{\mathbb{Z}}$. For any group $G$  equipped with a $\Gamma_F$ or $W_F$-action which is compatible with the structure of $G$, let $H^{i}(F,G):=H^{i}(\Gamma_F,G)$ denote the $i$-th Galois cohomology of $G$ and let $H^{i}(W_F,G)$ denote the $i$-th Weil cohomology of $G$. If $G$ is the complex points of an algebraic group over $\mathbb{C}$, we use $H^{i}(F,G)$ and $H^{i}(W_F,G)$ to stand for continuous cohomology, by which we mean the cochains are continuous with respect to the natural complex analytic topology of $G$.

Whenever there is a group homomorphism $f:G_1\rightarrow G_2$, we use the notation $G_2/G_1$ instead of $G_2/\mathrm{Im}(f)$ for simplicity, which is a group only when $\mathrm{Im}f$ is normal in $G_2$. For a group $G$, let $\mathbbm{1}_G$ denote the one-dimensional trivial representation, which is simply denoted by $\mathbbm{1}$ when the group is clear from the context. For any one-dimensional character $\chi:G_2\rightarrow \mathbb{C}^{\times}$ such that $\chi|_{\mathrm{Im}(f)}=\mathbbm{1}$, we sometimes use the notation $\chi\in \mathrm{Hom}(G_2/G_1,\mathbb{C}^{\times})$ for simplicity, when there is no confusion.

For a reductive group $G$ defined over $F$, let $G_{\mathrm{der}}$ denote the derived subgroup of $G$, which is semisimple. Let $G_{\mathrm{sc}}$ and $G_{\mathrm{ad}}$ denote the simply connected cover and adjoint quotient of $G_{\mathrm{der}}$. For any maximal torus $S\subset G$ defined over $F$ with splitting field $L$, let $\Phi(G,S):=\Phi(G_{F^{s}},S_{F^{s}})\subset X^{*}_{F^{s}}(S)$ denote the set of absolute roots associated to the adjoint action of $S_{F^{s}}$. When there is no confusion about the maximal torus, we simply write $\Phi$ for convenience. For any root $\alpha\in \Phi$, let $F_{\alpha}$ and $F_{\pm\alpha}$ be the subfields of $F^s$ corresponding to the subgroup $\mathrm{Stab}_{\Gamma_F}\alpha$ and $\mathrm{Stab}_{\Gamma_F}\{\pm\alpha\}$ of $\Gamma_F$. It is clear that $F_{\alpha}$ is an extension of $F_{\pm\alpha}$ of degree at most two. If the degree of the extension $F_{\alpha}/F_{\pm\alpha}$ is 2, then $\alpha$ is called symmetric. Otherwise, it is called asymmetric. Let $\Phi_{\mathrm{asy}}$ and $\Phi_{\mathrm{sym}}$ denote the set of asymmetric and symmetric roots in $\Phi$. Let $(X^{*}(S),\Phi(G,S), X_{*}(S),\Phi^{\vee}(G,S))$ be the absolute root datum of $G$. Let $\widehat{G}$ be the Langlands dual group of $G$, which is a complex algebraic group determined by the root datum $(X_{*}(S),\Phi^{\vee}(G,S),X^{*}(S),\Phi(G,S))$. For $\alpha\in \Phi(G,S)$, $\widehat{\alpha}$ denotes the corresponding element in $\Phi^{\vee}(\widehat{G},\widehat{S})$. Also, for $\alpha^{\vee}\in \Phi^{\vee}(G,S)$, $\widehat{\alpha^{\vee}}$ denotes the corresponding element in $\Phi(\widehat{G},\widehat{S})$. We will use the notation $\pi_1(G)=\pi_1(G_{\overline{F}}):=X_{*}(S)/X_{*}(S_{\mathrm{sc}})$ for the algebraic fundamental group of $G$, which is finite when $G$ is semi-simple. For any linear algebraic group $H$ over $F$, we will use the notation $\pi_{0}(H):=H/H^{\circ}$ for the component group of $H$. For any $g\in G$, we use the notation $\mathrm{Ad}(g)$ for the conjugation action of $g$ on $G$ as well as for the adjoint action of $g$ on $\mathfrak{g}$ depending on the context.

From now on till the end, we will always assume that $G$ splits over a tamely ramified extension $L/F$, and $S$ is a maximal torus of $G$ defined over $F$ which splits over $L$. Let $\mathcal{B}(G,L)$ denote the enlarged Bruhat--Tits building of $G$ over $L$. Set $\widetilde{\mathbb{R}}:=\mathbb{R}\bigsqcup\{r+|\, r\in \mathbb{R}\}\bigsqcup \{\infty\}$ with a natural partial order: for any $r,s\in \mathbb{R}$, $r<s+$ if $r\leq s$; $r+<s+$, $r+<s$ if $r<s$; and $r<\infty$, $r+<\infty$ for any $r\in \mathbb{R}$. For any $r\in \mathbb{R}_{\geq 0}$ and $x\in \mathcal{A}(G,S,L)\subset \mathcal{B}(G,L)$, let $G(L)_{x,r}$ denote the Moy--Prasad subgroup generated by $T(L)_r$ and $U_{\alpha}(L)_{x,r}$ for any $\alpha\in \Phi(G,S)$ defined in \cite{MP94}, and $G(L)_{x,r+}$ be $\bigcap\limits_{s>r}G(L)_{x,s}$, with corresponding filtration lattices of Lie algebra $\mathfrak{g}(F)_{x,r}$ and $\mathfrak{g}(F)_{x,r+}$. Let $G(F)_{x,r}$ denote $G(L)_{x,r}\cap G(F)$, and $G(F)_{x,r+}$ denote $\bigcap\limits_{s>r}G(F)_{x,s}$. For $r<s$, we also use the notation $G(F)_{x,r:s}$ for the Moy--Prasad quotient $G(F)_{x,r}/G(F)_{x,s}$. For a sequence of twisted Levi subgroups $\overrightarrow{G}=(G^0,\cdots,G^d=G)$, a sequence $\overrightarrow{r}=(r_0,\cdots,r_d)$ in $\widetilde{\mathbb{R}}$ and $x\in \mathcal{A}(G,S,F)$, we use the following notation for simplicity:
$$\overrightarrow{G}(F)_{x,\overrightarrow{r}}:=G^{0}(F)_{x,r_0}\cdots G^{d}(F)_{x,r_d}.$$

For an irreducible admissible representation $(\pi,V)$ of $G(F)$, let $\mathrm{depth}(\pi)$ be the depth of $\pi$, which is defined by
$$\mathrm{depth}(\pi):=\mathrm{inf}\{r\in \mathbb{R}_{\geq 0}|\,V^{G_{x,r+}}\neq 0 \text{ for some }x\in \mathcal{B}(G,F)\}.$$

We will also adopt the notations from \cite[Section 4.3]{Kal19b}. Let $S$ be an elliptic maximal torus of $G$ with $S\subset G^0\subset\cdots \subset G^d=G$. For any $\alpha\in \Phi(G,S)$ and $x\in \mathcal{B}(G,F)$, let $\mathrm{ord}_{x}(\alpha)$ denote the jumps of $\mathfrak{g}_{\alpha}$, that is,
$$\mathrm{ord}_{x}(\alpha):=\{r\in \mathbb{R}|\,\mathfrak{g}_{\alpha}(F_{\alpha})_{x,r+}\neq\mathfrak{g}_{\alpha}(F_{\alpha})_{x,r}\}.$$
Then $\mathrm{ord}_{x}(\alpha)$ is an $\frac{1}{e_\alpha}\mathbb{Z}$-torsor, where $e_{\alpha}$ denotes the ramification index of $F_\alpha/F$. Moreover, if $\alpha$ is symmetric, $\mathrm{ord}_{x}(\alpha)$ is either $\frac{1}{e_\alpha}\mathbb{Z}$ or $\frac{1}{e_\alpha}(\mathbb{Z}+\frac{1}{2})$.

For any $\alpha\in \Phi(G,S) $, there exists a unique $0\leq i\leq d$ denoted by $i_\alpha$, such that $\alpha\in \Phi(G^{i_\alpha},S)-\Phi(G^{i_\alpha-1},S)$, where we use the notation $G^{-1}:=S$. Let $r_i$ be a sequence of non-negative real numbers satisfying the condition in \Cref{Yuconstruction}. Let $\Phi_{r_{i}/2}$ denote the following subsets of roots
$$\Phi_{r_i/2}=\{\alpha\in \Phi(G^i,S)-\Phi(G^{i-1},S)|\,r_i\in 2\mathrm{ord}_{x}(\alpha)\}.$$

For any quadratic extension $E/F$, we use the notation $E_{E/F}^{1}$ for the set of norm-one elements inside $E^{\times}$. When there is no confusion about the subfield $F$, we simply write $E^{1}$. By Hilbert 90, there is an isomorphism between $E^1$ and $E^\times/F^\times$. Let $\sigma$ be the non-trivial Galois involution in $\mathrm{Gal}(E/F)$. We will adopt the notation $\iota$ for the isomorphism $\iota:E^1\rightarrow E^\times/F^\times$ whose inverse sends $x$ to $\frac{x}{\sigma(x)}$ for any $x\in E$.

For any locally compact abelian group $A$, we use the notation $A^*:=\mathrm{Hom}_{\rm{cts}}(A,\mathbb{C}^\times)$ for the $\mathbb{C}^\times$ dual, i.e., the group of one-dimensional continuous characters with values in $\mathbb{C}^\times$. Furthermore, if $A$ is finite, $A^*$ coincides with the Pontryagin dual of $A$, which is denoted by the same notation. For any finite cyclic group $A$ of even order, we use the notation $\mathrm{sgn}$ for the unique non-trivial quadratic sign character of $A$.

When $G$ is quasi-split over $F$ with a fixed Borel subgroup $B_0$ such that $S_0$ is a maximal torus of $B_0$ and $N_0$ is the unipotent radical of $B_0$, we use the notation $\mathcal{P}=(B_0,S_0,\{X_\alpha\}_{\alpha\in \Delta})$ for the pinning of $G$ determined by $\{X_\alpha\}_{\alpha\in \Delta}$, where $\{X_\alpha\}_{\alpha\in \Delta}$ is a collection of root vectors, one for each simple root determined by $(B_0,S_0)$. By a Whittaker datum $\mathfrak{w}$, we mean a $G(F)$-conjugacy class of pairs $(B_0,\psi_{N_0})$, where $\psi_{N_0}$ is a generic character of $N_0(F)$. Let $\psi:F\rightarrow\mathbb{C}^{\times}$ be a non-trivial additive character. Each choice of $\{X_{\alpha}\}_{\alpha\in \Delta}$ gives a homomorphism $N_0\rightarrow \prod\limits_{\alpha\in \Delta}\mathbb{G}_a$, and composing with the addition map $\Sigma:\prod\limits_{\alpha\in \Delta}\mathbb{G}_a\rightarrow \mathbb{G}_a$ yields a homomorphism $\Delta_{\Sigma}:N_0\rightarrow\mathbb{G}_a$, which is defined over $F$. Every generic character $\psi_{N_0}$ arises as a composition of $\psi$ and $\Delta_{\Sigma}:N_0(F)\rightarrow F$ for some choice of pinning $\mathcal{P}=(B_0,S_0,\{X_\alpha\}_{\alpha\in \Delta})$, and we also call $\mathfrak{w}=(B_0,\psi_{N_0})$ the Whittaker datum determined by the pinning $\mathcal{P}$. A representation $\pi$ of $G(F)$ is said to be generic if it is $(N_0,\psi_{N_0})$-distinguished for some Whittaker datum $(B_0,\psi_{N_0})$.

\subsection{Assumptions on $p$}
Since our work is based on the fundamental work of Kaletha \cite{Kal19b}, we also need to propose certain assumptions on the ground field as in \cite{Kal19b}. We list these assumptions for convenience of the readers. We assume the residual characteristic $p\neq 2$, $p$ is not a bad prime for $G$, $p \nmid|\pi_0(Z(G))|$ and $p\nmid |\pi_1(G_{\mathrm{der}})|$.

\section{Local Langlands correspondence for regular supercuspidal representations}\label{section3}
Throughout this section, we fix $G$ to be a quasi-split reductive group defined over a non-archimedean local field $F$. Let $\widehat{G}$ be the complex Langlands dual group of $G$ and let $\prescript{L}{}{G}=\widehat{G}\rtimes W_F$ be the Weil-form of the $L$-group of $G$.
\subsection{Langlands--Vogan bijection}
We first recall the conjectural local Langlands correspondence and Vogan's refinement.
\begin{definition}\label{lparameter}
A Langlands parameter $\phi$ of $G(F)$ is a homomorphism $\phi:WD_F\rightarrow \prescript{L}{}{G}=\widehat{G}\rtimes W_F$ such that
\begin{enumerate}[(1)]
\item $\phi|_{\mathrm{SL}_2(\mathbb{C})}:\mathrm{SL}_2(\mathbb{C})\rightarrow \widehat{G}$ is a morphism of algebraic groups over $\mathbb{C}$.
\item $\phi$ is continuous on $I_F$ and $\phi(\mathrm{Fr})$ is semi-simple in $\widehat{G}$.
\item $\mathrm{pr}_2\circ \phi:WD_F\rightarrow \prescript{L}{}{G}\rightarrow W_F$ is the natural projection.
\end{enumerate}
The homomorphism $\phi$ satisfying the above conditions is usually called an admissible homomorphism.
\end{definition}
By the third condition, for any $w\in WD_F$, we can write $\phi(w)=(\varphi(w),w)\in\prescript{L}{}{G} $, such that $\varphi(w)$ lies in $\widehat{G}$. Notice that $\phi$ is a homomorphism if and only if
$$\varphi(v)v(\varphi(w))=\varphi(vw),\quad \forall v,w\in WD_F,$$
which means that one can regard $\varphi$ as an element in $Z^{1}(WD_F,\widehat{G})$.
Let $\Pi(G)$ denote the set of equivalence classes of irreducible admissible representations of $G(F)$ and let $\Phi(G)$ denote the set of $\widehat{G}$-conjugacy classes of Langlands parameters of $G(F)$. The local Langlands correspondence for quasi-split groups predicts that there exists a surjective, finite to one map
$$\mathrm{LLC}:\Pi(G(F))\twoheadrightarrow \Phi(G(F)),$$
such that there exists a bijection between $\Pi_{\phi}(G(F))$ and $\mathrm{Irr}(\pi_0(S_{\phi}/Z(\widehat{G})^{\Gamma_F}))$ for any $\phi\in \Phi(G)$, where $\Pi_{\phi}(G(F))$ is the fiber of $\phi$ and $S_{\phi}$ is $C_{\widehat{G}}(\mathrm{Im}\phi)$. The set $\Pi_{\phi}(G(F))$ is called the $L$-packet associated to $\phi$. This bijection is not canonical and depends on a choice of Whittaker datum, under which the trivial representation of the component group corresponds to a generic representation. It is expected that LLC has nice properties, such as preserving local constants on both sides, being compatible with character twists on both sides, etc. We will list some of these properties in \Cref{Galoiscohomology}.

Later, Vogan \cite{Vog93} notices that one should consider the representation theory of pure inner twists of $G$ simultaneously in the local Langlands correspondence and he predicts that there should exist a bijection between $\bigsqcup\limits_{\alpha\in H^1(F,G)}\Pi_{\phi}(G_{\alpha}(F))$ and $\mathrm{Irr}(\pi_{0}(S_{\phi}))$ for any $\phi\in \Phi(G)$, such that the following diagram commutes
$$\xymatrix{
  \bigsqcup\limits_{\alpha\in H^1(F,G)}\Pi_{\phi}(G_{\alpha}(F)) \ar[d] \ar@{<->}[r]^-{1\mbox{-}1}
  & \mathrm{Irr}(\pi_{0}(S_{\phi})) \ar[d]^{\text{Taking central character}}\\
  H^1(F,G)  \ar@{<->}[r]^-{\mathrm{Kottwitz}}
  & \pi_{0}(Z(\widehat{G})^{\Gamma_F})^{*},}$$
where the bottom arrow is given by the Kottwitz isomorphism. Fixing a Whittaker datum of the quasi-split form $G$, one can associate a pair $(\phi_{\pi},\lambda_{\pi})$ to any irreducible representation $\pi$ of $G_{\alpha}(F)$, where $\phi\in \Phi(G)$ is the usual Langlands parameter and $\lambda_{\pi}$ is an irreducible representation of $\pi_{0}(S_{\phi})$. The pair $(\phi_{\pi},\lambda_{\pi})$ is usually called the Langlands--Vogan parameter (or enhanced Langlands parameter) of $\pi$.

It is expected that one can read the information of an irreducible representation from its Langlands parameter. We have the following definitions of various Langlands parameters.

\begin{definition}
Let $\phi:WD_F\rightarrow \prescript{L}{}{G}$ be a Langlands parameter of $G(F)$. It is called
\begin{enumerate}[(1)]
\item elliptic (discrete), if $\mathrm{Im}(\phi)$ is not contained in any proper Levi subgroup $\prescript{L}{}{L}$ of $\prescript{L}{}{G}$. (or equivalently $C_{\widehat{G}}(\mathrm{Im}\phi)$ is finite modulo $Z(\widehat{G})^{\Gamma_F}$.)
\item bounded, if the closure of $\varphi(W_F)$ in $\widehat{G}$ is compact.
\item unramified, if $\varphi(I_F)$ is trivial.
\item spherically unramified, if $\varphi(I_F\times \mathrm{SL}_2(\mathbb{C}))$ is trivial.
\item tamely ramified, if $\varphi(P_F)$ is trivial.
\item torally wild, if $\varphi(P_F)$ is contained in a torus of $\widehat{G}$.
\end{enumerate}
\end{definition}
These Langlands parameters are supposed to parameterize essentially discrete series, tempered, unipotent (unipotent reduction in the sense of Lusztig), unramified (spherical), depth zero and torally wild (essentially tame in the sense of Bushnell and Henniart for $\mathrm{GL}_n$) representations of $G(F)$.

The conjectural refined local Langlands correspondence has been verified for a large number of cases. The known cases are either for all irreducible representations of groups of special kinds such as $\mathrm{GL}_n$\cite{HT01, Hen00, Sch13}, quasi-split classical groups \cite{Art13, Mok15}, similitude groups of small ranks \cite{GT11} etc, or for particular representations of general classes of reductive groups \cite{DR09, RY14, Kal15, Kal19b, Sol23} etc. One can refer to \cite{Kal16a} for a more detailed description of the status of the conjecture.

\begin{remark}
As conjectured by Gross and Reeder \cite[Conjecture 7.1 (4)]{GR10}, for a discrete Langlands parameter, the condition $\phi|_{\mathrm{SL}_2(\mathbb{C})}=1$ should be equivalent to the condition that all the representations in $\Pi_{\phi}$ are supercuspidal representations. For these representations, the Langlands parameters are simply homomorphisms from $W_F$ to $^{L}G$ satisfying the conditions in \Cref{lparameter}.
\end{remark}	

\begin{definition}[{\cites[Definition 5.2.3]{Kal19b}}]
A regular supercuspidal parameter is a discrete Langlands parameter $\phi: W_F\rightarrow \prescript{L}{}{G}$ such that
\begin{enumerate}[(1)]
\item $\varphi(P_F)$ is contained in a torus of $\widehat{G}$.
\item $C:=C_{\widehat{G}}(\phi(I_F))^{\circ}$ is a torus. Let $\widehat{M}$ denote $C_{\widehat{G}}(\phi(P_F))^{\circ}$ and let $\widehat{T}$ denote $C_{\widehat{M}}(C)$ with the $W_F$-action on $\widehat{T}$ given by $\mathrm{Ad}(\phi(\cdot))$.
\item If $n\in N_{\widehat{M}}(\widehat{T})$ projects to a non-trivial element in $W(\widehat{M},\widehat{T})^{\Gamma_F}$, then $n\notin C_{\widehat{G}}(\phi(I_F))$.
\end{enumerate}
\end{definition}

\begin{remark}\label{bcregular}
Let $E/F$ be a quadratic extension and let $\phi:W_E\rightarrow \prescript{L}{}{G}$ be a regular supercuspidal Langlands parameter of $G(E)$. If $\tilde{\phi}$ is a Langlands parameter of $G(F)$ such that $\tilde{\phi}|_{W_E}$ is $\widehat{G}$-conjugate to $\phi$, then $\tilde{\phi}$ is also a regular supercuspidal parameter of $G(F)$. This is simply due to the fact that $p\neq 2$ implies $E/F$ is a tamely ramified extension, so that one has $P_E=P_F$ and $I_E\subset I_F$.
\end{remark}

\subsection{$L$-embedding and $\chi$-data}\label{Lembedding}

A crucial property of regular supercuspidal parameters is that they factor though the $L$-groups of certain elliptic maximal tori of $G$, which could be regarded as a special example of functoriality discussed in \Cref{Galoiscohomology}. Let $S$ be an arbitrary maximal torus of $G$. Langlands and Shelstad \cite{LS87} first realized that one could always extend a $\Gamma_F$-stable $\widehat{G}$-conjugacy class of embedding of dual groups $\widehat{j}:\widehat{S}\rightarrow \widehat{G}$ to an embedding of $L$-groups (Weil-form) $\prescript{L}{}{j}:\prescript{L}{}{S}=\widehat{S}\rtimes W_F\rightarrow \prescript{L}{}{G}=\widehat{G}\rtimes W_F$ with the help of some auxiliary data. We give a brief review about their constructions of $L$-embeddings using $\chi$-data. A more detailed description could be found in \cite{LS87,Tam16}.

Let $G$ be a quasi-split reductive group over $F$ and let $S\subset G$ be a maximal torus defined over $F$. We assume $S_{F^{s}}$ is contained in a Borel subgroup $B_{F^{s}}$ since $G_{F^{s}}$ is split. We may choose a $W_F$-invariant pinning $(\widehat{S}_0,\widehat{B}_0,\{\widehat{X_{\widehat{\alpha}}}\})$ of $\widehat{G}$ with $\widehat{j}(\widehat{S})\cong \widehat{S}_0$ such that $\widehat{j}$ maps the simple roots determined by $\widehat{B}$ to those determined by $\widehat{B}_0$. Notice that the map $\widehat{j}$ is not necessarily $W_F$-equivariant, we use subscripts $w_{\widehat{G}}$ and $w_{\widehat{S}}$ to denote the action on $\widehat{G}$ and $\widehat{S}$ for any $w\in W_F$.

\begin{definition}[{\cites[Section 2.6]{LS87}[6.2]{Tam16}}]
An admissible $L$-embedding $\prescript{L}{}{j}:\prescript{L}{}{S}\rightarrow \prescript{L}{}{G}$ is a group homomorphism satisfying
$$\prescript{L}{}{j}(t\rtimes w)=\widehat{j}(t)\varphi_{\prescript{L}{}{j}}(w)\rtimes w,$$
for some $\varphi_{\prescript{L}{}{j}}:W_F\rightarrow \widehat{G}$.
\end{definition}

For $\prescript{L}{}{j}$ to be a homomorphism, one must have
\begin{equation}\label{admissible embedding}
\begin{aligned}
\varphi_{\prescript{L}{}{j}}(v)v_{\widehat{G}}(\widehat{j}(t))\varphi_{\prescript{L}{}{j}}(v)^{-1}&=\widehat{j}(v_{\widehat{S}}(t))\\
\varphi_{\prescript{L}{}{j}}(vw)&=\varphi_{\prescript{L}{}{j}}(v)v_{\widehat{G}}(\varphi_{\prescript{L}{}{j}}(w)),
\end{aligned}
\end{equation}
for any $v,w\in W_F$. From the first equation, one can see that $\mathrm{Im}(\varphi_{\prescript{L}{}{j}})\subset N_{\widehat{G}}(\widehat{j}(\widehat{S}))$. From the second equation, one can find that $\phi=(\varphi_{\prescript{L}{}{j}},id):W_F\rightarrow \widehat{G}\rtimes W_F$ is a Langlands parameter of $G$.
By choosing a base point of an embedding $$\prescript{L}{}{j_0}:\prescript{L}{}{S}\rightarrow \prescript{L}{}{G},$$
one can obtain the following fact.

\begin{proposition}[{\cites{LS87}[Proposition 6.1]{Tam16}}]

The set of $\widehat{j}(\widehat{S})$-equivalence classes of $L$-embeddings
$\{\prescript{L}{}{j}:\prescript{L}{}{S}\rightarrow \prescript{L}{}{G}\}/\widehat{j}(\widehat{S})$ is a torsor under $H^1(W_F,\widehat{S})$. The explicit bijection with respect to a base point $\prescript{L}{}{j_0}$ is given by
\begin{equation}\label{embeddingcharacter}
\begin{aligned}
\{\prescript{L}{}{j}:\prescript{L}{}{S}\rightarrow \prescript{L}{}{G}\}/\widehat{j}(\widehat{S})&\leftrightarrow H^1(W_F,\widehat{S})\\
\prescript{L}{}{j}&\mapsto w\mapsto \varphi_{\prescript{L}{}{j}}(w)\cdot \varphi_{\prescript{L}{}{j_0}}(w)^{-1}.
\end{aligned}
\end{equation}
Conversely, for any given $\prescript{L}{}{j}$ and $\theta\in Z^1(W_F,\widehat{S})$, one can associate an equivalence class of $L$-embeddings:
\begin{equation}\label{twistembedding}
\begin{aligned}
\prescript{L}{}{j}^{\theta}:\prescript{L}{}{S}&\rightarrow \prescript{L}{}{G}\\
t\rtimes w&\mapsto \widehat{j}(t\theta(w))\varphi_{\prescript{L}{}{j}}(w)\rtimes w.
\end{aligned}
\end{equation}
\end{proposition}

\begin{example}[{\cites[Proposition 6.2]{Tam16}}]

For $\mathrm{GL}_n$ with an elliptic maximal torus $S=\mathrm{Res}_{L/F}\mathbb{G}_{m,L}$ such that $L/F$ is a degree $n$ extension, one can choose the standard pinning of $\widehat{G}$ such that $\widehat{S}_0\subset \widehat{G}$ is the diagonal torus. One can choose a $\widehat{j}$ such that $\widehat{j}(\widehat{S})= \widehat{S}_0$. There is a naive base point $\varphi_0$ in $Z^1(W_F,N_{\widehat{G}}(\widehat{S}_0))$ such that
\begin{equation}
\begin{aligned}
\varphi_0:W_F&\rightarrow N_{\widehat{G}}(\widehat{S}_0) \\
w&\mapsto N(w),
\end{aligned}
\end{equation}
with
$$N(w)\widehat{j}(t) N(w)^{-1}=\widehat{j}(w_{\widehat{S}}(t)),$$
for any $t\in \widehat{S}$, and $N(w)$ is given by a permutation matrix in $\mathrm{GL}_n(\mathbb{C})$ with entries being $0$ and $1$ determined by the above relation. If we fix a character $\theta:L^{\times}\rightarrow \mathbb{C}^{\times}$ and $\theta'\in H^{1}(W_F,\widehat{S})$ which corresponds to an equivalence class of $L$-embedding in \eqref{embeddingcharacter} with respect to the above base point $\varphi_0$, the composition gives the Langlands parameter of $\mathrm{GL}_n$:
\begin{equation}
\begin{aligned}
H^1(W_F,\widehat{S})\times \{\prescript{L}{}{j}:\prescript{L}{}{S}\rightarrow \prescript{L}{}{G}\}_{\varphi_0}/\widehat{S}_0&\rightarrow H^1(W_F,\widehat{G})\\
(\theta,\theta')&\mapsto \mathrm{Ind}_{W_L}^{W_F}(\theta\cdot \theta').
\end{aligned}
\end{equation}

Surprisingly, the base point $\varphi_0$ is not the ``correct" one in the sense that the map $H^1(W_F,\widehat{S})\stackrel{\varphi_0}{\rightarrow}H^1(W_F,\widehat{G})$ sending $(L^\times,\theta)$ to $\mathrm{Ind}_{W_L}^{W_F}\theta$ does not realize the correct local Langlands correspondence for $\mathrm{GL}_n$. Indeed, Bushnell and Henniart \cite{BH10} introduce a certain rectifying character $\mu_\theta$ (depending on $\theta$) in their construction of the essentially tame local Langlands for $\mathrm{GL}_n$, which is interpreted in \cite[Theorem 7.1]{Tam16} as choosing an admissible $L$-embedding for each $\theta$ via the theory of $\chi$-data introduced by Langlands and Shelstad \cite{LS87}.

\end{example}

Indeed, the construction of Langlands and Shelstad works for general quasi-split reductive groups, and determines a base point of admissible $L$-embedding with the help of a family of auxiliary data known as $\chi$-data.

\begin{definition}[{\cites[(2.2)]{LS87}}]
A set of $a$-data for $S\subset G$ is a collection of elements $\{a_{\alpha}\in F_{\alpha}^{\times}\}_{\alpha\in \Phi(G,S)}$ satisfying
$$a_{-\alpha}=-a_{\alpha},\quad a_{\gamma(\alpha)}=\gamma(a_{\alpha})$$
for any $\gamma\in \Gamma_F.$
\end{definition}

\begin{definition}[{\cites[(2.5)]{LS87}}]\label{chidata}
A set of $\chi$-data for $S\subset G$ is a collection of elements $\{\chi_{\alpha}: F_{\alpha}^{\times}\rightarrow \mathbb{C}^{\times}\}_{\alpha\in \Phi(G,S)}$ satisfying
$$\chi_{-\alpha}=\chi_{\alpha}^{-1},\quad\chi_{\gamma(\alpha)}=\chi_{\alpha}\circ \gamma^{-1}$$
for any $\gamma\in \Gamma_F,$ and
$\chi_{\alpha}|_{F_{\pm\alpha}^{\times}}$ is the quadratic character attached to $F_{\alpha}/F_{\pm \alpha}$ for any symmetric root $\alpha\in \Phi_{\mathrm{sym}}(G,S)$.
\end{definition}

Let $W(\widehat{G},\widehat{S}_0)$ denote the Weyl group $N_{\widehat{G}}(\widehat{S}_0)/\widehat{S}_0$. One can define a section of
$$W(\widehat{G},\widehat{S}_0)\rightarrow N_{\widehat{G}}(\widehat{S}_0)$$
in the following way. Let $\widehat{\alpha^{\vee}}$ be a simple root of $\widehat{\mathfrak{g}}$ with a root vector $X_{\widehat{\alpha}}\subset \widehat{\mathfrak{g}}$. Let $\{H_{\widehat{\alpha}},X_{\widehat{\alpha}},X_{-\widehat{\alpha}}\}$ be the corresponding $\mathfrak{sl}_2$-triple. Let $s_{\alpha}\in W(\widehat{G},\widehat{S}_0)$ denote the simple reflection associated to $\alpha$. We set
$$n(s_{\alpha}):=\exp (X_{\widehat{\alpha}})\exp (-X_{-\widehat{\alpha}})\exp (X_{\widehat{\alpha}}),$$
and set $n(s)=\prod n(s_\alpha)$ for $s=\prod s_{\alpha}$ with $n(1)=1$. This section gives a map
\begin{equation*}
\begin{aligned}
n:W_F&\rightarrow N_{\widehat{G}}(\widehat{S}_0)\rtimes W_F,\\
w&\mapsto n(w)=n(s_{w})\rtimes w,
\end{aligned}
\end{equation*}
where $n(s_{w})$ is the section of $s_{w}$ in $N_{\widehat{G}}(\widehat{S}_0)$ and the action of $s_w\in W(\widehat{G},\widehat{S}_0)$ is given by the difference of the $W_F$-action on $\widehat{S}_0$ and $W_F$-action on $\widehat{S}$, that is,
$$s_{w}\cdot w_{\widehat{G}}(\widehat{j}(t))=\widehat{j}(w_{\widehat{S}}(t)).$$

Notice that the map $n$ is not necessarily a homomorphism of groups. Hence to get a base point, Langlands and Shelstad \cite[(2.1)]{LS87} construct a 2-cocycle
$$t(v,w):=n(v)n(w)n(vw)^{-1}\in Z^2(W_F,\widehat{S}_0).$$
One needs to find a splitting $r_{\chi}^{-1}:W_F\rightarrow \widehat{S}$ of $t(v,w)$, that is, the map $r_\chi$ should satisfy:
\begin{equation}\label{**}
\widehat{j}(r^{-1}(v)v_{\widehat{S}}(r^{-1}(w))r^{-1}(vw)^{-1})=t(v,w).
\end{equation}

Such splittings are constructed by a family of $\chi$-data. We recall the explicit construction of the cochain $r_{\chi}:W_F\rightarrow \widehat{S}$. (In fact, the explicit realization of $r_{\chi}$ involves certain choices of coset representatives). Fixing a set of representatives $g_i$ of the left coset $W_{F_{\pm\alpha}}\backslash W_F\cong \Gamma_{F_{\pm\alpha}}\backslash \Gamma_F$, we can define a family of functions $u_{g_{i}}:W_F\rightarrow W_{F_{\pm\alpha}}$ by the following equations:
$$g_{i}w=u_{g_{i}}(w)g_{i'},$$
for suitable $g_{i'}$. One can also define $v_{0}:W_{F_{\pm\alpha}}\rightarrow W_{F_{\alpha}}$ similarly by fixing the coset representatives $\{1,s\}$ of $W_{F_{\alpha}}\backslash W_{F_{\pm\alpha}}$. More explicitly, for $w\in W_{F_{\alpha}}$, $v_{0}(w)=1$, and for $w\in W_{F_{\pm\alpha}}-W_{F_{\alpha}}$, $v_{0}(w)=ws^{-1}$.

We can define a cochain $r_{\chi}:W_F\rightarrow \widehat{S}$ by
\begin{equation}\label{basepoint}
r_{\chi}(w):=\prod_{\alpha\in \Phi/\Gamma\times\{\pm 1\}\atop g_{i}\in \Gamma_{F_{\pm\alpha}}\backslash \Gamma_F}\chi_{\alpha}(v_0(u_{g_{i}}(w)))^{\widehat{g_{i}^{-1}\alpha}}.
\end{equation}
Langlands and Shelstad prove that such $r_{\chi}^{-1}$ is a splitting of $t(v,w)$. A priori, the explicit construction of $r_{\chi}$ depends on a choice of certain coset representatives, which has a more elegant interpretation via the notion of gauge introduced in \cite{LS87}.

\begin{definition}[{\cites[Lemma 2.1.A]{LS87}}]
A gauge is a function $p:\Phi(G,S)\rightarrow \{\pm 1\}$ such that $p(-\alpha)=-p(\alpha)$ for any $\alpha\in \Phi(G,S)$.
\end{definition}

A set of coset representatives $g_i$ of $\Gamma_{F_{\pm\alpha}}\backslash\Gamma_F$ canonically determines a gauge by $p(g_{i}^{-1}\alpha)=+1$. For two different gauges $p$ and $q$, one can associate a cochain $s_{p/q}\in C^1(W_F,\widehat{S})$ as in \cite[Lemma 2.4]{LS87}. The map $r_{\chi,p}$ and $r_{\chi,q}$ are related by $r_{\chi,q}=s_{q/p}r_{\chi,p}$.

Given a cochain $r_{\chi,p}$ associated to a set of $\chi$-data $\chi$ and a fixed gauge $p$, we can extend the embedding $\widehat{j}$ to $\prescript{L}{}{j}_{\chi}$ by defining:
\begin{equation}\label{extensionofl}
1\rtimes w\mapsto \widehat{j}(r_{\chi}(w))n(s_w)\rtimes w.
\end{equation}

\begin{theorem}
The above map in \eqref{extensionofl} is a group homomorphism.
\end{theorem}
\begin{proof}
This is essentially contained in \cite{LS87}. We write down the explicit computation for completeness. To check that the map defines a homomorphism, one needs to check the following identity:
$$\widehat{j}(r(v))n(s_v)v_{\widehat{G}}(\widehat{j}(r(w))n(s_w))=\widehat{j}(r(vw))n(s_{vw}).$$
By \Cref{admissible embedding}, we know
$$n(s_v)v_{\widehat{G}}\widehat{j}(r(w))=\widehat{j}(v_{\widehat{T}}(r(w)))n(s_v).$$
We can also compute $t(v,w)$ explicitly by
$$(t(v,w)\rtimes 1)(n(s_{vw})\rtimes vw)=(n(s_v)\rtimes v)(n(s_w)\rtimes w)=(n(s_v)v_{\widehat{G}}(n(s_w))\rtimes vw).$$
Hence we only need to prove
$$\widehat{j}(r(v))\widehat{j}(v_{\widehat{T}}(r(w)))n(s_v)v_{\widehat{G}}(n(s_w))=\widehat{j}(r(vw))n(s_{vw}),$$
which is exactly given by \Cref{**}.
\end{proof}

In fact, the choice of $\chi$-data is not unique. The difference between two families of $\chi$-data provides a character of $S(F)$ in the following way.
\begin{definition}[{\cites[Definition 4.6.4]{Kal19b}}]\label{zetadata}
A set of $\zeta$-data for $S\subset G$ is a collection of elements $\{\zeta_{\alpha}: F_{\alpha}^{\times}\rightarrow \mathbb{C}^{\times}\}_{\alpha\in \Phi(G,S)}$ satisfying
$$\zeta_{-\alpha}=\zeta_{\alpha}^{-1},\quad\zeta_{\gamma(\alpha)}=\zeta_{\alpha}\circ \gamma^{-1}$$
for any $\gamma\in \Gamma_F,$ and
$\zeta_{\alpha}|_{F_{\pm\alpha}^{\times}}=\mathbbm{1}$ for any symmetric root $\alpha\in \Phi(G,S)_{\mathrm{sym}}$ .
\end{definition}

Following \cite[Lemma 3.3.D]{LS87} and \cite[Definition 4.6.5]{Kal19b}, for any set of $\zeta$-data, one can construct a character

\begin{equation}\label{zetacharacter}
\begin{aligned}
\zeta_{S(F),\zeta}:S(F)&\rightarrow \mathbb{C}^{\times}\\
t&\mapsto \prod_{\alpha\in \Phi_{\mathrm{sym}}/\Gamma_F}\zeta_{\alpha}(\iota_{\alpha} \alpha(t))\prod_{\alpha\in \Phi_{\mathrm{asy}}/\Gamma_F\times\{\pm1\}}\zeta_{\alpha}(\alpha(t)),
\end{aligned}
\end{equation}
where $\iota_{\alpha}$ denotes the isomorphism $\iota_{\alpha}: F_{\alpha,F_{\alpha}/F_{\pm\alpha}}^{1}\rightarrow F_{\alpha}^{\times}/F_{\pm\alpha}^{\times}$.

In particular, for two sets of $\chi$-data $\{\chi_{\alpha}\}_\alpha$ and $\{\chi_{\alpha}'\}_{\alpha}$, it is clear from definition that $\{\frac{\chi_{\alpha}}{\chi_{\alpha}'}\}_\alpha$ is a set of $\zeta$-data. In this case, the character constructed above is denoted by $\zeta_{S(F),\chi,\chi'}$. By the notation in \eqref{twistembedding}, the $L$-embeddings associated to $\{\chi_\alpha\}$ and $\{\chi'_\alpha\}$ are related by
$$\prescript{L}{}{j}_{\chi}=\prescript{L}{}{j}_{\chi'}^{\zeta_{S(F),\chi,\chi'}}.$$
We will also use the simplified notation $\zeta_{S(F)}$ for the character $\zeta_{S(F),\zeta}$, if the corresponding $\zeta$-data $\zeta$ is clear from the context.

Among the set of $\chi$-data, the following ones are crucially used in computing the characters of supercuspidal representations.
\begin{definition}[{\cites[Definition 4.6.1]{Kal19b}}]
A set of $\chi$-data: $\{\chi_{\alpha}\}_{\alpha}$ is called minimally ramified if
\begin{enumerate}[1]
\item $\chi_{\alpha}=\mathbbm{1}$, if $\alpha$ is asymmetric,
\item $\chi_{\alpha}$ is unramified, if $\alpha$ is symmetric unramified,
\item $\chi_{\alpha}$ is tamely ramified, if $\alpha$ is symmetric ramified.
\end{enumerate}
\end{definition}

Kaletha also introduced the following notion of mod-$a$-data, which can be better used in conjunction with tamely ramified $\chi$-data.

\begin{definition}[{\cites[Definition 4.6.8]{Kal19b}}]
A set of mod-$a$-data for $S\subset G$ is a collection of elements $\{(r_{\alpha},\overline{a}_{\alpha})| r_{\alpha}\in \mathbb{R},\overline{a}_{\alpha}\in [F_{\alpha}]_{r_{\alpha}}/[F_{\alpha}]_{r_{\alpha}+}\}_{\alpha\in \Phi(G,S)}$ satisfying
$$r_{\gamma(\alpha)}=r_{\alpha}=r_{-\alpha},\quad\overline{a}_{-\alpha}=-\overline{a}_{\alpha},\quad\overline{a}_{\gamma(\alpha)}=\gamma(\overline{a}_{\alpha})$$
for any $\gamma\in \Gamma_F.$
\end{definition}

For regular supercuspidal representations associated to a tame elliptic pair, Kaletha \cite[4.7.2, 4.7.3]{Kal19b} describes a canonical choice of minimally ramified $\chi$-data associated to the mod-$a$-data determined by the character $\mu:S(F)\rightarrow \mathbb{C}^{\times}$. We emphasis that those canonical minimally ramified $\chi$-data we will use are the ones (denoted by $\chi'$) in \cite[4.7.2]{Kal19b} instead of the ones in \cite[4.7.1]{Kal19b}. More precisely, these $\chi$-data satisfy the following conditions.

\begin{enumerate}
\item If $\alpha$ is asymmetric, $\chi_{\alpha}=\mathbbm{1}$.
\item If $\alpha$ is symmetric unramified, $\chi_{\alpha}$ is the quadratic character associated to the unique unramified quadratic extension of $F_{\alpha}$ by local class field theory.
\item If $\alpha$ is symmetric ramified, $\chi_{\mu,\alpha}$ is the tamely ramified character determined by
\begin{equation*}
	\begin{aligned}
		\chi_{\mu,\alpha}(2a_{\alpha})=\lambda_{F_{\alpha}/F_{\pm\alpha}},
	\end{aligned}
\end{equation*}
where $\lambda_{F_{\alpha}/F_{\pm\alpha}}$ is the $\lambda$-constant defined by Langlands. Here $a_{\alpha}$ is any lift in $[F_{\alpha}]_{-r_\alpha}$ of the mod-$a$-data $\overline{a_{\alpha}}$ in $[F_{\alpha}]_{-r_\alpha}/[F_{\alpha}]_{-r_\alpha+}$ determined by
\begin{equation}\label{symmetricramified}
	\begin{aligned}
		\phi_{i}(\mathrm{Nm}_{S(F_{\alpha})/S(F)}\alpha^{\vee}(1+X))=\psi_{F_{\alpha}}(\overline{a_{\alpha}}\cdot X)=\psi_{F}(\mathrm{tr}_{F_{\alpha}/F}(\overline{a_{\alpha}}\cdot X)),
	\end{aligned}
\end{equation}
for any $\alpha\in \Phi(S,G^{i+1})-\Phi(S,G^{i})$, $X\in [F_{\alpha}]_{r_\alpha}/[F_{\alpha}]_{r_\alpha+}$ and a Howe factorization $\mu=\mu_{-1}\cdot\prod\limits_{i=0}^{d} \phi_{i}|_{S(F)}$ in \Cref{howefactorization}. Moreover, the $\lambda$-constant has the following expression in terms of Deligne's local root numbers
$$\lambda_{F_{\alpha}/F_{\pm\alpha}}=\frac{\epsilon(\mathrm{Ind}_{\Gamma_{F_{\alpha}}}^{\Gamma_{F_{\pm\alpha}}}\mathbbm{1}_{\Gamma_{F_{\alpha}}},\frac{1}{2},\psi_{F_{\pm\alpha}})}
{\epsilon(\mathbbm{1}_{\Gamma_{F_{\alpha}}},\frac{1}{2},\psi_{F_{\alpha}})},$$
where $\psi_{F_{\pm\alpha}}$ denotes the additive character $\psi_F\circ \mathrm{tr}_{F_{\pm\alpha}/F}:F_{\pm\alpha}\rightarrow \mathbb{C}^{\times}$ and $\psi_{F_{\alpha}}$ similarly.

\end{enumerate}

The construction of \eqref{extensionofl} depends heavily on the choice of the $\chi$-data, which is not unique. For the case of $\mathrm{GL}_n$, Tam \cite{Tam16} also constructs certain $\chi$-data in order to understand the rectifiers defined by Bushnell and Henniart \cite{BH10}. The comparison of Tam's $\chi$-data and Kaletha's canonical minimally ramified $\chi$-data is carefully studied for regular supercuspidal representations of $\mathrm{GL}_n$ in \cite{OT21}, which also leads to a comparison of the local Langlands correspondences for these representations established by \cite{HT01, Hen00} and \cite{Kal19b}. Throughout this article, we work with Kaletha's canonical minimally ramified $\chi$-data.

Recently, Kaletha \cite{Kal19a} gives a new interpretation of the work of Langlands and Shelstad by introducing certain double cover of the elliptic torus, where the base point of the $L$-embedding associated to the double cover corresponds to a genuine character of the double cover.

\subsection{Kaletha's parametrization of regular supercuspidal representations}

Based on the earlier work of Langlands and Shelstad \cite{LS87}, Kaletha \cite{Kal19b} uses another type of data to describe these regular supercuspidal Langlands parameters, which he calls regular supercuspidal $L$-packet data. He identifies the set of equivalence classes of regular supercuspidal parameters with the set of equivalence classes of regular supercuspidal $L$-packet data, which forms a category. We give a brief review of his construction.

\begin{definition}[{\cites[Definition 3.7.5]{Kal19b}}]\label{tameellipticpair}

Let $S\subset G$ be a maximal torus and let $\mu:S(F)\rightarrow \mathbb{C}^{\times}$ be a character. The pair $(S,\mu)$ is called tame elliptic regular if
\begin{enumerate}
\item $S$ is elliptic and splits over a tamely ramified extension $L$;
\item the action of $I_F$ on the root system
\begin{equation*}
\begin{aligned}
\Phi_{0^{+}}:=\{\alpha\in \Phi(S,G)|\,\mu(\mathrm{Nm}_{S(L)/S(F)}(\alpha^{\vee}(L_{0^{+}}^{\times})))=1\}
\end{aligned}
\end{equation*}
preserves a positive set of roots;
\item the character $\mu|_{S(F)_{0}}$ has a trivial stabilizer for the action of $N_{G^0}(S)(F)/S(F)$, where $G^{0}$ is a reductive subgroup of $G$ with maximal torus $S$ and root system $\Phi_{0^{+}}$.
\end{enumerate}
\end{definition}

\begin{definition}[{\cites[Definition 5.2.4, 5.2.5]{Kal19b}}](Category of regular supercuspidal $L$-packet data.)\label{equivalenceclass}

An object in the category consists of a tuple
$(S,\widehat{j},\chi,\mu)$, where
\begin{enumerate}[(1)]
\item $S$ is an elliptic torus defined over $F$ such that $\dim S=\mathrm{rank}_{\overline{F}}G$ and $S$ splits over a tamely ramified extension of $F$;
\item $\widehat{j}:\widehat{S}\rightarrow \widehat{G}$ is an embedding of complex reductive groups whose $\widehat{G}$ conjugacy class is $\Gamma_F$-stable;
\item $\chi$ is minimally ramified $\chi$-data for $\Phi(G,S)$;
\item $\mu:S(F)\rightarrow \mathbb{C}^{\times}$ is a character, such that $(S,\mu)$ is a tame elliptic regular pair in the sense of \Cref{tameellipticpair}.
\end{enumerate}

A morphism between $(S_1,\widehat{j_1},\chi_1,\mu_1)$ and $(S_2,\widehat{j_2},\chi_2,\mu_2)$ is a triple $(\iota,g,\zeta)$, where
\begin{enumerate}[(1)]
\item $\iota:S_1\rightarrow S_2$ is an isomorphism of $F$-tori inducing an isomorphism of complex torus $\widehat{\iota}:\widehat{S_2}\rightarrow \widehat{S_1}$;
\item $g$ is an element in $\widehat{G}$ such that $\widehat{j_1}\circ\widehat{\iota}=\mathrm{Ad}(g)\circ\widehat{j_2}$;
\item $\zeta=\{\zeta_{\alpha_2}\}_{\alpha_{2}\in \Phi(G,S_2)}$ is a set of $\zeta$-data as in \Cref{zetadata} such that $\chi_{1,\alpha_2\circ\iota}=\chi_{2,\alpha_2}\cdot\zeta_{\alpha_2}$, and $\zeta_{S_2}^{-1}\cdot(\mu_2\circ\iota)=\mu_1$.
\end{enumerate}
\end{definition}

\begin{remark}
By Kaletha \cite[5.1]{Kal19b}, a $\Gamma_F$-stable $G(\overline{F})$ conjugacy class $J$ of embeddings $j:S_{\overline{F}}\rightarrow G_{\overline{F}}$ defined over $\overline{F}$ determines a $\Gamma_F$-stable $\widehat{G}$ conjugacy class $\widehat{J}$ of embeddings $\widehat{j}:\widehat{S}\rightarrow \widehat{G}$ defined over $\mathbb{C}$, and vice versa. Here, the structure of $S$ to be a maximal torus of $G$ is given by $J$, which is determined by $\widehat{J}$, and is not necessarily defined over $F$. Elements in $J$ are usually called admissible embeddings, and the subset of $\Gamma_F$-fixed points in $J$ corresponds to the subset of admissible embeddings defined over $F$, which is non empty due to \cite[Corollary 2.2]{Kot82}. As remarked by Kaletha, if we choose a $\Gamma_F$-invariant pinning $(S_0,B_0,\{X_{\alpha}\})$ of $G$, one can always choose $j\in J$ such that $j:S\rightarrow S_0$ is an isomorphism over $\overline{F}$, and pull back $\Phi(G,S_0)$ via $j$ to get a $\Gamma_F$-invariant subset $\Phi(G,S)\subset X^{*}_{F^{s}}(S)$. Here the notion of tame regular elliptic pair makes sense if one uses the embedding $j$ defined over $\overline{F}$ determined by $\widehat{j}$. The notion of symmetric roots or asymmetric roots also makes sense with the help of $j$.
\end{remark}

\begin{theorem}[{\cites[Proposition 5.2.7]{Kal19b}}]\label{regularLpacket}

There is a natural 1-1 correspondence between the $\widehat{G}$-conjugacy classes of regular supercuspidal parameters and the isomorphism classes of regular supercuspidal $L$-packet data.
\end{theorem}

For a given $L$-packet datum $(S,\widehat{j},\chi,\mu)$ (or equivalently, a given regular supercuspidal Langlands parameter $\phi_{(S,\widehat{j},\chi,\mu)}$ via the correspondence in \Cref{regularLpacket}), Kaletha also introduces another category to parameterize supercuspidal representations in the $L$-packet $\Pi_{\phi_{(S,\widehat{j},\chi,\mu)}}$ (or $\Pi_{(S,\widehat{j},\chi,\theta)}$ for short), which he calls the category of regular supercuspidal data. The objects in the category are formulated in terms of Kaletha's rigid inner twists of a quasi-split reductive group. Since Prasad's original conjecture is formulated in terms of pure inner twists, we only use the pure inner twist version of Kaletha's formulation. For the relationship between rigid inner forms and pure inner forms and related notations about rigid inner forms, one can refer to Kaletha's paper \cite{Kal16b}. Roughly speaking, for a finite central subgroup $Z$, the isomorphism classes of rigid inner twists of $G$ realized by $Z$ is parametrized by $H^1(u\rightarrow W,Z\rightarrow G)$, where $u$ is a certain pro-algebraic multiplicative group defined in \cite[Section 3.1]{Kal16b} and $W$ is the extension of $\Gamma_F$ by $u$ determined by the topological generator $-1\in H^{2}(\Gamma_F,u)\cong \widehat{\mathbb{Z}}$ as in \cite[Section 3.2]{Kal16b}. Then isomorphism classes of pure inner twists correspond to the case of $Z=\{1\}$, so that one has an isomorphism $H^1(u\rightarrow W,1\rightarrow G)\cong H^1(F,G)$ by \cite[3.5]{Kal16b}.

\begin{definition}[{\cites[Definition 5.3.2, 5.3.3]{Kal19b}}](Category of regular supercuspidal data with $Z=\{1\}$.)

An object in the category consists of a tuple
$(S,\widehat{j},\chi,\mu,(G_{\alpha},\xi_{\alpha}),j)$, where
\begin{enumerate}[(1)]
\item $(S,\widehat{j},\chi,\mu)$ is a regular supercuspidal $L$-packet datum;
\item $(G_{\alpha},\xi_{\alpha})$ is a pure inner twist of $G$, that is, $\xi_{\alpha}:G\rightarrow G_{\alpha}$ is an isomorphism over $F^{s}$ such that $\{\gamma\mapsto\xi_{\alpha}^{-1}\gamma(\xi_{\alpha}) =\mathrm{Ad}(\alpha(\gamma))\}\in \mathrm{Im}(Z^1(F,G)\rightarrow Z^1(F,\mathrm{Inn}(G))$ for some $\alpha\in Z^1(F,G)$;
\item $j:S\rightarrow G_{\alpha}$ is an admissible embedding defined over $F$.
\end{enumerate}

A morphism between $(S_1,\widehat{j_1},\chi_1,\mu_1,(G_{\alpha_1},\xi_{\alpha_2}),j_1)$ and $(S_2,\widehat{j_2},\chi_2,\mu_2,(G_{\alpha_2},\xi_{\alpha_2}),j_2)$ is a tuple $(\iota,g,\zeta,f)$, where
\begin{enumerate}[(1)]
\item $(\iota,g,\zeta)$ is an isomorphism of regular supercuspidal $L$-packet datum;
\item $f:(G_{\alpha_1},\xi_{\alpha_2})\rightarrow (G_{\alpha_2},\xi_{\alpha_2})$ is an isomorphism between pure inner twists of $G$, such that $j_2\circ\iota=f\circ j_1$.
\end{enumerate}
\end{definition}

There is a natural forgetful functor from the category of regular supercuspidal data to the category of regular supercuspidal $L$-packet data. In this way, the $L$-packet $\Pi_{(S,\widehat{j},\chi,\theta)}$ as a set is in bijection with a torsor under $H^1(F,S)$. In fact, one also wants to understand the construction of each member of a given regular supercuspidal $L$-packet data. Kaletha \cite{Kal19b} constructs a generic cuspidal $G$-datum from a tame elliptic regular pair, which produces a regular supercuspidal representation by Yu's construction \cite{Yu01}. We will describe this explicit construction later in \Cref{Yuconstruction}.

\begin{remark}
From a regular supercuspidal datum $(S,\widehat{j},\chi,\mu,(G_{\alpha},\xi_{\alpha}),j)$, one can construct a tame elliptic pair $(j(S),\mu\circ j^{-1}\cdot \epsilon_{\mathrm{Kal}})$, where $\epsilon_{\mathrm{Kal}}:j(S)(F)\rightarrow \mathbb{C}^{\times}$ is known as a certain rectifying character arising in the earlier work of Debacker--Spice \cite{DS18} and Kaletha \cite{Kal15}, which generalizes the well-known rectifying character appearing in the work of Bushnell and Henniart \cite{BH05} on the construction of the local Langlands correspondence for $\mathrm{GL}_n$. Then, one can run the process described in \Cref{Yuconstruction} to get a regular supercuspidal representation $\pi_{(j(S),\mu\circ j^{-1}\cdot \epsilon_{\mathrm{Kal}})}$ of $G_{\alpha}(F)$ from this tame elliptic pair.
\end{remark}

More precisely, the character $\epsilon_{\mathrm{Kal}}:j(S)(F)\rightarrow \mathbb{C}^{\times}$ has the following explicit expression.
\begin{definition}[{\cites[4.3.3, Lemma 4.7.4]{Kal19b}}]\label{Kalcharacter}
For $t\in j(S)(F)$,
$$\epsilon_{\mathrm{Kal}}(t):=\prod\limits_{\substack{\alpha\in \Phi_{\mathrm{asy}}/\Gamma_{F}\times\{\pm 1\},\\ \frac{r_{i_\alpha}}{2}\in\mathrm{ord}_x{\alpha}}}\mathrm{sgn}_{k_{F_{\alpha}}^{\times}}\alpha(t) \prod\limits_{\substack{\alpha\in \Phi_{\mathrm{sym,ram}}/\Gamma_{F},\\ \alpha(t)\neq 1, \\ \mathrm{ord}(\alpha(t)-1)=0}}f_{(G,S)}(\alpha) \prod\limits_{\substack{\alpha\in \Phi_{\mathrm{sym,un}}/\Gamma_{F},\\ \frac{r_{i_\alpha}}{2}\in\mathrm{ord}_x{\alpha}}}\mathrm{sgn}_{k_{F_{\alpha}}^{1}}\alpha(t),$$
where $\mathrm{sgn}_{k_{F_{\alpha}}^{\times}}$ and $\mathrm{sgn}_{k_{F_{\alpha}}^{1}}$ denote the unique non-trivial quadratic character of the cyclic group $k_{F_{\alpha}}^{\times}$ and $k_{F_{\alpha}}^{1}$, $f_{(G,S)}$ denotes the toral invariants defined in \cite{Kal15}, $i_\alpha$ denotes the non-negative integer associated to $\alpha$ in \Cref{notation} and $r_{i_\alpha}$ denotes the sequence of non-negative real numbers in Yu's construction in \Cref{Yuconstruction}. 
\end{definition}
In order to make the above definition reasonable, we need to add some extra explanation. Since $S$ is elliptic, we know that $\alpha(t)\in O_{F_\alpha}^
\times$. In the first product, we project $\alpha(t)\in O_{F_\alpha}^
\times$ to $k_{F_\alpha}^\times$. In the third product, we known that the $F_\alpha/F_{\pm\alpha}$-norm of the element $\alpha(t)\in O_{F_\alpha}^
\times$ is trivial, since $\alpha$ is symmetric. Moreover if $\alpha$ is symmetric unramified, the projection of $\alpha(t)\in O_{F_\alpha}^
\times$ to $k_{F_\alpha}^\times$ lies in the group $k_{F_\alpha}^1$ of norm-one elements with respect to $k_{F_\alpha}/k_{F_{\pm\alpha}}$. More detailed descriptions of these characters will be given in \Cref{toralinvariant}.


\subsection{Comparison of the $L$-packet}\label{comparisonl}

Now there are two ways to parameterize the set $\bigsqcup\limits_{\alpha\in H^1(F,G)}\Pi_{\phi}(G_{\alpha}(F))$ for a regular supercuspidal Langlands parameter $\phi$. One is conjecturally given by $\mathrm{Irr}(\pi_0(S_\phi))$, while the other one is given by a torsor under $H^1(F,S)$. The first parametrization depends on certain normalization of Whittaker datum, while the second parametrization depends on a base point of the set of equivalence classes of rational embeddings. According to the strong tempered $L$-packet conjecture, there exists a unique $\mathfrak{w}$-generic representation in a given tempered $L$-packet for a fixed Whittaker datum $\mathfrak{w}$. Hence, one may expect that a fixed Whittaker datum canonically and uniquely determines a base point inside the set of equivalence classes of $F$-embeddings of abstract tori.

For a toral regular supercuspidal representation $\pi$ with regular supercuspidal $L$-packet data $(S,\widehat{j},\chi,\mu)$, the following lemma due to Kaletha implies that there exists a canonical choice of rational embedding as a base point in a given $L$-packet for any chosen Whittaker datum.
\begin{lemma}[{\cites[Lemma 6.2.2]{Kal19b}}]\label{toral}
Fix a Whittaker datum $\mathfrak{w}$ for $G$ associated to a fixed pinning $(S_0,B_0,\{X_{\alpha}\})$ defined over $F$ and an additive character $\psi:F\rightarrow \mathbb{C}^{\times}$. Let $(S,\widehat{j},\chi,\mu)$ be a toral $L$-packet datum of generic depth $r$. There exists a unique (up to $G(F)$-conjugacy) admissible rational embedding $j_{\mathfrak{w}}$, such that the representation corresponding to $(S,\widehat{j},\chi,\mu,(G,\mathrm{id}),j_{\mathfrak{w}})$ is $\mathfrak{w}$-generic.
\end{lemma}

In a recent paper of Fintzen--Kaletha--Spice \cite[Page 2273]{FKS23}, \Cref{toral} is generalized to arbitrary regular supercuspidal representations. After fixing the base point, then these two parameterizations coincide as a consequence of the Kottwitz isomorphism for tori and the following lemma due to Kaletha.
\begin{lemma}[{\cites[Lemma 5.3.4]{Kal19b}}]
The embedding $\hat{j}:\widehat{S}\rightarrow \widehat{G}$ induces an isomorphism $\widehat{S}^{\Gamma_F}\cong S_{\phi}$\label{special isomorphism}.
\end{lemma}

Under this isomorphism, $\mathrm{Irr}(\pi_0(S_\phi))\cong\pi_{0}(\widehat{S}^{\Gamma_F})^{*}$ is isomorphic to $H^1(F,S)$ by the Kottwitz isomorphism for tori. Fix a base point of $F$-embedding $j_0:S\rightarrow G$. Then $j_{\pi}$ could be regarded as an element in $H^1(F,S)$ by $\mathrm{inv}(j_{\pi},j_0)$ as defined in \Cref{invofembedding}, which we still denote by $j_\pi$ for simplicity.

\subsection{Cohomological aspects of local Langlands correspondence}\label{Galoiscohomology}

In the study of functorial properties of Langlands parameters, it is often convenient for us to ignore the Frobenius semisimplicity condition and regard $\Phi(G)$ as a subset of
$$\{\text{Admissible homomorphisms}:WD_F\rightarrow \prescript{L}{}{G}\}/\widehat{G}\mbox{-}\text{conjugate}.$$
Actually, there is a natural bijection
\begin{equation*}
\begin{aligned}
\{\text{Admissible homomorphisms}:WD_F\rightarrow \prescript{L}{}{G}\}/\widehat{G}\mbox{-}\text{conjugate}&\leftrightarrow H^1(WD_F,\widehat{G}),\\
\phi&\mapsto \varphi,
\end{aligned}
\end{equation*}
which provides a parametrization of the (enlarged) space of equivalence classes of Langlands parameters by $H^1(WD_{F},\widehat{G})$ as a pointed set. In this way, many functorial maps between Langlands parameters could be described in terms of natural maps between non-abelian Galois cohomology.

\subsubsection{Functoriality}
The principle of Langlands functoriality provides evidence that there may exist a deep relation between representations of different reductive groups. In a simple word, the Langlands functoriality predicts the existence of certain transfer of representations whenever there exists a homomorphism between their $L$-groups. More precisely, we assume that $H$ is also a reductive group defined over $F$, and there exists an algebraic homomorphism
$$\prescript{L}{}{f}:\prescript{L}{}{H}\rightarrow \prescript{L}{}{G},$$
which projects to the identity map on $W_F$. For any Langlands parameter $\phi_H$ of $H(F)$, we can get a Langlands parameter $\phi_{G,\prescript{L}{}{f}}$ of $G(F)$ by composing $\phi_H$ with $\prescript{L}{}{f}$:

$$\xymatrix{
                & WD_F \ar[dl]_{\phi_H}\ar[dr]^{\prescript{L}{}{f}\circ \phi_H}             \\
 \prescript{L}{}{H}  \ar[rr]^{\prescript{L}{}{f}} & &     \prescript{L}{}{G},}$$
which implies that $\prescript{L}{}{f}$ defines a map from the set of Langlands parameters of $H(F)$ to those of $G(F)$. Indeed, $\prescript{L}{}{f}$ descends to a map from $\Phi(H)$ to $\Phi(G)$, hence defines a transfer from $\Pi_{\phi_H}(H(F))$ to $\Pi_{\prescript{L}{}{f}\circ \phi_{H}}(G(F))$. In fact, if we ignore the Frobenius semisimplicity condition, $\prescript{L}{}{f}$ also defines a map
$$\prescript{L}{}{f}:H^{1}(WD_F,\widehat{H})\rightarrow H^1(WD_F,\widehat{G}).$$
If there exists a $W_F$-equivariant map
$$\widehat{f}:\widehat{H}\rightarrow \widehat{G},$$
between complex dual groups, then one can easily extend this homomorphism to a homomorphism of $L$-groups simply by
\begin{equation*}
\begin{aligned}
\prescript{L}{}{f}:\prescript{L}{}{H}&\rightarrow \prescript{L}{}{G},\\
(h,w)&\mapsto (\widehat{f}(h),w).
\end{aligned}
\end{equation*}

In this case, $\prescript{L}{}{f}$ should coincide with $\widehat{f}_{*}:H^{1}(WD_F,\widehat{H})\rightarrow H^1(WD_F,\widehat{G})$ induced by the $WD_F$-equivariant map $\widehat{f}$, which provides a Galois cohomological interpretation of the functoriality.  However, the map $\widehat{f}$ is not $W_F$-equivariant in general and the $L$-homomorphism $\prescript{L}{}{f}$ is usually not an extension of a $W_F$-equivariant $\widehat{f}$. The $L$-embedding  $\prescript{L}{}{j}:\prescript{L}{}{S}\rightarrow \prescript{L}{}{G}$ associated to an elliptic torus $S$ inside a split group $G$ as in \Cref{Lembedding} provides a typical example, in which $\prescript{L}{}{j}$ is not an extension of a $W_F$-equivariant embedding $\widehat{S}\hookrightarrow \widehat{G}$. Nonetheless, $\prescript{L}{}{f}:H^{1}(WD_F,\widehat{H})\rightarrow H^1(WD_F,\widehat{G})$ may still enjoy some nice cohomological properties as $\widehat{f}_{*}$, which will be seen soon.

\subsubsection{Twisting by characters and taking central characters}

The set of continuous characters of $G(F)$ inherits a natural structure of abelian group with the group structure given by pointwise multiplication. In this case, we also have a parametrization of this set in terms of abelian Galois cohomology.

\begin{proposition}[{\cites{Lan97}[Lemma A.1]{LM15}}]\label{characters}
There is a natural map
$$L:H^1(W_F,Z(\widehat{G}))\rightarrow \mathrm{Hom}_{\mathrm{cts}}(G(F),\mathbb{C}^{\times}),$$
which is a bijection when $G$ is quasi-split.
\end{proposition}

In fact, there is a convenient description of the above map using hypercohomology of crossed modules. This has been described in detail in \cites{Bor98, Kal15, Lab08}. Let $(d:G_0\rightarrow G_1)$ be a crossed module. More precisely, $d$ is a group homomorphism and $G_0$ is endowed with an action of $G_1$ such that
\begin{equation*}
\begin{aligned}
d(g\cdot h)&=g dh g^{-1},\\
d(h)\cdot h'&=hh'h^{-1}.
\end{aligned}
\end{equation*}
for any $h,h'\in G_0, g\in G_1$. Notice that the second equation implies that $\mathrm{ker}d$ is abelian. Let $\Gamma$ be a group acting on the crossed module such that the action is compatible with the crossed module structure, that is,
$$\gamma(g\cdot h)=\gamma(g)\cdot \gamma(h).$$
Then one has the following cohomology theory for crossed modules with the usual degree
\begin{equation*}
\begin{aligned}
H^{0}(\Gamma,G_0\rightarrow G_1)&:=(\mathrm{ker}d)^{\Gamma}\\
H^{1}(\Gamma,G_0\rightarrow G_1)&:=Z^1(\Gamma,G_0\rightarrow G_1)/B^1(\Gamma,G_0\rightarrow G_1)\\
H^{2}(\Gamma,G_0\rightarrow G_1)&:=Z^2(\Gamma,G_0\rightarrow G_1)/B^2(\Gamma,G_0\rightarrow G_1).
\end{aligned}
\end{equation*}
where the hypercocycles $Z^{i}:=Z^{i}(\Gamma,G_0\rightarrow G_1)$ and hypercoboundaries $B^{i}:=B^{i}(\Gamma,G_0\rightarrow G_1)$ are defined by
\begin{equation*}
\begin{aligned}
Z^{1}&:=\{(a,b)|\,a\in Z^{1}(\Gamma,G_0),b\in G_1,d(a(\gamma))=b^{-1}\gamma(b)\},\\
B^{1}&:=\{(\partial c,dc)\in Z^{1}|\, c\in G_0,\partial c(\gamma):=c^{-1}\gamma(c) \},\\
Z^{2}&:=\{(a,b)|\,a\in Z^{2}(\Gamma,G_0),b\in C^{1}(\Gamma,G_1),da=\partial b\},\\
B^{2}&:=\{(\partial c,dc)\in Z^2|\,c\in C^1(\Gamma,G_0)\}.
\end{aligned}
\end{equation*}
Remark that we also adopt the notations in \cite{Kal15} such that the degree in our notation is shifted by $1$ from the usual degree in the cohomology theory of crossed modules. The reason is explained in \cite[Page 53]{Kal15} for the compatibility between the cohomology of crossed module and hypercohomology of a complex concentrated in degree $0$ and $1$, when we take the crossed module to be a complex of tori. Using the above shifted notation, we have the following isomorphisms.

\begin{proposition}[{\cites[Proposition 5.19]{Kal15}}]\label{Kalinter1}
There are canonical isomorphisms
$$H^{1}(W_F,Z(\widehat{G}))\cong \mathrm{Hom}_{\mathrm{cts}}(H^1(F,G_{\mathrm{sc}}\rightarrow G),\mathbb{C}^{\times}),$$
and
$$H^2(W_F,\widehat{G}_{\mathrm{sc}}\rightarrow \widehat{G})\cong\mathrm{Hom}(Z(G)(F),\mathbb{C}^{\times}).$$
\end{proposition}

The map $L$ and the map (on the parameter side) of taking the central character of an irreducible representation correspond to the following natural maps between Galois cohomology.

\begin{proposition}[{\cites[Proposition 5.20]{Kal15}}]\label{Kalinter2}
The map $L$ is given by
$$H^{1}(W_F,Z(\widehat{G}))\cong \mathrm{Hom}_{\mathrm{cts}}(H^1(F,G_{\mathrm{sc}}\rightarrow G),\mathbb{C}^{\times})\rightarrow \mathrm{Hom}_{\mathrm{cts}}(G(F)/G_{\mathrm{sc}}(F),\mathbb{C}^{\times}).$$
The map of taking the central character is given by
$$H^{1}(W_F,\widehat{G})\cong H^{2}(W_F,1\rightarrow \widehat{G})\rightarrow H^2(W_F,\widehat{G}_{\mathrm{sc}}\rightarrow \widehat{G})\cong\mathrm{Hom}(Z(G)(F),\mathbb{C}^{\times}). $$
\end{proposition}

\begin{example}
For $G=\mathrm{GL}_n$, the map of taking the central character has a simple description. Consider the short exact sequence of $W_F$-modules
$$1\rightarrow \widehat{G}_{\mathrm{sc}}=\mathrm{SL}_{n}(\mathbb{C})\rightarrow \widehat{G}=\mathrm{GL}_n(\mathbb{C})\stackrel{\mathrm{det}}{\longrightarrow}\mathbb{C}^{\times}\rightarrow 1.$$
We have a quasi-isomorphism of crossed modules induced by the determinant map
$$[\widehat{G}_{\mathrm{sc}}\rightarrow \widehat{G}]\cong [1\rightarrow \mathbb{C}^{\times}],$$ and the following commutative diagram
$$\xymatrix@R=0.5cm{
                &         [\widehat{G}_{\mathrm{sc}}\rightarrow \widehat{G}] \ar[dd]^{\cong}     \\
  [1\rightarrow \widehat{G}] \ar[ur] \ar[dr]_{\mathrm{det}}                 \\
                &         [1\rightarrow \mathbb{C}^{\times}].}$$
Hence the map of taking the central character
$$H^{2}(W_F,1\rightarrow \widehat{G})\rightarrow H^2(W_F,\widehat{G}_{\mathrm{sc}}\rightarrow\widehat{G})$$
coincides with the map of taking determinant on the parameter side
$$H^{1}(W_F,\mathrm{GL}_n(\mathbb{C}))\stackrel{\mathrm{det}}{\longrightarrow}  H^1(W_F,\mathbb{C}^{\times}).$$

\end{example}

Indeed, for quasi-split $G$, the natural bijection in \Cref{characters} has nice functorial properties. To be more precise, we need to recall the notion of normal homomorphisms
introduced by Kottwitz \cite[Section 1.8]{Kot84}. Let $G_1,G_2$ be connected reductive groups over $F$. We will say that an $F$-homomorphism $f:G_1\rightarrow G_2$ is normal if $f(G_1)$ is normal in $G_2$ and $G_1\rightarrow f(G_1)$ is separable. Moreover, by choosing pinnings of ${G_1}_{F^s},{G_2}_{F^s}$, there is a unique $\Gamma_F$-equivariant homomorphism $\widehat{f}:\widehat{G_2}\rightarrow \widehat{G_1}$ such that the $\widehat{G}^{\Gamma_F}$-conjugacy class of $\widehat{f}$ is canonical. In particular, the induced homomorphism $\widehat{f}:Z(\widehat{G_2})\rightarrow Z(\widehat{G_1}) $ is canonical. 
\begin{proposition}[{\cites[Section 1.8]{Kot84}}]
The map $G\mapsto Z(\widehat{G})$ defines a contravariant functor from the category of connected reductive $F$-groups with normal $F$-homomorphisms to
the category of diagonalizable $\mathbb{C}$-groups with $\Gamma_F$-actions.
\end{proposition}

In particular, by composing with $H^1(W_F,\cdot)$, we see that the map $G\mapsto H^1(W_F,Z(\widehat{G}))$ defines a  contravariant functor from the category of connected reductive $F$-groups with normal $F$-homomorphisms to
the category of abelian groups, hence is functorial with respect to normal $F$-homorphisms. In particular, we can take the normal homomorphism to be $G\rightarrow G_{\mathrm{ad}}$, in which case we have the following lemma.

\begin{lemma}\label{normalfunc}
The natural map in \Cref{characters} satisfies the commutative diagram 
\begin{equation*}
\xymatrix{
	H^{1}(W_F,Z(\widehat{G}_{\mathrm{sc}})) \ar[d] \ar[r]^-{L}
	& \mathrm{Hom}_{\mathrm{cts}}(G_{\mathrm{ad}}(F),\mathbb{C}^{\times}) \ar[d]  \\
	H^{1}(W_F,Z(\widehat{G})) \ar[r]^-{L}
	& \mathrm{Hom}_{\mathrm{cts}}(G(F),\mathbb{C}^{\times}),}
\end{equation*}
where the second vertical map is given by the composition with the natural map $G(F)\rightarrow G_{\mathrm{ad}}(F)$.
\end{lemma}

Let $S$ be a maximal torus of $G$. Notice that the restriction of a character of $G(F)$ to $S(F)$ is still a character. Actually, we still have the following functorial result with the help of the interpretation in \Cref{Kalinter1,Kalinter2}, although the inclusion $S\hookrightarrow G$ is not normal. 
\begin{lemma}\label{restrictiontotorus}
For quasi-split $G$, the natural bijection in \Cref{characters} satisfies the commutative diagram 
$$\xymatrix{
	H^{1}(W_F,Z(\widehat{G})) \ar[d] \ar[r]^-{L}
	& \mathrm{Hom}_{\mathrm{cts}}(G(F),\mathbb{C}^{\times}) \ar[d]^{|_{S(F)}}  \\
	H^{1}(W_F,\widehat{S}) \ar[r]^-{L}
	& \mathrm{Hom}_{\mathrm{cts}}(S(F),\mathbb{C}^{\times}) .}$$
\end{lemma}
\begin{proof}
By the proof of \cite[Proposition 5.19]{Kal15}, the natural maps between crossed-modules
\[[S_{\mathrm{sc}}\rightarrow S]\leftarrow [Z(G_{\mathrm{sc}})\rightarrow Z(G)]\rightarrow [G_{\mathrm{sc}}\rightarrow G]\]
induce isomorphisms on $\Gamma_F$-cohomology.
On the other hand, the short exact sequence of crossed-modules
\[1 \rightarrow [1\rightarrow G]\rightarrow [G_{\mathrm{sc}}\rightarrow G]\rightarrow[G_{\mathrm{sc}}\rightarrow1] \rightarrow 1\]
induces a long exact sequence on the Galois cohomology
\[G_{\mathrm{sc}}(F)\rightarrow G(F)\rightarrow H^1(F, G_{\mathrm{sc}}\rightarrow G) \rightarrow H^1(F,G_{\mathrm{sc}})=1.\]
Hence we have a natural commutative diagram
\begin{equation}
\begin{aligned}\label{commuinclu}
\xymatrix{
	S(F)/S_{\mathrm{sc}}(F) \ar@{^{(}->}[d] \ar[r]
	& H^1(F, S_{\mathrm{sc}}\rightarrow S) \ar[d]^-{\cong} \\
	G(F)/G_{\mathrm{sc}}(F) \ar[r]^-{\cong}
	& H^1(F, G_{\mathrm{sc}}\rightarrow G),}
\end{aligned}
\end{equation}
where the first vertical map is induced by the inclusion $S\hookrightarrow G$. Together with the isomorphism in \Cref{Kalinter1}, we have natural maps
\[H^{1}(W_F,Z(\widehat{G}))\cong  \mathrm{Hom}_{\mathrm{cts}}(H^1(F,S_{\mathrm{sc}}\rightarrow S),\mathbb{C}^{\times})\rightarrow \mathrm{Hom}_{\mathrm{cts}}(S(F)/S_{\mathrm{sc}}(F),\mathbb{C}^{\times})\rightarrow\mathrm{Hom}_{\mathrm{cts}}(S(F),\mathbb{C}^{\times}),\]
which is given by the restriction map by the commutativity of \eqref{commuinclu}.
\end{proof}

The set of continuous characters of $G(F)$ acts naturally on the set of equivalence classes of irreducible representations of $G(F)$ by character twists, that is, one has a natural action
\begin{equation*}
\begin{aligned}
\mathrm{Hom}_{\mathrm{cts}}(G(F),\mathbb{C}^{\times})\times \Pi(G(F))&\rightarrow \Pi(G(F)),\\
(\chi,[\pi])&\mapsto [\chi\cdot\pi],
\end{aligned}
\end{equation*}
which corresponds to the following action on the parameter side
\begin{equation*}
\begin{aligned}
H^1(W_F,Z(\widehat{G}))\times H^{1}(W_F,\widehat{G})&\rightarrow H^1(W_F,\widehat{G}),\\
(\varphi_{\chi},\varphi_{\pi})&\mapsto \varphi_{\chi\cdot \pi},
\end{aligned}
\end{equation*}
such that $\varphi_{\chi\cdot \pi}$ is given by $\varphi_{\chi\cdot \pi}(w)=\varphi_{\chi}(w)\cdot \varphi_{\pi}(w)$.

Fixing an elliptic maximal torus $S\subset G$, Kaletha's construction of regular supercuspidal representations in \Cref{Yuconstruction} and his construction of $L$-packet data for regular supercuspidal representations could be regarded as a map
\begin{equation}
\begin{aligned}
K_{\chi}:=\prescript{L}{}{j}_{\chi}:H^1(W_F,\widehat{S})^{\mathrm{reg}}&\rightarrow H^1(W_F,\widehat{G}),\\
\phi_{\mu}&\mapsto \phi_{\pi_{(S,\mu)}}.
\end{aligned}
\end{equation}

For any character $\eta:G(F)\rightarrow \mathbb{C}^{\times}$ with $(S,\mu\cdot \eta|_{S(F)})$ being regular elliptic, the following well-known isomorphism between supercuspidal representations \cite[Lemma 3.4.28]{Kal19b} yields the compatibility of character twists and the construction of regular supercuspidal representations
$$\pi_{(S,\mu)}\otimes \eta \cong \pi_{(S,\mu\cdot \eta|_{S(F)})},$$
which corresponds to the commutativity of the diagram on the parameter side:
\begin{equation*}
\begin{array}{ccccc}
  H^{1}(W_F,Z(\widehat{G}))  &\times &H^{1}(W_F,\widehat{S}) &\longrightarrow & H^{1}(W_F,\widehat{S})  \\
  \bigg\| & &\bigg\downarrow K_{\chi} & &\bigg\downarrow K_{\chi}\\
  H^{1}(W_F,Z(\widehat{G}))&\times  &H^{1}(W_F,\widehat{G}) &\longrightarrow & H^{1}(W_F,\widehat{G}).
\end{array}
\end{equation*}

\section{Base change of regular supercuspidal $L$-packet data}\label{section4}
Base change is one of the examples of Langlands functoriality. From the point of view of trace formula and harmonic analysis, the theory of base change could be understood as a special example of the theory of twisted endoscopy, which has been studied in detail in \cite{KS99}. We will not go deep into the theory of twisted endoscopy, but we only focus on the local counterpart of quadratic base change instead.
\subsection{Base change of Langlands parameters}\label{bc}
Let $E/F$ be a quadratic extension. For a Langlands parameter $\phi$ of $G(F)$, a base change $\mathrm{BC}(\phi)$ of $\phi:WD_F\rightarrow \widehat{G}\rtimes W_F$ is nothing but the composition of $\phi$ with the functorial map
\begin{equation*}
\begin{aligned}
i:\prescript{L}{}{G}&\rightarrow \prescript{L}{}{(\mathrm{Res}_{E/F}G_E)}=(\widehat{G}\times\widehat{G})\rtimes W_F,\\
(g,w)&\mapsto (g,g,w)
\end{aligned}
\end{equation*}
where the $W_F$-action on $\widehat{\mathrm{Res}_{E/F}(G_E)}=\widehat{G}\times\widehat{G}$ is given by
\begin{equation*}
\begin{aligned}
w(g_1,g_2)&:=(wg_1,wg_2) \quad\text{for } w\in W_E,\\
s(g_1,g_2)&:=(sg_2,sg_1) \quad\text{for } s\in W_F- W_E.
\end{aligned}
\end{equation*}
In other words, BC induces natural maps
\begin{equation}\label{basechangemap}
\begin{aligned}
\Phi(G(F))&\rightarrow \Phi(\mathrm{Res}_{E/F}(G_E)(F))\cong \Phi(G(E)),\\
H^1(WD_F,\widehat{G})&\rightarrow H^1(WD_F,\widehat{\mathrm{Res}_{E/F}G_E}).
\end{aligned}
\end{equation}

There is an alternative way to describe the base change map \eqref{basechangemap}, which is simply given by restricting $\phi$ to $WD_E$. To see this, consider the following $W_F$-group
$$\mathrm{Ind}_{W_E}^{W_F}\widehat{G}:=\{f:W_F\rightarrow \widehat{G}|\,f(ww')=wf(w')\text{ for any }w\in W_E,w'\in W_F\},$$
with the group law induced from the one on $\widehat{G}$. To be more precise, for any $f,g\in \mathrm{Ind}_{W_E}^{W_F}\widehat{G}$, one defines $(f\cdot g)(w):=f(w)g(w)$.
There is a natural $W_F$-action on $\mathrm{Ind}_{W_E}^{W_F}\widehat{G}$ given by right multiplication, which preserves the group law:
$$(u\cdot f)(w)=f(wu)\text{ for any } u,w\in W_F.$$
Choosing any element $t\in W_F - W_E$, we could construct a $W_F$-equivariant isomorphism
\begin{equation*}
\begin{aligned}
\mathrm{Ind}_{W_E}^{W_F}\widehat{G}&\rightarrow\widehat{\mathrm{Res}_{E/F}G_E} \\
f&\mapsto (f(1),t^{-1}f(t)).
\end{aligned}
\end{equation*}
For any $t'\in W_F-W_E$, we observe
$${t'}^{-1}f(t')={t'}^{-1}f(t't^{-1}t)={t'}^{-1}t't^{-1}f(t)=t^{-1}f(t),$$
which implies that the isomorphism does not depend on the choice of $t\in W_F-W_E$. Moreover, we have the following commutative diagram
$$\xymatrix{
                & \widehat{G} \ar[dl]_{i'}\ar[dr]^{i}             \\
 \mathrm{Ind}_{W_E}^{W_F}\widehat{G}  \ar@{<->}[rr]^{\cong} & &     \widehat{\mathrm{Res}_{E/F}G_E} ,}$$
where the $W_F$-equivariant map $i'$ is defined by
\begin{equation}\label{induced}
\begin{aligned}
i':\widehat{G}&\rightarrow \mathrm{Ind}_{W_E}^{W_F}\widehat{G},\\
g&\mapsto f_g:u\mapsto u\cdot g,\ \forall u\in W_F .
\end{aligned}
\end{equation}
By Shapiro's lemma for non-abelian Galois cohomology \cite{BS64}, we obtain
$$H^1(WD_{F},\widehat{\mathrm{Res}_{E/F}G_E})\cong H^1(WD_{F},\mathrm{Ind}_{W_E}^{W_F}\widehat{G})\cong H^1(WD_{E},\widehat{G}).$$
Hence $(i^{'})^{*}:H^1(WD_F,\widehat{G})\rightarrow H^1(WD_E,\widehat{G})$ coincides with the restriction map, which means that $\mathrm{BC}(\phi)$ could be identified with the restriction $\phi|_{WD_E}$.

Many aspects in the theory of base change have Galois cohomological interpretations. It turns out that the following remark in Prasad's paper \cite{Pra15}, which describes the fiber of base change map, is quite useful in our work.

Picking a base point $\phi_\pi:W_{E}\rightarrow \widehat{G}\rtimes W_{E}$ associated to an irreducible representation $\pi$ of $G(E)$, we can twist the action of the Weil group on $\widehat{G}$ by $\phi_\pi$, which means we can define a new $W_E$-action on $\widehat{G}$ by
$$w^{\phi_\pi}\cdot g=\varphi_\pi(w)w(g)\varphi_\pi(w)^{-1},$$
such that $\phi_\pi$ corresponds to the distinguished point in $H^1(W_{E},\widehat{G})$.
We use the notation ${}_{\phi_\pi}\widehat{G}$ for the group $\widehat{G}$ with the twisted action. One advantage of doing this is that there is an identification
\begin{equation}\label{twistedaction}
_{\phi_{\pi}}\widehat{G}^{{W_{E}}}=\{g\in \widehat{G}|\, \varphi_{\pi}(w)w(g)=g\varphi_{\pi}(w),\forall w\in W_E\} = S_{\phi_{\pi}}.
\end{equation}
 If we assume there exists an extension $\tilde{\phi}\in H^1(W_F,\widehat{G})$, such that $\tilde{\phi}|_{W_E}=\phi_{\pi}$, then under this twisted action, the set of possible extensions of $\phi_\pi$ to a parameter of $\widehat{G}$ can be described by the inflation-restriction sequence
\begin{equation}\label{inflationrestriction}
1\rightarrow H^1(\mathrm{Gal}(E/F),{}_{\phi_{\pi}}\widehat{G}^{W_{E}})\rightarrow H^1(W_F,{}_{\tilde{\phi}}\widehat{G})\rightarrow H^1(W_E,{}_{\phi_\pi}\widehat{G})^{\mathrm{Gal}(E/F)},
\end{equation}
where $H^1(W_E,{}_{\phi_\pi}\widehat{G})$ denotes the corresponding non-abelian Galois cohomology associated to the twisted $W_E$-action with the $\mathrm{Gal}(E/F)$-action on it given by conjugation of an element in $W_F- W_E$. The obstruction of extending the parameter of $G(E)$ to a parameter of $G(F)$ could also be read from this inflation restriction sequence. For example, the inflation restriction map implies the image of $H^1(W_F,{}_{\tilde{\phi}}\widehat{G})\rightarrow H^1(W_E,{}_{\phi_\pi}\widehat{G})$ lies in $H^1(W_E,{}_{\phi_\pi}\widehat{G})^{\mathrm{Gal}(E/F)}$. Hence a necessary condition for $\phi_\pi$ to be extendable to parameters of $\widehat{G}$ is that it is $\mathrm{Gal}(E/F)$-invariant. In fact, this is not the only obstruction. For general Langlands parameters, the inflation restriction sequence terminates due to the lack of $H^2$ for non abelian $W_F$-groups. However, for regular supercuspidal representations, $H^2(\mathrm{Gal}(E/F),\widehat{G}^{W_E})$ is still meaningful due to the special isomorphism in \Cref{special isomorphism}.

\subsection{Base change for regular supercuspidal $L$-packet data}
Now we use tuples $(S,\widehat{j},\chi,\mu)$ as Langlands parameters for regular supercuspidal representations of $G(F)$. We would like to describe the quadratic base change in terms of regular supercuspidal $L$-packet data. We will focus on those regular supercuspidal $L$-packet data whose base changes to $G(E)$ remain regular supercuspidal. Later we will see in \Cref{generalbasechange} that it's enough to study base changes of these regular supercuspidal $L$-packet data for the purpose of understanding the relationship between Galois distinction problems and functoriality.

\begin{definition}\label{basechangeregular}
The quadratic base change to $G(E)$ of a supercuspidal $L$-packet datum of $G(F)$ associated to
$$(S,\widehat{j},\chi,\mu)$$
is given by the 4-tuple
$$(S\times_F E,\widehat{j},\chi_{\mathrm{BC}},\mu\circ \mathrm{Nm}_{S(E)/S(F)}),$$
where $\chi_{\mathrm{BC}}$ is $\chi$-data of $G(E)$, such that $(\chi_{\mathrm{BC}})_{\alpha}=\chi_{\alpha}\circ \mathrm{Nm_{E_{\alpha}/F_{\alpha}}}:E_{\alpha}^{\times}\rightarrow F_{\alpha}^{\times}\rightarrow \mathbb{C}^{\times}$.
\end{definition}
Indeed, we are interested in the situation where this 4-tuple is still a supercuspidal $L$-packet datum for $G(E)$, although it will not always be in general. Another important property of the base change map defined above is that it only depends on the equivalence class of the supercuspidal $L$-packet datum. 

\begin{lemma}\label{naivebasechange}
Let $\mathcal{RSC}_{F,\mathrm{BC}}$ denote the subcategory of the category of regular supercuspidal $L$-packet data of $G$ over $F$ consisting of the objects whose base change (in the usual sense) to $G(E)$ remain regular supercuspidal. Then $\mathrm{BC}$ can be regarded as a functor from $\mathcal{RSC}_{F,\mathrm{BC}}$ to the category of regular supercuspidal $L$-packet data of $G_E$ over $E$.
\end{lemma}

\begin{proof}

It is enough to treat the case when the $F$-isomorphism of $S\rightarrow S$ is the identity map. General cases could be reduced to this case by composing with the isomorphism over $F$. Recall $(S,\widehat{j},\chi,\theta)$ and $(S,\widehat{j},\chi',\theta')$ are equivalent, if and only if
$$\theta=\theta'\cdot {\zeta}^{-1}_{S(F),\chi',\chi},$$
where $\zeta_{S(F),\chi',\chi}$ is the character $S(F)\rightarrow \mathbb{C}^{\times}$ constructed from the $\zeta$-datum $\{\zeta_{\alpha}\}$ (product of characters depending on whether the root is symmetric or asymmetric) determined by the ratio of $\chi$ and $\chi'$.
Composing both sides with $\mathrm{Nm}_{S(E)/S(F)}$, we need to prove the following identity in order to check the definition of base change is well-defined:
\begin{equation}\label{compatibility}
\begin{aligned}
\zeta_{S,\chi_{\mathrm{BC}},\chi'_{\mathrm{BC}}}=\zeta_{S,\chi,\chi'}\circ \mathrm{Nm}_{S(E)/S(F)}.
\end{aligned}
\end{equation}
We only need to check this identity for a fixed $\Gamma_{F}\times \{\pm 1\}$-orbit of $\alpha$, which will imply the whole identity by taking the product of components of these orbits. The general strategy is to decompose the $\Gamma_{F}\times \{\pm 1\}$-orbit of $\alpha$ into $\Gamma_{E}\times \{\pm 1\}$-orbits.

\begin{enumerate}
\item If the $\Gamma_{F}$-orbit of $\alpha$ remains a single $\Gamma_E$-orbit, then  $[E_{\alpha}:F_{\alpha}]=2$ and $\chi_{\mathrm{BC},\alpha}=\chi_\alpha\circ \mathrm{Nm}_{E_\alpha/F_\alpha}$. We have the following diagram
\begin{equation}
\begin{aligned}
\xymatrix{
  E_{\alpha} \ar@{-}[d]_{2} \ar@{-}[r]^{2} & E_{\pm\alpha} \ar@{-}[d]_{2} \ar@{-}[r] & E \ar@{-}[d] \\
  F_{\alpha} \ar@{-}[r]^{2} & F_{\pm\alpha} \ar@{-}[r] & F,   }
\end{aligned}
\end{equation}
or
\begin{equation}
\begin{aligned}
\xymatrix{
  E_{\alpha} \ar@{-}[d]_{2} \ar@{=}[r] & E_{\pm\alpha} \ar@{-}[d] \ar@{-}[r] & E \ar@{-}[d] \\
  F_{\alpha} \ar@{=}[r] & F_{\pm\alpha} \ar@{-}[r] & F. }
\end{aligned}
\end{equation}
Since 
\begin{align*}
	\Gamma_F\cdot\alpha&=\Gamma_E\cdot\alpha &\text{if $\alpha$ is symmetric},\\
	(\Gamma_F\times \{\pm1\})\cdot\alpha&=(\Gamma_E\times\{\pm1\})\cdot \alpha &\text{if $\alpha$ is asymmetric},
\end{align*}
the identity which we want to prove in this case (the $\alpha$-component of \Cref{compatibility}) is given by
$$\frac{\chi_{\alpha}}{\chi'_{\alpha}}(\iota_{F_{\alpha}}\alpha(\mathrm{Nm}_{S(E)/S(F)}t))=\frac{\chi_{\mathrm{BC},\alpha} }{\chi'_{\mathrm{BC},\alpha}}(\iota_{E_{\alpha}}\alpha(t)).$$

Notice that we have the following commutative diagram
\begin{equation}
\begin{aligned}\label{normcommute}
\xymatrix{
		S(E) \ar[d]_{\mathrm{Nm}_{S(E)/S(F)}} \ar[r]^{\alpha}
		& E_{\alpha}^{\times} \ar[d]^{\mathrm{BC}_{\alpha}=\mathrm{Nm}_{E_{\alpha}/F_{\alpha}}} \\
		S(F)  \ar[r]_{\alpha}
		& F_{\alpha}^{\times}.}
\end{aligned}
\end{equation}

Since the symmetric property of $\alpha$ is the same over $E$ and $F$, we know 
\begin{align}\label{iotacommute}
\iota_{F_{\alpha}}\mathrm{Nm}_{E_\alpha/F_\alpha}=\mathrm{Nm}_{E_\alpha/F_\alpha}\iota_{E_\alpha},
\end{align}
where $\iota_{F_{\alpha}}$ and $\iota_{E_{\alpha}}$ are identity maps when $\alpha$ is asymmetric over $E$ and $F$, and are isomorphisms $\iota_{F_{\alpha}}: {F_{\alpha}^{1}}_{F_{\alpha}/F_{\pm\alpha}}\cong F_{\alpha}^{\times}/F_{\pm\alpha}^{\times}$, $\iota_{E_{\alpha}}: {E_{\alpha}^{1}}_{E_{\alpha}/E_{\pm\alpha}}\cong E_{\alpha}^{\times}/F_{\pm\alpha}^{\times}$, when $\alpha$ is symmetric over $E$ and $F$. Based on \Cref{iotacommute} and the commutativity of the diagram \eqref{normcommute}, one has the following identity:
$$\iota_{F_{\alpha}}\alpha(\mathrm{Nm}_{S(E)/S(F)}t))=\mathrm{Nm}_{E_{\alpha}/F_{\alpha}}(\iota_{E_{\alpha}}\alpha(t)).$$
Composing both sides with $\frac{\chi_\alpha}{\chi'_{\alpha}}$, we obtain
\begin{equation}
\begin{aligned}
\frac{\chi_{\alpha}}{\chi'_{\alpha}}(\iota_{F_{\alpha}}\alpha(\mathrm{Nm}_{S(E)/S(F)}t))&=\frac{\chi_{\alpha}\circ \mathrm{Nm}_{E_{\alpha}/F_{\alpha}}}{\chi'_{\alpha}\circ \mathrm{Nm}_{E_{\alpha}/F_{\alpha}}}(\iota_{E_{\alpha}}\alpha(t))=\frac{\chi_{\mathrm{BC},\alpha}}{\chi'_{\mathrm{BC},\alpha}}(\iota_{E_{\alpha}}\alpha(t)).
\end{aligned}
\end{equation}

\item If the $\Gamma_{F}$-orbit of $\alpha$ breaks into two $\Gamma_E$-orbits, that is, $\Gamma_F\cdot\alpha=\Gamma_E\cdot\alpha\sqcup \Gamma_E \cdot s\alpha$ for $s\in \Gamma_F-\Gamma_E$, with projection $\sigma\in \mathrm{Gal}(E/F)$, then $E_{\alpha}=F_{\alpha}$
and $\chi_{\mathrm{BC},\alpha}=\chi_\alpha\circ \mathrm{Nm}_{E_\alpha/F_\alpha}=\chi_{\alpha}\circ \mathrm{id}$. We have the following cases:

\begin{enumerate}
\item The symmetry of $\alpha$ is different over $E$ and $F$.

The only possibility is that $\alpha$ is symmetric over $F$, but asymmetric over $E$. We can choose $s$ to satisfy $s\cdot\alpha=-\alpha$. In this case, we have the following diagram
\begin{equation}
\begin{aligned}
\xymatrix{
  E_{\alpha} \ar@{=}[d] \ar@{=}[r] & E_{\pm\alpha} \ar@{-}[d]_{2} \ar@{-}[r] & E \ar@{-}[d] \\
  F_{\alpha} \ar@{-}[r]^{2} & F_{\pm\alpha} \ar@{-}[r] & F   .}
\end{aligned}
\end{equation}

Since $$\Gamma_F\cdot\alpha=\Gamma_E\cdot\alpha\sqcup \Gamma_E s\cdot\alpha=(\Gamma_E \times \{\pm1\})\cdot \alpha,$$ 
the identity which we want to prove in this case ($\alpha$ component of \Cref{compatibility}) is given by
$$\frac{\chi_{\alpha}}{\chi'_{\alpha}}(\iota_{F_{\alpha}}\alpha(\mathrm{Nm}_{S(E)/S(F)}t))=\frac{\chi_{\mathrm{BC},\alpha} }{\chi'_{\mathrm{BC},\alpha}}(\alpha(t)).$$

Notice that we have the following commutative diagram
    $$\xymatrix{
  S(E) \ar[d]_{\mathrm{Nm}_{S(E)/S(F)}} \ar[r]^{\alpha}
                & E_{\alpha}^{\times} \ar[rrd]^{\mathrm{BC}_{\alpha}=\mathrm{id}} & &\\
  S(F)  \ar[r]^{\alpha}
                & F_{\alpha}^{1} \ar[r]^-{\iota_{F_{\alpha}}}_-{\cong}  & F_{\alpha}^{\times}/F_{\pm\alpha}^{\times}    & F_{\alpha}^{\times}\ar[l]      ,}$$
that is,
$$\alpha(\mathrm{Nm}_{S(E)/S(F)}(t))=\iota_{F_{\alpha}}^{-1}(\alpha(t)).$$
Composing both sides with $\frac{\chi_\alpha}{\chi'_{\alpha}}\circ \iota_{F_{\alpha}}$, we obtain
$$\frac{\chi_{\alpha}}{\chi'_{\alpha}}(\iota_{F_{\alpha}}\alpha(\mathrm{Nm}_{S(E)/S(F)}t))=\frac{\chi_{\alpha}}{\chi'_{\alpha}}(\alpha(t))=\frac{\chi_{\mathrm{BC},\alpha} }{\chi'_{\mathrm{BC},\alpha}}(\alpha(t)).$$

\item The symmetry of $\alpha$ is the same over $E$ and $F$.

In this case, $E_{\alpha}=E_{s\alpha}=F_{\alpha}$ and we have the following diagram

\begin{equation}
\begin{aligned}
\xymatrix{
  E_{\alpha} \ar@{=}[d] \ar@{-}[r]^{2} & E_{\pm\alpha} \ar@{=}[d] \ar@{-}[r] & E \ar@{-}[d] \\
  F_{\alpha} \ar@{-}[r]^{2} & F_{\pm\alpha} \ar@{-}[r] & F ,  }
\end{aligned}
\end{equation}

or
\begin{equation}
\begin{aligned}
\xymatrix{
  E_{\alpha} \ar@{=}[d] \ar@{=}[r] & E_{\pm\alpha} \ar@{-}[d] \ar@{-}[r] & E \ar@{-}[d] \\
  F_{\alpha} \ar@{=}[r] & F_{\pm\alpha} \ar@{-}[r] & F  . }
\end{aligned}
\end{equation}
Since
\begin{align*}
	\Gamma_F\cdot\alpha&=\Gamma_E\cdot{\alpha}\sqcup\Gamma_E\cdot{s\alpha}& \text{if } \alpha \text{ is symmetric},\\
	(\Gamma_F\times \{\pm 1\})\cdot \alpha&=(\Gamma_E\times\{\pm1\})\cdot{\alpha}\sqcup (\Gamma_E\times\{\pm1\})\cdot{s\alpha}
	&\text{if }\alpha \text{ is asymmetric},
\end{align*} 
the identity which we want to prove in this case ($\alpha$ component of \Cref{compatibility}) is given by

$$\frac{\chi_{\alpha}}{\chi'_{\alpha}}(\iota_{F_{\alpha}}\alpha(\mathrm{Nm}_{S(E)/S(F)}t))=\frac{\chi_{\mathrm{BC},\alpha} }{\chi'_{\mathrm{BC},\alpha}}(\iota_{E_{\alpha}}\alpha(t))\cdot \frac{\chi_{\mathrm{BC},s\alpha} }{\chi'_{\mathrm{BC},s\alpha}}(\iota_{E_{s\alpha}}(s\alpha)(t)).$$

Notice that we have the following commutative diagram
$$\xymatrix{
  S(E) \ar[d]_{\mathrm{Nm}_{S(E)/S(F)}} \ar[r]^{\alpha\cdot s\alpha}
                & E_{\alpha}^{\times} \ar[d]^{\mathrm{BC}_{\alpha}=\mathrm{id}} \\
  S(F)  \ar[r]_{\alpha}
                & F_{\alpha}^{\times},}$$

that is,
$$\alpha(\mathrm{Nm}_{S(E)/S(F)}(t))=(\alpha\cdot s\alpha)(t).$$
Composing both sides with $\frac{\chi_\alpha}{\chi'_{\alpha}}\circ \iota_{F_{\alpha}}$, we obtain
\begin{equation*}
\frac{\chi_{\alpha}}{\chi'_{\alpha}}(\iota_{F_{\alpha}}\alpha(\mathrm{Nm}_{S(E)/S(F)}t))=\frac{\chi_{\alpha}}{\chi'_{\alpha}}(\iota_{F_{\alpha}}\alpha(t))\cdot \frac{\chi_{\alpha}}{\chi'_{\alpha}}(\iota_{F_{\alpha}}(s\alpha)(t))=\frac{\chi_{\mathrm{BC},\alpha} }{\chi'_{\mathrm{BC},\alpha}}(\iota_{E_{\alpha}}\alpha(t))\cdot \frac{\chi_{\mathrm{BC},s\alpha} }{\chi'_{\mathrm{BC},s\alpha}}(\iota_{E_{s\alpha}}(s\alpha)(t)).\qedhere
\end{equation*}
\end{enumerate}
\end{enumerate}
\end{proof}

\begin{proposition}\label{basechangesame}
The base change map of supercuspidal $L$-packet data sending $(S,\widehat{j},\chi,\mu)$ to $(S\times_F E,\widehat{j},\chi_{\mathrm{BC}},\mu\circ \mathrm{Nm}_{S(E)/S(F)})$ matches with the usual base change map of the Langlands parameters sending $\phi$ to $\phi|_{W_E}$, i.e., there is a commutative diagram
\[\xymatrix@C=3cm{
	\phi \ar@{|->}[d]_{|_{W_E}} \ar@{<->}[r]^-{\Cref{regularLpacket}}
	& (S,\widehat{j},\chi,\mu) \ar@{|->}[d]^{\mathrm{BC}}  \\
	\phi|_{W_E}  \ar@{<->}[r]^-{\Cref{regularLpacket}}
	& (S\times_F E,\widehat{j},\chi\circ \mathrm{Nm}_{E_{\alpha}/F_{\alpha}},\mu\circ \mathrm{Nm}_{S(E)/S(F)}).}\]
\end{proposition}
Before we give a proof of this proposition, we highlight the following remarks on the base change maps of regular supercuspidal parameters.

\begin{remark}\label{basechangelpacketdata}
 One bad news is that $\chi_{\mathrm{BC}}$ may fail to be minimally ramified (see \Cref{footnote} for an explicit example). However, one can still find a representative in the equivalence class with minimally ramified $\chi$-data over $E$ by adjusting the character as in \Cref{equivalenceclass}. More precisely, one can find a 4-tuple $(T,\widehat{j},\chi_{\theta},\theta)$ where $\chi_{\theta}$ is the minimally ramified $\chi$-data determined by $\theta$ such that the tuple is equivalent to $(S\times_F E,\widehat{j},\chi_{\mu,\mathrm{BC}},\mu\circ \mathrm{Nm}_{S(E)/S(F)})$. The explicit relation between $\theta$ and $\mu\circ\mathrm{Nm}_{S(E)/S(F)}$ is described later in \Cref{basechangechi}.
\end{remark}

\begin{remark}\label{generalbasechange}
The base change map in the usual sense $\mathrm{BC}:\Phi(G(F))\rightarrow \Phi(G(E))$ given by $\mathrm{BC}(\phi)=\phi|_{W_E}$ is well-defined on the set of equivalence classes of all the Langlands parameters. However, for a discrete Langlands parameter $\phi$ of a regular supercuspidal representation of $G(F)$, $\mathrm{BC}(\phi)$ may not be a Langlands parameter of a regular supercuspidal representation of $G(E)$. Hence the base change of a regular supercuspidal $L$-packet datum is only meaningful when the datum after base change is still a regular supercuspidal parameter. Nonetheless, for a given regular supercuspidal parameter $\phi$, if $\mathrm{BC}(\tilde{\phi})=\phi$, then $\tilde{\phi}$ is a regular supercuspidal parameter by \Cref{bcregular}.
\end{remark}

To prove \Cref{basechangesame}, we need the following lemma.
\begin{lemma}
Let $p:\Phi(G,S)\rightarrow \{\pm 1\}$ be any gauge. The cocycle in $r_{\chi,p}:W_F\rightarrow \widehat{S}$ (\Cref{basepoint}) satisfies $r_{\chi,p}|_{W_E}=r_{\chi_{\mathrm{BC}},p}$.
\end{lemma}
\begin{proof}This is a direct consequence of \cite[Proposition 5.15.3]{Kal19a} by taking $S$ to be $\mathrm{Res}_{E/F}S_E$ and $T$ to be $S$. This is also proved in \cite[Theorem 64]{Sch24}.
\end{proof}

\begin{proof}[Proof of \Cref{basechangesame}]
Let $\phi_{(S,\mu)}$ denote the Langlands parameter of the regular supercuspidal representation associated to the tame elliptic pair $(S,\mu)$. As described in \Cref{Lembedding}, $\phi$ is of the form:
$$\phi_{(S,\mu)}(w)=\widehat{j}(\phi_{\mu}(w)r_{\chi}(w)n(s_w))\rtimes w.$$
Combining the above lemma with local Langlands correspondence for tori, we get
$$\phi_{(S,\mu)}|_{W_E}(w)=\widehat{j}(\phi_{\mu\circ\mathrm{Nm}_{S(E)/S(F)}}(w)r_{\chi_{\mathrm{BC}}}(w)n(s_w))\rtimes w,$$
which means the supercuspidal $L$-packet datum associated to $\phi|_{W_E}$ is $(S\times_F E,\widehat{j},\chi_{\mathrm{BC}},\mu\circ \mathrm{Nm}_{S(E)/S(F)})$.
In other words, we have the following commutative diagram
\[	\begin{gathered}[b]
	\xymatrix{
		\prescript{L}{}{S} \ar[d]_{\prescript{L}{}{j}=(\widehat{j},\chi)} \ar[r]^{i~~~~~~~}
		& \prescript{L}{}{(\mathrm{Res}_{E/F}S_E)}  \ar[d]^{\prescript{L}{}{j_E}=(\widehat{j},\chi_{\mathrm{BC}})}  \\
		\prescript{L}{}{G} \ar[r]^{i~~~~~~~}
		&   \prescript{L}{}{(\mathrm{Res}_{E/F}G_E)} .}\\[-\dp\strutbox]
\end{gathered}
\qedhere
\]

\end{proof}

\section{Prasad's conjecture for Galois pairs}\label{section5}

We give a brief review of Prasad's conjecture relating the distinction property of an irreducible admissible representation of $G(E)$ to the base change functorial property of its Langlands parameter and certain numerical invariants on both sides. Before going into details, we give a brief introduction to certain objects defined by Prasad. For more detailed description of these objects, one can refer to Prasad's original paper \cite{Pra15}. Let $G$ be a quasi-split reductive group defined over a non-archimedean local field $F$ and let $E/F$ be a quadratic extension with a non-trivial Galois involution $\sigma$.\\

\subsection{The quasi-split $F$-form $G^\mathrm{op}$ of $G$}\label{Gop}

We pick the split form $G_0$ of $G$ over $F$. Fix a split maximal torus $S_0$ of $G_0$ over $F$ and a Borel subgroup $B_0$ containing $S_0$, which determines a set of positive roots $\Phi^+\subset \Phi(G_0,S_0)$. Let $R_b=(X^{*}(S_0),\Phi(G_0,S_0),\Delta_0, X_{*}(S_0),\Phi^{\vee}(G_0,S_0),\Delta_0^{\vee})$ denote the based root datum determined by $(B_0,S_0)$, where $\Delta_0$ denotes the set of simple roots. Let $\mathrm{\mathrm{Aut}(G_0)}$ denote the automorphism group scheme over $F$, namely, for any $F$-algebra $S$, $\mathrm{Aut}(G_0)(S)=\mathrm{Aut}_S(G_{0,S})$. Then there is a natural exact sequence of $F$-groups
\begin{equation}\label{outer}
\begin{aligned}
1\rightarrow \mathrm{Inn}(G_0)\rightarrow \mathrm{Aut}(G_0)\rightarrow \mathrm{Out}(G_0)\rightarrow 1.
\end{aligned}
\end{equation}
Since $G_0$ is split, the outer-automorphism group scheme is a constant group scheme by \cite[7.19]{Con14} and 
$$\mathrm{Out}(G_0)(\overline{F})\cong \mathrm{Out}(G_0)(F)\cong \mathrm{Aut}(R_{b}),$$
We can produce a splitting of \eqref{outer} by fixing a pinning $\{X_{0,\alpha}\}_{\alpha\in \Delta}$ of $(G_0, B_0,S_0)$ over $F$. More precisely, for any $f_b\in \mathrm{Aut}(R_{b}$), let $f$ be a lifting of $f_b$ in $\mathrm{Aut}(G)$. It is direct to see that  $\{X'_\alpha=df(X_{0,f_b^{-1}(\alpha)})\}_{\alpha\in \Delta}$ is another set of pinning so that there exists $\{c_{\alpha}\in F^\times\}_{\alpha\in \Delta}$ such that $X'_\alpha=c_\alpha X_\alpha$. In this way, the set of all pinnings of $(G_0, B_0,S_0)$ is a torsor under $S_{0,\mathrm{ad}}(F)\cong (F^\times)^{|\Delta|}$. Since $t\in S_{0,\mathrm{ad}}(F)$ acts by multiplying $\alpha(t)$ on $X_\alpha$, we can assume $c_\alpha=1$ for any $\alpha$ by composing $f$ with an inner automorphism in $S_{0,\mathrm{ad}}(F)$, namely, we can assume $(df)(X_\alpha)=X_{f_b(\alpha)}$. The assignment $f_b\mapsto f$ provides a natural splitting
\begin{align}\label{splitting}
\mathrm{Out}(G_0)(F)\cong \mathrm{Aut}_F(G_0,B_0,S_0,\{X_{\alpha}\})\hookrightarrow \mathrm{Aut}(G_0).
\end{align}
Let $\omega_0$ be the longest Weyl element in $W(G_0,S_0)$, which is unique by the maximal length condition. Since $\omega_0^{-1}$ has the same length as $\omega_0$, the uniqueness property implies that $\omega_0^2=1\in W(G_0,S_0)$. It is also known that $\omega_0\cdot \Phi^+=-\Phi^+$ so that $-\omega_0\in \mathrm{Aut}(R_b)\cong \mathrm{Out}(G_0)(F)$. 

Chevalley involutions of $G_0$ are particular involutive lifts of $-\omega_0\in \mathrm{Aut}(R_b)$ in $\mathrm{Aut}(G_0)(F)$
defined by Prasad \cite{Pra19a} in the general case and Adams--Vogan \cite{Ada14, AV16} independently in the archimedean case. More precisely, a canonical representative of Chevalley involution can be  constructed and captured by the following properties.

Fix a pinning $\{X_\alpha\}_{\alpha \in \Delta}$ of $(G_0,B_0,S_0)$. Let $\{-X_{-\alpha}\}_{\alpha \in \Delta}$ denote the pinning of $(G_0,B_0^-,S_0)$, where $B_0^-$ denotes the opposite Borel subgroup and $X_{-\alpha}$ denotes the dual pinning, namely, the dual basis of $X_\alpha$ in $\mathfrak{g}_{-\alpha}$ by the natural pairing $\langle\,, \rangle:=\frac{1}{2}\alpha\circ [\,,\,]:\mathfrak{g}_\alpha\times \mathfrak{g}_{-\alpha}\rightarrow F$.
The following theorem ensures that there exists an automorphism $c_0$ on $G_0$, which restricts to $^{-1}$ on $S_0$, exchanges $\Phi^+$ and $-\Phi^+$ and maps $X_\alpha$ to $-X_{-\alpha}$. 
\begin{theorem}[{\cites[Theorem 2.2]{AV16}[Theorem 9.6.2]{Spr98}}]
Suppose $G_1$ and $G_2$ are split connected reductive groups with pinnings $\mathcal{P}_1=(B_1,S_1,\{X_\alpha\}_1)$ and $\mathcal{P}_2=(B_2,S_2,\{X_\alpha\}_2)$. Let $f_b:R_{b,1}\rightarrow R_{b,2}$ be an isomorphism of based root data. Then there is a unique $f:G_1\rightarrow G_2$ inducing $f_b$ on the based root data.
\end{theorem}
Indeed, the uniqueness part of the above theorem also implies that $c_0$ is an involution since both $c_0^2$ and $1$ induces the identity map on the based root datum of $G_0$. Chevalley involutions are then defined to be the involutions in the $G_{\mathrm{ad}}(F)$-conjugacy class of the canonical involution, namely, of the form $\mathrm{Ad}(g) \circ c_0\circ \mathrm{Ad}(g^{-1})$ for $\mathrm{Ad}(g)\in G_{\mathrm{ad}}(F)$. 

In order to get an involutive lift of $-\omega_0$ via the splitting \eqref{splitting}, we compose $c_0$ with the conjugate of a lift of longest Weyl element in $N_{G_0}(S_0)$ together with a proper inner automorphism in $S_{0,\mathrm{ad}}(F)$ so that it preserves $B_0$ and also preserves the pinning $\{X_\alpha\}_{\alpha\in \Delta}$. In this way, we get an involution $c_0^\mathrm{op}\in \mathrm{Aut}(G_0)(F)$ defined over $F$. Notice that the automorphism $c_0^{\mathrm{op}}$ does depend on the pinning we choose, and a canonical choice of lift with respect to the given pinning $\mathcal{P}=\{(G_0,B_0,S_0,\{X_\alpha\}_{\alpha\in\Delta})\}$ has been explicitly described by Prasad \cite[Definition 1, Example 1]{Pra19a}. Indeed, Prasad also proves the existence of $c^{\mathrm{op}}$ for all quasi-split reductive groups in \cite[Section 3]{Pra19a}.
\begin{example}
The prototype for the canonical Chevalley involution is the involution $(\cdot)^{t,-1}$ on $\mathrm{GL}_n$. To see what $c_0^{\mathrm{op}}$ is for $\mathrm{GL}_n$, we can first choose the canonical lift $n_0$ of $\omega_0$ given by the permutation matrix with $1$ on the anti-diagonal and $0$ otherwise. However, $c_0 \circ \mathrm{Ad}(n_0)$ maps $\alpha$ to $-\omega_0\alpha$ and $X_{\alpha}$ to $-X_{-\omega_0\alpha}$ for any $\alpha$. We need to further pick an element in $S_{0,\mathrm{ad}}(F)$ so that it acts on $-1$ on all simple root spaces $\mathfrak{g}_{\alpha}$ for any $\alpha\in \Delta$. A canonical construction of this element is given by $\rho^{\vee}(-1):=\frac{1}{2}\sum_{\alpha^\vee>0}\alpha^\vee(-1)$. In the case of $\mathrm{GL}_n$, this element is $\mathrm{Ad}(\mathrm{diag}((-1)^{n-1},(-1)^{n-2},\cdots,1))$ and $c_0^{\mathrm{op}}$ is given by $\mathrm{Ad}(J)\circ (\cdot)^{t,-1}$, where $J\in \mathrm{GL}_n(F)$ denotes the matrix with $J_{i,j}=(-1)^{n-i}\delta_{i,n-j+1}$.
\end{example}

It is known \cite[Proposition 7.2.11]{Con14} that quasi-split $F$-forms of $G_0$ (hence of $G$) are classified by
$$H^1(F,\mathrm{Out}(G_0))=\mathrm{Hom}(\Gamma_F,\mathrm{Out}(G_0)(\overline{F}))/ \mathrm{Out}(G_0)(\overline{F})\mbox{-}\text{conjugation}.$$
By composing the natural homomorphism $\mathrm{Gal}(E/F)\rightarrow \mathrm{Out}(G_0)(\overline{F})$ sending the non-trivial element $\sigma$ to $-\omega_{0}$ with the canonical projection $\Gamma_F\rightarrow \mathrm{Gal}(E/F)$, one gets an element $z_0^{\mathrm{op}}$ inside $H^1(F,\mathrm{Out}(G_0))$ corresponding to a quasi-split $F$-form $G_0^\mathrm{op}$ of $G$ which splits over $E$. More precisely, we have the following description of $F$-rational points of $G_0^{\mathrm{op}}$
$$G_0^{\mathrm{op}}(F)=\{g\in G_0(E)|\,\sigma(g)=c_0^{\mathrm{op}}(g)\},$$
where $c_0^{\mathrm{op}}$ is the opposite involution we construct via the splitting \eqref{splitting}. 

The above construction can be applied to any quasi-split reductive group $G$ instead of just the split form $G_0$ via the following trick in \cite[Section 7.2]{Con14}. The basic observation is that $\mathrm{Aut}(G)$ as an $F$-group is an $F$-form of $\mathrm{Aut}(G_0)$ via the inner twist of a 1-cocycle which determines the $F$-structure of $G$. More precisely, since $G$ is a $F$-form of $G_0$, the $F^s$-isomorphism $\xi:{G_{0}}_{F^s}\cong G_{F^s}$ determines an element $c_G:\gamma\mapsto \xi^{-1}\circ \gamma(\xi) \in Z^1(\Gamma_F,\mathrm{Aut}(G_0))$, where $\gamma(\xi)$ denotes $\gamma\circ \xi\circ \gamma^{-1}$. Via this $\xi$, one can construct an isomorphism $\eta_{\xi}:\mathrm{Aut}(G)_{F^s}\rightarrow \mathrm{Aut}(G_0)_{F^s}$ by sending $f$ to $\xi^{-1}\circ f\circ \xi$. One then has the following identity 
\[\eta_{\xi}(\gamma (f))=\xi^{-1}\circ\gamma(\xi)\circ\gamma(\xi^{-1}\circ f\circ\xi)\circ\gamma(\xi)^{-1}\circ\xi=\mathrm{Ad}(\xi^{-1}\circ\gamma(\xi))(\gamma(\eta_{\xi}(f))),\]
which means that the $F$-structure on $\mathrm{Aut}(G)$ is constructed from the $F$-structure on $\mathrm{Aut}(G_0)$ by the canonical inner twist  $c_G\in Z^1(\Gamma_F,\mathrm{Aut}(G_0))$ determined by $G$.
Indeed, the above statement also holds true for $\mathrm{Inn}(G)=G/Z(G)$ and $\mathrm{Out}(G)$, namely, $G/Z(G)$ is an $F$-form of $G_0/Z(G_0)$ and $\mathrm{Out}(G)$ is an $F$-form of $\mathrm{Out}(G_0)$. Moreover, by twisting the $1$-cocycle $c_G$, there are bijections of sets fitting into the following commutative diagram \cite[7.2.2]{Con14}
\[
\xymatrix{
	H^1(F,\mathrm{Aut}(G)) \ar[r]\ar[d]^{\cong}&H^1(F,\mathrm{Out}(G))\ar[d]^{\cong}\\
	H^1(F,\mathrm{Aut}(G_0)) \ar[r]&H^1(F,\mathrm{Out}(G_0)).
}
\]
Hence $z_0^\mathrm{op}\in H^1(F,\mathrm{Out}(G_0))$ can be regarded as an element inside $H^1(F,\mathrm{Out}(G))$ via this bijection, which defines a quasi-split $F$-form $G^\mathrm{op}$ of $G$. From the construction, we know that $G^{\mathrm{op}}$ is a quasi-split $F$-form of $G$, such that $G^{\mathrm{op}}_E\cong G_{E}$ as algebraic groups over $E$.

\begin{example}\label{Sop}
Let $S$ be a torus defined over $F$. Then there is an isomorphism between $S^{\mathrm{op}}$ and $\mathrm{Res}_{E/F}S_E/S$. In this case, we know $c^{\mathrm{op}}(t)=t^{-1}$ and
$$S^{\mathrm{op}}(F)=\{t\in S(E)|\sigma(t)=t^{-1}\}=\ker (\mathrm{Nm}_{S(E)/S(F)}).$$
The following cases of tori will be frequently used later:
\begin{enumerate}
\item $S=\mathbb{G}_{m}$, then $S^{\mathrm{op}}=\mathrm{U}_{1,E/F}$.
\item $S=\mathrm{U}_{1,E/F}$, then $S^{\mathrm{op}}=\mathbb{G}_{m}$.
\item $S=\mathrm{U}_{1,E_1/F}$, for a quadratic extension $E_1$ over $F$, which is different from $E$, then $S^{\mathrm{op}}=\mathrm{U}_{1,E_2/F}$ for the quadratic extension $E_2$ over $F$, which is different from $E$ and $E_1$.
\end{enumerate}
\end{example}

The complex dual group of $G^{\mathrm{op}}$ is equipped with a natural embedding $i:\widehat{G^{\mathrm{op}}}\rightarrow \widehat{\mathrm{Res}_{E/F}G_E}$ sending $g$ to $(g,c(g))$, which extends naturally to a homomorphism of $L$-groups $i:\prescript{L}{}{G^{\mathrm{op}}}\rightarrow \prescript{L}{}{(\mathrm{Res}_{E/F}G_E)}$ sending $(g,w)$ to $(g,c(g),w)$. This functoriality realizes the base change of $G^{\mathrm{op}}$, which coincides with the one in \Cref{bc}, except that we use the isomorphism $\widehat{\mathrm{Res}_{E/F}G^{\mathrm{op}}_{E}}\cong\widehat{\mathrm{Res}_{E/F}G_E}$ sending $(g,g)$ to $(g,c(g))$.

We have already seen from \eqref{inflationrestriction} that one of the necessary conditions for a Langlands parameter of $G(E)$ to be a functorial lift from a Langlands parameter of $G^{\mathrm{op}}(F)$ is that the parameter is $\mathrm{Gal}(E/F)$-invariant (the Galois action is the one associated to the pair $(G(E),G^{\mathrm{op}}(F))$), which is well known (for example, see \cite[Section 1.6]{AC89} and \cite[Section 2.10]{BP21}) in the following cases.
\begin{example}
\quad
\begin{enumerate}
\item If $G=\mathrm{U}_{n,E/F}$ is the quasi-split unitary group, then $G^{\mathrm{op}}=\mathrm{GL}_n$ is the general linear group over $F$. A Langlands parameter of $\mathrm{GL}_n(E)$ being $\mathrm{Gal}(E/F)$-invariant with respect to $G^{\mathrm{op}}$ means $\phi_{\pi}^{\sigma}\cong\phi_{\pi}$. Since each $L$-packet of $\mathrm{GL}_n(E)$ is a singleton, this simply means $\pi^{\sigma}\cong \pi$, that is, $\pi$ is conjugate invariant.

\item If $G=\mathrm{GL}_n$ is the general linear group, then $G^{\mathrm{op}}=\mathrm{U}_{n,E/F}$ is the quasi-split unitary group. In this case, the $\mathrm{Gal}(E/F)$-action on $\widehat{\mathrm{U}_{n,E/F}}$ is the usual Galois action twisted by the Chevalley involution $c:g\mapsto J(g^{t})^{-1}J^{-1}$, where $J$ denotes the matrix $J_{i,j}=(-1)^{n-i}\delta_{i,n-j+1}\in \mathrm{GL}_n(\mathbb{C})$. In this case, a Langlands parameter of $\mathrm{GL}_n(E)$ being $\mathrm{Gal}(E/F)$-invariant with respect to $G^{\mathrm{op}}$ means $\phi_{\pi}^{\sigma\cdot c}\cong\phi_{\pi}$, that is, $\phi_{\pi}^{\sigma}\cong \phi_{\pi}^{c}\cong\phi_{\pi^{\vee}}$. This simply means $\pi^{\sigma}\cong \pi^{\vee}$, that is, $\pi$ is conjugate self-dual.
    \end{enumerate}
In case 1, the conjugate invariant condition is also a sufficient condition for $\phi_\pi$ being a base change lift from $\mathrm{GL}_n(F)$. However, in case 2, the conjugate self-dual condition is not sufficient to determine whether $\phi_\pi$ is a (stable) base change lift from $\mathrm{U}_{n,E/F}(F)$. One needs further information, that is, the sign of the conjugate self-dual representation to determine whether it is a (stable) base change or not.
\end{example}

\subsection{The quadratic character $\omega_{G(F),E}:G(F)\rightarrow \{\pm 1\}$ associated to $E$}\label{omegaprasad}

For the quadratic extension $E/F$, Prasad constructs a quadratic character $\omega_{G(F),E}$ of $G(F)$ via the bijection in \Cref{characters}. To state his construction, we need to recall some basic facts about local class field theory and Galois cohomology.

Let $\omega_{E/F}$ be the quadratic character of $F^{\times}$ associated to the quadratic extension $E/F$ by local class field theory. More precisely, $\omega_{E/F}$ is given by the natural projection $F^{\times}\rightarrow F^{\times}/\mathrm{Nm}_{E/F}E^{\times}\cong \{\pm 1\}$. Alternatively, $\omega_{E/F}:F^\times \rightarrow \{\pm1\}$ could be described using the quadratic Hilbert symbol sending $t$ to $(t,a)$ for $a\in F$ such that $E=F(\sqrt{a})$. Although there is a uniform description of this quadratic character in terms of local class field theory, it is still convenient for us to write down the following explicit description of this quadratic character depending on whether $E/F$ is ramified or not.

When $E/F$ is unramified, we have
\begin{equation*}
\begin{aligned}
\omega_{E/F}:F^{\times}&\rightarrow \{\pm1\}\\
t&\mapsto (-1)^{v_F(t)},
\end{aligned}
\end{equation*}
where $v_F$ is the normalized valuation on $F$ such that $v_F(\varpi_F)=1$.

When $E/F$ is ramified, $a$ is chosen to be a uniformizer $\varpi_F$ of $F$. Then we have
\begin{equation*}
\begin{aligned}
\omega_{E/F}:F^{\times}&\rightarrow \{\pm1\}\\
a&\mapsto \omega_{E/F}(a)=\omega_{E/F}(-1),\\
t&\mapsto [t]^{\frac{q_F-1}{2}}:=\mathrm{sgn}(t\bmod(1+p_F)), t\in O_{F}^{\times}.
\end{aligned}
\end{equation*}

Notice that the Cartier dual of the finite group scheme $\mu_n$ over a $F$ is the constant group scheme $\mathbb{Z}/n\mathbb{Z}$. The corresponding Tate duality gives a perfect pairing
$$H^{1}(F,\mu_n)\times H^{1}(F,\mathbb{Z}/n\mathbb{Z})\rightarrow H^{2}(F,\mathbb{G}_m)\cong \mathbb{Q}/\mathbb{Z}.$$
When $n=2$, there is a natural isomorphism between $\mu_{2}(\overline{F})$ and $\mathbb{Z}/2\mathbb{Z}\cong \mu_2(\mathbb{C})=\{\pm1\}$ as trivial $\Gamma_F$-modules, where the second isomorphism is given by the exponential map $n\mapsto \exp(n\pi \sqrt{-1} )$. Together with the isomorphism $F^{\times}/(F^{\times})^{2}\cong H^{1}(F,\mu_2)$ induced by the Kummer sequence
$$1\rightarrow \mu_2\rightarrow \mathbb{G}_m\stackrel{\cdot 2}{\rightarrow} \mathbb{G}_m\rightarrow 1,$$
it is convenient for us to use the following isomorphisms to parameterize the set of isomorphism classes of quadratic \'{e}tale algebras over $F$
$$F^{\times}/(F^{\times})^{2}\cong H^{1}(F,\mu_2)\cong H^{1}(F,\{\pm 1\}).$$
For any quadratic \'{e}tale algebra $F'$ over $F$, let $[F']$ denote the corresponding element inside $H^1(F,\{\pm1\})$.

Let $S_0$ be a maximal torus of $G$ over $F$ contained in a Borel subgroup $B_0$ over $F$.
Let $\rho=\frac{1}{2}\sum\limits_{\alpha \in \Phi^{+}}\alpha$ be the half sum of positive roots of $G_{\mathrm{ad}}$, then $\rho\in X^{*}(S_{0,\mathrm{sc}})\cap \frac{1}{2}X^{*}(S_{0,\mathrm{ad}})$. Then we have the following commutative diagram
$$\xymatrix{
	1 \ar[r] & Z(G_{\mathrm{sc}}) \ar[d]^{\rho|_{Z(G_{\mathrm{sc}})}}\ar[r] & S_{0,\mathrm{sc}} \ar[d]^{\rho} \ar[r] & S_{0,\mathrm{ad}} \ar[d]^{2\rho} \ar[r] & 1  \\
	1 \ar[r] & \mu_2 \ar[r] & \mathbb{G}_m  \ar[r]^{2} & \mathbb{G}_m \ar[r] & 1 ,  }$$
which gives an exact sequence of dual Galois modules

$$\xymatrix{
	1 \ar[r] & \{\pm 1\} \ar[d]^{\widehat{2\rho}} \ar[r] & \mathbb{C}^{\times} \ar[d]^{\widehat{2\rho}} \ar[r]^{2} & \mathbb{C}^{\times} \ar[d]^{\widehat{\rho}} \ar[r] & 1  \\
	1 \ar[r] & Z(\widehat{G}_{\mathrm{sc}}) \ar[r] & \widehat{S}_{0,\mathrm{sc}}  \ar[r] & \widehat{S}_{0,\mathrm{ad}} \ar[r] & 1 . }$$
Moreover, there is a $\Gamma_F$-equivariant map
\begin{equation}\label{u}
	\begin{aligned}
		\widehat{2\rho}:\{\pm 1\}&\rightarrow Z(\widehat{G}_{\mathrm{sc}})\rightarrow Z(\widehat{G})\\
		-1&\mapsto \prod_{\alpha\in \Phi^{+}}\widehat{\alpha}(-1),
	\end{aligned}
\end{equation}
whose induced map on $H^1$ gives 
\begin{equation}\label{H1}
\begin{aligned}
\widehat{2\rho}:H^1(F,\{\pm 1\})\rightarrow H^1(F,Z(\widehat{G}_{\mathrm{sc}}))\rightarrow H^1(F,Z(\widehat{G})).
	\end{aligned}
\end{equation}

\begin{definition}\label{Prasadcharacter}
Prasad's quadratic character $\omega_{G(F),E}$ is defined to be the image of $[E]\in H^1(F,\{\pm1\})$ under $L\circ\widehat{2\rho}$, where $\widehat{2\rho}$ denotes the map in \eqref{H1} and $L$ denotes the map in \Cref{characters}.  
\end{definition}

\begin{remark}
Since $\widehat{2\rho}$ factors though $Z(\widehat{G}_{\mathrm{sc}})$, the character $\omega_{G(F),E}$ is indeed a quadratic character of $G_{\mathrm{ad}}(F)$, hence a quadratic character of $G(F)$ by inflation. Here the construction of Prasad's character differs slightly from the original one appearing in \cite{Pra01,Pra15}, which involves a choice of a regular unipotent element in $\widehat{G}_{\mathrm{sc}}$. However, it is elementary to check that these definitions coincide.
\end{remark}

From the bijection in \Cref{characters}, it is easy to see that $\omega_{G(F),E}$ is trivial if $G$ is semisimple, simply connected, since $\widehat{G}$ is adjoint so that $H^{1}(W_F,Z(\widehat{G}))$ is trivial. Moreover, the following lemma describes the relationship between  $\omega_{G(F),E}$ and $\omega_{G_{\mathrm{der}}(F),E}$.

\begin{lemma}\label{derived}
The quadratic character $\omega_{G(F),E}$ satisfies
$$\omega_{G(F),E}|_{G_{\mathrm{der}}(F)}=\omega_{G_{\mathrm{der}}(F),E}.$$
\end{lemma}
\begin{proof}
It follows from the constructions of $\omega_{G(F),E}$ and $\omega_{G_{\mathrm{der}}(F),E}$, which are both constructed from a quadratic character of $G_{\mathrm{ad}}(F)$, together with the following commutative diagram
\[
\begin{gathered}[b]
	\xymatrix{
		G_{\mathrm{der}}(F)\ar[d] \ar[r]
		&  G_{\mathrm{ad}}(F) \ar@{=}[d] \\
		G(F)  \ar[r]
		&  G_{\mathrm{ad}}(F).}\\[-\dp\strutbox]
\end{gathered}
\qedhere
\]
\end{proof}

\begin{example}\label{GL_n}

For $G=\mathrm{GL}_n$, the map $\mathbb{C}^{\times}\stackrel{\widehat{2\rho}}{\longrightarrow}{\widehat{S}_{0,\mathrm{sc}}=(\mathbb{C}^{\times}})^{n-1}$ is given by
\begin{equation*}
\begin{aligned}
\widehat{2\rho}&=\sum_{\alpha\in \Phi^{+}}\widehat{\alpha}=\sum_{1\leq i<j\leq n}(e_{i}^{\vee}-e_{j}^{\vee}),\\
&=(n-1)e_{1}^{\vee}+(n-3)e_{2}^{\vee}+\cdots+(1-n)e_{n}^{\vee}.
\end{aligned}
\end{equation*}
It is direct to see that
\begin{equation*}
\begin{aligned}
 \widehat{2\rho}(-1)=\left(\begin{matrix} (-1)^{n-1}& & & \\ & (-1)^{n-3}& & \\& & \ddots &\\  & & & (-1)^{-n+1} \end{matrix}\right)=\left(\begin{matrix} (-1)^{n-1}& & & \\ & (-1)^{n-1}& & \\& & \ddots &\\  & & & (-1)^{n-1} \end{matrix}\right),
\end{aligned}
\end{equation*}
such that the induced map on $H^1$ is given by
\begin{equation*}
\begin{aligned}
H^1(F,\{\pm 1\})&\rightarrow H^1(F,Z(\widehat{G}_{\mathrm{sc}}))\rightarrow H^1(F,Z(\widehat{G}))\rightarrow H^1(W_F,Z(\widehat{G}))\\
[E]&\mapsto \omega_{E/F}^{n-1}\circ \det.
\end{aligned}
\end{equation*}
For the quasi-split group $G=\mathrm{U}_{n,E'/F}$ with $E'/F$ quadratic, the computation is similar except that the $W_F$-action on $Z(\widehat{G})$ factors through $\mathrm{Gal}(E'/F)$, and is given by
$$\sigma\cdot z=J(z^{t})^{-1}J^{-1}=z^{-1}$$
for $\sigma\in\mathrm{Gal}(E'/F)$, which means that there is an identification $Z(\widehat{G})\cong \widehat{\mathrm{U}_{1,E'/F}}$.
By later computations in \Cref{torinorm1} and \Cref{torinorm2}, Prasad's character $\omega_{\mathrm{U}_{n,E'/F}(F),E}$ is given by
\begin{equation}\label{unitary}
\begin{aligned}
H^1(F,\{\pm 1\})&\rightarrow H^1(W_F,Z(\widehat{G}))\\
[E]&\mapsto (\omega_{\mathrm{U}_{1,E'/F}(F),E})^{n-1}\circ \det=\omega_{EE'/E'}^{n-1}\circ\iota_{E'}\circ\det,
\end{aligned}
\end{equation}
where $\iota_{E'}$ denotes the isomorphism $(E')_{E'/F}^{1}\rightarrow (E')^{\times}/F^{\times}$, and $\omega_{EE'/E'}$ denotes the trivial character if $E=E'$ and the quadratic character associated to $EE'/E'$ if $E\neq E'$.
\end{example}

In fact, Prasad's character has been computed explicitly for many examples \cites[Section 1]{BP18}[A.1]{MO22}. The above examples are also computed explicitly in \cite[A.1]{MO22} in a slightly different way. We list some of them.
\begin{example}
\quad
\begin{enumerate}
\item $G=\mathrm{SO}_{2n+1}$: the quasi-split special orthogonal group, $\omega_{G(F),E}$ is given by $\omega_{E/F}\circ \mathrm{SpinNm}$,
where $\mathrm{SpinNm}$ is the spinor norm given by the connecting homomorphism
$$\mathrm{SO}_{2n+1}(F)\rightarrow H^{1}(F,\mu_2)\cong F^{\times}/(F^{\times})^{2},$$
associated to the short exact sequence of $F$-groups
$$1\rightarrow \mu_2\rightarrow \mathrm{Spin}_{2n+1}\rightarrow \mathrm{SO}_{2n+1}\rightarrow 1.$$
\item $G=\mathrm{Sp}_{2n}$: $\omega_{G(F),E}$ is trivial, since $\mathrm{Sp}_{2n}$ is simply connected.
\item $G=\mathrm{SO}_{2n}$: the split special orthogonal group or the quasi-split but non-split orthogonal group, $\omega_{G(F),E}$ is trivial.
\end{enumerate}
\end{example}

\subsection{The restriction of $\omega_{G(F),E}$ to the elliptic maximal torus $S(F)$}

By \Cref{restrictiontotorus}, we obtain the following interpretation of restrictions of characters of $G(F)$ to $S(F)$ in terms of Galois cohomology
$$\xymatrix{
  H^{1}(W_F,Z(\widehat{G})) \ar[d] \ar[r]
                & \mathrm{Hom}_{\mathrm{cts}}(G(F),\mathbb{C}^{\times}) \ar[d]^{|_{S(F)}}  \\
  H^{1}(W_F,\widehat{S}) \ar[r]
                & \mathrm{Hom}_{\mathrm{cts}}(S(F),\mathbb{C}^{\times}) .}$$
Moreover, we have the following commutative diagram
$$\xymatrix{
  1 \ar[r] & Z(\widehat{G}_{\mathrm{sc}}) \ar[d]\ar[r] & \widehat{S}_{\mathrm{sc}} \ar[d] \ar[r] & \widehat{S}_{\mathrm{ad}} \ar@{=}[d] \ar[r] & 1  \\
  1 \ar[r] & Z(\widehat{G}) \ar[r] & \widehat{S}  \ar[r] & \widehat{S}_{\mathrm{ad}} \ar[r] & 1 ,  }$$
which gives the following commutative diagram at the level of $H^1$:
$$\xymatrix{
  H^1(W_F,Z(\widehat{G}_{\mathrm{sc}})) \ar[d]\ar[r] & H^1(W_F,Z(\widehat{G})) \ar[d]  \\
  H^1(W_F,\widehat{S}_{\mathrm{sc}}) \ar[r] & H^1(W_F,\widehat{S}).}$$
Hence the image of $[E]$ under the natural map
$$H^1(F,\{\pm 1\})\rightarrow H^1(F,Z(\widehat{G}_{\mathrm{sc}}))\rightarrow H^1(F,Z(\widehat{G}))\rightarrow H^1(F,\widehat{S})\hookrightarrow H^1(W_F,\widehat{S})$$
is the restriction of Prasad's character to $S(F)$.

Notice that the local Langlands correspondence for tori could be regarded as an isomorphism of two functors from the category of tori over $F$ to the category of abelian groups \cite[Section 7.5]{Yu09} sending a torus $S$ to
$$\mathrm{Hom}_{\mathrm{cts}}(S(F),\mathbb{C}^{\times})\cong H^1(W_F,\widehat{S}).$$
Hence it enjoys the following functorial property.

\begin{proposition}\label{Prop3.4}
Let $f:S_1\rightarrow S_2$ be a morphism of $F$-tori, which induces a morphism of $W_F$-modules $\widehat{f}:\widehat{S_2}\rightarrow \widehat{S_1}$. Then we have the following commutative diagram
$$\xymatrix{
  \mathrm{Hom}_{\mathrm{cts}}(S_{2}(F),\mathbb{C}^{\times}) \ar[d]^-{f^*} \ar[r]^-{\mathrm{LLC}}
                & H^{1}(W_F,\widehat{S_2})  \ar[d]^-{\widehat{f}}  \\
  \mathrm{Hom}_{\mathrm{cts}}(S_{1}(F),\mathbb{C}^{\times}) \ar[r]^-{\mathrm{LLC}}
                & H^{1}(W_F,\widehat{S_1}),}$$
where the first vertical map $f^*$ is induced by the composition with $f$ sending a character $\mu$ of $S_2(F)$ to $\mu\circ f$ of $S_1(F)$.
\end{proposition}
This functorial property indeed has many useful implications. One of these implications is that the quadratic character $\omega_{E/F}$ satisfies the following nice property.
\begin{lemma}\label{quadraticnorm}
Let $E'/F$ be a finite extension such that $E$ and $E'$ are disjoint. Then one has an equality of quadratic characters
$$\omega_{E/F}\circ \mathrm{Nm}_{E'/F}= \omega_{EE'/E'}.$$
\end{lemma}
\begin{proof}
	Indeed this equality follows directly from \Cref{Prop3.4}. More precisely, if we take $S_1=\mathrm{Res}_{E'/F}\mathbb{G}_{m,E'}$, $S_2=\mathbb{G}_m$ and $f=\mathrm{Nm}_{E'/F}$, then $H^1(W_F,\widehat{\mathrm{Res}_{E'/F}\mathrm{\mathbb{G}_{m,E'}}})\cong \mathrm{Hom}(W_{E'},\mathbb{C}^\times)$ and $\widehat{f}$ is given by restriction to $W_{E'}$. Under the local Langlands for tori, $\omega_{E/F}$ corresponds to the $1$-cocycle $c_{E/F}$ trivial on $W_E$. On the one hand, $c_{E/F}|_{W_{E'}}$ is a $1$-cocycle trivial on $W_{E'}\cap W_E=W_{E'E}$ so that it must be $c_{EE'/E'}$ corresponding to $\omega_{EE'/E'}$. On the other hand, the commutativity of the diagram in \Cref{Prop3.4} implies $c_{E/F}|_{W_{E'}}$ corresponds to $\omega_{E/F}\circ \mathrm{Nm}_{E'/F}$, which finishes the proof of this lemma.
\end{proof}
The following functoriality is also frequently used in the later computation.
\begin{example}
Let $L/F$ be a finite Galois extension and $S$ be a torus defined over $F$. Let $\iota$ be the natural embedding $\iota:S\hookrightarrow \mathrm{Res}_{L/F}S_L$. There is a natural identification
$$\widehat{\mathrm{Res}_{L/F}S_L}\cong \mathrm{Ind}_{W_L}^{W_F}\widehat{S}:=\{f:W_F\rightarrow \widehat{S}|\,f(ww')=wf(w'), \forall w\in W_L, w'\in W_F\},$$
such that $\widehat{\iota}$ is given by the corestriction map
\begin{equation}\label{corestriction}
\begin{aligned}
\widehat{\iota}:\mathrm{Ind}_{W_L}^{W_F}\widehat{S}&\rightarrow \widehat{S},\\
f&\mapsto \sum_{g_i\in W_L\backslash W_F}g_{i}^{-1}f(g_{i}).
\end{aligned}
\end{equation}
By Shapiro's lemma, we obtain the following commutative diagram
$$\xymatrix{
  \mathrm{Hom}_{\mathrm{cts}}(S(L),\mathbb{C}^{\times}) \ar[d]^-{|_{S(F)}} \ar[r]^-{\mathrm{LLC}}
                & H^{1}(W_L,\widehat{S})  \ar[d]^-{\text{cor}}  \\
  \mathrm{Hom}_{\mathrm{cts}}(S(F),\mathbb{C}^{\times}) \ar[r]^-{\mathrm{LLC}}
                & H^{1}(W_F,\widehat{S}).}$$
\end{example}
Moreover, if we apply \Cref{Prop3.4} to the case when $S$ is a split maximal torus of $G$ and $\alpha:S\rightarrow \mathbb{G}_m$ is a root of $S$, we can get the following factorization of the restriction of Prasad's character to the split maximal torus.

\begin{corollary}
Let $S_0$ be a split maximal torus of a split reductive group $G$. We have the following identity
$$\omega_{G(F),E}|_{S_0(F)}=\omega_{E/F}(\prod_{\alpha\in \Phi^{+}}\alpha(t)).$$
\end{corollary}

\begin{proof}
This is a direct consequence of the functorial property of the local Langlands correspondence for tori in \Cref{Prop3.4} and the explicit description of the map $\widehat{2\rho}$ in \eqref{u}.
\end{proof}

In fact, we can also obtain a similar factorization for the restriction of Prasad's quadratic character to an elliptic maximal torus $S$ of $G$. However, since there will be no Borel subgroup defined over $F$ containing $S$, we need to find certain replacement of the set of positive roots, which is given by a section of $\Phi\rightarrow \Phi/\{\pm 1\}$. In the case when $S$ is a maximal torus contained in a Borel subgroup $B$, both of which are defined over $F$, $\Phi^{+}$ could be regarded as a section of $\Phi\rightarrow \Phi/\{\pm 1\}$ determined by the $B$. If $S$ is an arbitrary maximal torus of $G$ defined over $F$, one can choose a Borel subgroup $B_{F^{s}}$ of $G$ defined over $F^{s}$ containing $S$ such that $\Phi/\{\pm 1\}$ can be identified with $\Phi^{+}_{B_{F^{s}},S_{F^{s}}}$ as a set, where the former set may have a non-trivial $\Gamma_F$-action. Let $S_{\mathrm{sc}}$ be the corresponding maximal torus in $G_{\mathrm{sc}}$ such that $S_{\mathrm{sc}}\subset B_{\mathrm{sc},F^{s}}$. The following lemma is due to Kottwitz.

\begin{lemma}[{\cites[Page 292]{Kot83}}]
The restriction of $\rho$ to $Z(G_{\mathrm{sc}})$ is independent of the choice of $(B_{F^{s}},S_{F^{s}})$, hence is defined over $F$ and $\rho|_{Z(G_{\mathrm{sc}})}$ is preserved by any automorphism of $G_{\mathrm{sc}}$.
\end{lemma}
Hence the definition of Prasad's character also does not depend on the choice of positive roots $\Phi^+$. Notice that $\Phi/\{\pm1\}$ indeed inherits a $\Gamma_F$-action given by
$$\gamma\cdot [\alpha]:=[\gamma\cdot \alpha],$$
for any $\gamma\in \Gamma_F$. We have the following identification
$$(\Phi/\{\pm1\})/\Gamma_F\leftrightarrow \Phi/(\{\pm1\}\times \Gamma_F),$$
hence the following partition
\begin{equation}\label{partition}
\begin{aligned}
\Phi/\{\pm1\}&=\bigsqcup_{\mathcal{O}_{\mathrm{sym}}\in \Phi_{\mathrm{sym}}/\Gamma_F} \mathcal{O}_{\mathrm{sym}}/\{\pm 1\}\sqcup \bigsqcup_{\mathcal{O}_{\mathrm{asym}}\in \Phi_{\mathrm{asym}}/\Gamma_F}  (\mathcal{O}_{\mathrm{asym}}\cup -\mathcal{O}_{\mathrm{asym}})/\{\pm1\},\\
&=\bigsqcup_{\mathcal{O}_{\mathrm{sym}}\in \Phi_{\mathrm{sym}}/\Gamma_F} \mathcal{O}_{\mathrm{sym}}/\{\pm 1\}\sqcup \bigsqcup_{\mathcal{O}_{\mathrm{asym}^{\pm}}\in \Phi_{\mathrm{asym}}/\Gamma_F\times \{\pm 1\}}  \mathcal{O}_{\mathrm{asym}}^{\pm},
\end{aligned}
\end{equation}
where $\mathcal{O}_{\mathrm{sym}}$, $\mathcal{O}_{\mathrm{asym}}$ and $\mathcal{O}_{\mathrm{asym}}^{\pm}$ stand for symmetric, asymmetric $\Gamma_F$-orbits and $\Gamma_F\times\{\pm1\}$-orbits for asymmetric roots. This decomposition enables us to factorize Prasad's character in terms of certain $\Gamma_F\times \{\pm 1\}$-orbits.

Before we present a general description of the restriction of Prasad's character to an elliptic maximal torus, we first prove the following lemma.
\begin{lemma}\label{torinorm1}
Let $S$ be any one-dimensional torus defined over a non-archimedean local field $F$. Consider the natural $\Gamma_F$-equivariant inclusion $\{\pm 1\}\rightarrow \widehat{S}$ (embedding of $\{\pm 1\}$ into $\mathbb{C}^{\times}$), which gives the natural map $H^1(F,\{\pm 1\})\rightarrow H^1(W_F,\widehat{S})$. Let $E$ be a quadratic extension of $F$, which is regarded as an element in $H^1(F,\{\pm1\})$. Then the image of $[E]$ under the above map is the character of $S(F)$ given by the natural map $S(F)\rightarrow S(F)/\mathrm{Nm}\,S(E)$.
\end{lemma}

\begin{proof}
First we need to prove that the map $S(F)\rightarrow S(F)/\mathrm{Nm}\,S(E)$ is indeed a character, that is, $S(F)/\mathrm{Nm}\,S(E)$ could be embedded into $\mathbb{C}^{\times}$. Notice that there are four kinds of one-dimensional tori over a $p$-adic field, when the residue characteristic is not $2$, i.e., $\mathbb{G}_m,U_{1,E_1/F},U_{1,E_2/F}$ and $U_{1,E_3/F}$, where $E_1,E_2,E_3$ are three quadratic extensions of $F$.

\begin{enumerate}[$\bullet$]
\item If $S$ is $\mathbb{G}_m$, we obtain $S(F)/\mathrm{Nm}\,S(E)=F^{\times}/\mathrm{Nm}_{E/F}E^{\times}\cong \{\pm1\}$ by local class field theory.

\item If $S$ is $\mathrm{U}_{1,E/F}$, we obtain $S(F)/\mathrm{Nm}\,S(E)=E^{1}_{E/F}/\mathrm{Nm}_{E^{\times}/E^1}E^{\times}\cong 1$ by Hilbert $90$, where $\mathrm{Nm}_{E^{\times}/E^1}$ is given by $t\mapsto \frac{t}{\sigma(t)}$.

\item If $S$ is $\mathrm{U}_{1,E_1/F}$ with $E_1\neq E$ and $K:=E_1E$, we obtain $S(F)=E_{1,E_1/F}^{1}\cong E_{1}^{\times}/F^{\times}$ and $S(E)=K_{K/E}^{1}\cong K^{\times}/E^{\times}$. Then we have the following identification $$S(F)/\mathrm{Nm}\,S(E)\cong \{\pm 1\},$$
due to the following isomorphisms
$$S(F)/\mathrm{Nm}\,S(E)\cong (E_{1}^{\times}/F^{\times})/\mathrm{Nm}(K^{\times}/E^{\times})\cong \mathrm{coker}(F^{\times}/\mathrm{Nm}_{E/F}E^{\times}\rightarrow E_{1}^{\times}/\mathrm{Nm}_{K/E_1}K^{\times})\cong\{\pm 1\}.$$
\end{enumerate}
The last isomorphism is due to the fact that the map $F^{\times}/\mathrm{Nm}_{E/F}E^{\times}\rightarrow E_{1}^{\times}/\mathrm{Nm}_{K/E_1}K^{\times}$ is trivial since $\mathrm{Nm}_{K/E_1}K^{\times}\supset\mathrm{Nm}_{E_2/F}E_{2}^{\times}\cdot\mathrm{Nm}_{E/F}E^{\times}=F^{\times}$.

Since $\Gamma_F$ acts on $\{\pm 1\}$ trivially, we also have the identifications $$H^1(F,\{\pm 1\})\cong \mathrm{Hom}(\Gamma_F,\{\pm 1\})\cong \mathrm{Hom}(F^{\times},\{\pm 1\}).$$
Then the element $[E]$ gives rise to a natural quadratic character of $\Gamma_F$ which is trivial on the index two subgroup $\Gamma_E$. Let $\phi_{\omega_{S(F),E}}$ denote the image of $[E]$ under the map $H^1(F,\mu_2)\rightarrow H^1(W_F,\widehat{S})$. Then $\phi_{\omega_{S(F),E}}$ is a cocycle which is trivial on $\Gamma_E$. Under the local Langlands correspondence for tori, $\phi_{\omega_{S(F),E}}$ corresponds to a character of $S(F)$ trivial on $\mathrm{Nm}(S(E))$, which is the natural character $\omega_{S(F),E}:S(F)\rightarrow S(F)/\mathrm{Nm}\,S(E)$.

\[
\begin{gathered}[b]
	\xymatrix{
		H^{1}(W_F,\widehat{S}) \ar[d]^{|_{W_{E}}} \ar@{<->}[r]
		& \mathrm{Hom}_{\mathrm{cts}}(S(F),\mathbb{C}^{\times}) \ar[d]^{\mathrm{Nm}_{S(E)/S(F)}^*}  \\
		H^{1}(W_E,\widehat{S}) \ar@{<->}[r]
		& \mathrm{Hom}_{\mathrm{cts}}(S(E),\mathbb{C}^{\times}) }
	\\[-\dp\strutbox]
\end{gathered}
\qedhere
\]	
\end{proof}

\begin{remark}\label{torinorm2}
The quadratic character in \Cref{torinorm1} has a more convenient description in terms of quadratic characters of the splitting field of the one-dimensional torus $S$. More precisely, let $S$ be a one-dimensional torus over $F$ with splitting field $L$. Then the image of $[E]$ is $\omega_{EL/L}|_{S(F)}$ for $S(F)\subset S(L)=L^{\times}$, where $\omega_{EL/L}$ is trivial if $E$ is contained in $L$, and is the quadratic character associated to $EL/L$ if $E$ and $L$ are disjoint. To see this, observe that there exists a natural embedding
$$S\hookrightarrow \mathrm{Res}_{L/F}S_L=\mathrm{Res}_{L/F}\mathbb{G}_{m,L},$$
with a $W_F$-equivariant morphism of dual torus given by the corestriction map in \eqref{corestriction}
$$\mathrm{Ind}_{W_L}^{W_F}\mathbb{C}^{\times}\rightarrow \widehat{S}.$$
Furthermore, there is a commutative diagram
$$\xymatrix{
  \mathrm{Ind}_{W_{L}}^{W_F}\{\pm1\}  \ar@{^{(}->}[r]
                & \mathrm{Ind}_{W_{L}}^{W_F}\mathbb{C}^{\times} \ar[d]^{\mathrm{cores}}  \\
  \{\pm1\}\ar[u]^{i} \ar@{^{(}->}[r]
                &  \widehat{S}.}$$
Applying the functor $H^1(W_F,\cdot)$, we get a commutative diagram at the level of $H^1$
$$\xymatrix{
  H^{1}(L,\{\pm1\})  \ar[r]
                & H^{1}(W_{L},\widehat{S}) \ar[d]^{\mathrm{cores}}  \\
  H^{1}(F,\{\pm1\})\ar[u] \ar[r]
                &  H^{1}(W_F,\widehat{S}).}$$
By transferring the diagram to continuous characters of tori via local Langlands correspondence for tori, we see that the image of $[E]$ is $\omega_{EL/L}|_{S(F)}$ for $S(F)\subset S(L)=L^{\times}$.
\end{remark}

\begin{example}
In fact, for any torus $S$ of the form $\mathrm{Res}_{E'/F}S^1$, where $S^{1}$ is a one-dimensional torus over $E'$, one can also construct the following map
$$H^1(F,\{\pm1\})\rightarrow H^1(E',\{\pm 1\})\rightarrow H^1(W_{E'},\widehat{S^1})\cong H^1(W_F,\widehat{\mathrm{Res}_{E'/F}S^{1}}).$$
The simplest example is when we take $S$ to be $\mathrm{Res}_{E'/F}\mathbb{G}_{m,E'}$. For disjoint $E$ and $E'$, the image of $[E]$ is the quadratic character $\omega_{EE'/E'}:{E'}^{\times}\rightarrow {E'}^{\times}/\mathrm{Nm}(EE')^{\times}\cong\{\pm 1\}$. For $E\subset E'$, the image of $[E]$ is trivial.
\end{example}

Now we give a new interpretation of the restriction of Prasad's character to $S(F)$, where $S$ is an elliptic maximal torus of $G$. This interpretation is largely inspired by the work of Kaletha and Langlands--Shelstad.

\begin{lemma}

For each absolute root $\alpha\in \Phi(G,S)$, the following diagram of $W_F$-modules is commutative:
$$\xymatrix{
  \mathrm{Ind}_{W_{F_{\alpha}}}^{W_F}\{\pm1\}  \ar[r]^{\widehat{\alpha}}
                & \mathrm{Ind}_{W_{F_{\alpha}}}^{W_F}\widehat{S} \ar[d]^{\mathrm{cores}}  \\
  \{\pm1\}\ar[u]^{i} \ar[r]^{\sum\limits_{\alpha\in \mathcal{O}}\widehat{\alpha}}
                &  \widehat{S}.}$$
\end{lemma}
\begin{proof}
Notice that the $W_F$-action on $\{\pm1\}$ is trivial. By \eqref{induced}, the natural map $i$ sends $\pm 1$ to the constant functions $\pm1$ on $W_F$. For a $\Gamma_F$-orbit $\mathcal{O}$ of a root $\alpha$, there are natural identifications $\mathcal{O}\cong \Gamma_{F_{\alpha}}\backslash\Gamma_F\cong W_{F_{\alpha}}\backslash W_F$. One can directly check the commutativity of the diagram by explicit computation
\[\sum\limits_{\alpha_i\in \mathcal{O}_{\alpha}}\widehat{\alpha_i}(\pm 1)=\sum_{g_i\in W_{F_{\alpha}}\backslash W_F}(g_i^{-1}\widehat{\alpha})(\pm 1).\qedhere\]

\end{proof}

We can apply the functor $H^1(W_F,\cdot)$ to the above commutative diagram so that we get a commutative diagram at the level of $H^1$
\begin{equation}
\begin{aligned}\label{orbitcharacter}
\xymatrix{
  H^{1}(F_{\alpha},\{\pm1\})  \ar[r]^{\widehat{\alpha}}
                & H^{1}(W_{F_{\alpha}},\widehat{S}) \ar[d]^{\mathrm{cores}}  \\
  H^{1}(F,\{\pm1\})\ar[u] \ar[r]^{\sum\limits_{\alpha\in \mathcal{O}}\widehat{\alpha}}
                &  H^{1}(W_F,\widehat{S}).}
\end{aligned}
\end{equation}
Moreover, if $\mathcal{O}$ is symmetric, we also have
\begin{equation}
\begin{aligned}\label{symmetricorbitcharacter}
\xymatrix{
  H^{1}(F_{\pm\alpha},\{\pm 1\})  \ar[r]^{\widehat{\alpha}}
                & H^{1}(W_{F_{\pm\alpha}},\widehat{S}) \ar[d]^{\mathrm{cores}}  \\
  H^{1}(F,\{\pm1\})\ar[u] \ar[r]^{\sum\limits_{\alpha\in \mathcal{O}/\{\pm 1\}}\widehat{\alpha}}
                &  H^{1}(W_F,\widehat{S}).}
\end{aligned}
\end{equation}

For any $\Gamma_F$-orbit $\mathcal{O}$ of an absolute root $\alpha\in \Phi(G,S)$, the bottoms of \eqref{orbitcharacter} and \eqref{symmetricorbitcharacter} define quadratic characters of $S(F)$ associated to the quadratic extension $E/F$ by 
\begin{equation*}
\begin{aligned}
(L\circ (\sum_{\alpha\in \mathcal{O}/\{\pm 1\}}\widehat{\alpha}))([E]) &\quad \text{if }\alpha\text{ is symmetric},\\
(L\circ (\sum_{\alpha\in \mathcal{O}}\widehat{\alpha}))([E]) &\quad \text{if }\alpha\text{ is asymmetric},
\end{aligned}
\end{equation*}
where $L$ denotes the map in \Cref{characters}. Actually, with the help of the commutative diagrams \eqref{orbitcharacter} and \eqref{symmetricorbitcharacter}, we can compute these characters explicitly in the following way.
\begin{enumerate}
\item
   $\alpha\in \Phi(G,S)\leftrightarrow \widehat{\alpha}\in \Phi^{\vee}(\widehat{G},\widehat{S})$ is symmetric.

   Let $S_\alpha$ be the one-dimensional anisotropic torus over $F_{\pm\alpha}$, whose $F_{\pm\alpha}$-points correspond to the norm-one elements of $F_{\alpha}$ with respect to $F_{\pm\alpha}$. A $\Gamma_F\times \{\pm 1\}$-orbit of $\alpha$ induces a natural map $S(F_{\pm\alpha})\rightarrow S_{\alpha}(F_{\pm\alpha})$. Then we consider the following diagram

$$\xymatrix{
  H^{1}(W_{F_{\pm\alpha}},\widehat{S_{\alpha}}) \ar[d]^{\widehat{\alpha}} \ar@{<->}[r]
                & \mathrm{Hom}_{\mathrm{cts}}((F_{\alpha}^{1}),\mathbb{C}^{\times}) \ar[d]^{\circ\alpha}  \\
  H^{1}(W_{F_{\pm\alpha}},\widehat{S}) \ar@{<->}[r]
                & \mathrm{Hom}_{\mathrm{cts}}(S(F_{\pm\alpha}),\mathbb{C}^{\times}) .}$$

The character of $S(F)$ associated to $E$ and the $\Gamma_F$-orbit of $\alpha$ is defined by the image of $[E]$ under the following composition of maps
\begin{equation*}
H^1(F,\{\pm 1\})\stackrel {|_{\Gamma_{F_{\pm\alpha}}}}{\rightarrow} H^{1}(F_{\pm\alpha},\{\pm1\})\rightarrow  H^{1}(W_{F_{\pm\alpha}},\widehat{S_{\alpha}})\stackrel{\widehat{\alpha}}{\rightarrow} H^{1}(W_{F_{\pm\alpha}},\widehat{S})\stackrel{\mathrm{cores}}{\rightarrow}H^{1}(W_{F},\widehat{S}),
\end{equation*}
where the second map is the one appearing in \Cref{torinorm1}. More explicitly, the above composite map sends $[E]$ to $(\omega_{S_{\alpha}(F_{\pm\alpha}),E}\circ\alpha)|_{S(F)}$, where $\omega_{S_{\alpha}(F_{\pm\alpha}),E}$ is the natural map
\begin{equation*}
F_{\alpha,F_{\alpha}/F_{\pm\alpha}}^{1}\rightarrow
F_{\alpha,F_{\alpha}/F_{\pm\alpha}}^{1}/\mathrm{Nm}_{EF_{\alpha}/F_{\alpha}}((EF_{\alpha})^{1}_{EF_{\alpha}/EF_{\pm\alpha}}).
\end{equation*}

Moreover, we also have the following commutative diagram
\begin{equation*}
\begin{aligned}
\xymatrix{
  (EF_{\alpha})^{\times}/(EF_{\pm\alpha})^{\times} \ar[r]^-{\mathrm{Nm}_{E/F}} & F_{\alpha}^{\times}/F_{\pm\alpha}^{\times}  \ar[r]^-{\omega_{EF_{\alpha}/F_{\alpha}}}  & \{\pm 1\}  \ar@{=}[d]\\
  (EF_{\alpha})_{EF_{\alpha}/EF_{\pm\alpha}}^{1} \ar[u]^{\iota_{E_\alpha}}\ar[r]^-{\mathrm{Nm}_{EF_{\alpha}/F_{\alpha}}} & {F_{\alpha}^{1}}_{F_{\alpha}/F_{\pm\alpha}} \ar[u]^{\iota_{F_\alpha}}\ar[r]^-{\omega_{S_{\alpha}(F_{\pm\alpha}),E}} & \{\pm 1\}, }
\end{aligned}
\end{equation*}
where $\iota_{F_\alpha}$ denotes the isomorphism $F_{\alpha}^{1}\rightarrow F_{\alpha}^{\times}/F_{\pm\alpha}^{\times}$.
Hence we obtain
\begin{equation}\label{sym}
(\omega_{S_{\alpha}(F_{\pm\alpha}),E}\circ\alpha)|_{S(F)}=(\omega_{EF_{\alpha}/F_{\alpha}}\circ\iota_{F_\alpha}\circ\alpha)|_{S(F)}.
\end{equation}

\item $\alpha$ is asymmetric.

The character of $S(F)$ associated to $E$ and the $\Gamma_F$-orbit of $\alpha$ is defined by the image of $[E]$ under the following composition of maps
\begin{equation*}
H^1(F,\{\pm 1\})\stackrel {|_{\Gamma_{F_{\alpha}}}}{\rightarrow} H^{1}(F_{\alpha},\{\pm1\})\rightarrow  H^{1}(W_{F_{\alpha}},\widehat{\mathbb{G}_m})\stackrel{\widehat{\alpha}}{\rightarrow} H^{1}(W_{F_{\alpha}},\widehat{S})\stackrel{\mathrm{cores}}{\rightarrow}H^{1}(W_{F},\widehat{S}),
\end{equation*}
where the second map is also the one appearing in \Cref{torinorm1}. More explicitly, the composite map sends $[E]$ to $(\omega_{F_{\alpha}^{\times},E}\circ\alpha)|_{S(F)}$.

By \cite[Page 6]{Kal19a}, any root $\alpha$ gives a morphism $S\rightarrow \mathrm{Res}_{F_{\alpha}/F}\mathbb{G}_m$, and summing over the $\Gamma_F$-orbit of $\alpha$ gives a morphism
$$S\rightarrow \mathrm{Res}_{F_{\alpha}/F}\mathbb{G}_m \stackrel{\mathrm{Nm}_{F_{\alpha}/F}}{\longrightarrow}\mathbb{G}_{m}$$
defined over $F$. (In fact these two morphisms are trivial on $Z(G)$, hence define a morphism from $S/Z(G)$ to $\mathbb{G}_m$ over $F$, which is in fact trivial due to the ellipticity of $S$.) This produces a morphism of dual tori
$$\mathbb{C}^{\times}\stackrel{\prod\limits_{\mathcal{O}} \widehat{\alpha}}{\longrightarrow}\widehat{S}.$$
Notice that the $W_F$-map
$$\{\pm1\}\stackrel{\prod\limits_{\mathcal{O}}\widehat{\alpha}}{\longrightarrow}\widehat{S}$$
factors through the above map, which implies the character $(\omega_{F_{\alpha}^{\times},E}\circ\alpha)|_{S(F)}$ factors through
\begin{equation}\label{asym}
S(F)\stackrel{\prod\limits_{\mathcal{O}} \alpha}{\rightarrow} F^{\times}.
\end{equation}
Hence the character is trivial.

\end{enumerate}

\begin{theorem}\label{rootcharacter}
The restriction of Prasad's character to an elliptic maximal torus $S(F)$ inherits the following factorization:
$$\omega_{G(F),E}|_{S(F)}(t)=\prod_{\alpha\in \Phi_{\mathrm{sym}}/\Gamma_F}\omega_{EF_{\alpha}/F_{\alpha}}(\iota_{F_\alpha}\alpha(t)).$$
\end{theorem}

\begin{proof}
By \Cref{restrictiontotorus}, the restriction of Prasad's character to $S(F)$ is given by the image of $[E]$ under the following map
$$H^1(F,\{\pm 1\})\stackrel {\sum\limits_{\alpha\in \Phi/\{\pm 1\}}\widehat{\alpha}}{\longrightarrow} H^{1}(W_F,\widehat{S}).$$
By \Cref{partition}, we obtain
$$\sum\limits_{\alpha\in \Phi/\{\pm 1\}}\widehat{\alpha}=\sum_{\mathcal{O}_{\mathrm{sym}}\in \Phi_{\mathrm{sym}}/\Gamma_F}\sum\limits_{\alpha\in \mathcal{O}_{\mathrm{sym}}/\{\pm 1\}}\widehat{\alpha}+\sum_{\mathcal{O}_{\mathrm{asym}}\in \Phi_{\mathrm{asym}}/\Gamma_F}\sum\limits_{\alpha\in (\mathcal{O}_{\mathrm{asym}}\cup -\mathcal{O}_{\mathrm{asym}})/\{\pm1\}}\widehat{\alpha}.$$
Since each summand is defined over $F$, we get the following factorization by \eqref{sym} and \eqref{asym}
\begin{align*}
\omega_{G(F),E}|_{S(F)}&=\prod_{\mathcal{O}_\alpha\in \Phi_{\mathrm{sym}}/\Gamma_F}(\omega_{S_{\alpha}(F_{\pm\alpha}),E}\circ\alpha)|_{S(F)}\cdot \prod_{\mathcal{O}_\alpha\in \Phi_{\mathrm{asym}}/\Gamma_F}\mathbbm{1}\\
&=\prod_{\alpha\in \Phi_{\mathrm{sym}}/\Gamma_F}\omega_{EF_{\alpha}/F_{\alpha}}(\iota_{F_\alpha}\circ\alpha).\qedhere
\end{align*}

\end{proof}

One advantage of our formulation is that we can compute the restriction of Prasad's character to an elliptic torus $S$ according to the property of absolute roots associated to $S$.

\begin{example}\label{examplegln}

For $G=\mathrm{GL}_n$, computations in \Cref{GL_n} have already shown that
$$\omega_{G(F),E}=\left\{
             \begin{array}{lr}
             \mathbbm{1} & n \text{ is odd},\\
             \omega_{E/F}\circ\mathrm{det} & \quad  n \text{ is even}.\\
            \end{array}\right.$$

For applications in the distinction problem, we give a new way of computing $\omega_{G(F),E}|_{S(F)}$ based on \Cref{rootcharacter}, where $S$ is a maximal elliptic torus of $\mathrm{GL}_n$ of the form $\mathrm{Res}_{E'/F}\mathbb{G}_m$ for a degree $n$ extension $E'/F$ such that $S\times_{F}E$ is still an elliptic torus of $\mathrm{GL}_{n,E}$. In this case, we know that $E\nsubseteq E'$. One can easily see that $\mathrm{det}|_{S(F)}=\mathrm{Nm}_{E'/F}$, which means that 
$$\omega_{G(F),E}|_{S(F)}=\left\{
             \begin{array}{lr}
             \mathbbm{1} & n \text{ is odd},\\
             \omega_{E/F}\circ\mathrm{Nm}_{E'/F}=\omega_{E'E/E'} & \quad n \text{ is even}.\\
            \end{array}\right.$$
From our interpretation, we know that $F_{\alpha}\cong E'$ for each root $\alpha\in \Phi(G,S)$. Since each $\Gamma_F$-orbit of roots has exactly $n$ roots, the number of $\Gamma_F$-orbits equals $\frac{n(n-1)}{n}=n-1$.
Hence we know
\[\#\{\Gamma_F \text{-orbits of symmetric roots}\}+2\#\{\Gamma_F \text{-orbits of asymmetric roots} \}=\#\{\Gamma_F \text{-orbits}\}=n-1,\]
which means that
\begin{align*}
\#\{\Gamma_F \text{-orbits of symmetric roots}\}=n-1(\bmod2).
\end{align*}

By \Cref{rootcharacter}, we can deduce
\begin{equation*}
\begin{aligned}
\omega_{G(F),E}|_{S(F)}(t)&=\prod_{\Phi_{\mathrm{sym}}/\Gamma_F}\omega_{\alpha}(\iota_{EF_\alpha/F_{\alpha}}\alpha(t))=\omega_{E'E/E'}^{n-1}(t)\\
&=\left\{
\begin{array}{lr}
1 & n \text{ is odd},\\
\omega_{E'E/E'}(t) & \quad n \text{ is even}.\\
\end{array}\right.
\end{aligned}
\end{equation*}
\end{example}

\subsection{Prasad's conjecture for regular supercuspidal representations}

Although Prasad makes his conjecture for an arbitrary irreducible admissible representation of $G(E)$ in a generic $L$-packet, the conjecture has a simpler form for regular supercuspidal representations. Let $\pi$ be a regular supercuspidal representation of $G(E)$ and $(\phi_{\pi},\lambda_{\pi})$ be its Langlands--Vogan parameter. Namely, $\phi_{\pi}$ is the Langlands parameter of $\pi$ and $\lambda_{\pi}$ is an irreducible representation of the component group $\pi_{0}(S_{\phi_{\pi}})$. Let $G^{\mathrm{op}}$ be the quasi-split $F$-form of $G$ defined in \Cref{Gop}, and let $\omega_{G(F),E}$ be the quadratic character of $G(F)$ associated to $E$ defined in \Cref{omegaprasad}.
\begin{conjecture}\label{conjecture1}
Fix a Whittaker datum $\mathfrak{w}=(B_0,\psi_{N_0})$ of $G(E)$ such that $\psi_{N_0}:N_0(E)\rightarrow \mathbb{C}^{\times}$ satisfies $\psi_{N_0}|_{N_0(F)}=\mathbbm{1}$. Let $\pi$ be a regular supercuspidal representation of $G(E)$ lying in the generic $L$-packet with Langlands--Vogan parameter $(\phi_{\pi},\lambda_{\pi})$. Then 
\begin{equation}\label{prasadidentity}
\begin{aligned}
\sum_{\alpha\in H^1(\mathrm{Gal}(E/F),G(E))}\dim\mathrm{Hom}_{G_{\alpha}(F)}(\pi,\omega_{G_{\alpha}(F),E})=\sum\limits_{\tilde{\phi}}\dim \mathrm{Hom}_{\pi_{0}(S_{\widetilde{\phi}})}(\mathbbm{1},\lambda_\pi),
\end{aligned}
\end{equation}
where the sum of the right-hand side runs over all the Langlands parameters $\tilde{\phi}:W_F\rightarrow \prescript{L}{}G^{\mathrm{op}}$ such that $\tilde{\phi}|_{W_{E}}=\phi_\pi$.
\end{conjecture}


We can simplify the parameter side of this identity based on the nice parametrization properties of regular supercuspidal representations, which will be described in detail in \Cref{RHSofprasad}.

\begin{remark}\label{regularinfla}
The inflation restriction sequence in \eqref{inflationrestriction} already implies that the set of possible extensions of $\phi_\pi$ can be identified with
$$H^1(\mathrm{Gal}(E/F),{}_{\phi_{\pi}}\widehat{G^{\mathrm{op}}}^{W_E})\cong \ker\left(H^1(W_F,\widehat{G^{\mathrm{op}}})\rightarrow H^1(W_E,{}_{\phi_{\pi}}\widehat{G})\right).$$

By \eqref{twistedaction}, under the twisted (by $\phi_\pi$) $W_E$-action, there is an identification ${}_{\phi_{\pi}}\widehat{G^{\mathrm{op}}}^{W_E}=S_{\phi_\pi}$. In fact, for a regular supercuspidal parameter $\phi_\pi$, Lemma 5.3.4 in \cite{Kal19b} implies $$S_{\phi_\pi}=\widehat{j}(\widehat{T}^{\Gamma_E})=\widehat{j}(\widehat{S^{\mathrm{op}}}^{\Gamma_E}).$$
However, there may exist different $F$-structures on $S^{\mathrm{op}}$ (hence different $\Gamma_F$-structure on $\widehat{S^{\mathrm{op}}}$) such that $\widehat{S^{\mathrm{op}}}^{\Gamma_E}$ is $S_{\phi_{\pi}}$, and we should count them all. If we fix an $F$-structure on $S^{\mathrm{op}}$, then the set of possible extensions of $\phi_\pi$ can be identified with the pointed set
$$H^1(\mathrm{Gal}(E/F),\widehat{S^{\mathrm{op}}}^{W_E})$$
under the above isomorphism. We will give a comparison between the above parametrization of possible extensions of $\phi_\pi$ and the one obtained from the base change of regular supercuspidal $L$-packet data (\Cref{basechangeregular}) in \Cref{compareextension}.
\end{remark}

\section{Hakim--Murnaghan's formula for distinguished regular supercuspidal representations}\label{section6}

\subsection{Yu's construction of tamely ramified supercuspidal representations}\label{Yuconstruction}

Let $G$ be a reductive group over a non-archimedean local field $F$ of residual characteristic $p\neq 2$, such that it splits over a tamely ramified extension of $F$. Yu develops a general method to construct tame supercuspidal representations of $G(F)$ using cuspidal data. We give a brief review of the construction here. For more details, one can refer to \cite{Yu01, HM08, Fin21a}. For convenience, we fix an additive character $\psi:F\rightarrow \mathbb{C}^{\times}$ of conductor $p_F$, i.e., $\psi|_{p_F}=\mathbbm{1},\psi|_{O_F}\neq\mathbbm{1}$.

A generic cuspidal datum is a $5$-tuple:
$$(\overrightarrow{G},x,\overrightarrow{r},\rho,\overrightarrow{\phi}),$$
where
\begin{enumerate}
\item $\overrightarrow{G}=(G^0,\cdots,G^d=G)$ is a sequence of tamely ramified twisted Levi subgroups of $G$, such that $Z(G^0)/Z(G)$ is anisotropic.
\item $x\in \mathcal {B}(G^{0},F)$ is a point in the enlarged Bruhat--Tits building of $G^0$ and $[x]$ denotes its image in the reduced building.
\item $\overrightarrow{r}=(r_0,\cdots,r_d)$ is a sequence of real numbers, such that $0<r_0<\cdots<r_{d-1}\leq r_d$ if $d>0$, and $r_0\geq 0$ if $d=0$.
\item $\rho$ is an irreducible representation of $G^{0}(F)_{[x]}$ such that $\rho|_{G^{0}_{x,0^{+}}}$ is $\mathbbm{1}$-isotypic and $\rho|_{G^{0}_{x,0}}$ contains an inflation of a cuspidal representation of $G^{0}_{x,0}/G^{0}_{x,0^{+}}$, and also $\pi_{-1}=\mathrm{c}\mbox{-}\mathrm{Ind}_{G^{0}(F)_{[x]}}^{G^{0}(F)}\rho$ is irreducible. (The last condition is also equivalent to the condition that $\pi_{-1}$ is a depth-0 supercuspidal representation of $G^{0}(F)$.)
\item $\overrightarrow{\phi}=(\phi_0,\cdots,\phi_d)$ is a sequence of quasi-characters $\phi_i$ of $G^{i}(F)$ of depth $r_i$, which are $G^{i+1}(F)$-generic in the sense of Yu. More precisely, $\phi_{i}|_{G^{i}(F)_{x,r_{i}+}}$ is trivial and there exists a $G^{i+1}(F)$-generic element $X_{i}^{*}\in (\mathfrak{z}^{i})^{*}_{-r_i}$ such that
\begin{equation}\label{genericelement}
\phi_i(\mathrm{MP}(Y+\mathfrak{g}^{i}(F)_{x,r_{i}+}))=\psi(\langle X_{i}^{*},Y\rangle),
\end{equation}
for any $Y\in \mathfrak{g}_{x,r_i}^{i}(F)$, where $\mathrm{MP}$ denotes the Moy--Prasad isomorphism
\[\mathrm{MP}:\mathfrak{g}^{i}(F)_{x,r_i}/\mathfrak{g}^{i}(F)_{x,r_i+}\cong G^{i}(F)_{x,r_i}/G^{i}(F)_{x,r_i+}.
\]

\end{enumerate}

To understand Yu's construction, we first need to understand the representation theory of abstract Heisenberg $p$-groups and Yu's special isomorphism, which is crucial in the whole construction. Let $(W,\langle \cdot,\cdot\rangle)$ be a symplectic vector space over a finite field $k$ such that $\mathrm{char}\,k\neq 2$, and let $H(W)$ be the usual Heisenberg group associated to $W$. More precisely, $H(W)$ is identified with $W\oplus k$ with the multiplication given by
$$(w_1,a_1)\cdot(w_2,a_2)=(w_{1}+w_{2},a_1+a_2+\frac{1}{2}\langle w_1,w_2\rangle).$$

\begin{definition}[{\cites[Section 10]{Yu01}}]
A Heisenberg $p$-group is an abstract group $\mathcal{H}$ which is isomorphic to a Heisenberg group $H(W)$ associated to a symplectic vector space $W$ over $\mathbb{F}_p$.
\end{definition}
For an abstract Heisenberg $p$-group $\mathcal{H}$ with center $\mathcal{Z}$, let $\mathcal{W}$ denote $\mathcal{H}/\mathcal{Z}$. The commutator map defines a symplectic form on $\mathcal{W}$ by
\begin{equation}
\begin{aligned}
\mathcal{H}/\mathcal{Z}\times \mathcal{H}/\mathcal{Z}&\rightarrow \mathcal{Z}\\
(h_1\mathcal{Z},h_2\mathcal{Z})&\mapsto \langle h_1\mathcal{Z},h_2\mathcal{Z}\rangle:=h_{1}h_{2}h_{1}^{-1}h_{2}^{-1}.
\end{aligned}
\end{equation}
\begin{definition}
A special isomorphism on $\mathcal{H}$ is an isomorphism $\nu:\mathcal{H}\rightarrow H(\mathcal{W})$ such that the following diagram commutes:
$$\xymatrix{
  1 \ar[r] & \mathcal{Z} \ar@{=}[d] \ar[r] & \mathcal{H}\ar[d]_{\nu} \ar[r] & \mathcal{W} \ar@{=}[d] \ar[r] & 1  \\
  1 \ar[r] & \mathcal{Z} \ar[r] & H(\mathcal{W}) \ar[r] & \mathcal{W} \ar[r] & 1.   }$$
\end{definition}
Fix a non-trivial additive character $\psi_{\mathbb{F}_p}:\mathbb{F}_p\rightarrow \mathbb{C}^{\times}$. The Heisenberg--Weil representation of $\mathrm{Sp}(\mathcal{W})\ltimes_{\nu}\mathcal{H}$ is the pull back of the usual Heisenberg--Weil representation of $\mathrm{Sp}(\mathcal{W})\ltimes H(\mathcal{W})$ associated to $\psi_{\mathbb{F}_p}$.

The explicit construction of the tamely ramified supercuspidal representations \cite{Yu01} could be described as follows.
Let $S_{\mathrm{ms}}$ be a maximal $F$-split torus of $G^{0}$, let $S_{\mathrm{mus}}$ be an $F$-torus containing $S_{\mathrm{ms}}$ such that $S_{\mathrm{mus}}\times_{F}F^{\mathrm{ur}}$ is maximal $F^{\mathrm{ur}}$-split, and let $S'=C_{G^0}(S_{\mathrm{mus}})$ be the centralizer of $S_{\mathrm{mus}}$ in $G^0$. Since $G^0$ is quasi-split over $F^{\mathrm{ur}}$, $S'$ is a maximal torus of $G^0$ defined over $F$ with maximal $F$-split subtorus $S_{\mathrm{ms}}$. Let $L/F^{\mathrm{ur}}$ be the splitting field of $S'$ such that $L/F$ is a tamely ramified extension. We have canonical inclusions of apartments:
$$\mathcal{A}(G^0,S_{\mathrm{ms}},F)\subset \mathcal{A}(G^0,S_{\mathrm{mus}},F^{\mathrm{ur}})\subset \mathcal{A}(G^0,S',L). $$
When $G^0$ is quasi-split over $F$, then $S'$ could be chosen to be the maximal torus $S^{0}_{0}$, such that $(B^{0}_{0},S^{0}_{0},\{X^{0}_{\alpha}\})$ determines the pinning of $G^0$ over $F$.

For an arbitrary maximal torus $S\subset G^0$ defined over $F$, which splits over a tamely ramified extension $L/F$, we adopt the notation in \cite{Yu01} and \cite{HM08}.
Let $\mathcal{A}(G^0,S,F)$ denote the intersection
$$\mathcal{A}(G^0,S,F):=\mathcal{A}(G^0,S,L)\cap \mathcal{B}(G^0,F)$$
in $\mathcal{B}(G^0,L)$.
Since $L/F$ is tame, by Galois descent on the Bruhat--Tits building, there is a natural identification
$$\mathcal{A}(G^0,S,F)\cong \mathcal{A}(G^0,S,L)^{\mathrm{Gal}(L/F)}.$$

Fix a cuspidal datum $(\overrightarrow{G},x,\overrightarrow{r},\rho,\overrightarrow{\phi})$ such that $x\in \mathcal{A}(G^0,S,F)$ for a tamely ramified maximal torus $S\subset G^0$, and set $s_i:=\frac{r_i}{2}$. Yu defines a family of subgroups of $G(F)$:
$$K^{i}:=G^{0}(F)_{[x]}(G^{0}(F),\cdots,G^{i}(F))_{x,(0+,s_0,\cdots,s_{i-1})},$$
$$K^{i}_{+}:=G^{0}(F)_{x,0+}(G^{0}(F),\cdots,G^{i}(F))_{x,(0+,s_0,\cdots,s_{i-1})},$$
$$J^{i}:=(G^{i-1}(F),G^{i}(F))_{x,(r_{i-1},s_{i-1})},$$
$$J^{i}_{+}:=(G^{i-1}(F),G^{i}(F))_{x,(r_{i-1},s_{i-1}+)}.$$
These subgroups satisfy the following relations:
$$K^{i}=K^{i-1}J^{i}=K^{i-1}G^{i}(F)_{x,s_{i-1}},$$
$$K^{i}_{+}=K^{i-1}_{+}J^{i}_{+}=K^{i-1}_{+}G^{i}(F)_{x,s_{i-1}+}.$$
Since $\phi_i$ is of depth $r_i$, it is trivial on $G^{i}(F)_{x,r_{i}+}$, so that $\phi_i|_{G^{i}(F)_{x,s_{i}+}}$ factors through a character of $G^{i}(F)_{x,s_{i}+}/G^{i}(F)_{x,r_{i}+}$. We can get a character of $G(F)_{x,s_{i}+}$ by inflation
$$\mathrm{inf}_{G^{i}(F)_{x,s_{i}+}}^{G(F)_{x,s_{i}+}}(\phi_i)=\mathrm{inf}_{G^{i}(F)_{x,s_{i}+}/G^{i}(F)_{x,r_{i}+}}^{G(F)_{x,s_{i}+}/G(F)_{x,r_{i}+}}(\phi_i|_{G^{i}(F)_{x,s_{i}+}}).$$
Let $\widehat{\phi_i}$ be the unique character of $G^{0}(F)_{[x]}G^{i}(F)_{x,0}G(F)_{x,s_{i}+}$ such that
$$\widehat{\phi_i}|_{G^{0}(F)_{[x]}G^{i}(F)_{x,0}}=\phi_i|_{G^{0}(F)_{[x]}G^{i}(F)_{x,0}},$$
$$\widehat{\phi_i}|_{G(F)_{x,s_{i}+}}=\mathrm{inf}_{G^{i}(F)_{x,s_{i}+}}^{G(F)_{x,s_{i}+}}(\phi_i).$$
Notice that $J^{i}/\mathrm{ker}(\widehat{\phi_i}|_{J^{i}_{+}})$ has a natural structure of abstract Heisenberg group over $\mathbb{F}_p$, such that $J^{i}/J^{i}_{+}$ has a natural structure of symplectic vector space over $\mathbb{F}_p$. There exists an irreducible representation $\widetilde{\phi_i}$ of $K^{i}\ltimes J^{i+1}$, which corresponds to the Heisenberg representation via the special isomorphism such that
\begin{enumerate}
\item the restriction of $\widetilde{\phi_i}$ to $J^{i+1}_{+}=1\ltimes J^{i+1}_{+} $ is $\widehat{\phi_i}|_{J^{i+1}_{+}}$-isotypic,
\item the restriction of $K^{i}_{+}\ltimes 1$ is $\mathbbm{1}$-isotypic.
\end{enumerate}
It is also proved that $\mathrm{inf}(\phi_i)\otimes \widetilde{\phi_i}$ factors through the map
$$K^{i}\ltimes J^{i+1}\rightarrow K^{i+1}=K^{i}J^{i+1},$$
where $\mathrm{inf}(\phi_i)$ is the inflation of $\phi_{i}|_{K^{i}}$ to $K^{i}\ltimes J^{i+1}$.
Let $\phi_{i}^{'}$ be the representation of $K^{i+1}$ such that its inflation to $K^{i}\ltimes J^{i+1}$ is $\inf(\phi_i)\otimes \widetilde{\phi_i}$. One can produce a sequence of representations $\{\kappa_i\}$ of $K^{d}$ by inflation. More precisely, we have
$\kappa_{-1}=\mathrm{inf}_{G^{0}(F)_{[x]}}^{K^d}\rho$, $\kappa_{i}=\mathrm{inf}_{K^{i+1}}^{K^{d}}(\phi_{i}^{'})$ for $i\leq d-1$, and $\kappa_d=\phi_d|_{K^d}$.
Then one can further produce a representation of $K^{d}$ by
$$\rho_d=\kappa_{-1}\otimes\cdots\otimes \kappa_{d}.$$
By \cite{Yu01} and \cite{Fin21a}, the representation $\pi:=\mathrm{c}\mbox{-}\mathrm{Ind}_{K^d}^{G(F)}\rho_d$ is irreducible, hence a supercuspidal representation of $G(F)$.

\begin{example}

An important class of such representations are the so-called toral supercuspidal representations, that is, the representations with Yu data: $$(\overrightarrow{G}=(S,G),\mathbbm{1},(\theta,\mathbbm{1})),$$
where $S$ is an elliptic maximal torus of $G$ and $\theta$ is a character of $S(F)$. As remarked by \cite{Kal19b}, when the residual characteristic $p$ is good, these representations are already general enough to include a class of supercuspidal representations called epipelagic representations, which is first constructed by Reeder and Yu \cite{RY14} using stable vectors in geometric invariant theory. Later, Kaletha \cite{Kal15} proves that their constructions could be recovered from Adler's construction \cite{Adl98} for good $p$.

\end{example}

Based on Yu's work, Kaletha provides a strategy to construct regular supercuspidal representations from a tame elliptic pair $(S,\mu)$ satisfying certain regularity conditions. We give a short review of his construction here. The key point to recover Yu's data from $(S,\mu)$ is the following existence result of Howe factorization proved by Kaletha.

Let $L$ be the splitting field of $S$. One can produce a sequence of Levi subsystems of $\Phi(S,G)$ from $\mu$ for a sequence of positive real numbers $\overrightarrow{r}$ by
\begin{equation*}
\begin{aligned}
\Phi_{r_i}:=\{\alpha\in \Phi(S,G)|\,\mu(\mathrm{Nm}_{S(L)/S(F)}(\alpha^{\vee}(L_{r_i}^{\times})))=1\}.
\end{aligned}
\end{equation*}

Set $\Phi_{r+}:=\bigcap\limits_{s>r}\Phi_s$. One can define $G^{i}$ to be the connected reductive subgroups of $G$ with a maximal torus $S$ and root system $\Phi_{r+}$. We also set $G^{-1}:=S$.

\begin{definition}[{\cites[Definition 3.6.2]{Kal19b}}]\label{howefactorization}
A Howe factorization of $(S,\mu)$ is a sequence of characters $\phi_i:G^{i}(F)\rightarrow \mathbb{C}^{\times}$ for $i=-1,\cdots,d$ such that
\begin{enumerate}
\item $\mu=\prod\limits_{i=-1}^{d}\phi_{i}|_{S(F)}$.
\item $\phi_i$ is trivial on $G^{i}_{\mathrm{sc}}(F)$.
\item For $0\leq r\leq d-1$, $\phi_i$ is $G^{i+1}$-generic of depth $r_i$. For $i=d$, $\phi_d$ is the trivial character if $r_d=r_{d-1}$, and has depth $r_d$ if $r_d\neq r_{d-1}$. For $i=-1$, $\phi_{-1}$ is the trivial character if $G^{0}=S$, and is a depth zero character if $G^{0}\neq S$. In this case we usually denote it by $\mu_{-1}:S(F)\rightarrow \mathbb{C}^{\times}$ instead of $\phi_{-1}$.
\end{enumerate}
\end{definition}

\begin{proposition}[{\cites[Proposition 3.6.7]{Kal19b}}]

Any pair $(S,\mu)$ consisting of a tame elliptic torus $S\subset G$ and a character $\mu:S(F)\rightarrow \mathbb{C}^{\times}$ has a Howe factorization $(\phi_{-1},\cdots,\phi_d)$.
\end{proposition}

If $G^0\neq S$, then $S$ is a maximally unramified elliptic maximal torus of $G^0$ \cite[Definition 3.4.2]{Kal19b}, and $\mu_{-1}:S(F)\rightarrow \mathbb{C}^{\times}$ is a regular depth zero character such that $\mu_{-1}|_{S(F)_{0}}$ factors through $\mathsf{S}'(k_F)\cong S(F)_0/S(F)_{0^{+}}$. Let $\kappa_{\mathsf{S}',\mu_{-1}}$ denote the inflation of the irreducible Deligne--Lusztig cuspidal representation $\pm R_{\mathsf{S}',\mu_{-1}}$ of $\mathsf{G}^{0,\circ}(k_F)$ to $G^{0}(F)_{x,0}$. Let $\overline{\kappa_{\mathsf{S}',\mu_{-1}}}$ denote the canonical extension of $\kappa_{\mathsf{S}',\mu_{-1}}$ to $S(F)G^{0}(F)_{x,0}$ established by Kaletha with the help of the $\mathsf{S}'_{\mathrm{ad}}(k_F)$-action on the corresponding Deligne--Lusztig variety. By \cite[Proposition 3.4.27]{Kal19b}, $\rho$ could be chosen to be $\mathrm{Ind}_{S(F)G^{0}(F)_{x,0}}^{G^{0}(F)_{[x]}}\overline{\kappa_{\mathsf{S}',\mu_{-1}}}$.

Hence, the above process associates a generic cuspidal datum $(\overrightarrow{G},x,\overrightarrow{r},\rho,\overrightarrow{\phi})$ to $(S,\mu)$. One can then get a regular tamely ramified supercuspidal representation of $G(F)$ from Yu's construction.

\subsection{Hakim--Murnaghan's formula}

For a tame supercuspidal representation $(\pi,V)$ of $G'(F)$ with a fixed Yu-datum, Hakim and Murnaghan \cite{HM08} establish a formula to describe the dimension of the space of $(G')^{\vartheta}(F)$ invariant linear forms on $V$ for an involution $\vartheta:G' \rightarrow G'$ defined over $F$. We can apply their machinery to the case when $G'$ is $\mathrm{Res}_{E/F}G_E$ and $\vartheta$ is the $F$-involution induced by $\sigma\in \mathrm{Gal}(E/F)$ (also denoted by $\sigma$). Indeed their formula is a refinement of the classical Mackey formula

\begin{equation}
\begin{aligned}\label{naivemackey}
\mathrm{Hom}_{G(F)}(\pi,\mathbbm{1})&=\mathrm{Hom}_{G(F)}(\mathrm{c}\mbox{-}\mathrm{Ind}_{K^d}^{G(E)}\kappa,\mathbbm{1})\\
&=\bigoplus\limits_{K^{d}gG(F)\in K^{d}\backslash G(E)/G(F)}\mathrm{Hom}_{K^{d}\cap gG(F)g^{-1}}(\kappa,\mathbbm{1}).\\
\end{aligned}
\end{equation}
Observe that there is a natural left $G'(F)=G(E)$-action on the set of $F$-involutions of $G'$ given by $g\cdot\vartheta:=\mathrm{Ad}(g)\circ\vartheta \circ \mathrm{Ad}(g^{-1})$ for any $F$-involution $\vartheta$ and any $g\in G$. By \cite[Proposition 5.9]{HM08}, if $\mathrm{Hom}_{K^{d}\cap gG(F)g^{-1}}(\kappa,\mathbbm{1})\neq 0$, one can pick a suitable involution $\vartheta'$ in the same $K$-orbit with $g\cdot\sigma$ such that the generic cuspidal datum associated to $\pi$ is $\vartheta'$-symmetric \cite[Definition 3.13]{HM08}. We may simply denote this involution to be $g\cdot \sigma$. Then all these open compact modulo center subgroups $K^d$ are $\vartheta$-stable. Furthermore, one has the following identifications of subgroups
\[(K^d)^{ g\cdot \sigma}=g(g^{-1}K^dg)^\sigma g^{-1}=g(g^{-1}K^dg\cap G(F)) g^{-1}=K^d\cap gG(F)g^{-1}.\]
Then they have the following refined Mackey formula
\begin{equation}\label{HMoriginal}
\begin{aligned}
\dim\mathrm{Hom}_{G(F)}(\pi,\mathbbm{1})&=\sum\limits_{\Theta'\in\{K^d\text{-orbits in the }G(E)\text{-orbit of }\sigma\}}m(K,\Theta')\dim\mathrm{Hom}_{(K^{d})^{\sigma'}}(\kappa,\mathbbm{1})\\
&=\sum\limits_{\Theta'\in\{K^0\text{-orbits in the }G(E)\text{-orbit of }\sigma\}}m(K,\Theta')\dim\mathrm{Hom}_{(K^{0})^{\sigma'}}(\rho\otimes \prod_{i=0}^{d}(\phi_i|_{K^{0}}),\epsilon_{K^0,\sigma'}),
\end{aligned}
\end{equation}
where $m(K,\Theta')$ is a certain multiplicity \cite[Lemma 2.6]{HM08}, $\sigma'=g\cdot \sigma$ (for some $g\in G(E)$) is an $F$-involution in $\Theta'$, and $\epsilon_{K^0,\sigma'}$ is a certain quadratic character of ${K^{0}}^{\sigma'}$. Moreover, they also explicitly describe the relevant condition (i.e., the condition of double coset or equivalently $\Theta'$ such that $\dim\mathrm{Hom}_{K^{\theta'}}(\kappa,\mathbbm{1})\neq 0$). In our reformulation of their formula, these multiplicities $m(K,\Theta')$ will never appear so that we will not mention explicit definitions of them.

\begin{remark}
In order to fit better with the ``moduli interpretation" of the set of double cosets, we prefer to use $G(F)\backslash G(E)/K^d$ instead of $K^{d}\backslash G(E)/G(F)$ in our reformulation, which can be achieved simply by twisting with the inverse map. By doing this, \Cref{naivemackey} becomes
$$\bigoplus\limits_{K^{d}gG(F)\in K^{d}\backslash G(E)/G(F)}\mathrm{Hom}_{K^{d}\cap gG(F)g^{-1}}(\kappa,\mathbbm{1})=\bigoplus\limits_{G(F)gK^d\in G(F)\backslash G(E)/ K^{d}}\mathrm{Hom}_{K^{d}\cap g^{-1}G(F)g}(\kappa,\mathbbm{1}).$$
Then we consider the right action of $G(E)$ on the set of $F$-involution given by $\vartheta\cdot g:=\mathrm{Ad}(g^{-1})\circ\theta\circ\mathrm{Ad}(g)$ instead of the left action $g\cdot \vartheta$ defined earlier. We already see that Hakim and Murnaghan's formula involves the fixed-point subgroup of an open compact modulo center subgroup $K_0=G^0(E)_{[x]}$ under a twisted involution $g\cdot \sigma $. In terms of our right action, the identifications become
\[(G^0(E)_{[x]})^{\sigma\cdot g}=g^{-1}(gG^0(E)_{[x]}g^{-1})^\sigma g=g^{-1}(gG^0(E)_{[x]}g^{-1}\cap G(F)) g=G^0(E)_{[x]}\cap g^{-1}G(F)g.\]
Moreover, if the generic cuspidal datum is $\sigma\cdot g$ symmetric, namely,  $G^0(E)$ and $[x]$ are both $\sigma\cdot g$ stable, this subgroup has a simpler description due to the compatibility between the right $G(E)$-action on the set of involutions and the $G(E)$-action on the Bruhat--Tits building. More precisely, for any $x\in \mathcal{B}(G^0,E)\subset \mathcal{B}(G,E), g\in G(E)$, one has $(\sigma\cdot g)(x)=x$ if and only if $\sigma(gx)=gx$. To see the coincidence of the above two characterization, let $\mathcal{A}(G^0_E,T_{\mathrm{ms}},E)$ be an apartment which $x$ lies in. Then the $G(E)$-action on $\mathcal{B}(G,E)$ $x\mapsto gx$ is indeed induced by the conjugate action of $G(E)$ on the set of apartments (maximal $F$-split tori) $T_{\mathrm{ms}}\mapsto \mathrm{Ad}(g)(T_{\mathrm{ms}})$ for any $g\in G(E)$. Hence $(\sigma\cdot g)(x)=x$ if and only if $(\mathrm{Ad}(g^{-1})\circ\sigma\circ\mathrm{Ad}(g))(x)=x$ which is equivalent to $\sigma (gx)=gx$ by the above compatibility of actions. In this case, the group $G(E)_{[x]}$ is $\sigma\cdot g$-stable and $(G^0(E)_{[x]})^{\sigma\cdot g}=g^{-1}((gG^0(E)g^{-1})_{[gx]})^\sigma g.$

\end{remark}

Recently, Hakim \cite{Hak18} transforms the relevant condition on $K^{0}$-orbits of involutions to $j(T)(E)$-orbits of involutions, that is, from the double coset $G(F)\backslash G(E)/K^0$ to the double coset $G(F)\backslash G(E)/j(T)(E)$ for regular supercuspidal representations, and gets a refined formula for these representations. We remark here that the set of $K^d$-orbits of involutions within a fixed right $G(E)$-orbit of $\sigma$ can be identified with $G_{\sigma}\backslash G(E)/K^d$, where $G_\sigma$ is the stabilizer of $\sigma$ in $G(E)$. It is clear that $G_{\sigma}$ contains $G(F)$ and the map $K^{d}\backslash G(E)/G(F)\rightarrow K^{d}\backslash G(E)/G_{\sigma}$ is finite \cite[Lemma 2.6]{HM08}. However, in our reinterpretation, we prefer to use the double coset $K^d\backslash G(E)/G(F)$ instead of $K^{d}\backslash G(E)/G_{\sigma}$ for later Galois cohomological interpretation, which also avoids the factor $m(K^d,\Theta')$. The following theorem (in a slightly different form) is proved by Hakim \cite{Hak18} for regular supercuspidal representations.

\begin{theorem}[{\cites[Theorem 3.4.1]{Hak18}}]\label{hakimcharacter}
Let $\pi$ be a regular supercuspidal representation of $G(E)$ attached to a tame regular elliptic pair $(j_{\pi}(T),\theta)$. Then
$$\dim\mathrm{Hom}_{G(F)}(\pi,\mathbbm{1})=\sum_{\{\substack{[g]\in G(F)\backslash G(E)/j_\pi(T)(E),\\ \mathrm{Ad}(g)(j_\pi(T)(E))\text{ is }\sigma\mbox{-}\text{stable}}\}}\dim\mathrm{Hom}_{gj_\pi(T)(E)g^{-1}\cap G(F)}(\theta\circ \mathrm{Ad}(g)^{-1},\epsilon_{\mathrm{HM}}).$$
\end{theorem}

\noindent More precisely, the quadratic character $\epsilon_{\mathrm{HM}}:i_g(S)(F)=gj_{\pi}(T)(E)g^{-1}\cap G(F)\rightarrow \{\pm 1\}$ has a factorization $$\epsilon_{\mathrm{HM}}(s)=\epsilon^{-}_{i_g(S)}(s)\cdot\epsilon^{+}_{G^0(E)_{[x]}, [g]}|_{i_g(S)}(s),$$
such that both $\epsilon^{-}_{i_g(S)}(s)$ and $\epsilon^{+}_{G^0(E)_{[x]}, [g]}$ are quadratic characters. More explicitly, for any $s\in i_g(S)(F)$ and $h\in (\mathrm{Ad}(g)G^{0}(E))_{[gx]}^\sigma$, these quadratic characters are defined by
\[\epsilon^{-}_{i_g(S)}(s):=\det(\mathrm{Ad}(s)_{\mathfrak{g}^{0}(E)_{[gx]}^\sigma}),\]
\[\epsilon^{+}_{G^{0}(E)_{[x]},[g]}(h)=\prod\limits_{i=0}^{d-1}\mathrm{sgn}_{k_{F}^{\times}}(\det \mathrm{Ad}(h)|_{\mathfrak{W}_i^+}).\]
The first $\mathrm{Ad}$ denotes the adjoint action of $i_g(S)$ on ${}^g\mathsf{g}^0(k_E)^\sigma:= (\mathrm{Ad}(g)(\mathfrak{g}^{0}(E))_{gx,0:0+})^\sigma$, where $\mathfrak{g}^{0}(E)_{gx,0:0+}$ denotes the Moy--Prasad quotient of the Lie algebra $\mathfrak{g}^{0}(E)_{gx,0}$ of $G^{0}(E)_{[gx]}$. The second $\mathrm{Ad}$ denotes the adjoint action of $(\mathrm{Ad}(g)G^{0}(E))_{[gx]}^\sigma$ on the $k_F$-vector space 
$${}^g\mathfrak{W}_i^+:=(\mathrm{Ad}(g)(\bigoplus\limits_{\alpha\in\Phi\left(G^{i+1},\mathrm{Ad}(g)(j_{\pi}(T))\right)-\Phi\left(G^{i},\mathrm{Ad}(g)(j_{\pi}(T))\right)}\mathfrak{g}_{\alpha})^{\Gamma_E})_{gx,s_i:s_{i}+}^{\sigma}.$$ Indeed the character $\epsilon^{+}_{G^{0}(E)_{[x]},[g]}$ is the quadratic character $\epsilon_{K^0,\sigma'}\circ \mathrm{Ad}(g^{-1})$ in \Cref{HMoriginal}.

\subsection{Some facts about $p$-adic tori}
Before we give a reinterpretation of the relevant condition on the set of $T(E)$-orbits in the $G(E)$-orbit of the Galois involution $\sigma$, we first review some basic facts about $p$-adic tori.

\noindent$\bullet$ \textbf{Conjugacy classes}.

\noindent For a connected quasi-split reductive group $G$ defined over $F$, let $A$ be a maximal $F$-split torus of $G$. Then $S:=C_G(A)$ is a maximal torus of $G$ defined over $F$. We set $S$ as a base point in the set of all maximal tori defined over $F$. Let $W$ denote the corresponding Weyl group. 

\begin{definition}
Let $S_1$ and $S_2$ be two maximal tori of $G$ defined over $F$. They are called rationally conjugate if they are conjugate by an element in $G(F)$. They are called stably conjugate if $S_1(F)$ and $S_2(F)$ are conjugate by an element in $G(\overline{F})$.
\end{definition}
These conjugacy classes have the following Galois cohomological parameterization.
\begin{proposition}[{\cites[Lemma 6.2]{Ree11}}]
The set of rational conjugacy classes of maximal tori of $G$ defined over $F$ is in bijection with the pointed set $\ker \left(H^1(F,N_{G}(S))\rightarrow H^1(F,G)\right)$.
\end{proposition}
Explicitly, the bijection is given by
\begin{equation*}
\begin{aligned}
\left\{\begin{array}{l}
             \text{Rational conjugacy classes of}\\
             \text{maximal tori defined over }F
            \end{array}\right\}& \longleftrightarrow \ker\left(H^1(F,N_{G}(S))\rightarrow H^1(F,G)\right)\\
S_g&\longmapsto (\gamma \mapsto g^{-1}\gamma(g)),
\end{aligned}
\end{equation*}
where $S_g$ is a maximal torus of $G$ such that $S_g(\overline{F})=gS(\overline{F})g^{-1}$.

In fact this identification can also be read from the long exact sequence associated to the short exact sequence of Galois pointed sets
$$ 1\rightarrow N_{G}(S)\rightarrow G\rightarrow G/N_{G}(S)\rightarrow 1,$$
where the set of rational conjugacy classes of maximal tori defined over $F$ can be identified with the set 
$$(G/N_{G}(S))(F)/G(F)\cong \ker\left(H^1(F,N_{G}(S))\rightarrow H^1(F,G)\right).$$
In fact, the variety $G/N_{G}(S)$ is known as the variety of maximal tori, whose rational points $(G/N_{G}(S))(F)$ parameterize the set of maximal tori of $G$ defined over $F$.

\begin{proposition}[{\cites{Rag04}[Proposition 6.1]{Ree11}}]

The set of stable conjugacy classes of maximal tori defined over $F$ is in bijection with the pointed set $H^1(F,W)$ and there is a surjective map $$\pi:\ker\left(H^1(F,N_{G}(S))\rightarrow H^1(F,G)\right)\rightarrow H^{1}(F,W)$$
induced by the natural $F$-homomorphism $N_G(S)\rightarrow W$, such that each fiber consists of the set of rational conjugacy classes of maximal tori inside a stable conjugacy class of maximal tori.
\end{proposition}

Now we fix a cocycle $x:\Gamma_F\rightarrow W$ with $[x]\in H^1(F,W)$. Let $S_{[x]}$ be the $F$-torus corresponding to $[x]$ (regarded as an element in $H^1(F,\mathrm{Aut}(S))$ by the natural map $W\rightarrow \mathrm{Aut}(S)$). Observe that there is an $F$-isomorphism $\mathrm{Ad}(g_{[x]}): S_{[x]}\cong g_{[x]}Sg_{[x]}^{-1}$, where $g_{[x]}$ lies in $G(\overline{F})$, such that the cocycle $[\gamma \mapsto g_{[x]}^{-1}\gamma(g_{[x]})]$ in $\ker (H^1(F,N_G(S))\rightarrow H^1(F,G))$ projects to $[x]$. Hence we get an embedding of $S_{[x]}$ into $G$ for any choice of $g_{[x]}$, such that $S_{[x]}$ is a maximal torus of $G$. Moreover, we can also define twisted groups $N_{[x]}$ and $G_{[x]}$ such that there is a commutative diagram 
\[\xymatrix{
	H^1(k,N_{[x]}) \ar[d]_{t_{[x]}}\ar[r] & H^1(k,G_{[x]})\ar[d]_{\mathrm{Ad}(g_{[x]})} \\
	H^1(k,N) \ar[r] &H^1(k,G),}\]
where $t_{[x]}$ denotes the twisting bijection \cite[5.3]{Ser02} and induces isomorphisms
\begin{equation}\label{twistmap}
\begin{aligned}
t_{[x]}:\ker \left(H^1(F,N_{G_{[x]}}(S_{[x]}))\rightarrow H^1(F,G_{[x]})\right)&\cong\ker\left (H^1(F,N_{G}(S))\rightarrow H^1(F,G)\right),\\
t_{[x]}:H^1(F,W_{[x]})&\cong H^1(F,W):\quad[y]\mapsto [y\cdot x].
\end{aligned}
\end{equation}
By the short exact sequence of $F$-groups
$$1\rightarrow S_{[x]}\rightarrow N_{G_{[x]}}(S_{[x]})\rightarrow W_{[x]}\rightarrow 1,$$
we obtain a long exact sequence
$$1\rightarrow N_{G_{[x]}}(S_{[x]})(F)/S_{[x]}(F)\rightarrow W_{[x]}(F)\rightarrow H^1(F,S_{[x]})\rightarrow H^{1}(F,N_G(S_{[x]}))\rightarrow H^1(F,W_{[x]})\rightarrow\cdots,$$ hence also a long exact sequence
\begin{equation*}
\begin{aligned}
1&\rightarrow N_{G}(S_{[x]})(F)/S_{[x]}(F)\rightarrow W_{[x]}(F)\rightarrow \ker\left(H^1(F,S_{[x]})\rightarrow H^{1}(F,G)\right)\\
&\rightarrow \ker\left(H^1(F,N_{G}(S_{[x]}))\rightarrow H^{1}(F,G)\right)\rightarrow H^1(F,W_{[x]})\rightarrow\cdots.
\end{aligned}
\end{equation*}
where the last map is given by $\pi_{t_{[x]}}:=t_{[x]}^{-1}\circ\pi\circ t_{[x]}:\ker\left(H^1(F,N_{G_{[x]}}(S_{[x]}))\rightarrow H^1(F,G)\right)\rightarrow H^1(F,W_{[x]})$. This further implies that there exists a bijection between pointed sets $$\pi^{-1}([x])=\pi^{-1}_{t[x]}(1)\cong \ker\left(H^1(F,S_{[x]})\rightarrow H^{1}(F,G)\right)/W_{[x]}(F).$$ 
Indeed in \cite[Section 6.2]{Ree11}, Reeder also constructs an explicit affine action of $W_{[x]}(F)$ on
$\ker\left(H^1(F,S_{[x]})\rightarrow H^{1}(F,G)\right)$ such that the long exact sequence gives a bijection
\[\pi^{-1}([x])\leftrightarrow \{W_{[x]}(F)\text{-orbits in }\ker\left(H^1(F,S_{[x]})\rightarrow H^{1}(F,G)\right)\}.\]
In a recent preprint of Debacker \cite{Deb23}, he alternatively provides a parameterization of rational conjugacy classes of tamely-ramified maximal tori in a reductive $p$-adic group via Bruhat--Tits theory.

\noindent$\bullet$ \textbf{Rational Embeddings}.

Fix a stable conjugacy class of maximal tori $[x]\in H^1(F,W)$. There is another interpretation of the set
\begin{equation*}
\begin{aligned}
\ker\left(H^1(F,S_{[x]})\rightarrow H^{1}(F,G)\right)&\cong (G/S_{[x]})(F)/G(F)\\
&\cong\{G(F)gS_{[x]}(\overline{F}) \in G(F)\backslash G(\overline{F})/S_{[x]}(\overline{F})|\,g^{-1}\gamma(g)\in S_{[x]}(\overline{F})\}.
\end{aligned}
\end{equation*}

\begin{lemma}\label{embedding}
This set parametrizes rational equivalence classes of admissible $F$-embeddings of a maximal torus $S_{[x]}$ into $G$, that is,
$$\{j|\,j:S_{[x]}\rightarrow G \text{ is an admissible embedding over }F\}/G(F).$$
(two embeddings $j_1$ and $j_2$ are equivalent if there exists some $h\in G(F)$ such that $j_2=\mathrm{Ad}(h)\circ j_1$)
\end{lemma}

\begin{proof}
If we write $j(S_{[x]})=g S_{[x]}g^{-1}$ for some $g\in G(\overline{F})$, we not only have $g^{-1}\gamma(g)\in N_{G}(S_{[x]})(\overline{F})$ (this is the condition for $j(S_{[x]})$ to be defined over $F$), but also have $g\gamma(t)g^{-1}=\gamma(gtg^{-1})$ for any $\gamma\in \Gamma_F$ and any $t\in S(\overline{F})$ (this is the condition for $j$ to be defined over $F$, that is, $j(\gamma(t))=\gamma(j(t))$), which means $g^{-1}\gamma(g)\in C_{G}(S)(\overline{F})=S(\overline{F})$.
\end{proof}

\begin{remark}\label{invofembedding}
In the above lemma, we have already assumed that $S_{[x]}$ is a maximal torus of $G$, which means that we have already chosen a base point of an $F$-embedding of the underlying abstract torus of $S_{[x]}$ to $G$. Moreover, if $S$ is an abstract torus with a fixed $G(F)$-conjugacy class of an embedding $j_0:S\rightarrow G$ such that $j_0(S)$ is a maximal torus of $G$, then the set $\{j|\,j:S\rightarrow G \text{ is an admissible embedding over }F\}/G(F)$ is a torsor under $\ker\left(H^1(F,S)\stackrel{{j_0}_*}{\longrightarrow} H^{1}(F,G)\right)$ by sending $j$ to $\mathrm{inv}(j,j_0):=[\gamma\mapsto g^{-1}\gamma(g)]$ for $g\in G(\overline{F})$ such that $j(S)=gj_0(S)g^{-1}$.
\end{remark}

Notice that the set $$\bigsqcup\limits_{[x]\in H^1(F,W)}\ker\left(H^1(F,S_{[x]})\rightarrow H^{1}(F,G)\right)$$ parameterizes all the $G(F)$-equivalence classes of all possible admissible embeddings of maximal tori into $G$. More precisely, this set is in bijection with the set of equivalence classes of the pairs $(S,j)$ where $S$ is an $F$-torus with $\dim_{F}S=\mathrm{rank}_{\overline{F}}G$ and $j:S\rightarrow G$ is a $G(F)$ equivalence class of $F$-embeddings. $(S_1,j_1)$ is equivalent to $(S_2,j_2)$ if and only if $S_1\cong S_2$ and there exists $h\in G(F)$ such that $j_2=\mathrm{Ad}(h)\circ j_1$.

\begin{lemma}[{\cites[Lemma 3.2.1]{Kal19b}}]\label{elliptic}
If $S'$ is an elliptic maximal torus of $G$, then $H^1(F,S')\stackrel{{j_0}_{*}}{\longrightarrow} H^1(F,G)$ is surjective.
\end{lemma}

Based on this lemma, we can see that the pointed set $H^1(F,S)$ with $S$ being an elliptic maximal torus of $G$ can be identified with the set of rational conjugacy classes of $F$-embeddings of $S$ into all pure inner forms $G_{\alpha}$ of $G$ for $\alpha\in H^1(F,G)$. For a regular supercuspidal parameter $\phi$, we have already seen in \Cref{comparisonl} that there is a bijection
$$\Pi_{\phi}=\bigsqcup\limits_{\alpha\in H^1(F,G)} \Pi_{\phi}(G_{\alpha}(F))\longleftrightarrow \mathrm{Irr}(\pi_{0}(S_{\phi}))\cong \pi_{0}(\widehat{S}^{\Gamma_{F}})^{*}\cong H^1(F,S).$$
What's more, we also have the following refined Proposition
$$\Pi_{\phi}(G_{\alpha}(F))\longleftrightarrow ({j_0}_*)^{-1}(\alpha)\longleftrightarrow \{G_{\alpha}(F) \text{ equivalence classes of admissible }F\mbox{-}\text{embeddings of }S\text{ into }G_{\alpha}\}.$$

\subsection{Reinterpretation of the indexing set appearing in Hakim--Murnaghan's formula}

Let $\sigma\cdot G(E)=\{\mathrm{Ad}(g^{-1})\circ\sigma\circ\mathrm{Ad}(g)\}_{g\in G(E)}$ denote the right $G(E)$-orbit of the Galois involution $\sigma$. For a regular supercuspidal representation $\pi$ associated to $(T,\theta,j_{\pi})$ with $T$ an elliptic torus over $E$ such that $\dim_E(T)=\mathrm{rank}_{\overline{F}}(G)$ and $j_{\pi}(T)$ is a maximal torus of $G_E$, the relevant condition for $\pi$ being distinguished becomes that $j_{\pi}(T(E))$ is stable under $\sigma\cdot g$ for some $g\in G(E)$ by \cite{Hak18}. This condition is equivalent to the condition $\sigma(gj_{\pi}(T(E))g^{-1})=gj_{\pi}(T(E))g^{-1}$, which means $gj_{\pi}(T)g^{-1}$ is defined over $F$. Let $S_g$ denote this $F$-torus with $S_g(E)\cong T(E)$ as locally compact topological groups.

\begin{lemma}\label{bcindividual}
There is a bijection between the relevant double coset $$\{G(F)gj_{\pi}(T(E))\in G(F)\backslash G(E)/j_{\pi}(T(E))|\, \mathrm{Ad}(g)(j_{\pi}(T(E)))\text{ is }\sigma\mbox{-}\text{stable}\}$$
and the set $\mathrm{BC}^{-1}([j_{\pi}])$, where $\mathrm{BC}$ is defined to be
\begin{equation*}
\begin{aligned}
\bigsqcup\limits_{S\times_{F}E=T}\{i:S\rightarrow G\text{ admissible embedding}/F\}/G(F)&\stackrel{\mathrm{BC}}{\longrightarrow} \{i_E:T\rightarrow G_E \text{ admissible embedding}/E\}/G(E)\\
[i]&\mapsto [i\times_{F}E].
\end{aligned}
\end{equation*}
Let $j_0$ be a fixed base point of an embedding of an abstract torus $T$ to $G_E$ over $E$. Then $j_{\pi}$ could be regarded as an element of $H^{1}(E,T)$ via the map $j_{\pi}\mapsto\mathrm{inv}(j_{\pi},j_0)$. Furthermore, if $j_{\pi}$ is fixed as the base point in $\ker\left(H^1(E,T)\rightarrow H^1(E,G)\right)$, this set can also be identified with the set
$$\bigsqcup\limits_{S\times_{F}E=T}\mathrm{ker}\left(H^1(\mathrm{Gal}(E/F),S(E))\rightarrow H^1(\mathrm{Gal}(E/F),G(E))\right).$$
\end{lemma}
\begin{proof}
Notice that $S_g:=gj_{\pi}(T)g^{-1}$ is $\sigma$-stable hence defined over $F$ with $S_g(E)\cong T(E)$. We partition the relevant double coset according to the isomorphism class of the $F$-structure of $S_g$. The bijection is given by sending $[g]$ to $[{i}_{g}]:S_g\rightarrow G/F$ such that $\mathrm{Im}(i_g)=gj(T)g^{-1}$. It is clear from definition that $[i_g]$ only depends on $[g]\in G(F)\backslash G(E)/j_{\pi}(T(E))$ and $\mathrm{BC}([i_g])=[j_{\pi}]$. It is also direct to check this map is indeed a bijection. The second statement follows from Proposition \eqref{embedding}, Lemma \eqref{elliptic} and the following commutative diagram for elliptic tori, which is due to the snake lemma:
$$\xymatrix{
  1 \ar[r] & \ker(\mathrm{BC}|_{S})\ar[d] \ar[r] & H^1(\mathrm{Gal}(E/F),S(E)) \ar[d] \ar[r] & H^1(\mathrm{Gal}(E/F),G(E)) \ar[d] &  \\
  1 \ar[r] & \ker^1(F,S,G) \ar[d]_{\mathrm{BC}|_{S}} \ar[r] & H^1(F,S) \ar[d]_{\mathrm{res}} \ar[r]&  H^1(F,G)\ar[d]_{\mathrm{res}} \ar[r] & 1  \\
  1 \ar[r] & \ker^1(E,T,G)\ar[r] & H^1(E,T) \ar[r] & H^{1}(E,G) \ar[r] & 1, }$$
where $\ker^1(F,S,G)$ is the short notation for $\ker \left(H^1(F,S)\rightarrow H^1(F,G)\right)$.
We remark here that all rows and columns are exact in this diagram and we are applying the snake lemma to the bottom two rows of the diagram. It is known that all these Galois cohomology sets are pointed sets endowed with a compatible structure of abelian groups so that all maps between them are group homomorphisms \cites[Corollary 5.4.1]{Bor98}[Theorem 1.2]{Kot86}, which ensures the validity of applying the snake lemma for abelian groups.
\end{proof}

We can also consider the above BC map for pure inner forms $G_{\alpha}$ with $\alpha\in H^1(\mathrm{Gal}(E/F),G(E))$, which is denoted by $\mathrm{BC}_\alpha$ with $\mathrm{BC}_1:=\mathrm{BC}$. These maps can be glued to a map
\begin{equation}\label{totalbc}
\overline{\mathrm{BC}}:\bigsqcup\limits_{S\times_{F}E=T}\bigsqcup\limits_{\alpha\in H^1(\mathrm{Gal}(E/F),G(E))}({j_\pi}_{*})^{-1}(\alpha)\rightarrow \ker \left(H^1(E,T)\rightarrow H^1(E,G)\right),
\end{equation}
which coincides with the usual restriction map
$$\bigsqcup\limits_{S\times_{F}E=T}H^1(F,S)\rightarrow H^1(E,T).$$
Moreover, we get a decomposition
\begin{equation}\label{sumoverinnertwist}
\bigsqcup\limits_{\alpha\in H^1(\mathrm{Gal}(E/F),G(E))}\mathrm{BC}_{\alpha}^{-1}([j_{\pi}])=\overline{\mathrm{BC}}^{-1}([j_{\pi}]).
\end{equation}
In fact, if we choose $j_{\pi}$ to be the base point in $\ker\left(H^1(E,S)\rightarrow H^1(E,G)\right)$, there is a natural identification
\begin{equation}\label{preimageofbc}
\overline{\mathrm{BC}}^{-1}([j_{\pi}])\longleftrightarrow \bigsqcup \limits_{S\times_{F} E=T}H^1(\mathrm{Gal}(E/F),S(E)).
\end{equation}
In terms of this interpretation, Hakim--Murnaghan's formula becomes
\begin{equation}\label{HMformularein}
\begin{aligned}
\dim\mathrm{Hom}_{G(F)}(\pi,\mathbbm{1})&=\dim\mathrm{Hom}_{G(F)}(\mathrm{ind}_{K}^{G(E)}\kappa,\mathbbm{1})\\
&=\sum\limits_{KgG(F)\in K\backslash G(E)/G(F)}\dim\mathrm{Hom}_{K\cap gG(F)g^{-1}}(\kappa,\mathbbm{1})\\
&=\sum\limits_{\{[g]\in G(F)\backslash G(E)/j_{\pi}(T(E)):\mathrm{relevant}\}}\dim \mathrm{Hom}_{gj_{\pi}(T(E))g^{-1}\cap G(F)}(\theta\circ j_\pi^{-1}\circ \mathrm{Ad}(g^{-1}) ,\epsilon_{\mathrm{HM},[g]})\\
&=\sum\limits_{[i]\in \mathrm{BC}^{-1}([j_{\pi}])}\dim \mathrm{Hom}_{i(S(F))}( \theta\circ [j_\pi]^{-1},\epsilon_{\mathrm{HM},i(S(F))}).
\end{aligned}
\end{equation}

\subsection{$z$-extension and a reduction step}\label{zextension}
Since we are dealing with the distinction problem with respect to a specific quadratic character which may be nontrivial, we can't apply Hakim--Murnaghan's formula directly. An obstruction is that $\omega_{G(F),E}$ may not always be extended to be a character of $G(E)$.
If such an extension $\widetilde{\omega}$ always exists, then for any irreducible representation $\pi$ of $G(E)$, we have an isomorphism
\[\mathrm{Hom}_{G(F)}(\pi,\omega_{G(F),E})\cong \mathrm{Hom}_{G(F)}(\pi\otimes \widetilde{\omega}^{-1},\mathbbm{1}),\]
so that the latter space can be understood via Hakim--Murnaghan's formula. Unfortunately, such phenomenon even happens when $G$ is $\mathrm{PGL_2}$.
\begin{example}
In the case of $\mathrm{PGL}_2$, Prasad's character is given by 
\[\omega_{\mathrm{PGL}_2(F),E}=\omega_{E/F}\circ \mathrm{det}_F:\mathrm{PGL}_2(F)\rightarrow F^\times/(F^{\times})^2\rightarrow F^{\times}/\mathrm{Nm}_{E/F}E^\times\cong \{\pm1\}.\]
We claim that $\omega_{E/F}$ can not be extended to a character of $\mathrm{PGL}_2(E)$ when $E/F$ is ramified and $q_E=3 \bmod 4$. Observe that $[\mathrm{PGL}_2(E),\mathrm{PGL}_2(E)]= \ker \left(\det_E:\mathrm{PGL}_2(E)\rightarrow E^\times/(E^{\times})^2\right)$. Hence every character of $\mathrm{PGL}_2(E)$ factors through a character of $E^\times/(E^\times)^2$. If such an extension $\omega_{\mathrm{PGL}_2(E),E}$ exists, then there exists a character $\widetilde{\omega}$ of $E^\times/(E^\times)^2$ such that $\omega_{\mathrm{PGL}_2(E),E}=\widetilde{\omega}\circ \det$. Restricting to $\mathrm{PGL}_2(F)$, one sees that $\widetilde{\omega}$ as a character of $E^\times$ extends the quadratic character $\omega_{E/F}$ of $F^\times$. Choose $a$ to be a uniformizer of $F$ such that $E=F(\sqrt{a})$. Since $\widetilde{\omega}$ is trivial on $(E^\times)^2$, we know that $1=\widetilde{\omega}(\sqrt{a}^2)=\widetilde{\omega}(a)=\omega_{E/F}(a)=\omega_{E/F}(-1)=(-1)^{\frac{q_E-1}{2}}=-1,$ which is a contradiction. The main reason for invalidity of extension in this example is that the map $F^{\times}/(F^{\times})^2\rightarrow E^{\times}/(E^{\times})^2$ induced by the inclusion $F^\times\hookrightarrow E^\times$ is not always injective so that one cannot always extend a character of $F^{\times}/(F^{\times})^2$ to a character of $E^{\times}/(E^{\times})^2$. Nonetheless, a character of $F^{\times}/(F^{\times})^2$ regarded as a character of $F^{\times}$ can always be extended to a character of $E^\times$ but not necessarily trivial on $(E^\times)^2$.
\end{example}

To overcome the obstruction of applying Hakim--Murnaghan's machinery, we take consideration of the $z$-extension of $G$ motivated by the above example.
\begin{definition}[{\cites[Section 1]{Kot82}}]
Let $G$ be a connected reductive group over $F$. A $z$-extension of $G$ is a central extension $\widetilde{G}$ of $G$
$$1\rightarrow \widetilde{Z}\rightarrow \widetilde{G}\rightarrow G\rightarrow 1,$$
such that
\begin{enumerate}
\item $\widetilde{G}$ is a connected reductive group over $F$ whose derived subgroup is simply connected.
\item $\widetilde{Z}$ is an induced torus, that is, $\widetilde{Z}$ is isomorphic to a finite product $\prod\limits_{i}\mathrm{Res}_{L_i/F}\mathbb{G}_m$ for finite extensions $L_i/F$.
\end{enumerate}
\end{definition}

\begin{lemma}[{\cites[Section 11.4]{KP23}}]

For any reductive group $G$ over a field $F$, a $z$-extension always exists.
\end{lemma}

Taking the long exact sequence of Galois cohomology associated to the short exact sequence of $F$-groups, we get
$$1\rightarrow \widetilde{Z}(F)\rightarrow \widetilde{G}(F)\rightarrow G(F)\rightarrow H^1(F,\widetilde{Z})\rightarrow \cdots.$$
By Shapiro's lemma, there are isomorphisms
$$H^1(F,\widetilde{Z})\cong H^1(F,\prod\limits_{i}\mathrm{Res}_{L_i/F}\mathbb{G}_m)\cong \prod_{i}H^{1}(L_i,\mathbb{G}_m)\cong \{1\},$$
which imply that $G(F)\cong \widetilde{G}(F)/\widetilde{Z}(F)$.

In this way, we can treat an irreducible representation of $G(F)$ as an irreducible representation of $\widetilde{G}(F)$, which $\widetilde{Z}(F)$ acts on trivially. In particular, Prasad's character $\omega_{G(F),E}$ is indeed a quadratic character of $\widetilde{G}(F)$.

\begin{lemma}\label{extensionz}
$\omega_{G(F),E}$, as a quadratic character of $\widetilde{G}(F)$, can be extended to a character of $\widetilde{G}(E)$ for any $z$-extension $\widetilde{G}$ of $G$.
\end{lemma}

\begin{proof}
We first prove that $\omega_{G(F),E}$ regarded as a character of $\widetilde{G}(F)$ is the same as $\omega_{\widetilde{G}(F),E}$. Notice that we have the following commutative diagram
$$\xymatrix@R=0.5cm{
                &         Z(\widehat{G}_{\mathrm{sc}}) \ar@{^{(}->}[dd]     \\
  \{\pm 1\} \ar[ur]^{\widehat{2\rho}} \ar[dr]                \\
                &         Z(\widehat{\widetilde{G}}_{\mathrm{sc}}).}$$
By applying \Cref{normalfunc} to 
\[\xymatrix{
\widetilde{G}\ar@{->>}[d] \ar@{->>}[r]
	&  \widetilde{G}_{\mathrm{ad}} \ar@{=}[d] \\
	G  \ar@{->>}[r]
	&  G_{\mathrm{ad}},}\]
with each morphism being normal, we have the following commutative diagram
$$\xymatrix@R=0.5cm{
                &         H^1(F,Z(\widehat{G}_{\mathrm{sc}})) \ar[dd] \ar[r] & H^1(F,Z(\widehat{G})) \ar[dd]\ar[r] & \mathrm{Hom}_{\mathrm{cts}}(G(F),\mathbb{C}^{\times})\ar[dd]    \\
  H^1(F,\{\pm 1\}) \ar[ur] \ar[dr]                 \\
                &         H^1(F,Z(\widehat{\widetilde{G}}_{\mathrm{sc}})) \ar[r] & H^1(F,Z(\widehat{\widetilde{G}}))  \ar[r] &\mathrm{Hom}_{\mathrm{cts}}(\widetilde{G}(F),\mathbb{C}^{\times}).}$$
The commutativity of the above diagram implies that $\omega_{G(F),E}=\omega_{\widetilde{G}(F),E}$ as a character of $\widetilde{G}(F)$. By Lemma \eqref{derived}, $\omega_{G(F),E}|_{\widetilde{G}_{\mathrm{der}}(F)}$ coincides with $\omega_{\widetilde{G}_{\mathrm{der}}(F),E}$, which is trivial since $\widetilde{G}_{\mathrm{der}}$ is simply connected. This simply means that $\omega_{G(F),E}$ is a character of $\widetilde{G}(F)/\widetilde{G}_{\mathrm{der}}(F)$.

We also have a second exact sequence of $F$-groups
$$1\rightarrow \widetilde{G}_{\mathrm{der}}\rightarrow \widetilde{G}\rightarrow C\rightarrow 1,$$
where $\widetilde{G}_{\mathrm{der}}=[\widetilde{G},\widetilde{G}]$ denotes the derived subgroup of $\widetilde{G}$ and $C\cong \widetilde{G}/\widetilde{G}_{\mathrm{der}}$ denotes the cocenter of $\widetilde{G}$, which is a torus defined over $F$.

Taking the long exact sequence of Galois cohomology, we get
$$1\rightarrow \widetilde{G}_{\mathrm{der}}(F)\rightarrow \widetilde{G}(F)\rightarrow C(F)\rightarrow H^{1}(F,\widetilde{G}_{\mathrm{der}})=1,$$
since $\widetilde{G}_{\mathrm{der}}$ is simply connected. This provides an identification between $\widetilde{G}(F)/\widetilde{G}_{\mathrm{der}}(F)$ and $C(F)$. Since $C(F)$ is a closed subgroup of the locally compact abelian group $C(E)$, $\omega_{G(F),E}$ as a character of $C(F)$ could be always extended to be a character of $C(E)\cong \widetilde{G}(E)/\widetilde{G}_{\mathrm{der}}(E) $, hence a character of $\widetilde{G}(E)$.
\end{proof}

Let $\widetilde{\omega}_{\widetilde{G}(E)}$ be an extension of $\omega_{\widetilde{G}(F),E}=\omega_{G(F),E}$ to $\widetilde{G}(E)$. Now we review the relationship between the representation theory of $G(E)$ and $\widetilde{G}(E)$. Indeed there is a natural inflation map
$$\mathrm{Irr}(G(E))\rightarrow \mathrm{Irr}(\widetilde{G}(E))$$
by regarding an irreducible representation of $G(E)$ as an irreducible representation of $\widetilde{G}(E)$, on which $\widetilde{Z}(E)$ acts trivially.

Let $\pi$ be a regular supercuspidal representation of $G(E)$ constructed from a tame regular elliptic pair $(j_{\pi}(T),\theta)$. The image of $\pi$ under this map is the regular supercuspidal representation of $\widetilde{G}(E)$ attached to $(\widetilde{j_{\pi}}(\widetilde{T}),\widetilde{\theta})$, where $\widetilde{T}$ is the preimage of $T$ in $\widetilde{G}$ and $\widetilde{\theta}:\widetilde{j_{\pi}}(\widetilde{T})(E)\rightarrow \mathbb{C}^{\times}$ is the pull back of $\theta$ along $\widetilde{T}(E)\rightarrow T(E)$. We have the following commutative diagram
$$\xymatrix{
  1  \ar[r] & \widetilde{Z} \ar@{=}[d] \ar[r] & \widetilde{T} \ar@{^{(}->}[d]^{\widetilde{j_{\pi}}} \ar[r] & T \ar@{^{(}->}[d]^{j_{\pi}} \ar[r] & 1  \\
  1 \ar[r] & \widetilde{Z} \ar[r] & \widetilde{G} \ar[r] & G \ar[r] & 1 .  }$$
Taking long exact sequence of Galois cohomology, we get the following commutative diagram
$$\xymatrix{
H^{1}(E,\widetilde{Z}) \ar@{=}[d] \ar[r] & H^1(E,\widetilde{T}) \ar[d] \ar[r] & H^1(E,T) \ar[d] \ar[r] & H^2(E,\widetilde{Z})\ar@{=}[d]  \\
H^{1}(E,\widetilde{Z}) \ar[r] &H^1(E,\widetilde{G}) \ar[r] & H^{1}(E,G) \ar[r] & H^2(E,\widetilde{Z})  . }$$
Since $\widetilde{Z}$ is an induced torus, we know that $H^{1}(E,\widetilde{Z})=1$ and $H^{2}(E,\widetilde{Z})=\prod\limits_{i}\mathbb{Q}/\mathbb{Z}$. By the snake lemma, we also have an isomorphism
$$\ker\left(H^1(E,T)\rightarrow H^1(E,G)\right)\cong \ker\left(H^1(E,\widetilde{T})\rightarrow H^1(E,\widetilde{G})\right). $$

If we fix a stable conjugacy class of an elliptic torus $T$ in $G_E$, then we obtain a bijection
$$\{i_E:T\rightarrow G \text{ embedding}/E\}/G(E)\leftrightarrow \{\widetilde{i}_E:\widetilde{T}\rightarrow \widetilde{G} \text{ embedding}/E\}/\widetilde{G}(E).$$
Moreover, there is a bijection
$$\mathrm{BC}^{-1}([j])\leftrightarrow \mathrm{BC}^{-1}(\widetilde{[j]}),$$
since any $F$-embedding $(S,i)$ into $G$, whose base change to $E$ is $(T,j)$, canonically determines an $F$-embedding $(\widetilde{S},\widetilde{i})$ whose base change to $E$ is $(\widetilde{T},\widetilde{j})$ and vice versa. Hence we obtain
\begin{equation}\label{individual}
\begin{aligned}
&\dim\mathrm{Hom}_{G(F)}(\pi,\omega_{G(F),E})\\
=&\dim\mathrm{Hom}_{\widetilde{G}(F)}(\pi\otimes \widetilde{\omega}_{\widetilde{G}(E)}^{-1},\mathbbm{1})\\
=&\sum\limits_{\{[\widetilde{g}]\in \widetilde{G}(F) \backslash \widetilde{G}(E)/\widetilde{j_{\pi}}(\widetilde{T}(E)):\text{relevant}\}}\dim \mathrm{Hom}_{\widetilde{g}\widetilde{j_{\pi}}(\widetilde{T}(E))\widetilde{g}^{-1}\cap \widetilde{G}(F)}(\widetilde{\theta}\circ\widetilde{j_\pi}^{-1}\circ\mathrm{Ad}(\widetilde{g}^{-1})\cdot  \widetilde{\omega}_{\widetilde{G}(E)}^{-1}|_{\widetilde{g}\widetilde{j_{\pi}}(T(E))\widetilde{g}^{-1}},\widetilde{\epsilon}_{\mathrm{HM},[g]})\\
=&\sum\limits_{[\widetilde{i}]\in \mathrm{BC}^{-1}([\widetilde{j_{\pi}}])}\dim \mathrm{Hom}_{[\widetilde{i}](S(F))}(\widetilde{\theta}\circ [\widetilde{j_\pi}]^{-1} \cdot\  \widetilde{\omega}_{\widetilde{G}(E)}^{-1}|_{\widetilde{[j_{\pi}]}(\widetilde{T}(E))},\widetilde{\epsilon}_{\mathrm{HM},[\widetilde{i}](S(F))})\\
=&\sum\limits_{[i]\in \mathrm{BC}^{-1}([j_{\pi}])}\dim \mathrm{Hom}_{[\widetilde{i}](S(F))}(\widetilde{\theta}\circ [\widetilde{j_\pi}]^{-1} \cdot\  \widetilde{\omega}_{\widetilde{G}(E)}^{-1}|_{[\widetilde{j_{\pi}]}(\widetilde{T}(E))},\widetilde{\epsilon}_{\mathrm{HM},[\widetilde{i}](S(F))}),
\end{aligned}
\end{equation}
i.e., we reduce the original distinction problem of the pair $(G(E),G(F))$ with respect to a non-trivial quadratic character of $G(F)$ to a distinction problem of the pair of $z$-extensions $(\widetilde{G}(E),\widetilde{G}(F))$.
Indeed, the above computation can be applied to compute $\dim\mathrm{Hom}_{G_\alpha(F)}(\pi,\omega_{G_\alpha(F),E})$ for any $\alpha\in H^1(\mathrm{Gal}(E/F),G(E))$ such that the index set of the summation has nice interpretations. We will use these formulas crucially in \Cref{LHS}.

\section{Proof of the theorem under \Cref{productofcharacters}}\label{section7}

The main goal of this section is twofold. The first is to provide a proof of Prasad’s
conjecture for tori. The second is to build a bridge between the regular supercuspidal case and the torus case, which then reduces the proof of Prasad's conjecture for regular supercuspidal representations to a conjectural identity involving various quadratic characters.

More precisely, let $G$ be a quasi-split reductive group over $F$. Let $(T,\theta)$ be a tame regular elliptic pair for $G_E$ and let $(S^\mathrm{op},\mu)$ be a tame regular elliptic pair for $G^\mathrm{op}$ with base change $(T,\theta)$ over $E$. Let $\epsilon_{\mathrm{HM}},\epsilon_{\mathrm{Kal}},\zeta$ be the quadratic characters associated to these data defined in \Cref{hakimcharacter}, \Cref{Kalcharacter} and \eqref{zetacharacter}. Let $\omega_{G(F),E}$ be Prasad's character in \Cref{Prasadcharacter}. Then these characters are conjectured to have the following relation.
\begin{conjecture}[Character Comparison Conjecture]\label{productofcharacters}
Let $i:S\rightarrow G$ be an $F$-embedding of a maximal torus such that the base change $i(S)$ over $E$ is $T$. Then the following equality of characters of $i(S)(F)$
$$\epsilon_{\mathrm{HM}}\cdot \omega_{G(F),E}|_{i(S)(F)}\cdot \epsilon_{\mathrm{Kal}}|_{i(S)(F)}=\zeta|_{i(S)(F)}$$
holds for any reductive group $G$ whose derived subgroup is simply connected.
\end{conjecture}

\subsection{Torus case}\label{toruscase}
We start with Prasad's conjecture for torus $G=S$. This has already been treated in Prasad's paper \cite{Pra15}, and we give a brief review of the torus case here. Since $H^1(F,\mathrm{Inn}(S))$ is trivial, $S_{\alpha}$ is identified with $S$ for any $\alpha\in H^1(F,S)$ which may be non-trivial. From the construction of $\omega_{G(F),E}$, we know that it is in fact pulled back from a character of $G^{\mathrm{ad}}(F)$, hence is trivial in the torus case.

Let $\chi$ be a character of $S(E)$. The left-hand side of Prasad's conjecture gives
$$\sum_{\alpha \in H^{1}(\mathrm{Gal}(E/F),S(E))}\dim \mathrm{Hom}_{S_{\alpha}(F)}(\chi,\mathbb{C})=\#|H^1(\mathrm{Gal}(E/F),S(E))|\cdot\dim \mathrm{Hom}_{S(F)}(\chi,\mathbb{C}).$$

By \Cref{Sop}, $S^{\mathrm{op}}$ is an $F$-torus, whose $F$-rational points are given by
$$S^{\mathrm{op}}(F)=\{t\in S(E)|\,\sigma(t)=t^{-1}\}.$$
From the long exact sequence associated to the short exact sequence of $F$-groups
$$1\rightarrow S\rightarrow \mathrm{Res}_{E/F}S_E\rightarrow S^{\mathrm{op}}\rightarrow 1,$$
we get
$$1\rightarrow S(E)/S(F)\rightarrow S^{\mathrm{op}}(F)\rightarrow \ker\left(H^1(F,S)\rightarrow H^1(E,S)\right)\rightarrow 1.$$
Applying the exact functor $\mathrm{Hom}_{\mathrm{cts}}(\cdot,\mathbb{C}^{\times})$ of taking the $\mathbb{C}^\times$-dual, we can see that a character $\chi$ of $S(E)$ is distinguished by $S(F)$ if and only if it is a base change of a character of $S^{\mathrm{op}}$ that is, $\chi=\chi^{\mathrm{op}}\circ \mathrm{Nm}_{S(E)/S^{\mathrm{op}}(F)}$ where $\chi^{\mathrm{op}}$ is a character of $S^{\mathrm{op}}(F)$ and $\mathrm{Nm}_{S(E)/S^{\mathrm{op}}(F)}:S(E)\rightarrow S^{\mathrm{op}}(F)$ is the norm map associated to $(S^{\mathrm{op}}(F),S(E)=S^{\mathrm{op}}(E))$ given by $ t\mapsto \frac{t}{\sigma(t)}.$ Hence we obtain
$$\text{The left-hand side}=\left\{
             \begin{array}{lr}
             \#|H^1(\mathrm{Gal}(E/F),S(E))| & \quad\text{if }\chi\text{ comes from base change},\\
             0 & \quad\text{otherwise}.\\
            \end{array}\right.$$
On the other hand, by \eqref{inflationrestriction}, the set of equivalence classes of possible extensions to $W_F$ of the parameter of $\chi$ (if not empty) can identified with the set $$\ker\left(H^1(W_F,\widehat{S^{\mathrm{op}}})\rightarrow H^1(W_E,{}_{\phi_{\chi}}\widehat{S})\right)\cong H^1(\mathrm{Gal}(E/F),\widehat{S^{\mathrm{op}}}^{W_E}).$$ The right hand side of Prasad's conjecture reads:
$$\text{The right-hand side}=\sum\limits_{\tilde{\phi}}m(\lambda,\tilde{\phi})=\left\{
             \begin{array}{lr}
             \#|H^1(\mathrm{Gal}(E/F),\widehat{S^{\mathrm{op}}}^{W_E})|  & \quad\text{if } \phi_{\chi}\text{ comes from base change},\\
             0 & \text{otherwise}.\\
            \end{array}\right.$$
In fact, by local Langlands for tori and the exact sequence
$$1\rightarrow S(E)/S(F)\rightarrow S^{\mathrm{op}}(F)\rightarrow \ker\left(H^1(F,S)\rightarrow H^1(E,S)\right)\rightarrow 1,$$
we can already see that the number of base change lifts of $\chi$ from $S^{\mathrm{op}}(F)$ (if not $0$) equals the cardinality of
$$H^1(\mathrm{Gal}(E/F),S(E))=\ker\left(H^1(F,S)\rightarrow H^1(E,S)\right).$$

Now we give a direct proof of this fact based on Galois cohomology. Then the proof of Prasad’s conjecture for tori becomes a direct corollary of the following lemma
\begin{lemma}\label{toruscaseidentity}
$\#|H^1(\mathrm{Gal}(E/F),S(E))|=\#|H^1(\mathrm{Gal}(E/F),\widehat{S^{\mathrm{op}}}^{W_E})|.$
\end{lemma}
	
\begin{proof}
From the short exact sequence of tori over $F$
$$1\rightarrow S\rightarrow \mathrm{Res}_{E/F}(S_E)\rightarrow S^{\mathrm{op}}\rightarrow 1,$$
we get a dual exact sequence of $W_F$-modules
$$1\rightarrow\widehat{S^{\mathrm{op}}}\rightarrow \mathrm{Ind}_{W_E}^{W_F}(\widehat{S})\rightarrow \widehat{S}\rightarrow 1.$$
By combining a result of Kottwitz \cite[Corollary 2.3]{Kot84} and Shapiro's lemma, we get the following long exact sequence
$$1\rightarrow\pi_{0}(\widehat{S^{\mathrm{op}}}^{W_F})\rightarrow \pi_{0}(\widehat{S}^{W_E})\rightarrow \pi_{0}(\widehat{S}^{W_F})\rightarrow H^1(W_F,\widehat{S^{\mathrm{op}}})\rightarrow H^1(W_E,\widehat{S})\rightarrow\cdots,$$
which leads to an isomorphism
$$H^1(\mathrm{Gal}(E/F),\widehat{S^{\mathrm{op}}}^{W_E})\cong \pi_{0}(\widehat{S}^{W_F})/\pi_{0}(\widehat{S}^{W_E}),$$
where the map $\pi_{0}(\widehat{S}^{W_E})\rightarrow\pi_{0}(\widehat{S}^{W_F})$ is induced by the corestriction map. Since $W_F$ is dense in $\Gamma_F$, $A^{W_F}=A^{\Gamma_F}$ for any Galois module $A$. On the other hand, we also have the commutative diagram
$$\xymatrix{
  H^1(F,S) \ar[d]_{\mathrm{res}} \ar[r]^{\mathrm{Kottwitz}}
  & \pi_{0}(\widehat{S}^{\Gamma_F})^* \ar[d]^{\mathrm{cores}^*}\\
  H^1(E,S)  \ar[r]^{\mathrm{Kottwitz}}
  & \pi_{0}(\widehat{S}^{\Gamma_E})^*.}$$
Observe that for any homomorphism of finite abelian groups $f:A\rightarrow B$, there is a natural identification $(B/A)^{*}\cong \ker(B^{*}\rightarrow A^{*})$, which implies the identity $\#|\mathrm{coker} f|=\#|\ker f^*|$. Specializing the above observation to $f^*=\mathrm{res}$ together with the isomorphisms
$$H^1(\mathrm{Gal}(E/F),\widehat{S^{\mathrm{op}}}^{W_E})\cong \pi_{0}(\widehat{S}^{W_F})/\pi_{0}(\widehat{S}^{W_E})=\pi_{0}(\widehat{S}^{\Gamma_F})/\pi_{0}(\widehat{S}^{\Gamma_E}),$$
$$H^1(\mathrm{Gal}(E/F),S(E))\cong \ker\left(H^1(F,S)\rightarrow H^1(E,S)\right),$$ we finish the proof of the lemma.
\end{proof}

\subsection{Regular supercuspidal case}\label{whittakerembeddingbasepoint}

Now we try to reduce the proof of Prasad's conjecture for regular supercuspidal representations to the case of tori. We first summarize Kaletha's construction of an $L$-packet associated to a regular supercuspidal $L$-packet datum in a few words. 

Fix a regular supercuspidal $L$-packet datum $(T,\widehat{j},\chi,\theta)$ of $G(E)$ such that $\chi$ is the minimally ramified $\chi$-datum determined by $\theta$. The Vogan $L$-packet associated to this Langlands parameter consists of representations
$$\{ \pi_{(j(T),\ \theta\circ j^{-1}\cdot \epsilon_{\mathrm{Kal}})}  \},$$
where $j$ runs over the set of $G_{\beta}(E)$-conjugacy classes of embeddings of $T$ into $G_{E,\beta}$ over $E$ for any $\beta\in H^{1}(E,G)$, and the representation $\pi_{(j(T),\ \theta\circ j^{-1}\cdot \epsilon_{\mathrm{Kal}})}$ is the one constructed from Yu's generic cuspidal datum associated to the tame elliptic pair in \Cref{Yuconstruction}.

Let $\pi$ be a regular supercuspidal representation associated to a regular supercuspidal datum $(T,\widehat{j},\chi,\theta,(G,1),j_{\pi})$, where $j_{\pi}$ is an admissible embedding of $T$ to $G$ defined over $E$. Now we give an explicit computation of the left-hand side and the right-hand side of \Cref{prasadidentity}. Notice the tame elliptic regular pair associated to this regular supercuspidal datum is $(j_{\pi}(T),\theta\circ j_{\pi}^{-1}\cdot \epsilon_{\mathrm{Kal}})$. We also need the following generalization of \Cref{toral} to fix a base point in a given regular supercuspidal $L$-packet.

\begin{theorem}[{\cites{DR10}[Lemma 6.2.2]{Kal19b}[Page 2273]{FKS23}}]\label{whittakerembeeding}
For any Whittaker datum $\mathfrak{w}$ for $G(E)$, there is a unique admissible embedding (up to $G(E)$ conjugacy) $j_{\frak{w}}:T\rightarrow G_E  $ such that the regular supercuspidal representation corresponding to $(T,\widehat{j},\chi,\theta,(G_E,\mathrm{id}),j_{\mathfrak{w}})$ is $\mathfrak{w}$-generic.
\end{theorem}

Indeed the unique embedding (up to $G(E)$-conjugacy) $j_{\mathfrak{w}}:T\rightarrow G_E$ such that $\pi_{j_{\mathfrak{w}}}$ is $\mathfrak{w}$-generic is characterized by the condition that
the $G(E)$-orbit of $dj_{\mathfrak{w}}(X)$ meets the Kostant section of $f_{\psi_{N_0(E)}}$, where $f_{\psi_{N_0(E)}}$ denotes the nilpotent element in $\mathfrak{n}_0^-(E)$ as in \Cref{liealgebrageneric}. This statement is first proved in the proof of \cite[Lemma 6.2.2]{Kal19b} and later generalized to all non-singular supercuspidal representations in \cite[Page 2273]{FKS23} (See \Cref{characterizationremark} for more detailed explanation). With the help of \Cref{whittakerembeeding}, we get the following proposition characterizing the $G(E)$-conjugacy class of rational embeddings for $\mathfrak{w}$ satisfying Prasad's condition.

\begin{proposition}\label{embeddingoverF}
If the Whittaker datum $\mathfrak{w}=(B_0,\psi_{N_0(E)})$ is chosen to satisfy Prasad's condition $\psi_{N_0(E)}|_{N_0(F)}=\mathbbm{1}$, then the base point of the embedding $[j_{\mathfrak{w}}]$ is defined over $F$.
\end{proposition}

\begin{proof}
By \Cref{whittakerF}, the condition $\psi_{N_0(E)}|_{N_{0}(F)}=\mathbbm{1}$ is equivalent to the condition $-f_{\psi_{N_0(E)}}=\sigma (f_{\psi_{N_0(E)}})$. Let $X^*\in \mathfrak{t}^{*}_{-r}$ be the generic semisimple element associated to the character $\theta=\prod_{i=-1}^{d}\phi_{i}|_{T(E)}$, and let $X\in \mathfrak{t}_{-r}$ be the corresponding element in the Lie algebra under the isomorphism $\mathfrak{g}\cong\mathfrak{g}^*$ induced by a $G$-equivariant non-degenerate bilinear form $\langle\,,\,\rangle$ on $\mathfrak{g}$. More precisely, $X$ is determined by $\theta(\exp Y)=\psi(\langle X,Y\rangle)$ for any $Y\in \mathfrak{s}(F)_{0+}$. Since we also assume the $L$-parameter of $\pi$ is a base change lift from one of $G^{\mathrm{op}}$, we know $X=-\sigma(X)$. By \cites[Lemma 6.2.2]{Kal19b}[Page 2273]{FKS23}, the base point $[j_{\mathfrak{w}}]:T\rightarrow G_E$ determined by the Whittaker datum $\mathfrak{w}$ satisfies the property that the $G(E)$-orbit of $dj_{\mathfrak{w}}(X)$ meets the Kostant section of $f_{\psi_{N_0(E)}}$. More precisely, there exists $g\in G(E)$ such that $\mathrm{Ad}(g)(dj_{\mathfrak{w}}(X))\in f_{\psi_{N_0(E)}}+C_{\mathfrak{g}(\overline{F})}(e_{\psi_{N_0(E)}})$. To prove $[j_{\mathfrak{w}}]$ is defined over $F$, we need to verify
$$[\sigma\circ j_{\mathfrak{w}}\circ \sigma^{-1}]=[j_{\mathfrak{w}}].$$
Let $j'$ denote $\sigma\circ j_{\mathfrak{w}}\circ\sigma^{-1}$. Then we find $dj'(X)=\sigma \circ dj_{\mathfrak{w}}\circ\sigma^{-1}(X)=-\sigma (dj_{\mathfrak{w}}(X))$. Notice that the facts $-f_{\psi_{N_0(E)}}=\sigma (f_{\psi_{N_0(E)}})$ and $\mathrm{Ad}(g)(dj_{\mathfrak{w}}(X))\in f_{\psi_{N_0(E)}}+C_{\mathfrak{g}(\overline{F})}(e_{\psi_{N_0(E)}})$ imply that
$$\mathrm{Ad}(\sigma(g))(dj'(X))\in f_{\psi_{N_0(E)}}+C_{\mathfrak{g}(\overline{F})}(e_{\psi_{N_0(E)}}).$$
By the uniqueness of the characterization in \Cref{whittakerembeeding}, we have $[j']=[j_{\mathfrak{w}}]$, that is, $[j_{\frak{w}}]$ is defined over $F$.
\end{proof}

For any $\sigma$-stable $G(E)$ conjugacy class of embeddings $[j_{\mathfrak{w}}]$, there is an $\sigma$-fixed point by \cite[Corollary 2.2]{Kot82}, which is denoted by $j_{\mathfrak{w}}$. We shall henceforth fix a Whittaker datum $\mathfrak{w}$ satisfying Prasad's condition and
use $j_{\mathfrak{w}}$ in \Cref{embeddingoverF} as a base point to fix the bijection:
$$\xymatrix{
                & \{\pi_{(j(T),\ \theta\circ j^{-1}\cdot \epsilon_{\mathrm{Kal}})}\} \ar@{<->}[dl]_-{j_\mathfrak{w}}   \ar@{<->}[dr]^-{\mathfrak{w}}          \\
 H^1(E,T)\ar@{<->}[rr]^-{\text{Kottwitz isomorphism}} & &   \pi_{0}(\widehat{T}^{\Gamma_E})^{*}\cong \mathrm{Irr}(\pi_{0}(S_{\phi}))       .}$$

\subsubsection{The right-hand side of \Cref{prasadidentity}}\label{RHSofprasad}
Now we try to compute the right-hand side of Prasad's conjecture in \Cref{prasadidentity} explicitly for regular supercuspidal representations in terms of their supercuspidal data. We need to keep in mind that $j_{\pi}$ plays the role of a character $\lambda_\pi$ of the component group associated to $\pi$.

To find an extension of a parameter $\phi_\pi:W_{E}\rightarrow\widehat{G}$ to a parameter $\tilde{\phi}:W_F\rightarrow\widehat{G^{\mathrm{op}}}$ is equivalent to find a regular supercuspidal $L$-packet datum $(S^{\mathrm{op}},\widehat{j}^{\mathrm{op}},\chi_{\mu},\mu)$ of $G^{\mathrm{op}}(F)$, where $S^{\mathrm{op}}$ is an elliptic maximal torus of $G^{\mathrm{op}}$ defined over $F$ and $\mu:S^{\mathrm{op}}(F)\rightarrow\mathbb{C}^{\times}$ is a character of $S^{\mathrm{op}}(F)$ such that
$$(S^{\mathrm{op}}\times_{F}E,\widehat{j},\chi_{\mu,\mathrm{BC}},\mu\circ \mathrm{Nm}_{S(E)/S^{\mathrm{op}}(F)})\cong (T,\widehat{j},\chi_{\theta},\theta)$$
as in \Cref{basechangelpacketdata}. The isomorphism is equivalent to the condition that there exists an abstract isomorphism of $E$-tori
$f:S^{\mathrm{op}}\times_{F}E\cong T$ such that
\begin{equation}\label{abstractbasechange}
\begin{aligned}
\theta=\mu\circ\mathrm{Nm}_{T(E)/S^{\mathrm{op}}(F)}\cdot \zeta_{T(E),\chi_{\mu,\mathrm{BC}},\chi_{\theta}},
\end{aligned}
\end{equation}
where the norm map $T(E)\rightarrow S^{\mathrm{op}}(F)$ is the composite of the isomorphism $f(E):T(E)\cong S(E)$ and the norm map on $S(E)\rightarrow S^{\mathrm{op}}(F)$ sending $t$ to $\frac{t}{\sigma(t)}$. We take this abstract isomorphism $f$ as part of our initial data.

Since $\zeta_{T(E),\chi_{\mu,\mathrm{BC}},\chi_{\theta}}$ is tamely ramified, we know that $\theta$ and $\mu\circ\mathrm{Nm}_{T(E)/S^{\mathrm{op}}(F)}$ determine the same mod-$a$-data, which implies $\chi_{\theta}=\chi_{\mu\circ\mathrm{Nm}_{T(E)/S^{\mathrm{op}}(F)}}$ by \cite[Page 1154]{Kal19b}. Then \Cref{abstractbasechange} becomes
\begin{equation}\label{basechangechi}
\begin{aligned}
\theta=\mu\circ\mathrm{Nm}_{T(E)/S^{\mathrm{op}}(F)}\cdot \zeta_{T(E),{\chi_{\mu}}_{\mathrm{BC}},\chi_{\mu\circ\mathrm{Nm}_{T(E)/S^{\mathrm{op}}(F)}}}.
\end{aligned}
\end{equation}
We will adopt the notation $\mathrm{Nm}_{T/S^{\mathrm{op}}}^*(\mu):=\mu\circ\mathrm{Nm}_{T(E)/S^{\mathrm{op}}(F)}$ as in \Cref{Prop3.4}.
\begin{remark}
A priori, it seems that the character $\zeta_{T(E),\chi_{\mu,\mathrm{BC}},\chi_{\mu\circ\mathrm{Nm}_{T(E)/S^{\mathrm{op}}(F)}}}$ produced by $\zeta$-data depends on the character $\mu$. However, the computations in \Cref{explicitzeta} imply that this character does not depend on $\mu$, but only depends on $\theta$. We will omit the subscript of $\zeta_{T(E),\chi_{\mu,\mathrm{BC}},\chi_{\mu\circ\mathrm{Nm}_{T(E)/S^{\mathrm{op}}(F)}}}$ and simply write $\zeta$ if the subscript is clear from the context.
\end{remark}

\begin{proposition}\label{RHS}
The right-hand side of \Cref{prasadidentity} is non-zero if and only if
$$\theta\cdot \zeta^{-1}\in \mathrm{Im}(\mathrm{Nm}_{T/S^{\mathrm{op}}}^*)\quad\text{and}\quad j_\pi\in\ker\left(H^1(E,T)\rightarrow H^1(F,S^{\mathrm{op}})\right),$$
in which case it equals
\[\sum\limits_{\substack{S^{\mathrm{op}}/F,\\S^{\mathrm{op}}\times_{F}E\cong T }}|H^1(\mathrm{Gal}(E/F),\widehat{S^{\mathrm{op}}}^{W_E})|.\]
\end{proposition}

\begin{proof}

For a regular supercuspidal parameter $\phi$, Lemma 5.3.1 in \cite{Kal19b} implies $$\pi_{0}(S_{\phi_\pi})=\pi_0(\widehat{j}(\widehat{T}^{\Gamma_E})),\quad\pi_{0}(S_{\tilde{\phi}})=\pi_{0}(\widehat{j}(\widehat{S^{\mathrm{op}}}^{\Gamma_F})).$$

The parameter side of Prasad's conjecture in \Cref{prasadidentity} then becomes
\begin{equation}\label{reductiontoral}
\sum\limits_{\substack{\tilde{\phi}:W_F\rightarrow \widehat{G^{\mathrm{op}}}\\ \tilde{\phi}|_{W_E}=\phi_\pi}}m(\lambda_\pi,\tilde{\phi})=\sum\limits_{\substack{S^{\mathrm{op}}/F,\\S^{\mathrm{op}}\times_{F}E\cong T }}\sum\limits_{\substack{\mu\in\mathrm{Irr}(S^{\mathrm{op}}(F))\\ \mu\circ \mathrm{Nm}_{T(E)/S^{\mathrm{op}}(F)}=\theta\cdot \zeta^{-1}}}\dim \mathrm{Hom}_{\pi_{0}(\widehat{S^{\mathrm{op}}}^{\Gamma_{F}})}(\lambda_\pi,\mathbbm{1}),
\end{equation}
where the notation $S^{\mathrm{op}}\times_F E\cong T$ means that we fix an abstract isomorphism $f:S^{\mathrm{op}}\times_F E\cong T$ of abstract tori over $E$. Now we will compute each term appearing in the right-hand side of \Cref{reductiontoral} explicitly.
\begin{enumerate}[(i)]
\item By local Langlands correspondence for tori,
$\theta\in \mathrm{Hom}_{\mathrm{cts}}(T(E),\mathbb{C}^{\times})$ corresponds to $\theta \in H^1(W_E,\widehat{T}).$ We also have a commutative diagram
$$\xymatrix{
  \mathrm{Hom}_{\mathrm{cts}}(S^{\mathrm{op}}(F),\mathbb{C}^\times) \ar[d]_-{ \mathrm{Nm}_{T/S^{\mathrm{op}}}^*} \ar@{<->}[r]^-{\mathrm{LLC}}
  & H^1(W_F,\widehat{S^{\mathrm{op}}}) \ar[d]^-{\mathrm{Restriction}}\\
  \mathrm{Hom}_{\mathrm{cts}}(T(E),\mathbb{C}^\times)  \ar@{<->}[r]^-{\mathrm{LLC}}
  & H^1(W_E,\widehat{T}).}$$
Hence for $\theta\cdot\zeta^{-1}\in \mathrm{Im}(\mathrm{Nm}_{T/S^{\mathrm{op}}}^{*})$, the set $$\{\mu|\,\mu\circ\mathrm{Nm}_{T(E)/S^{\mathrm{op}}(F)}=\theta\cdot \zeta^{-1}\}$$ can be identified with a torsor for $\ker\left(H^1(W_F,\widehat{S^{\mathrm{op}}})\rightarrow H^1(W_E,\widehat{T})\right)=H^1(\mathrm{Gal}(E/F),\widehat{S^{\mathrm{op}}}^{W_E})$.

\item Fix a Whittaker datum $\mathfrak{w}$ satisfying Prasad's condition, which determines a base point $j_{\mathfrak{w}}$ of the set of equivalence classes of rational embeddings of $T$ to $G_E$. By \Cref{embeddingoverF}, $j_{\mathfrak{w}}$ is in fact defined over $F$, and the representation $\pi$ has the label $j_{\pi}$ such that $\mathrm{inv}(j_\pi,j_{\mathfrak{w}})\in H^{1}(E,T)$, which we still denote by $j_{\pi}$ for simplicity. Consider the following correspondence established in \Cref{whittakerembeeding}:
$$\lambda_{\pi}\in \mathrm{Irr}(\pi_{0}(S_{\phi}))\cong\pi_{0}(\widehat{T}^{\Gamma_E})^{*}\longleftrightarrow j_{\pi}\in H^1(E,T).$$
The condition $\dim \mathrm{Hom}_{\pi_0(\widehat{S^{\mathrm{op}}}^{\Gamma_{F}})}(\lambda_\pi,\mathbbm{1})\neq0$ (hence must be $1$) is equivalent to the condition
$$\lambda_{\pi}\in \ker\left(\pi_{0}(\widehat{T}^{\Gamma_E})^{*}\rightarrow\pi_{0}({S^{\mathrm{op}}}^{\Gamma_F})^{*}\right).$$ 
We also have a commutative diagram
$$\xymatrix{
  \pi_{0}(\widehat{T}^{\Gamma_E})^{*} \ar[d]_-{\mathrm{Restriction}} \ar@{<->}[r]^-{\mathrm{Kottwitz}}
  & H^1(E,T) \ar[d]^-{\mathrm{Corestriction}}\\
  \pi_{0}(\widehat{S^{\mathrm{op}}}^{\Gamma_F})^{*}  \ar@{<->}[r]^-{\mathrm{Kottwitz}}
  & H^1(F,S^{\mathrm{op}}),}$$
where the horizontal arrows are given by Kottwitz isomorphisms. Hence the following two conditions are equivalent
$$\lambda_{\pi}\in \ker\left(\pi_{0}(\widehat{T}^{\Gamma_E})^{*}\rightarrow\pi_{0}(\widehat{S^{\mathrm{op}}}^{\Gamma_F})^{*}\right)\Leftrightarrow j_{\pi}\in \ker\left(H^1(E,T)\rightarrow H^1(F,S^{\mathrm{op}})\right).$$

\end{enumerate}

By the above two interpretations of the formula, we finish the proof.

\end{proof}

\begin{remark}\label{compareextension}
Comparing the sets of base change liftings of the supercuspidal $L$-packet datum in \Cref{RHS} with the one given by inflation and restriction sequence \eqref{inflationrestriction}, \Cref{regularinfla}, we can see that both sets are parametrized by $$\bigsqcup_{\substack{S^{\mathrm{op}}/F,\\S^{\mathrm{op}}\times_{F}E\cong T }} H^1(\mathrm{Gal}(E/F),\widehat{S^{\mathrm{op}}}^{W_E}),$$
but in different ways. 

However, the regularity condition of the character $\theta$ seems to pose some constraints on the set of pair $(S,\mu)$ whose base change over $E$ is $(T,\theta)$. 

One natural question is whether $(T,\theta)$ could be a base change of $(S_1,\mu_1)$ and $(S_2,\mu_2)$ simultaneously such that the $F$-structure of $S_1$ and $S_2$ are different. Notice that this question could be asked in a general manner. But here we will only provide a complete answer in the special cases of $\mathrm{GL}_2$ and $\mathrm{SL}_2$, and leave open the question of what happens
in general. More precisely, we will see immediately that, for $G=\mathrm{GL}_2$ or $\mathrm{SL}_2$, there is an elliptic maximal torus $T$ over $E$ such that
\begin{enumerate}
\item there are two non-isomorphic $F$-tori $S_1$ and $S_2$ of $G$ whose base changes to $E$ are $G(E)$-conjugate. 
\item if we further fix a regular character $\theta$ of $T(E)$, then it is not
possible to find characters $\mu_i$ of $S_i(F)$ such that both the base changes of $(S_1,\mu_1)$ and $(S_2, \mu_2)$ are $G(E)$-conjugate to $(T,\theta)$.
\end{enumerate}

\end{remark}

\begin{example}
Let $T$ be an elliptic maximal torus of $\mathrm{GL}_2$ over $E$ and let $\theta$ be a regular character of $T(E)$. Let $T^1=T\cap \mathrm{SL}_2$ be an elliptic torus of $\mathrm{SL}_2$ and let $\theta'$ be a regular character of $T^1(E)$. Then $T(E)=K^\times$ and $T^1(E)=K^{1}_{K/E}$ can be identified with non-zero and norm-one elements inside a quadratic extension $K/E$. A sufficient and necessary condition for the existence of an elliptic maximal torus of $\mathrm{GL}_{2,F}$ or $\mathrm{SL}_{2,F}$ whose base change equals to $T$ or $T^1$ is that $K/F$ is a bi-quadratic extension. Let $E_1$ and $E_2$ be the other $2$ intermediate fields between $K$ and $F$. Let $\tau\in\mathrm{Gal}(E_1/F)$ denote the non-trivial involution. Then $\mathrm{Gal}(K/F)\cong \{1,\sigma,\tau,\sigma\tau|\,\sigma^2=\tau^2=1,\sigma\tau=\tau\sigma\}$. Set $E=F(\sqrt{\alpha})$, $K=E(\sqrt{\beta})=F(\sqrt{\alpha},\sqrt{\beta})$ with $\alpha,\beta, \alpha\beta\in F^{\times}-(F^{\times})^{2}$. Let $E_1$ denote $F(\sqrt{\beta})$ and let $E_2$ denote $F(\sqrt{\alpha\beta})$. We may choose an embedding of $T$ into $\mathrm{GL}_{2,E}$ sending $a+b\sqrt{\beta}$ to $\left(\begin{matrix} a&b\beta\\ b&a\end{matrix}\right)$. With this particular embedding, one has identifications $N_G(T)=\left(\begin{matrix} 1&0\\ 0&-1\end{matrix}\right)T \sqcup T$ and $N_G(T)/T\cong \mathrm{Gal}(K/E)$.

Let $S_1$ and $S_2$ be elliptic tori of $\mathrm{GL}_{2,F}$ with $S_1\times_FE= S_2\times_FE=T,S_1(F)=E_1^\times, S_2(F)=E_2^\times$, and let $S_1^1$ and $S_2^1$ be elliptic tori of $\mathrm{SL}_{2,F}$ with $S_1^1(F)=E_1^1, S_2^1(F)=E_2^1$ as in \Cref{examplesl2}. 
$$\xymatrix{& K \ar@{-}[dl]_{\sigma} \ar@{-}[d]^{\sigma\tau}
\ar@{-}[dr]^{\tau} \\ E_1 \ar@{-}[dr]_{\tau} & E_2
\ar@{-}[d] & E \ar@{-}[dl]^{\sigma} \\ & F}$$

\begin{proposition}
In the above example for $\mathrm{GL}_2$ and $\mathrm{SL}_2$, $(T,\theta)$ cannot be a base change of $(S_1,\mu_1)$ and $(S_2,\mu_2)$ with $S_1(F)\ncong S_2(F)$ simultaneously. $(T^1,\theta')$ cannot be a base change of $(S_1^1,\mu_1')$ and $(S_2^1,\mu_2')$ with $S_1^1(F)\ncong S_2^1(F)$ simultaneously.
\end{proposition}

\begin{proof}
If $(T,\theta)$ comes from two elliptic torus $S_1$ and $S_2$ over $F$ with $S_1(F)=E_{1}^\times, S_2(F)=E_2^\times$, then there exist characters $\mu_1:E_1^\times\rightarrow \mathbb{C}^{\times},\mu_2:E_2^\times\rightarrow \mathbb{C}^{\times}$ such that $\theta(k)=\mu_{1}\circ \mathrm{Nm}_{K/E_1}(k)$ and $\theta(k)=\mu_{2}\circ \mathrm{Nm}_{K/E_2}(k)$ for any $k\in K^1$. An immediate consequence is that $\theta^{\sigma}=\theta=\theta^{\sigma\tau}$, which implies $\theta^\tau=\theta$. This contradicts with the regularity condition, since there are natural identifications $N_G(T)/T= N_G(T)(E)/T(E)\cong\mathrm{Gal}(K/E).$ In the case of $\mathrm{SL}_2$, we know that $N_{G^1}(T^1)=N_G(T)\cap G_1$ so that $N_{G^1}(T^1)(E)=T^1(E)$ if and only if $-1\notin \mathrm{Nm}_{K/E}(K^\times)$. If $K/E$ is unramified, we always have $-1\in \mathrm{Nm}_{K/E}(K^\times)$ since the restriction of the norm map to the ring of units $\mathrm{Nm}_{K/E}:O_K^\times\rightarrow O_E^\times$ is surjective. If $K/E$ is ramified, $-1\in \mathrm{Nm}_{K/E}(K^\times)$ if and only if $-1$ is a square in $k_K^\times=k_E^\times$, i.e., $q_K=1 \bmod 4$, which is always true in the bi-quadratic setting since $q_K=q_{E_1}^2=1 \bmod 4$. Hence in the case of $\mathrm{SL}_2$, one also has $N_{G^1}(T^1)(E)/T^1(E) \cong \mathbb{Z}/2\mathbb{Z}\cong \mathrm{Gal}(K/E)$. Then the same argument as the case of $\mathrm{GL}_2$ can be applied. More precisely, if $\theta'=\mu'_{1}\circ \mathrm{Nm}_{K/E_1}=\mu'_{2}\circ \mathrm{Nm}_{K/E_2}$, we then have $\theta'^\tau=\theta'$ which contradicts with the regularity condition of $\theta'$.
\end{proof}

\end{example}

\begin{remark}

Indeed the same phenomenon happens on the representation side. More precisely, $\theta:T(E)\rightarrow \mathbb{C}^{\times}$ is base changed from $S_1^\mathrm{op}$ and $S_2^\mathrm{op}$ simultaneously with $S_1^\mathrm{op}(F)\ncong S_2^\mathrm{op}(F)$ if and only if $\theta$ satisfies $\theta|_{S_1(F)}=\mathbbm{1}$ and $\theta|_{S_2(F)}=\mathbbm{1}$ with $S_1(F)\ncong S_2(F)$ simultaneously. Since these conditions are equivalent as a simple consequence of local Langlands correspondence for tori, we do not need to exclude these possibilities to prove Prasad's identity.

\end{remark}

\subsubsection{The left-hand side of \Cref{prasadidentity}}

\begin{proposition}\label{LHS}
The left-hand side of \Cref{prasadidentity} is
non-zero if and only if $$\theta|_{S(F)}=\epsilon_{\mathrm{HM}}\cdot \omega_{G(F),E}|_{S(F)}\cdot \epsilon_{\mathrm{Kal}}|_{S(F)}\quad\text{and}\quad \mathrm{inv}(j_\pi,j_{\mathfrak{w}})\in\mathrm{Im} (H^1(F,S)\rightarrow H^1(E,T)),$$ 
in which case it equals
$$\sum\limits_{\substack{S/F,\\S\times_{F}E\cong T }}|H^1(\mathrm{Gal}(E/F),S(E))|.$$
\end{proposition}
\begin{proof}
As we have explained in \Cref{zextension}, we need to pass to $z$-extensions to deal with $(G_{\alpha}(F),\omega_{G_{\alpha}(F),E})$ distinction problem. For $\alpha\in H^1(\mathrm{Gal}(E/F),G(E))$, we have a $z$-extension of $G_{\alpha}$ depending on $\alpha$. Let $\widetilde{G_{\alpha}}$ denote a $z$-extension for $G_{\alpha}$. i.e. there is an induced torus $\widetilde{Z_\alpha}$ such that we have an exact sequence
\[1\rightarrow \widetilde{Z_\alpha}\rightarrow \widetilde{G}_\alpha\rightarrow G_\alpha\rightarrow 1.\]
Summing these dimensions over $F$-pure inner forms $G_{\alpha}$ of $G$ which are trivial over $E$, we get the following identity
\begin{equation}\label{computationstep}
	\begin{aligned}
		&\sum_{\alpha\in H^1(\mathrm{Gal}(E/F),G(E))}\dim\mathrm{Hom}_{G_{\alpha}(F)}(\pi,\omega_{G_{\alpha}(F),E})\\
		&=\sum_{\alpha\in H^1(\mathrm{Gal}(E/F),G(E))}\sum\limits_{[i_{\alpha}]\in \mathrm{BC}_{\alpha}^{-1}([j_{\pi}])}\dim \mathrm{Hom}_{\widetilde{[i_{\alpha}]}(\widetilde{S}(F))}((\widetilde{\theta}\circ \widetilde{[j_{\pi}]}^{-1}\cdot\widetilde{\epsilon}_{\mathrm{Kal}})\cdot \widetilde{\omega}^{-1}_{\widetilde{G}(E)}|_{[\widetilde{j_{\pi}}](\widetilde{T}(E))},\widetilde{\epsilon}_{\mathrm{HM},\widetilde{[i_{\alpha}]}(\widetilde{S}(F))})\\
		&=\sum\limits_{[i]\in \overline{\mathrm{BC}}^{-1}([j_{\pi}])}\dim \mathrm{Hom}_{[\widetilde{i}](\widetilde{S}(F))}((\widetilde{\theta}\circ \widetilde{[j_{\pi}]}^{-1}\cdot\widetilde{\epsilon}_{\mathrm{Kal}})\cdot \widetilde{\omega}^{-1}_{\widetilde{G}(E)}|_{[\widetilde{j_{\pi}}](\widetilde{T}(E))},\widetilde{\epsilon}_{\mathrm{HM},\widetilde{[i]}(\widetilde{S}(F))})\\
		&=\sum\limits_{[i]\in \overline{\mathrm{BC}}^{-1}([j_{\pi}])}\dim \mathrm{Hom}_{\widetilde{[i]}(\widetilde{S}(F))}((\widetilde{\theta}\circ [\widetilde{j_{\pi}]}^{-1}),\widetilde{\epsilon}_{\mathrm{HM},[\widetilde{i}](\widetilde{S}(F))}\cdot\omega_{\widetilde{G}(F),E}|_{\widetilde{[i]}(\widetilde{S}(F))}\cdot \widetilde{\epsilon_{\mathrm{Kal}}}|_{\widetilde{[i]}(\widetilde{S}(F))})\\
		&=\sum\limits_{[i]\in \overline{\mathrm{BC}}^{-1}([j_{\pi}])}\dim \mathrm{Hom}_{\widetilde{S}(F)}(\widetilde{\theta} ,(\widetilde{\epsilon}_{\mathrm{HM},\widetilde{[i]}(\widetilde{S}(F))}\cdot\omega_{\widetilde{G}(F),E}|_{\widetilde{[i]}(\widetilde{S}(F))}\cdot \widetilde{\epsilon_{\mathrm{Kal}}}|_{\widetilde{[i]}(\widetilde{S}(F))})\circ \widetilde{[i]}).
	\end{aligned}
\end{equation}
The first equality follows from summing \Cref{individual} over each pure inner twist $\widetilde{G_\alpha}$ of $\widetilde{G}$. The second equality follows from the decomposition \eqref{sumoverinnertwist}, where $\mathrm{BC}$ is defined in \Cref{bcindividual} and $\overline{\mathrm{BC}}$ is defined in \eqref{totalbc}.

Since $\widetilde{Z_\alpha}\subset Z(\widetilde{G_\alpha})$, we know that for any root $\beta\in \Phi(\widetilde{G_\alpha}, [\widetilde{i}](\widetilde{S}))$, $\beta|_{\widetilde{Z_\alpha}}=\mathbbm{1}$. Hence by \Cref{Kalcharacter} and \Cref{Hakimfactor}, we know that $\widetilde{\epsilon_\mathrm{Kal}}|_{Z_\alpha(E)}\cdot\widetilde{\epsilon}_{\mathrm{HM}}|_{Z_\alpha(E)}=\mathbbm{1}$ so that $\widetilde{\epsilon_\mathrm{Kal}}|_{[\widetilde{i}](\widetilde{S}(F))}\cdot\widetilde{\epsilon}_{\mathrm{HM}}$ descends to a character of $i(S(F))$ and satisfies $\widetilde{\epsilon_\mathrm{Kal}}|_{[\widetilde{i}](\widetilde{S}(F))}\cdot\widetilde{\epsilon}_{\mathrm{HM}}=\epsilon_\mathrm{Kal}|_{[i](S(F))}\cdot\epsilon_{\mathrm{HM}}$ as a character of $i(S(F))$. Since $\omega_{\widetilde{G}(F),E}$ is constructed as a character of $\widetilde{G}_{\mathrm{ad}}(F)$, it is trivial on $\widetilde{Z_\alpha}(F)\subset Z(\widetilde{G_\alpha})(F)$. Furthermore, by \Cref{extensionz}, $\omega_{\widetilde{G}(F),E}|_{i(S)(F)}=\omega_{G(F),E}|_{i(S)(F)}$.
Then \Cref{computationstep} can be reduced to 
\[\sum\limits_{[i]\in \overline{\mathrm{BC}}^{-1}([j_{\pi}])}\dim \mathrm{Hom}_{S(F)}(\theta ,(\epsilon_{\mathrm{HM},[i](S(F))}\cdot\omega_{G(F),E}|_{[i](S(F))}\cdot \epsilon_{\mathrm{Kal}}|_{[i](S(F))})\circ [i]).\]
Combining with \ref{preimageofbc}, we finish the proof.
\end{proof}

If we assume \Cref{productofcharacters}, then the proof of \Cref{conjecture1} is a direct consequence of the proof of the torus case in \Cref{toruscaseidentity} and the following comparison of Proposition \ref{RHS} and \ref{LHS}.
\begin{lemma}\label{comparisonoflr}
If \Cref{productofcharacters} is true, the non-vanishing conditions for both sides of \Cref{prasadidentity} in \Cref{RHS} and \Cref{LHS} are the same.
\end{lemma}
\begin{proof}
The coincidence of the first conditions is a consequence of local Langlands of tori and is proved in \Cref{toruscase}. More precisely, by the long exact sequence associated to the short exact sequence of $F$-groups
$$1\rightarrow S\rightarrow \mathrm{Res}_{E/F}T\rightarrow S^{\mathrm{op}}\rightarrow 1,$$
we obtain a short exact sequence
\[1\rightarrow T(E)/S(F)\xrightarrow{\mathrm{Nm}_{T/S^{\mathrm{op}}}}S^{\mathrm{op}}(F)\xrightarrow{\quad\quad} \ker\left(H^1(F,S)\rightarrow H^1(E,T)\right)\rightarrow 1.\]
Taking the $\mathbb{C}^\times$-dual of the above exact sequence, we get
\[1\rightarrow \mathrm{Hom}_{\mathrm{cts}}(\ker\left(H^1(F,S)\rightarrow H^1(E,T)\right),\mathbb{C}^\times)\rightarrow \mathrm{Hom}_{\mathrm{cts}}(S^{\mathrm{op}}(F),\mathbb{C}^\times)\rightarrow\mathrm{Hom}_{\mathrm{cts}}(T(E)/S(F),\mathbb{C}^\times)\rightarrow 1.\]
The surjectivity of the second map implies that 
$\theta|_{S(F)}=\mathbbm{1}$ if and only if $\theta\in \mathrm{Im}(\mathrm{Nm}_{T/S^{\mathrm{op}}}^*)$. Hence the condition $\theta\cdot \zeta^{-1}\in \mathrm{Im}(\mathrm{Nm}_{T/S^{\mathrm{op}}}^*)$ is equivalent to the condition $\theta\cdot \zeta^{-1}|_{S(F)}=\mathbbm{1}$, which by \Cref{productofcharacters} is further equivalent to the condition $\theta|_{S(F)}=(\epsilon_{\mathrm{HM}}\cdot \omega_{G(F),E}|_{[i](S(F))}\cdot \epsilon_{\mathrm{Kal}}|_{[i](S(F))})\circ [i]$.

It is left to prove that the second conditions are the same, namely, \[\mathrm{inv}(j_\pi,j_{\mathfrak{w}})\in\mathrm{Im} \left(H^1(F,S)\rightarrow H^1(E,T)\right)\Leftrightarrow \mathrm{inv}(j_\pi,j_{\mathfrak{w}})\in \ker\left(H^1(E,T)\rightarrow H^1(F,S^{\mathrm{op}})\right).\]

Notice that a stable conjugacy class of a maximal torus $S$ of $G$ uniquely determines a stable conjugacy class of a maximal torus $S^{\mathrm{op}}$ of $ G^{\mathrm{op}}$, and vice versa. This has already been proved in \cite[Lemma 11.4]{Pra20} for the archimedean case, which inspires our argument here. To see this, let $S$ be a maximal torus of $G$ so that $T=S_E$ determines a stable conjugacy class of $G_E$. By invoking the twisted bijection in \eqref{twistmap}, we can always make the stable conjugacy class of $T$ correspond to the trivial element in $H^1(E,W)$. Then we are led to prove that there is a natural bijection between pointed sets
\[\ker\left(H^1(F,W)\rightarrow H^1(E,W)\right)\cong \ker\left(H^1(F,W^{\mathrm{op}})\rightarrow H^1(E,W)\right),\]
i.e.,
\begin{equation}\label{stablecorrespond}
H^1(\mathrm{Gal}(E/F),W(E))\cong H^1(\mathrm{Gal}(E/F),W^{\mathrm{op}}(E)),
\end{equation}
where the $\sigma$-action on $W^\mathrm{op}(E)=W(E)$ is given by $\sigma^{\omega_0}(\sigma):=\omega_0 \sigma(\omega)\sigma(\omega_0)^{-1}=\omega_0 \sigma(\omega)\sigma(\omega_0)$ for any $\omega\in W$. It is known that there is a natural bijection between $H^1(\mathbb{Z}/2\mathbb{Z},W(E))$ and the set of twisted conjugacy classes of twisted involutions $\{\omega\in W(E)|\,\omega\sigma(\omega)=1\}/\sim_\sigma$
via the map $c\mapsto c(\sigma)$, where $\omega_1\sim_\sigma \omega_2$ if and only if there exists $\omega\in W$ such that $\omega_2=\omega \omega_1\sigma(\omega)^{-1}$. Then the bijection \eqref{stablecorrespond} is simply realized by 
\begin{equation}
\begin{aligned}
\{\omega\in W(E)|\,\omega\sigma(\omega)=1\}/\sim_\sigma&\rightarrow \{\omega\in W(E)|\,\omega\omega_0\sigma(\omega)\sigma(\omega_0)=1\}/\sim_{\sigma^{\omega_0}},\\
\omega&\mapsto \omega\omega_0,
\end{aligned}
\end{equation}

Again, by the long exact sequence associated to the short exact sequence of $F$-tori
$$1\rightarrow S\rightarrow \mathrm{Res}_{E/F}T\rightarrow S^{\mathrm{op}}\rightarrow 1,$$
we obtain $\ker\left(H^1(E,T)\rightarrow H^1(F,S^{\mathrm{op}})\right)=\mathrm{Im} \left(H^1(F,S)\rightarrow H^1(E,T)\right).$
\end{proof}

\begin{corollary}
Prasad's \Cref{conjecture1} is true for regular supercuspidal representations if and only if \Cref{productofcharacters} is true.
\end{corollary}

\section{Comparison of various characters}\label{section8}
To prove \Cref{productofcharacters}, we need to detect the relation between the following four quadratic characters: $\epsilon_{\mathrm{HM}, i(S)(F)}$, $\omega_{G_{\alpha}(F),E}|_{i(S)(F)}$, $\epsilon_{\mathrm{Kal}}|_{i(S)(F)}$ and $\zeta|_{i(S),F}$, where $\epsilon_{\mathrm{HM}, i(S)(F)}$ appears in \cite{HM08} for the computation the dimension of invariant forms, $\omega_{G_{\alpha}(F),E}|_{i(S)(F)}$ is the restriction of Prasad's quadratic character to $i(S)(F)$, $\epsilon_{\mathrm{Kal}}:T(E)\rightarrow \{\pm 1\}$ is the rectifying character, which is used to construct regular supercuspidal representations from $(T,\hat{j},\chi,\theta)$ in \cite{Kal19b}, and $\zeta:T(E)\rightarrow \mathbb{C}^{\times}$ is the character of $T(E)$ associated to the $\zeta$-data, which measures the difference of two different $\chi$-data. The proof of \Cref{productofcharacters} is not completely done in this article, we manage to prove \Cref{productofcharacters} for some particular examples, and when $E/F$ is unramified. A priori, it is hard to imagine there is any relation between these characters, since the constructions of these characters are of different nature. In fact, these characters are quite mysterious, and of independent interest themselves. Certain deep arithmetic hides behind these characters. We give a short summary of these characters.

The motivation of Prasad's character is to detect the sign of a conjugate self-dual representation. More precisely, for the case of $\mathrm{GL}_n$, the representations distinguished by the trivial character are conjugate orthogonal, and the ones distinguished by Prasad's character are conjugate symplectic. Hakim's character consists of two parts, both appearing in the computation of the dimension of the linear form. The first factor of Hakim's character originates from different choices of Yu's special isomorphism in his construction of tame supercuspidal representations. The second factor of Hakim's character appears in the depth zero part, which follows from Lusztig's work on the computation of the dimension of linear forms for finite groups of Lie type. Kaletha's character naturally appears in the construction of local Langlands for regular supercuspidal representations, which aims to produce the correct character identity based on the work of Adler, Debacker and Spice \cite{AS09, DS18}. The character $\zeta$ appears naturally measuring the difference of the base point of $L$-embedding given by different tamely ramified $\chi$-data determined by the characters of the torus.

In the next few subsections, we will present a description, or sometimes a reinterpretation of these characters, and we will also compute these characters explicitly taking into consideration of the distinction problem. For convenience, we use the notation $T$ for a maximal torus of $G_E$, and $S$ for a maximal torus of $G$ such that $S(E)$ is $G(E)$ conjugate to $T(E)$ instead of writing $j(T)$ and $i(S)$ for abstract tori $S,T$ and embeddings $i,j$. The basic strategy to understand these characters is to factorize these characters into products of components according to the $\Gamma_F  ,\Gamma_F\times\{\pm1\}, \Gamma_E$, or $\Gamma_E\times\{\pm1\}$-orbits of the roots. We will then try to compare these characters componentwise and provide a proof of the Character Comparison \Cref{productofcharacters} in the following particular cases.

\begin{proposition}\label{characterunramified}
The Character Comparison \Cref{productofcharacters} is true when $E/F$ is unramified.
\end{proposition}

\begin{proposition}\label{characterexample}
The Character Comparison \Cref{productofcharacters} is true for $\mathrm{SL}_2$, $\mathrm{GL}_2$, $\mathrm{GL}_n$ ($n$ odd) and $\mathrm{U}_{n,E/F}$ ($n$ odd) for arbitrary quadratic extension $E/F$.
\end{proposition}

\subsection{Prasad's character}
By \Cref{rootcharacter}, Prasad's quadratic character has the following factorization
$$\omega_{G(F),E}|_{S(F)}(t)=\prod_{\alpha\in \Phi_{\mathrm{sym}}/\Gamma_F}\omega_{EF_{\alpha}/F_{\alpha}}(\iota_{F_\alpha}\alpha(t)).$$

\subsection{Kaletha's character}
In our setting, we work with regular supercuspidal representations of $G(E)$. Let $T$ be an elliptic maximal torus of $G_E$. By \Cref{Kalcharacter}, Kaletha's character $\epsilon_{\mathrm{Kal}}:T(E)\rightarrow\ \{\pm1\}$ is given by
$$\epsilon_{\mathrm{Kal}}(t):=\prod\limits_{\substack{\alpha\in \Phi_{\mathrm{asy}}/\Gamma_{E}\times\{\pm 1\},\\ \frac{r_{i_\alpha}}{2}\in\mathrm{ord}_x{\alpha}}}\mathrm{sgn}_{k_{E_{\alpha}}^{\times}}\alpha(t) \prod\limits_{\substack{\alpha\in \Phi_{\mathrm{sym,ram}}/\Gamma_{E},\\ \alpha(t)\neq 1, \\ \mathrm{ord}(\alpha(t)-1)=0}}f_{(G,S)}(\alpha) \prod\limits_{\substack{\alpha\in \Phi_{\mathrm{sym,un}}/\Gamma_{E}\times\{\pm 1\},\\ \frac{r_{i_\alpha}}{2}\in\mathrm{ord}_x{\alpha}}}\mathrm{sgn}_{k_{E_{\alpha}}^{1}}\alpha(t),$$
where $\mathrm{sgn}_{k_{E_{\alpha}}^{\times}}$ and $\mathrm{sgn}_{k_{E_{\alpha}}^{1}}$ denote the unique non-trivial quadratic character of the cyclic groups $k_{E_{\alpha}}^{\times}$ and $k_{F_{\alpha}}^{1}$, and $f_{(G,S)}$ denotes the toral invariants defined in \cite{Kal15},  $i_\alpha$ denotes the non-negative integer associated to $\alpha$ in \Cref{notation} and $r_{i_\alpha}$ denotes the sequence of non-negative real numbers in Yu's construction in \Cref{Yuconstruction}. Moreover, the toral invariant can be explicitly described by
\begin{equation*}
	\begin{aligned}
		f_{(G,T)}:\Phi(G_E,T)_{\mathrm{sym}}&\rightarrow \{\pm1\}\\
		\alpha&\mapsto \omega_{E_\alpha/E_{\pm\alpha}}(\frac{[X_{\alpha},\tau X_{\alpha}]}{H_{\alpha}}), \\
	\end{aligned}
\end{equation*}
where $H_{\alpha}$ denotes $d \alpha^{\vee}(1)$, $X_{\alpha}$ denotes a non-zero element in $\mathfrak{g}_{\alpha}(E_{\alpha})$, $E_{\alpha}/E_{\pm\alpha}$ is a quadratic extension with $\tau\in \Gamma_{E_{\pm\alpha}}-\Gamma_{E_\alpha}$, and $\omega_{E_\alpha/E_{\pm\alpha}}$ denotes the quadratic character $E_{\pm\alpha}^{\times}\rightarrow E_{\pm\alpha}^{\times}/\mathrm{Nm}(E_{\alpha}^{\times})$ associated to $E_{\alpha}/E_{\pm\alpha}$ by local class field theory.

Taking into consideration of the distinction problem, we will only focus on the situation when $T$ is a base change of a maximal torus $S$ of $G$ over $F$. In this case, we can identify the absolute character groups $X_{F^{s}}^{*}(T)=X_{F^{s}}^{*}(S)$. The remaining task of this section is to compute $\epsilon_{\mathrm{Kal}}|_{S(F)}$.

\subsubsection{Toral invariant}\label{toralinvariant}
In \cite{Kal19b}, he gives a characterization of the toral invariant in terms of the Kottwitz sign of certain absolute rank $1$ group. Now we describe his formulation.

For a symmetric root $\alpha$ over $E$, let $G_{E,\pm\alpha}$ denote the subgroup of $G_E$ generated by the root subgroups for $\alpha$ and $-\alpha$.
We know that $G_{E,\pm\alpha}$ is defined over $E_{\pm\alpha}$, and a semisimple group of absolute rank $1$. If $G_{E,\pm\alpha}$ is simply connected, then the toral invariant equals the Kottwitz sign of $G_{E,\pm\alpha}$.
\begin{lemma}\label{toralkottwitz}
$$f_{(G,T)}(\alpha)=e_{E_{\pm\alpha}}(G_{E,\pm\alpha})=\begin{cases}
1     &   \quad\text{if } G_{E,\pm\alpha}\cong \mathrm{SL}_{2,E_{\pm\alpha}},\\
-1    &  \quad\text{if } G_{E,\pm\alpha}\cong \mathrm{SL}_{1,D_{E_{\pm\alpha}}}.
\end{cases}$$
where $D_{E_{\pm\alpha}}$ is a quaternion algebra over $E_{\pm\alpha}$ and $e(G)$ denotes the Kottwitz sign \rm{\cite{Kot83}}.
\end{lemma}

\begin{proof}
This is essentially proved by Kaletha. There is a decomposition of the Lie algebra $\mathfrak{g}_{E,\pm\alpha}$ of $G_{E,\pm\alpha}$
$$\mathfrak{g}_{E,\pm\alpha}=\mathfrak{t}_{E,\pm\alpha}\oplus \mathfrak{g}_{E,\alpha}\oplus \mathfrak{g}_{E,-\alpha},$$
such that $\mathfrak{t}_{E,\pm\alpha}$ is the Lie algebra of a one-dimensional anisotropic torus $T_{E,\pm\alpha}$ defined over $E_{\pm\alpha}$, which splits over $E_{\alpha}$. Since $G_{E,\pm\alpha}$ is simply connected, semisimple of rank $1$ over $E_{\alpha}$, we know $G_{E,\pm\alpha}$ must be an inner form of $\mathrm{SL}_2$ over $E_{\pm\alpha}$. From the definition of the toral invariant $f_{(G,T)}(\alpha)$, we can see that $f_{(G,T)}(\alpha)=f_{(G_{E,\pm\alpha},T_{E,\pm\alpha})}(\alpha)$. Hence the computation reduces to a computation of toral invariant of absolute rank one semisimple simply connected group. Notice that $f_{(\mathrm{SL}_2,T_{E,\pm\alpha})}$ is $+1$ by \cite[Lemma 4.5.3, Page 1126]{Kal19b}. By \cite[Proposition 4.3]{Kal15}, we have the following formula for an inner twist $\xi:G\rightarrow G'$ which restricts to an $F$-isomorphism $S\rightarrow S'$:
$$\prod_{\alpha\in \Phi(G',S')_{\mathrm{sym}}/\Gamma_F}f_{(G',S')}(\alpha)=e(G)e(G')\prod_{\alpha\in \Phi(G,S)_{\mathrm{sym}}/\Gamma_F}f_{(G,S)}(\alpha).$$
We apply this formula to the case of absolute rank one group over $E_{\pm\alpha}$, and get
\[f_{(\mathrm{SL}_{1,D_{E,\pm\alpha}},T_{E,\pm\alpha})}(\alpha)=e(\mathrm{SL}_{1,D_{ E_{\pm\alpha}}})=-1.\qedhere\]
\end{proof}

In fact, we can compute the toral invariant directly and present a uniform proof of \Cref{toralkottwitz}  for $\mathrm{SL}_{2,E_{\pm\alpha}}$ and $\mathrm{SL}_{1,D_{E_{\pm\alpha}}}$ in a rather elementary manner.

\begin{proof}[Another elementary proof of \Cref{toralkottwitz}]
Let $E_{\alpha}=E_{\pm\alpha}(x)$ for $x^{2}=a\in E_{\pm\alpha}^{\times}-(E_{\pm\alpha}^{\times})^{2}$. Let $B$ be a quaternion algebra over $E_{\pm\alpha}$, which possibly splits. Then $E_{\alpha}$ could be embedded into $B$ such that $B$ as a (right) $E_{\alpha}$-vector space could be written as $E_{\alpha}+yE_{\alpha}$ for $y^2=b\in E_{\pm\alpha}^{\times}-(E_{\pm\alpha}^{\times})^{2}$ with multiplication satisfying $xy=-yx$. Then the quaternion algebra is determined by the Hilbert symbol $(a,b)$. The left multiplication of $B^{\times}$ on $B$ is indeed $E_\alpha$-linear since $B$ is a right $E_\alpha$ vector space. By choosing the ordered basis $(1,y)$ over $E_{\alpha}$ , one can embed $B$ to $B_{E_{\alpha}}=\mathrm{M}_2(E_{\alpha})$ sending $m+yn$ to $\left(\begin{matrix} m&\tau(n)b\\ n&\tau(m)\end{matrix}\right)$. Notice that $B$ is the fixed-point set of $B_{E_{\alpha}}=\mathrm{M}_2(E_{\alpha})$ of the following involution
$$\left(\begin{matrix} p&q\\ r&s\end{matrix}\right)\mapsto \left(\begin{matrix} 0&b\\ 1&0\end{matrix}\right)\left(\begin{matrix} \tau(p)&\tau(q)\\ \tau(r)&\tau(s)\end{matrix}\right)\left(\begin{matrix} 0&1\\ b^{-1}&0\end{matrix}\right)=\left(\begin{matrix} \tau(s)&b\tau(r)\\ b^{-1}\tau(q)&\tau(p)\end{matrix}\right),$$
which defines the action of $\tau$. Over $E_{\alpha}$, the adjoint action of $E_{\alpha}^{\times}$ is given by $\mathrm{Ad}\left(\begin{matrix} t&\\ &t^{-1}\end{matrix}\right)$. We can choose
$X_{\alpha}=\left(\begin{matrix} 0&1\\ 0&0\end{matrix}\right)$, then $\tau(X_{\alpha})=\left(\begin{matrix} 0&0\\ b^{-1}&0\end{matrix}\right)$ and $H_{\alpha}=d\alpha^{\vee}(1)=\left(\begin{matrix} 1&0\\ 0&-1\end{matrix}\right)$, which implies $\frac{[X_{\alpha},\tau X_{\alpha}]}{H_{\alpha}}=b^{-1}$. Hence we obtain
\[ \begin{gathered}[b]
	f_{(G,T)}(\alpha)=\omega_{E_\alpha/E_{\pm\alpha}}(\frac{[X_{\alpha},\tau X_{\alpha}]}{H_{\alpha}})=(a,b^{-1})=(a,b)=
	\begin{cases}
		1     &  \quad\text{if }B\text{ is split},\\
		-1    &  \quad\text{if }B\text{ is non-split}.
	\end{cases}
	\\[-\dp\strutbox]
\end{gathered}
\qedhere   
\]
\end{proof}

For the purpose of dealing with the distinction problem, we only focus on the case that $T$ is a base change of a torus $S$ over $F$. It is clear from the definition that $\alpha$ must be symmetric over $F$ if it is symmetric over $E$. We then have the following cases:
\begin{enumerate}
\item If $E_{\alpha}\neq F_{\alpha}$, that is, $E_{\alpha}/F_{\alpha}$ and $E_{\pm\alpha}/F_{\pm\alpha}$ are both quadratic extensions, then $G_{E,\pm\alpha}$ is also a base change of a rank one group over $F_{\pm\alpha}$. Since $\mathrm{SL}_{1,D_{F_{\pm\alpha}}}$ splits over any quadratic extension, in particular, over $E_{\pm\alpha}$, we know that $f_{(G,T)}(\alpha)=1$.

\item If $E_{\alpha}=F_{\alpha}$, then the $\Gamma_F$-orbit of $\alpha$ breaks into two $\Gamma_E$-orbits, $\Gamma_E \alpha$ and $\Gamma_Es\alpha$ for $s\in \Gamma_F-\Gamma_E$. Then we can put the toral invariant of $\alpha$ and that of $s\alpha$ together, and get
    $$f_{(G,T)}(\alpha)\cdot f_{(G,T)}(s\alpha)=f_{(G,T)}(\alpha)^{2}=1.$$
\end{enumerate}
This implies that the components contributed by the toral invariants in $\epsilon_{\mathrm{Kal}}$ are trivial in our distinction problem.
\subsubsection{Explicit computation}
We try to compute the quadratic character $\epsilon_{\mathrm{Kal}}(t)|_{S(F)}$ explicitly. Notice that Kaletha's character is given by a product of characters over $\Gamma_{E}\times\{\pm 1\}$-orbits of roots, we will deal with the component of a fixed $\Gamma_E\times \{\pm 1\}$-orbit at each time.

\begin{enumerate}
\item $\alpha$ is asymmetric over $F$. Then $\alpha$ is automatically asymmetric over $E$, that is, $F_{\pm\alpha}=F_{\alpha}$ and $E_{\pm\alpha}=E_{\alpha}$. There are two cases:

\begin{enumerate}
\item The $\Gamma_F$-orbit of $\alpha$ becomes two $\Gamma_E$-orbits. In this case, we have $E_{\alpha}=F_{\alpha}$.
$$\xymatrix{
  E_{\alpha} \ar@{=}[d] \ar@{=}[r] & E_{\pm\alpha} \ar@{-}[d] \ar@{-}[r] & E \ar@{-}[d] \\
  F_{\alpha} \ar@{=}[r] & F_{\pm\alpha} \ar@{-}[r] & F   }$$
Let $\alpha$ and $\sigma(\alpha)$ be the two $\Gamma_E$-orbits, both of which contribute to the character. We can put the $\alpha$-component and the $\sigma(\alpha)$-component of Kaletha's character together. More precisely, for $s\in S(F)$, that is, $\sigma(s)=s$, we obtain
\begin{equation*}
\begin{aligned}
\mathrm{sgn}_{k_{E_{\alpha}}^{\times}}(\alpha(s))\mathrm{sgn}_{k_{E_{\sigma(\alpha)}}^{\times}}(\sigma(\alpha)(s))
&=\mathrm{sgn}_{k_{E_{\alpha}}^{\times}}(\alpha(s))\mathrm{sgn}_{k_{E_{\sigma(\alpha)}}^{\times}}(\alpha(\sigma(s)))\\
&=\mathrm{sgn}_{k_{E_{\alpha}}^{\times}}(\alpha(s))^{2}=1.
\end{aligned}
\end{equation*}

\item The $\Gamma_F$-orbit of $\alpha$ remains one $\Gamma_E$-orbit. In this case, we have $[E_{\alpha}:F_{\alpha}]=2$.
$$\xymatrix{
  E_{\alpha} \ar@{-}[d]_{2} \ar@{=}[r] & E_{\pm\alpha} \ar@{-}[d] \ar@{-}[r] & E \ar@{-}[d] \\
  F_{\alpha} \ar@{=}[r] & F_{\pm\alpha} \ar@{-}[r] & F   }$$
\begin{enumerate}
\item If $E/F$ is unramified, $E_{\alpha}/F_{\alpha}$ must be unramified.

Since $E/F$ is unramified, we know $\alpha(s)(\bmod \,1+\varpi_{F_{\alpha}})\in k_{F_{\alpha}}^{\times}\subset k_{E_{\alpha}}^{\times}$ for $s\in S(F)$. Let $a$ be a generator of the cyclic group $k_{E_{\alpha}}^{\times}$, that is,
$$k_{E_{\alpha}}^{\times}=\langle a|\, a^{q_{E_\alpha}-1}=1\rangle.$$
Then we have
$$k_{F_{\alpha}}^{\times}=\langle a^{q_{F_{\alpha}}+1}\rangle=\langle (a^{\frac{q_{F_{\alpha}}+1}{2}})^{2}\rangle,$$
since $p\neq 2$, which means that every element in $k_{F_{\alpha}}^{\times}$ is a square in $k_{E_{\alpha}}^{\times}$. Hence the restriction of the character $\mathrm{sgn}_{k_{E_{\alpha}}^{\times}}\circ\alpha$ to $S(F)$ is trivial.

\item If $E/F$ is ramified, $E_{\alpha}/F_{\alpha}$ could be ramified or unramified. If it is unramified, the character is trivial by the same reason. If it is ramified, then $k_{E_{\alpha}}^{\times}=k_{F_{\alpha}}^{\times}$, and the quadratic character is given by $(\mathrm{sgn}_{k_{E_{\alpha}}^{\times}}\circ\alpha)|_{S(F)}=\mathrm{sgn}_{k_{F_{\alpha}}^{\times}}\circ\alpha|_{S(F)} $.
\end{enumerate}
\end{enumerate}

\item $\alpha$ is symmetric unramified over $F$, that is, $F_{\alpha}/F_{\pm\alpha}$ is unramified quadratic.
\begin{enumerate}
\item The $\Gamma_F$-orbit of $\alpha$ becomes two $\Gamma_E$-orbits, that is, $E_{\alpha}=F_{\alpha}$. In this case, we have
$$\xymatrix{
  E_{\alpha} \ar@{=}[d] \ar@{-}[r] & E_{\pm\alpha} \ar@{-}[d] \ar@{-}[r] & E \ar@{-}[d] \\
  F_{\alpha} \ar@{-}[r]^{2} & F_{\pm\alpha} \ar@{-}[r] & F  . }$$

\begin{enumerate}
\item $\alpha$ is asymmetric over $E$, that is, $E_{\alpha}=E_{\pm\alpha}=F_{\alpha}$ and $E_{\pm\alpha}/F_{\pm\alpha}$ is unramified quadratic.
$$\xymatrix{
  E_{\alpha} \ar@{=}[d] \ar@{=}[r] & E_{\pm\alpha} \ar@{-}[d]_{2} \ar@{-}[r] & E \ar@{-}[d] \\
  F_{\alpha} \ar@{-}[r]^{2} & F_{\pm\alpha} \ar@{-}[r] & F   }$$

In this case, one has two $\Gamma_E$-orbits. We can choose a lift $s\in \Gamma_F-\Gamma_E$ of $\sigma$ such that $s\cdot\alpha=-\alpha$, and we only need to compute the contribution of $\alpha$ or $-\alpha$ in the asymmetric case. Hence the quadratic character is given by $\mathrm{sgn}_{k_{E_{\alpha}}^{\times}}\circ\alpha$ on $T(E)$.

Since $\alpha$ is symmetric unramified over $F$ in this case, we know $\alpha(s)(\bmod\,1+\varpi_{F_{\alpha}})$ lies in $k_{F_{\alpha}}^{1}\subset k_{F_{\alpha}}^{\times}$ for any $s\in S(F)$. Similarly, let $a$ be a generator of the cyclic group $k_{E_{\alpha}}^{\times}=k_{F_{\alpha}}^{\times}$, that is,
$$k_{F_{\alpha}}^{\times}=\langle a|\, a^{q_{F_\alpha}-1}=1\rangle.$$
Then we have
$$k_{F_{\alpha}}^{1}=\langle  a^{q_{F_{\pm\alpha}}-1}\rangle=\langle (a^{\frac{q_{F_{\pm\alpha}}-1}{2}})^{2}\rangle,$$
since $p\neq 2$, which means that every element in $k_{F_{\alpha}}^{1}$ is a square in $k_{F_{\alpha}}^{\times}$. Hence the restriction of the character $\mathrm{sgn}_{k_{E_{\alpha}}^{\times}}\circ\alpha$ to $S(F)$ is trivial.

\item $\alpha$ is symmetric (necessarily unramified) over $E$, that is, $E_{\alpha}/E_{\pm\alpha}$ is an unramified quadratic extension with $E_{\pm\alpha}=F_{\pm\alpha}$.
$$\xymatrix{
  E_{\alpha} \ar@{=}[d] \ar@{-}[r]^{2} & E_{\pm\alpha} \ar@{=}[d] \ar@{-}[r] & E \ar@{-}[d] \\
  F_{\alpha} \ar@{-}[r]^{2} & F_{\pm\alpha} \ar@{-}[r] & F   }$$
  In this case, one has two $\Gamma_E$-orbits $\alpha$ and $\sigma(\alpha)$, both of which are symmetric over $E$. We can put the $\alpha$-component and the $\sigma(\alpha)$-component of Kaletha's character together. For $s\in S(F)$, that is, $\sigma(s)=s$, we obtain
\begin{equation*}
\begin{aligned}
\mathrm{sgn}_{k_{E_{\alpha}}^{1}}(\alpha(s))\mathrm{sgn}_{k_{E_{\sigma(\alpha)}}^{1}}(\sigma(\alpha)(s))
&=\mathrm{sgn}_{k_{E_{\alpha}}^{1}}(\alpha(s))\mathrm{sgn}_{k_{E_{\sigma(\alpha)}}^{1}}(\alpha(\sigma( s)))\\
&=\mathrm{sgn}_{k_{E_{\alpha}}^{1}}(\alpha(s))^{2}=1.
\end{aligned}
\end{equation*}
\end{enumerate}
\item The $\Gamma_F$-orbit of $\alpha$ remains one $\Gamma_E$-orbit, that is, $E_{\alpha}/F_{\alpha}$ is a quadratic extension. In this case, both $E_{\alpha}/E_{\pm\alpha}$ and $F_{\alpha}/F_{\pm\alpha}$ are unramified quadratic. Since $\mathrm{Stab}_{\Gamma_F}\alpha\neq \mathrm{Stab}_{\Gamma_E}\{\alpha,-\alpha\}$ inside $\mathrm{Stab}_{\Gamma_F}\{\alpha,-\alpha\}$, we know that $E_{\pm\alpha}\neq F_{\alpha}$, which means $E_{\alpha}$ is a biquadratic extension of $F_{\pm\alpha}$.
    $$\xymatrix{
  E_{\alpha} \ar@{-}[d]_{2} \ar@{-}[r]^{2} & E_{\pm\alpha} \ar@{-}[d]_{2} \ar@{-}[r] & E \ar@{-}[d] \\
  F_{\alpha} \ar@{-}[r]^{2} & F_{\pm\alpha} \ar@{-}[r] & F   }$$
This implies that $E_{\alpha}/E_{\pm\alpha}$ is unramified and $E/F,E_{\pm\alpha}/F_{\pm\alpha},E_{\alpha}/F_{\alpha}$ are all ramified. Hence $\alpha$ is symmetric unramified over $E$. The quadratic character is given by $\mathrm{sgn}_{k_{E_{\alpha}}^{1}}\circ \alpha$ and its restriction to $S(F)$ is $\mathrm{sgn}_{k_{F_{\alpha}}^{1}}\circ \alpha|_{S(F)}$ since $k_{E_\alpha}=k_{F_\alpha}$ and $k_{E_{\pm\alpha}}=k_{F_{\pm\alpha}}$.
\end{enumerate}

\item $\alpha$ is symmetric ramified over $F$, that is, $F_{\alpha}/F_{\pm\alpha}$ is ramified quadratic.
\begin{enumerate}

\item The $\Gamma_F$-orbit of $\alpha$ becomes two $\Gamma_E$-orbits, that is, $E_{\alpha}=F_{\alpha}$. In this case, we have

$$\xymatrix{
  E_{\alpha} \ar@{=}[d] \ar@{-}[r] & E_{\pm\alpha} \ar@{-}[d] \ar@{-}[r] & E \ar@{-}[d] \\
  F_{\alpha} \ar@{-}[r]^{2} & F_{\pm\alpha} \ar@{-}[r] & F  . }$$

\begin{enumerate}
\item $\alpha$ is asymmetric over $E$, that is, $E_{\alpha}=E_{\pm\alpha}=F_{\alpha}$ and $E_{\pm\alpha}/F_{\pm\alpha}$ is ramified quadratic.
$$\xymatrix{
  E_{\alpha} \ar@{=}[d] \ar@{=}[r] & E_{\pm\alpha} \ar@{-}[d]_{2} \ar@{-}[r] & E \ar@{-}[d] \\
  F_{\alpha} \ar@{-}[r]^{2} & F_{\pm\alpha} \ar@{-}[r] & F   }$$
In this case, one has two $\Gamma_E$-orbits, $\alpha$ and $\sigma(\alpha)$. Since $\sigma(\alpha)=-\alpha$, we only need to compute the contribution of $\alpha$ or $-\alpha$ in the asymmetric case. Hence the quadratic character is given by $\mathrm{sgn}_{k_{E_{\alpha}}^{\times}}\circ \alpha$ and its restriction to $S(F)$ is given by $\mathrm{sgn}_{k_{F_{\alpha}}^{\times}}\circ \alpha|_{S(F)}.$ 
\item $\alpha$ is symmetric (necessarily ramified) over $E$, that is, $E_{\alpha}/E_{\pm\alpha}$ ramified quadratic extension with $E_{\pm\alpha}=F_{\pm\alpha}$.
$$\xymatrix{
  E_{\alpha} \ar@{=}[d] \ar@{-}[r]^{2} & E_{\pm\alpha} \ar@{=}[d] \ar@{-}[r] & E \ar@{-}[d] \\
  F_{\alpha} \ar@{-}[r]^{2} & F_{\pm\alpha} \ar@{-}[r] & F   }$$
In this case, one has two $\Gamma_E$-orbits, $\alpha$ and $\sigma(\alpha)$, both of which are symmetric over $E$. We can put the $\alpha$-component and the $\sigma(\alpha)$-component of Kaletha's character together. The quadratic character is given by the toral invariant
$$f_{(G,T)}(\alpha)\cdot f_{(G,T)}(s\alpha)=f_{(G,T)}(\alpha)^{2}=1,$$
and its restriction to $S(F)$ is trivial.
\end{enumerate}
\item The $\Gamma_F$-orbit of $\alpha$ remains one $\Gamma_E$-orbit, that is, $E_{\alpha}/F_{\alpha}$ is a quadratic extension. In this case, both $E_{\alpha}/E_{\pm\alpha}$ quadratic and $F_{\alpha}/F_{\pm\alpha}$ is ramified quadratic. By the same reason $E_{\alpha}$ is biquadratic over $F_{\pm\alpha}$ and $E_{\alpha}/F_{\alpha}$ is unramified.
    $$\xymatrix{
  E_{\alpha} \ar@{-}[d]_{2} \ar@{-}[r]^{2} & E_{\pm\alpha} \ar@{-}[d]_{2} \ar@{-}[r] & E \ar@{-}[d] \\
  F_{\alpha} \ar@{-}[r]^{2} & F_{\pm\alpha} \ar@{-}[r] & F   }$$
\begin{enumerate}
\item $E/F$ is unramified. In this case, $E_{\pm\alpha}/F_{\pm\alpha}$ is unramified and $F_{\alpha}/F_{\pm\alpha}$ is ramified, which means that $\alpha$ is symmetric ramified over $E$.  Hence the character is given by $f_{(G,T)}(\alpha)=1$ so that its restriction to $S(F)$ is also trivial.
\item $E/F$ is ramified, $E_{\pm\alpha}/F_{\pm\alpha}$ could be unramified or ramified.
\begin{enumerate}
\item $E_{\pm\alpha}/F_{\pm\alpha}$ is unramified, and $E_{\alpha}/F_{\pm\alpha}$ is biquadratic. In this case, $\alpha$ is symmetric ramified over $E$. The quadratic character is given by $f_{(G,T)}(\alpha)=1$ so that its restriction to $S(F)$ is also trivial.
\item $E_{\pm\alpha}/F_{\pm\alpha}$ is ramified, and $E_{\alpha}/F_{\pm\alpha}$ is biquadratic. In this case, $E_{\alpha}/E_{\pm\alpha}$ is unramified, that is, $\alpha$ is symmetric unramified over $E$. Hence the character is given by $\mathrm{sgn}_{k_{E_{\alpha}}^{1}}\circ \alpha$ and its restriction to $S(F)$ is $\mathrm{sgn}_{k_{E_{\alpha}}^{1}}\circ \alpha|_{S(F)}$.
\end{enumerate}
\end{enumerate}
\end{enumerate}
\end{enumerate}

\subsection{Hakim's character}
According to \cite[3.2]{Hak18} and \Cref{hakimcharacter}, Hakim's character is given by
$$\epsilon_{\mathrm{HM}}(s)=\epsilon^{-}_{i_g(S)}(s)\cdot\epsilon^{+}_{G^0(E)_{[x]}, [g]}|_{i_g(S)}(s)=\det(\mathrm{Ad}(s)|_{\mathfrak{g}^{0}(E)_{[gx]}^\sigma})\cdot \prod\limits_{i=0}^{d-1}(\det \mathrm{Ad}(s)|_{{}^g\mathfrak{W}_i^+})^{\frac{q_F-1}{2}},$$
where ${}^g\mathfrak{W}_i^+:=(\mathrm{Ad}(g)(\bigoplus\limits_{\alpha\in\Phi^{i+1}-\Phi^{i}}\mathfrak{g}_{\alpha})^{\Gamma_E})_{gx,s_i:s_{i}+}^{\sigma}.$
is a vector space over the residue field $k_F$ and $\Phi^{i+1}$ denotes the set of roots $\Phi(G^{i+1},\mathrm{Ad}(g)(j_{\pi}(T)))$.

In this section, we will present a new interpretation of Hakim's character. Before we do this, we first list some properties of Hakim's character. These properties are mainly proved by Hakim \cite{Hak21, Hak18} and Zhang \cite{Zha20a, Zha20b}.

\begin{proposition}[{\cites[Lemma 3.2.1]{Hak18}}]
The character $\epsilon^+$ satisfies
$\epsilon^{+}|_{(((\mathsf{G}^{0}_{[x]})^{\circ})^{g\cdot \sigma})^{\circ}(k_F)}=\mathbbm{1}.$
\end{proposition}
The proof depends on the algebraic nature of this quadratic character, which means that it is the restriction of an algebraic quadratic character of $\mathsf{G}_{i}^{\circ}(\overline{k_F})\cong G^{i}(F^{\mathrm{ur}})_{x,0:0^{+}}$, and the fact that one has a decomposition $\mathsf{G}_{i}^{\circ}={\mathsf{G}_{i}^{\circ}}_{\mathrm{der}} Z(\mathsf{G}_{i}^{\circ})$.

\begin{proposition}[{\cites{Hak18}[Definition 3.4]{Zha20a}}]
The character $\epsilon^{-}:{i(S)(F)}\rightarrow \{\pm 1\}$ factors through $\epsilon_{\mathsf{T}}: \mathsf{S}(k_F)=\mathsf{T}(k_E)^{g\cdot\sigma}\rightarrow \{\pm 1\}$, where the latter is the Lusztig quadratic character defined by $$\epsilon_{\mathsf{T}}(s)=\sigma_{\mathrm{L}} (C_{\mathsf{G}}((\mathsf{T}^{g\cdot\sigma})^{\circ})) \sigma_{\mathrm{L}}(C_{G}(s)^{\circ}\cap C_{\mathsf{G}}((\mathsf{T}^{g\cdot\sigma})^{\circ})),$$
where $\sigma_{\mathrm{L}}(G)$ is the sign defined by Lusztig, which is $(-1)^{\mathrm{rank}_{k_F}(G)}$ for any reductive group $G$ defined over $k_F$.
\end{proposition}

\begin{proposition}[{\cites[Proposition 4.3.2]{Hak21}}]
Let $\vartheta$ be any $\mathbb{F}_q$-involution of $\mathsf{G}$ preserving $\mathsf{T}$. Lusztig's quadratic character has the following factorization
$$\epsilon_{\mathsf{T}}(s)=\prod_{\alpha\in \Gamma_{\mathbb{F}_q}\backslash \Phi_{\vartheta}}\alpha(s),$$
where $\Phi_{\vartheta}$ denotes $\Phi(C_{\mathsf{G}}({\mathsf{T}^{\vartheta}}^{\circ}),\mathsf{T})$.
\end{proposition}

\begin{proposition}[{\cites[Corollary 3.18]{Zha20a}}]\label{zhang}

For Galois pairs, $\epsilon^{-}$ is always trivial.
\end{proposition}

\subsubsection{A new interpretation of $\epsilon^{+}_{T(E),[g]}|_{i(S)}=\epsilon^{+}_{i(S)}$}\label{Hakimcharacter}

Since we have already mentioned before \Cref{characterunramified} that we will directly use the notations $S,T$ for maximal tori of $G$ and $G_E$ instead of $i(S), j(T)$ with $S,T$ being abstract tori and $i, j$ being embeddings, we will compute the following character
$$\epsilon_{\mathrm{HM}}(s)=\epsilon^{-}_{S}(s)\cdot\epsilon^{+}_{G^0(E)_{[x]}}|_{S}(s)=\epsilon^{+}_{G^0(E)_{[x]}}|_{S}(s)= \prod\limits_{i=0}^{d-1}(\det \mathrm{Ad}(s)|_{\mathfrak{W}_i^+})^{\frac{q_F-1}{2}},$$
where the second equality holds because of the triviality of $\epsilon^-$ in \Cref{zhang}.

\begin{lemma}\label{Hakimfactor}
	Hakim's character has a factorization
	\begin{equation}\label{hakimfactorization}
		\epsilon_{\mathrm{HM}}(s)=\prod\limits_{\substack{\alpha\in \Phi_{\mathrm{asy}}/\Gamma_{F},\\ \frac{r_{i_\alpha}}{2}\in\mathrm{ord}_x{\alpha}/E}}\mathrm{sgn}_{k_{F_{\alpha}}^{\times}}\alpha(s) \prod\limits_{\substack{\alpha\in \Phi_{\mathrm{sym,ram}}/\Gamma_{F},\\ \frac{r_{i_\alpha}}{2}\in\mathrm{ord}_x{\alpha}/E}}\mathrm{sgn}_{k_{F_{\alpha}}^\times}\alpha(s).
	\end{equation}
\end{lemma}

\begin{proof}
Since ${\mathfrak{W}_i}^+=((\bigoplus\limits_{\alpha\in\Phi^{i+1}-\Phi^{i}}\mathfrak{g}_{\alpha})^{\Gamma_E})_{x,s_i:s_{i}+}^{\sigma}$
is a vector space over the residue field $k_F$, we rewrite $\mathfrak{W}_i^+$ as a direct sum over $\Gamma_F$-orbits of roots appearing in $\Phi^{i+1}-\Phi^{i}$. Namely, we get the following decomposition
\begin{equation}\label{newinterpretation}
\begin{aligned}
\mathfrak{W}_i^+=\bigoplus_{\mathcal{O}\in (\Phi^{i+1}-\Phi^{i})/\Gamma_{F}}(\bigoplus _{\alpha\in \mathcal{O}}\mathfrak{g}_{\alpha})(F)_{x, s_i: s_{i}+}.
\end{aligned}
\end{equation}
Moreover, there is an isomorphism of $F$-vector spaces
\begin{equation}\label{sumoverorbit}
	\begin{aligned}(\bigoplus_{\alpha\in\mathcal{O}}\mathfrak{g}_{\alpha})(F)\cong F_{\alpha}.
\end{aligned}
\end{equation}
The adjoint action of $S$ on $\mathfrak{W}_i^+$ becomes multiplication of $\alpha(s)$ on each piece $(F_{\alpha})_{x,s_1:s_i+}$ under the identifications \eqref{newinterpretation} and \eqref{sumoverorbit}, where the determinant  map corresponds to the norm map $\mathrm{Nm}_{k_{F_{\alpha}}/k_{F}}$ on the corresponding piece. 

Furthermore, if $\alpha\in \Phi_{\frac{r_{i_\alpha}}{2},E}$, i.e., $s_{i_\alpha}$ is a jump for $\alpha$ over $E$, then $\epsilon^{+}$ is given by $$\mathrm{sgn}_{k_{F}^{\times}}(\mathrm{Nm}_{k_{F_{\alpha}}/k_F}\alpha(s))=\mathrm{sgn}_{k_{F_{\alpha}}^{\times}}(\alpha(s)).$$
Otherwise, $\epsilon ^+$ is trivial. Moreover, if $[F_{\alpha}:F_{\pm\alpha}]=2$ and $F_{\alpha}/F_{\pm\alpha}$ is unramified, this character is trivial since
$$\mathrm{Nm}_{k_{F_{\alpha}}/k_{F}}(\alpha(s))=\mathrm{Nm}_{k_{F_{\pm\alpha}}/k_{F}}\mathrm{Nm}_{k_{F_{\alpha}}/k_{F_{\pm\alpha}}}(\alpha(s))=\mathrm{Nm}_{k_{F_{\pm\alpha}}/k_{F}}(1)=1.$$
\end{proof}

\subsection{The character associated to $\zeta$-data}\label{zeta}
Since the character associated to $\zeta$-data is constructed from the parameter side measuring the difference between ${\chi_{\mu}}_{\mathrm{BC}}$ and $\chi_{\mu\circ\mathrm{Nm}_{T(E)/S^{\mathrm{op}}(F)}}$, we need to deal with the roots $\{\alpha^{\mathrm{op}}\}$ associated to a $\Gamma_F$-stable $G(\overline{F})$ conjugacy class of embeddings $j^{\mathrm{op}}:S^{\mathrm{op}}_{\overline{F}}\rightarrow G^{\mathrm{op}}_{\overline{F}}$ instead of $\{\alpha\}$. The first step towards the explicit computation of this character is to understand the relation between $\alpha$ and $\alpha^{\mathrm{op}}$. Miraculously, there is a ``bi-quadratic phenomena" which makes it possible for us to take a comparison between the character from $\zeta$-data and the other $3$ quadratic characters. More precisely, by \eqref{zetacharacter}, the $\zeta_{T(E),{\chi_{\mu}}_{\mathrm{BC}},\chi_{\mu\circ\mathrm{Nm}_{T(E)/S^{\mathrm{op}}(F)}}}:T(E)\rightarrow \mathbb{C}^\times$ has the following expression
$$\zeta_{T(E),{\chi_{\mu}}_{\mathrm{BC}},\chi_{\mu\circ\mathrm{Nm}_{T(E)/S^{\mathrm{op}}(F)}}}=\prod_{\alpha\in \Phi_{\mathrm{sym}}/\Gamma_E}\frac{{\chi_{\mu}}_{\mathrm{BC},\alpha}}{\chi_{\mu\circ\mathrm{Nm}_{T(E)/S^{\mathrm{op}}(F)},\alpha}}(\iota_{\alpha} \alpha(t))\prod_{\alpha\in \Phi_{\mathrm{asy}}/\Gamma_E\times\{\pm1\}}\frac{{\chi_{\mu}}_{\mathrm{BC},\alpha}}{\chi_{\mu\circ\mathrm{Nm}_{T(E)/S^{\mathrm{op}}(F)},\alpha}}(\alpha(t)),$$
where $\chi_{\mu\circ\mathrm{Nm}_{T(E)/S^{\mathrm{op}}(F)}}$ denotes the minimally ramified $\chi$-data determined by the character $\mu\circ\mathrm{Nm}_{T(E)/S^{\mathrm{op}}(F)}$ as in \Cref{Lembedding} and ${\chi_{\mu}}_{\mathrm{BC}}$ denotes the base-changed $\chi$-data as in \Cref{basechangeregular}.
We will compute each $\alpha$ component of $\zeta_{T(E),{\chi_{\mu}}_{\mathrm{BC}},\chi_{\mu\circ\mathrm{Nm}_{T(E)/S^{\mathrm{op}}(F)}}}|_{S(F)}$ explicitly by determining ${\chi_{\mu}}_{\mathrm{BC},\alpha}$ and $\chi_{\mu\circ\mathrm{Nm}_{T(E)/S^{\mathrm{op}}(F)},\alpha}$ explicitly.

\subsubsection{Relation between $F_\alpha$ and $F_{\alpha^{\mathrm{op}}}$}

Observe that $S^\mathrm{op}$ is constructed from $S$ via twisting the $\mathrm{Gal}(E/F)$-action by the inverse map. Hence $\alpha^{\mathrm{op}}$ as an element of the absolute root system $\Phi(G_{F^{s}},S_{F^{s}})$ is the same as $\alpha$, but carries a twisted Galois action. More precisely, the $\Gamma_F$-action on $\alpha^{\mathrm{op}}$ is described by
\[\gamma \ast \alpha:=\begin{cases}
	\gamma \cdot\alpha,&\quad\forall \gamma\in \Gamma_E, \\
	-\gamma\cdot \alpha,&\quad \forall\gamma\in \Gamma_F-\Gamma_E,
\end{cases}\]
where $\cdot$ denotes the original action on $\alpha$, and $\ast$ denotes the twisted action. We use the notation $\alpha^{\mathrm{op}}$ for $\alpha$ with this twisted $\Gamma_F$-action. Let $F_{\alpha^{\mathrm{op}}}$ denote the subfield of $F^s$ corresponding to the stabilizer of $\alpha$ in $\Gamma_F$ under this $\ast$ action, which we denote by $\mathrm{Stab}_{\Gamma_F}(\alpha^{\mathrm{op}})$.

\begin{lemma}
There is an equality of fields $F_{\pm\alpha}=F_{\pm\alpha^{\mathrm{op}}}$.
\end{lemma}
\begin{proof}
We only need to prove that $\mathrm{Stab}_{\Gamma_F}\{\pm\alpha\}=\mathrm{Stab}_{\Gamma_F}\{\pm\alpha^{\mathrm{op}}\}$. For any $\gamma\in\mathrm{Stab}_{\Gamma_F}\{\pm\alpha\}$, we know
$$\gamma\ast \alpha=\pm\gamma\cdot\alpha=\pm \pm\alpha=\pm\alpha,$$
which means $\mathrm{Stab}_{\Gamma_F}\{\pm\alpha\} \subset \mathrm{Stab}_{\Gamma_F}\{\pm\alpha^{\mathrm{op}}\}$. The reverse direction is similar.
\end{proof}

\begin{lemma}\label{biquadratic}
If $E_{\alpha}$ is a quadratic extension of $E_{\pm\alpha}$ and $F_{\alpha}$ with $E_{\pm\alpha}\neq F_{\alpha}$, then $F_{\alpha^{\mathrm{op}}}$ is the other intermediate field of the biquadratic extension $E_{\alpha}$ over $F_{\pm\alpha}=F_{\pm\alpha^{\mathrm{op}}}$.
\end{lemma}

\begin{proof}
Let $\sigma\in \mathrm{Gal}(E_{\alpha}/F_{\alpha})$ and $\tau\in \mathrm{Gal}(E_{\alpha}/E_{\pm\alpha})$ be the two non-trivial elements. We have
$$\sigma\cdot \alpha=\alpha,\quad\tau\ast\alpha=\tau\cdot\alpha=-\alpha,$$
which implies
$$(\sigma\tau)\ast\alpha=\sigma\ast(\tau\ast \alpha)=\sigma\ast(-\alpha)=-\sigma\ast\alpha=\sigma\cdot\alpha=\alpha.$$
Hence $F_{\alpha^{\mathrm{op}}}$ is the third intermediate field of $E_{\alpha}$ corresponding to the order two element $\sigma\tau$.
\end{proof}

\begin{lemma}
If $E_\alpha=F_{\alpha}$ and $\alpha$ is both symmetric over $F$ and $E$, then $F_{\alpha^{\mathrm{op}}}=F_{\alpha}$.
\end{lemma}
\begin{proof}
Since $\alpha^{\mathrm{op}}$ is symmetric over $E$, it is necessarily symmetric over $F$. Hence we get $$F_{\pm\alpha^{\mathrm{op}}}=E_{\pm\alpha}\subsetneq F_{\alpha^{\mathrm{op}}}\subset E_{\alpha},$$
which means that $F_{\alpha^{\mathrm{op}}}=E_{\alpha}=F_{\alpha}$.
\end{proof}

\begin{lemma}\label{biquadraticdegenerate2}
If $E_\alpha=F_{\alpha}$ and $\alpha$ is symmetric over $F$ but asymmetric over $E$, then $F_{\alpha^{\mathrm{op}}}=F_{\pm\alpha^{\mathrm{op}}}$.
\end{lemma}

\begin{proof}
On the one hand, we know $F_{\pm\alpha^{\mathrm{op}}}\subset F_{\alpha^{\mathrm{op}}}\subset E_{\alpha}$ and $E_{\alpha}=F_{\alpha}$ over $F_{\pm\alpha^{\mathrm{op}}}=F_{\pm\alpha}$ is quadratic. On the other hand, we also know $F_{\alpha^{\mathrm{op}}}\neq F_{\alpha}$ since the $\sigma$-action on $\alpha^{\mathrm{op}}$ is non-trivial. Hence $ F_{\alpha^{\mathrm{op}}}=F_{\pm\alpha^{\mathrm{op}}}$, which means that $\alpha^{\mathrm{op}}$ is asymmetric over $F$.
\end{proof}

\subsubsection{Explicit computation}\label{explicitzeta}

\begin{enumerate}
\item The $\Gamma_F$-orbit of $\alpha^{\mathrm{op}}$ remains a single $\Gamma_E$-orbit. The symmetry of $\alpha^{\mathrm{op}}$ is necessarily the same over $E$ and $F$.

\begin{enumerate}

\item Symmetric case

\begin{enumerate}
\item $\alpha^{\mathrm{op}}$ is symmetric ramified both over $E$ and over $F$.

In this case, $\alpha$ is also symmetric over $F$ since $\alpha$ is symmetric over $E$. Then $E_{\alpha}$ is a biquadratic extension of $F_{\pm{\alpha}^{\mathrm{op}}}=F_{\pm{\alpha}}$ with three intermediate fields $F_{\alpha}, F_{{\alpha}^{\mathrm{op}}},E_{\pm\alpha}$. We have the following diagrams
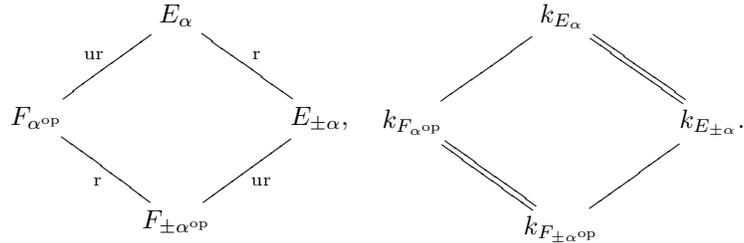
$$\xymatrix{& E_{\alpha} \ar@{-}[dl]_{\mathrm{ur}}
\ar@{-}[dr]^{\mathrm{r}} \\ F_{\alpha^{\mathrm{op}}} \ar@{-}[dr]_{\mathrm{r}} &
 & E_{\pm\alpha} \ar@{-}[dl]^{\mathrm{ur}}, \\ & F_{\pm\alpha^{\mathrm{op}}}}  \quad\xymatrix{& k_{E_{\alpha}} \ar@{-}[dl]
\ar@{=}[dr] \\ k_{F_{\alpha^{\mathrm{op}}}} \ar@{=}[dr] &
 & k_{E_{\pm\alpha}}. \ar@{-}[dl] \\ & k_{F_{\pm\alpha^{\mathrm{op}}}}}$$

Under the isomorphisms $E_{\alpha}^{\times}\cong O_{E_{\alpha}}^{\times}\times \varpi_{E_{\alpha}}^{\mathbb{Z}}\cong k_{E_{\alpha}}^{\times}\times (1+\varpi_{E_{\alpha}} O_{E_{\alpha}})\times  \varpi_{E_{\alpha}}^{\mathbb{Z}} $, $\chi_{\mu\circ \mathrm{Nm}_{S(E)/S^{\mathrm{op}}(F)},\alpha}$ is the trivial character on
$1+\varpi_{E_{\alpha}} O_{E_{\alpha}}$, the unique non-trivial character on $k_{E_{\alpha}}^{\times}$, and is determined by its value at $\varpi_{E_{\alpha}}^{\mathbb{Z}}$. Notice that $\chi_{\mathrm{BC},\alpha^{\mathrm{op}}}=\chi_{\mu,\alpha}\circ \mathrm{Nm}_{E_{\alpha}/F_{\alpha^{\mathrm{op}}}}$ is also the trivial character on $1+\varpi_{E_{\alpha}} O_{E_{\alpha}}$, the unique non-trivial character on $k_{E_{\alpha}}^{\times}$, and is also determined by its value at $\varpi_{E_{\alpha}}^{\mathbb{Z}}$ by the surjectivity of
$$\mathrm{Nm}_{E_{\alpha}/F_{\alpha^{\mathrm{op}}}}:1+\varpi_{E_{\alpha}} O_{E_{\alpha}}\rightarrow 1+\varpi_{F_{\alpha^{\mathrm{op}}}}O_{F_{\alpha^{\mathrm{op}}}},$$
together with the fact
$$\mathrm{sgn}_{k_{E_\alpha}^{\times}}=\mathrm{sgn}_{k_{F_{\alpha^{\mathrm{op}}}}^{\times}}\circ \mathrm{Nm}_{k_{E_\alpha}/k_{F_{\alpha^{\mathrm{op}}}}}.$$
This implies that $\chi_{\mathrm{BC},\alpha^{\mathrm{op}}}$ is also tamely ramified over $E$, and $\zeta_{\alpha}$ is trivial on $O_{E_{\alpha}}^{\times}$. Hence to compare these two $\chi$-data at a symmetric ramified root $\alpha$ over $E$, one only needs to evaluate these two characters at $2a_{\alpha}$ for certain mod-$a$-data over $E$.

\begin{proposition}
Let $a_{\alpha}\in F_{\alpha}$ be the $a$-data corresponding to a character $\mu:S(F)\rightarrow\mathbb{C}^{\times}$. Then $a_{\alpha}\in E_{\alpha}$ is also the $a$-data corresponding to the character $\mu\circ\mathrm{Nm}_{S(E)/S(F)}:S(E)\rightarrow\mathbb{C}^{\times}$.
\end{proposition}
\begin{proof}
Let $\mu=\mu_{-1}\cdot\prod\limits_{i=0}^{d} \phi_{i}|_{S(F)}$ be a Howe factorization of $\mu$ as in \Cref{howefactorization}. By \eqref{symmetricramified}, $a_{\alpha}$ is an arbitrary lift in $[F_{\alpha}]_{-r_\alpha}$ of the mod-$a$-data $\overline{a_{\alpha}}$ in $[F_{\alpha}]_{-r_\alpha}/[F_{\alpha}]_{-r_\alpha+}$ determined by
\begin{equation*}
\begin{aligned}
		\phi_{i}(\mathrm{Nm}_{S(F_{\alpha})/S(F)}\alpha^{\vee}(1+X))=\psi_{F_{\alpha}}(\overline{a_{\alpha}}\cdot X)=\psi_{F}(\mathrm{tr}_{F_{\alpha}/F}(\overline{a_{\alpha}}\cdot X)),
	\end{aligned}
\end{equation*}
for any $\alpha\in \Phi(S,G^{i+1})-\Phi(S,G^{i})$ and $X\in [F_{\alpha}]_{r_\alpha}/[F_{\alpha}]_{r_\alpha+}$. 

As pointed out in \cite[Proof of Lemma 4.3.3]{FKS23}, a key property of Kaletha's construction of the Howe factorization in \cite[Proposition 3.6.7]{Kal19b} is that $\phi_i$ is not only trivial on $G_{\mathrm{sc}}^i(F)$, but also trivial on $G^i_{\mathrm{der}}(F)$. Let $C^i$ denote the cokernel of $G^i_{\mathrm{der}}\rightarrow G^i$, which is a torus over $F$. Namely, there is an exact sequence of $F$-groups
\[1\rightarrow G^i_{\mathrm{der}}\rightarrow G^i\rightarrow C^i\rightarrow 1.\]
By taking the $F$-rational points, we have an inclusion $G^i(F)/G^i_{\mathrm{der}}(F)\hookrightarrow C^i(F)$. Since $C^i$ is an $F$ torus, there is a norm map $\mathrm{Nm}_{C^i(E)/C^i(F)}:C^i(E)\rightarrow C^i(F)$, which also satisfies the following commutative diagram
\[
\xymatrix{
	S(E)/S_{\mathrm{der}}(E)\ar@{^{(}->}[r]\ar[d]^{\mathrm{Nm}_{S(E)/S(F)}}&	G(E)/G_{\mathrm{der}}(E)\ar@{^{(}->}[r]&C^i(E)\ar[d]^{\mathrm{Nm}_{C^i(E)/C^i(F)}}\\
	S(F)/S_{\mathrm{der}}(F)\ar@{^{(}->}[r]&G(F)/G_{\mathrm{der}}(F) \ar@{^{(}->}[r]&C^i(F)
}
\]
Since $\phi_i$ is a character of $G^i(F)$ which are trivial on $G^i_{\mathrm{der}}(F)$, it can be extended to character $\widetilde{\phi}_i$ of $C^i(F)$. Then the restriction of $\widetilde{\phi}_i\circ \mathrm{Nm}_{C^i(E)/C^i(F)}$ to the subgroup $G^i(E)/G^i_{\mathrm{der}}(E)$ of $C^i(E)$ satisfies
\[\widetilde{\phi}_i\circ \mathrm{Nm}_{C^i(E)/C^i(F)}|_{S(E)}={\phi_i}|_{S(F)} \circ\mathrm{Nm}_{S(E)/S(F)}.\]
It is direct to check that $(\mu_{-1}\circ \mathrm{Nm}_{S(E)/S(F)},(\widetilde{\phi}_0\circ \mathrm{Nm}_{C^0(E)/C^0(F)})|_{G^0(E)},\cdots, (\widetilde{\phi}_d\circ \mathrm{Nm}_{C^d(E)/C^d(F)})|_{G^d(E)})$ is a Howe factorization of $\mu\circ\mathrm{Nm}_{S(E)/S(F)}$. Furthermore, since $a_\alpha\in F_\alpha$, we have 
\begin{align*}
(\widetilde{\phi}_i\circ \mathrm{Nm}_{C^0(E)/C^0(F)})|_{S(E)}(\mathrm{Nm}_{S(E_{\alpha})/S(E)}\alpha^{\vee}(1+X))
=&\phi_i|_{S(F)}\circ\mathrm{Nm}_{S(E)/S(F)}(\mathrm{Nm}_{S(E_{\alpha})/S(E)}\alpha^{\vee}(1+X))\\
=&\phi_i(\mathrm{Nm}_{S(F_{\alpha})/S(F)}(\mathrm{Nm}_{S(E_{\alpha})/S(F_{\alpha})}\alpha^{\vee}(1+X)))\\
=&\phi_i(\mathrm{Nm}_{S(F_{\alpha})/S(F)}(\alpha^{\vee}(1+\mathrm{tr}_{E_{\alpha}/F_{\alpha}}X)))\\
=&\psi_{F}(\mathrm{tr}_{F_{\alpha}/F}(a_{\alpha}\cdot \mathrm{tr}_{E_{\alpha}/F_{\alpha}}X))\\
=&\psi_{F}(\mathrm{tr}_{F_{\alpha}/F} \mathrm{tr}_{E_{\alpha}/F_{\alpha}}(a_{\alpha}\cdot X))\\
=&\psi_{F}( \mathrm{tr}_{E_{\alpha}/F}(a_{\alpha}\cdot X)).
\end{align*}
This implies that $\{a_\alpha\}$ is also the $a$-data corresponding to $\mu\circ\mathrm{Nm}_{S(E)/S(F)}$, where $a_\alpha$ is regarded as an element of $E_\alpha$.
\end{proof}

\begin{corollary}
$$\chi_{\mathrm{BC}_{\mu},\alpha}(2a_{\alpha})=\left\{\begin{matrix}
\chi_{\mu,\alpha}(2a_\alpha)=\lambda_{F_{\alpha}/F_{\pm\alpha}}     &  F_{\alpha}=E_{\alpha},\\
(\chi_{\mu,\alpha}(2a_\alpha))^{2}=\lambda_{F_{\alpha}/F_{\pm\alpha}}^{2}     & [E_{\alpha}:F_{\alpha}]=2.
\end{matrix}
\right.$$
$$\chi_{\mu\circ \mathrm{Nm}_{S(E)/S(F)},\alpha}(2a_{\alpha})=\lambda_{E_{\alpha}/E_{\pm\alpha}}.$$
\end{corollary}
Applying the above proposition and corollary for $S^\mathrm{op}$, we know that $\zeta$ is given by $\frac{\lambda_{F_{\alpha^{\mathrm{op}}}/F_{\pm\alpha^{\mathrm{op}}}}^{2}}{\lambda_{E_{\alpha}/E_{\pm\alpha}}}$ on $2a_{\alpha^{\mathrm{op}}}\in F_{\alpha^{\mathrm{op}}}\subset E_{\alpha}$. According to Bushnell and Henniart \cite[1.5.2]{BH05}, for inclusions $F\subset E\subset K$ such that $K/F$ is a biquadratic extension, we have a formula:
$$\lambda_{K/F}=\lambda_{K/E}\lambda_{E/F}^{[K:E]}=\lambda_{K/E}\lambda_{E/F}^{2}.$$
Hence we obtain
$$\lambda_{E_{\alpha}/F_{\pm\alpha^{\mathrm{op}}}}=\lambda_{E_{\alpha}/E_{\pm\alpha}}\lambda_{E_{\pm\alpha}/F_{\pm\alpha^{\mathrm{op}}}}^{2}=\lambda_{E_{\pm\alpha}/F_{\alpha^{\mathrm{op}}}}\lambda_{F_{\alpha^{\mathrm{op}}}/F_{\pm\alpha^{\mathrm{op}}}}^{2},$$
which implies
$$\frac{\lambda_{F_{\alpha^{\mathrm{op}}}/F_{\pm\alpha^{\mathrm{op}}}}^{2}}{\lambda_{E_{\alpha}/E_{\pm\alpha}}}=\frac{\lambda_{E_{\pm\alpha}/F_{\pm\alpha^{\mathrm{op}}}}^{2}}{\lambda_{E_{\alpha}/F_{\alpha^{\mathrm{op}}}}}.$$

By \cite[Lemma 1.5]{BH05}, for any unramified extension $E/F$ of degree $n$, $\lambda_{E/F}=(-1)^{n-1}$. Hence $\zeta_{\alpha}(2a_{\alpha^{\mathrm{op}}})=-1$, which means that $\zeta_\alpha$ is the unramified quadratic character of $E_{\alpha}^\times$ associated to the unique unramified quadratic extension $E_{\alpha}'/E_{\alpha}$. This further implies that $\zeta_{\alpha}|_{F_{\alpha}^{\times}}$ is the unramified quadratic character $\omega_{E_{\alpha}/F_{\alpha}}$. Hence, the $\alpha$-component of $\zeta_{T(E)}|_{S(F)}$ is given by
\begin{equation}\label{symram}
\zeta_{({\chi_{\mathrm{BC}}}_{\mu,\alpha},\chi_{\mu\circ\mathrm{Nm}_{T(E)/S^{\mathrm{op}}(F)},\alpha})}|_{F_{\alpha}^{\times}}(\iota_{F_\alpha}\alpha(t))
=\omega_{E_{\alpha}/F_{\alpha}}(\iota_{F_\alpha}\alpha (t)).
\end{equation}

\item $\alpha^{\mathrm{op}}$ is symmetric unramified over $E$ and symmetric ramified over $F$. We have the following diagrams
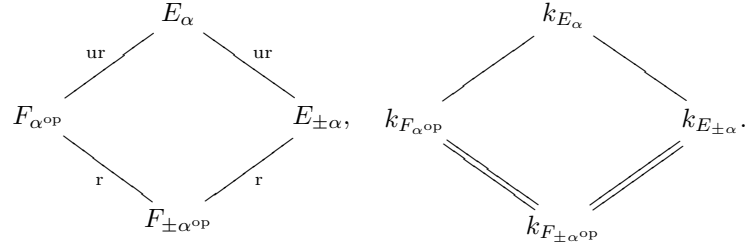
\begin{equation}
\begin{aligned}\label{bidiag1}
\xymatrix{& E_{\alpha} \ar@{-}[dl]_{\mathrm{ur}}
\ar@{-}[dr]^{\mathrm{ur}} \\ F_{\alpha^{\mathrm{op}}} \ar@{-}[dr]_{\mathrm{r}} &
 & E_{\pm\alpha} \ar@{-}[dl]^{\mathrm{r}}, \\ & F_{\pm\alpha^{\mathrm{op}}}}  \quad\xymatrix{& k_{E_{\alpha}} \ar@{-}[dl]
\ar@{-}[dr] \\ k_{F_{\alpha^{\mathrm{op}}}} \ar@{=}[dr] &
 & k_{E_{\pm\alpha}} \ar@{=}[dl] \\ & k_{F_{\pm\alpha^{\mathrm{op}}}}}
\end{aligned}
\end{equation}
together with
$$\zeta_{({\chi_{\mathrm{BC}}}_{\mu,\alpha},\chi_{\mu\circ\mathrm{Nm}_{T(E)/S^{\mathrm{op}}(F)},\alpha})}(\iota_{E_\alpha}\alpha(t))=\frac{\chi_{\alpha^{\mathrm{op}}}\circ \mathrm{Nm}_{E_{\alpha}/F_{\alpha^{\mathrm{op}}}}}{\omega_{{E_{\alpha}}'/E_{\alpha}}}(\iota_{E_\alpha}\alpha (t)),\footnote{The computation here also provides an explicit example where the $\chi_{\mathrm{BC},\alpha^\mathrm{op}}=\chi_{\alpha^\mathrm{op}}\circ \mathrm{Nm}_{E_\alpha/F_{\alpha^\mathrm{op}}}$ fails to be minimally ramified. More precisely, from the diagram \eqref{bidiag1}, we know that $\alpha^{\mathrm{op}}$ is unramified over $E$. If $\chi_{\mathrm{BC},\alpha^\mathrm{op}}$ is minimally ramified, it must be the unramified quadratic character $\omega_{E_{\alpha}'/E_\alpha}$ associated to the unique unramified extension $E_{\alpha}'/E_\alpha$. However, this cannot be true since we prove that $\chi_{\mathrm{BC},\alpha^\mathrm{op}}|_{F_{\alpha}^\times}=\omega_{E_\alpha/F_\alpha}$ in \Cref{missing}
and $\omega_{{E_{\alpha}}'/E_{\alpha}}|_{F_{\alpha}^{\times}}=\mathbbm{1}$ so that $\chi_{\mathrm{BC},\alpha^\mathrm{op}}|_{F_{\alpha}^\times}\neq\omega_{E_{\alpha}'/E_\alpha}|_{F_{\alpha}^\times}.$\label{footnote}}$$
where ${E_{\alpha}}'$ is the unramified quadratic extension of $E_{\alpha}$, and $\chi_{\alpha^{\mathrm{op}}}:F_{\alpha^{\mathrm{op}}}^{\times}\rightarrow \{\pm1\}$ is the tamely ramified $\chi$-data determined by the mod-$a$-data. It is clearly that
$$(\omega_{{E_{\alpha}}'/E_{\alpha}}\circ\iota_{E_{\alpha}}\circ \alpha)|_{S(F)}=\omega_{{E_{\alpha}}'/E_{\alpha}}|_{F_{\alpha}^{\times}}\circ \iota_{E_{\alpha}}|_{F_{\alpha}^1}\circ \alpha|_{S(F)}.$$

By \Cref{biquadratic} and Diagram \eqref{bidiag1}, we know that $E_{\alpha}/F_{\alpha}$ is ramified, which implies $\varpi_{F_{\alpha}}=\varpi_{E_{\alpha}}^{2}$ mod $O_{E_{\alpha}}^{\times}$. By the fact that $\omega_{{E_{\alpha}}'/E_{\alpha}}$ is an unramified character and $F_{\alpha}^{\times}=O_{F_{\alpha}}^{\times}\times \varpi_{F_{\alpha}}^{\mathbb{Z}}$, we know that $\omega_{{E_{\alpha}}'/E_{\alpha}}|_{F_{\alpha}^{\times}}$ is trivial. Hence we only need to compute $(\chi_{\alpha^{\mathrm{op}}}\circ \mathrm{Nm}_{E_{\alpha}/F_{\alpha^{\mathrm{op}}}})|_{F_{\alpha}^{\times}}$ explicitly.

\begin{lemma}\label{missing}
$(\chi_{\alpha^{\mathrm{op}}}\circ \mathrm{Nm}_{E_{\alpha}/F_{\alpha^{\mathrm{op}}}})|_{F_{\alpha}^{\times}}$ coincides with $\omega_{E_{\alpha}/F_{\alpha}}$.
\end{lemma}

\begin{proof}

Since $\alpha_{\mathrm{op}}$ is symmetric over $F$, we know $\chi_{\alpha^{\mathrm{op}}}|_{F_{\pm\alpha}^{\times}}=\omega_{F_{\alpha^{\mathrm{op}}}/F_{\pm\alpha}}$ by the properties of $\chi$-data in \Cref{chidata}, which implies
\begin{equation*}
\begin{aligned}
(\chi_{\alpha^{\mathrm{op}}}\circ \mathrm{Nm}_{E_{\alpha}/F_{\alpha}^{\mathrm{op}}})|_{F_{\alpha}^{\times}}&=\chi_{\alpha^{\mathrm{op}}}|_{F_{\pm\alpha}^{\times}}\circ \mathrm{Nm}_{F_{\alpha}/F_{\pm\alpha}}\\
&=\omega_{F_{\alpha^{\mathrm{op}}}/F_{\pm\alpha}}\circ\mathrm{Nm}_{F_{\alpha}/F_{\pm\alpha}}=\omega_{E_{\alpha}/F_{\alpha}},
\end{aligned}
\end{equation*}
by considering the middle missing diagram
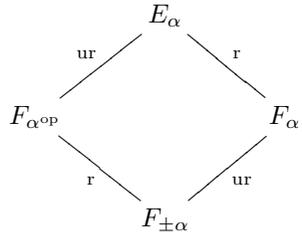
\[
\begin{gathered}[b]
\xymatrix{
	& E_{\alpha} \ar@{-}[dl]_{\mathrm{ur}}\ar@{-}[d]^{\mathrm{r}} \\ F_{\alpha^{\mathrm{op}}} \ar@{-}[dr]_{\mathrm{r}} 	& F_{\alpha} \ar@{-}[d]^{\mathrm{ur}} \\ 
	& F_{\pm\alpha}.}
	\\[-\dp\strutbox]
\end{gathered}
\qedhere
\]

\end{proof}

Hence the $\alpha$-component of $\zeta_{T(E)}|_{S(F)}$ is given by
\begin{equation}\label{symur'}
\zeta_{({\chi_{\mathrm{BC}}}_{\mu,\alpha},\chi_{\mu\circ\mathrm{Nm}_{T(E)/S^{\mathrm{op}}(F)},\alpha})}|_{F_{\alpha}^{\times}}(\iota_{F_\alpha}\alpha(s))
=\omega_{E_{\alpha}/F_{\alpha}}(\iota_{F_\alpha}\alpha (s)).
\end{equation}

\item $\alpha^{\mathrm{op}}$ is symmetric unramified both over $E$ and over $F$.

We have the following diagrams
$$\xymatrix{& E_{\alpha} \ar@{-}[dl]_{\mathrm{r}}
\ar@{-}[dr]^{\mathrm{ur}} \\ F_{\alpha^{\mathrm{op}}} \ar@{-}[dr]_{\mathrm{ur}} &
 & E_{\pm\alpha} \ar@{-}[dl]^{\mathrm{r}}, \\ & F_{\pm\alpha^{\mathrm{op}}}}  \quad\xymatrix{& k_{E_{\alpha}} \ar@{=}[dl]
\ar@{-}[dr] \\ k_{F_{\alpha^{\mathrm{op}}}} \ar@{-}[dr] &
 & k_{E_{\pm\alpha}}. \ar@{=}[dl] \\ & k_{F_{\pm\alpha^{\mathrm{op}}}}}$$

In this case, we obtain
\begin{equation*}
\begin{aligned}
\zeta_{({\chi_{\mathrm{BC}}}_{\mu,\alpha},\chi_{\mu\circ\mathrm{Nm}_{T(E)/S^{\mathrm{op}}(F)},\alpha})}(\iota_{E_\alpha}\alpha(t))&=\frac{\omega_{F_{\alpha^{\mathrm{op}}}'/F_{\alpha^{\mathrm{op}}}}\circ \mathrm{Nm}_{E_{\alpha}/F_{\alpha^{\mathrm{op}}}}}{\omega_{{E_{\alpha}}'/E_{\alpha}}}(\iota_{E_\alpha}\alpha (t)),
\end{aligned}
\end{equation*}
where $F'_{\alpha^{\mathrm{op}}}$ denotes the unique unramified quadratic extension of $F_{\alpha^{\mathrm{op}}}$. Moreover, by considering the diagram
$$\xymatrix{& E_{\alpha}' \ar@{-}[dl]_{\mathrm{r}}
\ar@{-}[dr]^{\mathrm{ur}} \\ F_{\alpha^{\mathrm{op}}}' \ar@{-}[dr]_{\mathrm{ur}} &
 & E_{\alpha}, \ar@{-}[dl]^{\mathrm{r}} \\ & F_{\alpha^{\mathrm{op}}}}$$
together with \Cref{quadraticnorm}, we know that \[\omega_{F_{\alpha^{\mathrm{op}}}'/F_{\alpha^{\mathrm{op}}}}\circ \mathrm{Nm}_{E_{\alpha}/F_{\alpha^{\mathrm{op}}}}=\omega_{{E_{\alpha}}'/E_{\alpha}}\]
Hence the $\alpha$-component of $\zeta_{T(E)}|_{S(F)}$ is given by
\begin{equation}\label{symur}
\zeta_{({\chi_{\mathrm{BC}}}_{\mu,\alpha},\chi_{\mu\circ\mathrm{Nm}_{T(E)/S^{\mathrm{op}}(F)},\alpha})}|_{F_{\alpha}^{\times}}(\iota_{F_\alpha}\alpha(s))
=\mathbbm{1}(\iota_{F_\alpha}\alpha (s))=1.
\end{equation}

\end{enumerate}

\item Asymmetric case

If $\alpha^{\mathrm{op}}$ is asymmetric both over $E$ and over $F$, then we obtain
$$\zeta_{({\chi_{\mathrm{BC}}}_{\mu,\alpha},\chi_{\mu\circ\mathrm{Nm}_{T(E)/S^{\mathrm{op}}(F)},\alpha})}(\alpha(t))=\frac{\mathbbm{1}_{E_{\alpha}^{\times}}}{\mathbbm{1}_{F_{\alpha^{\mathrm{op}}}^{\times}}\circ \mathrm{Nm}_{E_{\alpha}/F_{\alpha}^{\mathrm{op}}}}(\alpha(t))=1.$$

Hence, the $\alpha$-component of $\zeta_{T(E)}|_{S(F)}$ is given by
\begin{equation}\label{asymef}
\begin{aligned}
\zeta_{({\chi_{\mathrm{BC}}}_{\mu,\alpha},\chi_{\mu\circ\mathrm{Nm}_{T(E)/S^{\mathrm{op}}(F)},\alpha})}|_{F_{\alpha}^{\times}}(\alpha(s))
=1.
\end{aligned}
\end{equation}
\end{enumerate}
\item The $\Gamma_F$-orbit of $\alpha^{\mathrm{op}}$ breaks into two $\Gamma_E$-orbits.

\begin{enumerate}

\item The symmetry of $\alpha^{\mathrm{op}}$ is the same over $E$ and $F$.

We have the following diagrams
$$\xymatrix{& E_{\alpha} \ar@{=}[dl]
\ar@{-}[dr]^{\mathrm{ur}} \\ F_{\alpha^{\mathrm{op}}} \ar@{-}[dr]_{\mathrm{ur}} &
 & E_{\pm\alpha}, \ar@{=}[dl] \\ & F_{\pm\alpha^{\mathrm{op}}}}  \quad\xymatrix{& k_{E_{\alpha}} \ar@{=}[dl]
\ar@{-}[dr] \\ k_{F_{\alpha^{\mathrm{op}}}} \ar@{-}[dr] &
 & k_{E_{\pm\alpha}}. \ar@{=}[dl] \\ & k_{F_{\pm\alpha^{\mathrm{op}}}}}$$
or
$$\xymatrix{& E_{\alpha} \ar@{=}[dl]
\ar@{-}[dr]^{\mathrm{r}} \\ F_{\alpha^{\mathrm{op}}} \ar@{-}[dr]_{\mathrm{r}} &
 & E_{\pm\alpha}, \ar@{=}[dl] \\ & F_{\pm\alpha^{\mathrm{op}}}}  \quad\xymatrix{& k_{E_{\alpha}} \ar@{=}[dl]
\ar@{=}[dr] \\ k_{F_{\alpha^{\mathrm{op}}}} \ar@{=}[dr] &
 & k_{E_{\pm\alpha}}. \ar@{=}[dl] \\ & k_{F_{\pm\alpha^{\mathrm{op}}}}}$$
or
$$\xymatrix{& E_{\alpha} \ar@{=}[dl]
\ar@{=}[dr] \\ F_{\alpha^{\mathrm{op}}} \ar@{=}[dr] &
 & E_{\pm\alpha}, \ar@{=}[dl] \\ & F_{\pm\alpha^{\mathrm{op}}}}  \quad\xymatrix{& k_{E_{\alpha}} \ar@{=}[dl]
\ar@{=}[dr] \\ k_{F_{\alpha^{\mathrm{op}}}} \ar@{=}[dr] &
 & k_{E_{\pm\alpha}}. \ar@{=}[dl] \\ & k_{F_{\pm\alpha^{\mathrm{op}}}}}$$

The equalities $E_{\alpha}=F_{\alpha^{\mathrm{op}}}$ and $E_{\pm\alpha}=F_{\pm\alpha^{\mathrm{op}}}$ imply $\chi_{\mu\circ\mathrm{Nm}_{T(E)/S^{\mathrm{op}}(F)}}={\chi_{\mathrm{BC}}}_{\mu,\alpha}$. Hence the $\alpha$-component of $\zeta_{T(E)}$ is trivial in all these three cases, which also implies that the $\alpha$ component of $\zeta_{T(E)}|_{S(F)}$ is trivial in these cases.

\item The symmetry of $\alpha^{\mathrm{op}}$ is different over $E$ and $F$.

We have the following diagrams

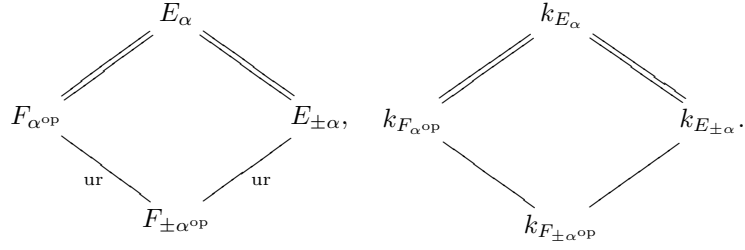
\begin{enumerate}
\item
$$\xymatrix{& E_{\alpha} \ar@{=}[dl]
\ar@{=}[dr] \\ F_{\alpha^{\mathrm{op}}} \ar@{-}[dr]_{\mathrm{ur}} &
 & E_{\pm\alpha}, \ar@{-}[dl]^{\mathrm{ur}} \\ & F_{\pm\alpha^{\mathrm{op}}}}  \quad\xymatrix{& k_{E_{\alpha}} \ar@{=}[dl]
\ar@{=}[dr] \\ k_{F_{\alpha^{\mathrm{op}}}} \ar@{-}[dr] &
 & k_{E_{\pm\alpha}}. \ar@{-}[dl] \\ & k_{F_{\pm\alpha^{\mathrm{op}}}}}$$

In this case, $\chi_{\mu\circ\mathrm{Nm}_{T(E)/S^{\mathrm{op}}(F)}}=\mathbbm{1}$ since $\alpha$ is asymmetric over $E$. The $\alpha$-component of $\zeta_{T(E)}|_{S(F)}$ is given by
$$(\frac{\chi_{\alpha^{\mathrm{op}}}\circ \mathrm{id}}{\mathbbm{1}}\circ \alpha)|_{S(F)}=(\omega_{E_{\alpha}'/E_{\alpha}}\circ \alpha)|_{S(F)}.$$
Furthermore, by \Cref{biquadraticdegenerate2}, we know $F_{\alpha}^\times=F_{\pm\alpha}^\times=F_{\pm\alpha^{\mathrm{op}}}^\times$. Since both $E_{\alpha}'/E_{\alpha}$ and $E_\alpha/F_\alpha$ are unramified, then  $\omega_{E_{\alpha}'/E_{\alpha}}|_{F_{\alpha}^\times}=\omega_{E_{\alpha}/F_{\alpha}}$ so that $(\omega_{E_{\alpha}'/E_{\alpha}}\circ \alpha)|_{S(F)}=\omega_{E_{\alpha}/F_{\alpha}}\circ \alpha$. However, since $\alpha(s)\in O_{F_{\alpha}^\times}$ for $s\in S(F)$ and $\omega_{E_{\alpha}/F_{\alpha}}$ is unramified, we know $\omega_{E_{\alpha}'/E_{\alpha}}\circ \alpha|_{S(F)}=\mathbbm{1}$. Hence the $\alpha$-component of $\zeta_{T(E)}|_{S(F)}$ is $\mathbbm{1}$.

\item
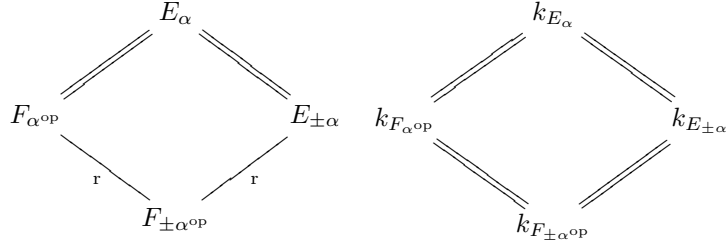
$$\xymatrix{& E_{\alpha} \ar@{=}[dl]
\ar@{=}[dr] \\ F_{\alpha^{\mathrm{op}}} \ar@{-}[dr]_{\mathrm{r}} &
 & E_{\pm\alpha} \ar@{-}[dl]^{\mathrm{r}} \\ & F_{\pm\alpha^{\mathrm{op}}}}  \quad\xymatrix{& k_{E_{\alpha}} \ar@{=}[dl]
\ar@{=}[dr] \\ k_{F_{\alpha^{\mathrm{op}}}} \ar@{=}[dr] &
 & k_{E_{\pm\alpha}} \ar@{=}[dl] \\ & k_{F_{\pm\alpha^{\mathrm{op}}}}}$$

In this case, $\chi_{\mu\circ\mathrm{Nm}_{T(E)/S^{\mathrm{op}}(F)}}=\mathbbm{1}$ since $\alpha$ is asymmetric over $E$. The $\alpha$-component of $\zeta_{T(E)}|_{S(F)}$ is given by
$$(\frac{\chi_{\alpha^{\mathrm{op}}}\circ \mathrm{id}}{\mathbbm{1}}\circ \alpha)|_{S(F)}=(\chi_{\alpha^{\mathrm{op}}}\circ \alpha)|_{S(F)},$$
where $\chi_{\alpha^{\mathrm{op}}}$ is a tamely ramified $\chi$-data being the unique nontrivial quadratic character on $k_{E_{\alpha}}^{\times}$, and is determined by the mod-$a$-data. Furthermore, by \Cref{biquadraticdegenerate2}, we know $F_{\alpha}^\times=F_{\pm\alpha}^\times=F_{\pm\alpha^{\mathrm{op}}}^\times$. Again since $\alpha(t)\in O_{E_{\alpha}}^{\times}$ and $E_{\alpha}/F_{\alpha}$ is ramified, the $\alpha$-component of $\zeta_{T(E)}|_{S(F)}$ is given by
$$\mathrm{sgn}_{k_{F_{\alpha}}^{\times}}(\alpha(s)).$$
\end{enumerate}

\end{enumerate}
\end{enumerate}

\subsection{Tables}\label{table}
We list tables describing the contributions of each type of roots to these four quadratic characters.

\begin{longtable}{ |c|c| }
\hline
$\alpha/F$  & contribution\\[0.5ex]
\hline\hline
asym & $\mathbbm{1}$ \\
\hline
sym ur & $\omega_{E_{\alpha}/F_{\alpha}}(\iota_{F_\alpha}\circ\alpha )$  \\
\hline
sym r & $\omega_{E_{\alpha}/F_{\alpha}}(\iota_{F_\alpha}\circ\alpha)$\\
[1ex]
\hline
\caption{\textbf{Prasad's character.}}\label{table1}
\end{longtable}

\begin{longtable}{ |c|c|c|c|c| }
\hline
$[E_{\alpha}:F_{\alpha}]$ & $\alpha/F$ & $\alpha/E$ & $E/F$ & contribution\\[0.5ex]
\hline\hline
1 & asym & asym & r/ur &  $\mathbbm{1}$\\
\hline
2 ur & asym & asym & r/ur & $\mathbbm{1}$  \\
\hline
2 r & asym & asym & r & $\mathrm{sgn}_{k_{F_{\alpha}}^{\times}}\circ\alpha|_{S(F)}$ \\
\hline
1 & sym ur & asym & ur & $\mathbbm{1}$\\
\hline
1 & sym ur & sym ur& r/ur & $\mathbbm{1}$\\
\hline
2 r & sym ur & sym ur& r & $\mathrm{sgn}_{k_{F_{\alpha}}^{1}}\circ \alpha|_{S(F)}$ \\
\hline
1 & sym r & asym& r & $\mathrm{sgn}_{k_{F_{\alpha}}^{\times}}\circ\alpha|_{S(F)}$\\
\hline
1 & sym r & sym r& r/ur &  $\mathbbm{1}$\\
\hline
2 ur & sym r & sym r& r/ur & $\mathbbm{1}$\\
\hline
2 ur & sym r & sym ur& r & $\mathrm{sgn}_{k_{E_{\alpha}}^{1}}\circ \alpha|_{S(F)}$\\
[1ex]
\hline
\caption{\textbf{Kaletha's character for $\alpha$ with $\frac{r_{i_\alpha}}{2}\in \mathrm{ord}_x\alpha$.}}\label{table2}
\end{longtable}

\begin{longtable}{ |c|c| }
\hline
$\alpha/F$  & contribution\\[0.5ex]
\hline\hline
asym & $\mathrm{sgn}_{k_{F_{\alpha}}^{\times}}\circ\alpha$ \\
\hline
sym ur & $\mathbbm{1}$  \\
\hline
sym r & $\mathrm{sgn}_{k_{F_{\alpha}}^{\times}}\circ\alpha$\\
[1ex]
\hline
\caption{\textbf{Hakim's character for $\alpha$ with $\frac{r_{i_\alpha}}{2}\in \mathrm{ord}_x\alpha$.}}\label{table3}
\end{longtable}

\begin{longtable}{ |c|c|c|c|c| }
\hline
$[E_{\alpha}:F_{\alpha^{\mathrm{op}}}]$ & $\alpha^{\mathrm{op}}/F$ & $\alpha/E$ & $E/F$ & contribution\\[0.5ex]
\hline\hline
1 & asym & asym & r/ur &  $\mathbbm{1}$\\
\hline
2 ur & asym & asym & r/ur & $\mathbbm{1}$  \\
\hline
2 r & asym & asym & r & $\mathbbm{1}$ \\
\hline
1 & sym ur & asym & ur & $\mathbbm{1}$\\
\hline
1 & sym ur & sym ur& r/ur & $\mathbbm{1}$\\
\hline
2 r & sym ur & sym ur& r & $\mathbbm{1}$ \\
\hline
1 & sym r & asym& r & $\mathrm{sgn}_{k_{F_{\alpha}}^{\times}}\circ\alpha|_{S(F)}$\\
\hline
1 & sym r & sym r& r/ur &  $\mathbbm{1}$\\
\hline
2 ur & sym r & sym r& r/ur & $\omega_{E_{\alpha}/F_{\alpha}}\circ\iota_{F_\alpha}\circ\alpha|_{S(F)} $\\
\hline
2 ur & sym r & sym ur& r & $\omega_{E_{\alpha}/F_{\alpha}}\circ\iota_{F_\alpha}\circ\alpha |_{S(F)}$\\
[1ex]
\hline
\caption{\textbf{The character associated to $\zeta$-data.}}\label{table4}
\end{longtable}

\begin{longtable}{ |c|c|c|c|c| }
\hline
 $[E_{\alpha}:F_{\alpha}]$ & $\alpha/F$ & $\alpha/E$ & $\alpha^{\mathrm{op}}/F$ &$[E_{\alpha}:F_{\alpha^{\mathrm{op}}}]$\\[0.5ex]
\hline\hline
1 & asym & asym & asym &  1\\
\hline
2 ur & asym & asym & sym ur & 1  \\
\hline
2 r & asym & asym & sym r & 1 \\
\hline
1 & sym ur & asym & asym & 2 ur\\
\hline
1 & sym ur & sym ur& sym ur & 1\\
\hline
2 r & sym ur & sym ur& sym r & 2 ur \\
\hline
1 & sym r & asym& asym & 2 r\\
\hline
1 & sym r & sym r& sym r &  1\\
\hline
2 ur & sym r & sym r& sym r & 2 ur\\
\hline
2 ur & sym r & sym ur& sym ur & 2 r\\
[1ex]
\hline
\caption{\textbf{Comparison between $\alpha$ and $\alpha^{\mathrm{op}}$.}}\label{table5}
\end{longtable}

\subsection{$E/F$ unramified case}

In particular, we can use our methods to deal with all cases such that $E/F$ is unramified. The unramified assumption enables us to prove that many components of these characters are trivial due to some arithmetic reasons.

\begin{theorem}\label{unramifiedprasad}
Prasad's \Cref{conjecture1} holds for regular supercuspidal representations when $E/F$ is unramified.
\end{theorem}

\paragraph{Prasad's character.}
Recall that Prasad's character has the following factorization by \Cref{rootcharacter}
$$\omega_{G(F),E}(s)=\prod _{\alpha\in \Phi_{\mathrm{sym}}/\Gamma_F}\omega_{EF_{\alpha}/F_\alpha}(\iota_{\alpha} \alpha(s)),$$
for any $s\in S(F)$.
If $E$ is contained in $F_{\alpha}$, then $\omega_{EF_{\alpha}/F_{\alpha}}$ is trivial.
If $E$ and $F_\alpha$ are disjoint, then $EF_{\alpha}=E_{\alpha}$. Hence the only non-trivial term is the following case, where $E_\alpha$ is a biquadratic extension of $F_{\pm\alpha}$:
$$\xymatrix{
  E_{\alpha} \ar@{-}[d]_{2} \ar@{-}[r]^{2} & E_{\pm\alpha} \ar@{-}[d]_{2} \ar@{-}[r] & E \ar@{-}[d] \\
  F_{\alpha} \ar@{-}[r]^{2} & F_{\pm\alpha} \ar@{-}[r] & F .  }$$
Since $E/F$ is unramified, both $E_{\pm\alpha}/F_{\pm\alpha}$ and $E_\alpha/F_\alpha$ are also unramified, which means $\alpha$ is symmetric ramified over $F$ and $E$.

Due to the above reason, Prasad's character could be simplified by
$$\omega_{G(F),E}(s)=\prod \limits_{\substack{\alpha\in \Phi_{\mathrm{sym,ram}}/\Gamma_F\\ \alpha \in \Phi_{\mathrm{sym,ram}}/\Gamma_E}}\omega_{E_{\alpha}/F_\alpha}(\iota_{\alpha} \alpha(s)).$$

\paragraph{Kaletha's character.}
By \Cref{table2}, we get the following description of the character $\epsilon_{\mathrm{Kal}}|_{S(F)}$.
$$\epsilon_{\mathrm{Kal}}|_{S(F)}(s)=\prod\limits_{\substack{\alpha\in \Phi_{\mathrm{asy}}/\Gamma_{E}\times\{\pm 1\},\\
		\frac{r_{i_\alpha}}{2}\in\mathrm{ord}_x{\alpha},\\ \alpha \in \Phi_{\mathrm{asy}}/\Gamma_F,\\ [E_{\alpha}:F_{\alpha}]=2,\mathrm{ram} }} \mathrm{sgn}_{k_{E_{\alpha}}^{\times}}\alpha(s)    \prod\limits_{\substack{\alpha\in \Phi_{\mathrm{asy}}/\Gamma_{E}\times\{\pm 1\},\\ \frac{r_{i_\alpha}}{2}\in\mathrm{ord}_x{\alpha},\\ \alpha \in \Phi_{\mathrm{sym,ram}}/\Gamma_F }}\mathrm{sgn}_{k_{E_{\alpha}}^{\times}}\alpha(s)
   \prod\limits_{\substack{\alpha\in \Phi_{\mathrm{sym,ur}}/\Gamma_{E},\\
   		\frac{r_{i_\alpha}}{2}\in\mathrm{ord}_x{\alpha},\\ \alpha \in \Phi_{\mathrm{sym,ur}}/\Gamma_F ,\\ [E_{\alpha}:F_{\alpha}]=2,\mathrm{ram}}}\mathrm{sgn}_{k_{E_{\alpha}}^{1}}\alpha(s)    \prod\limits_{\substack{\alpha\in \Phi_{\mathrm{sym,ur}}/\Gamma_{E},\\ \frac{r_{i_\alpha}}{2}\in\mathrm{ord}_x{\alpha},\\ \alpha \in \Phi_{\mathrm{sym,ram}}/\Gamma_F.}}\mathrm{sgn}_{k_{E_{\alpha}}^{1}}\alpha(s).$$
Moreover, if $E/F$ is unramified, from \Cref{table2} again, we see that
$$\epsilon_{\mathrm{Kal}}|_{S(F)}(s)=1.$$
\paragraph{Hakim's character.}
By \Cref{hakimfactorization}, we know that $\epsilon_{\mathrm{HM}}^{+}$ has the following factorization
\begin{equation}
\epsilon_{\mathrm{HM}}(s)=\prod\limits_{\substack{\alpha\in \Phi_{\mathrm{asy}}/\Gamma_{F},\\ \frac{r_{i_\alpha}}{2}\in\mathrm{ord}_x{\alpha}/E}}\mathrm{sgn}_{k_{F_{\alpha}}^{\times}}\alpha(s) \prod\limits_{\substack{\alpha\in \Phi_{\mathrm{sym,ram}}/\Gamma_{F},\\ \frac{r_{i_\alpha}}{2}\in\mathrm{ord}_x{\alpha}/E}}\mathrm{sgn}_{k_{F_{\alpha}}^\times}\alpha(s).
\end{equation}

However, we cannot detect more information of this character from this expression. Luckily, when $E/F$ is unramified, $\epsilon_{\mathrm{HM}}^{+}$ is also trivial by a result of Zhang \cite[Proposition 4.1]{Zha20b}. His proof relies on the existence of depth-zero good element of trace zero, which is only available when $E/F$ is unramified. In fact, his argument is global in nature in the sense that he constructs an $F$-symplectic structure preserved by the adjoint action of $S(F)$, instead of analyzing a fixed root component of Hakim--Murnaghan's character. His method is somehow complementary to ours. Hence when $E/F$ is unramified, we obtain
$$\epsilon_{\mathrm{HM}}(s)=1.$$
\paragraph{The character associated to $\zeta$-data.}
By \Cref{table4}, we get the following description of the character of $T(E)$ associated to $\zeta$-data.
$$\zeta_{T(E),{\chi_{\mu}}_{\mathrm{BC}},\chi_{\mu\circ\mathrm{Nm}_{T(E)/S^{\mathrm{op}}(F)}}}(s)=\prod\limits_{\substack{\alpha\in \Phi_{\mathrm{asy}}/\Gamma_{E}\times\{\pm 1\},\\ \alpha^{\mathrm{op}} \in \Phi_{\mathrm{sym,ram}}/\Gamma_F }}\mathrm{sgn}_{k_{E_{\alpha}}^{\times}}\alpha(s)     \prod\limits_{\substack{\alpha\in \Phi_{\mathrm{sym,ur}}/\Gamma_{E},\\ \alpha^{\mathrm{op}} \in \Phi_{\mathrm{sym,ram}}/\Gamma_F }}\omega_{E_{\alpha}/F_{\alpha}}(\iota_{F_\alpha}\alpha (s)) \prod\limits_{\substack{\alpha\in \Phi_{\mathrm{sym,ram}}/\Gamma_{E},\\ \alpha^{\mathrm{op}} \in \Phi_{\mathrm{sym,ram}}/\Gamma_F }}\omega_{E_\alpha/F_\alpha}(\iota_{\alpha}\alpha(s)).$$
Moreover, if $E/F$ is unramified, the character is simply given by
$$\zeta_{T(E)}|_{S(F)}(s)=\prod\limits_{\substack{\alpha\in \Phi_{\mathrm{sym,ram}}/\Gamma_{E},\\ \alpha^{\mathrm{op}} \in \Phi_{\mathrm{sym,ram}}/\Gamma_F }}\omega_{E_\alpha/F_\alpha}(\iota_{\alpha}\alpha(s)).$$

\begin{proof}[Proof of \Cref{unramifiedprasad}]
Summarizing the above computation, we get
\begin{equation}
\omega_{G(F),E}|_{S(F)}(s)=\prod \limits_{\substack{\alpha\in \Phi_{\mathrm{sym,ram}}/\Gamma_F\\ \alpha \in \Phi_{\mathrm{sym,ram}}/\Gamma_E}}\omega_{E_{\alpha}/F_\alpha}(\iota_{\alpha} \alpha(s)),
\end{equation}
\begin{equation}
\epsilon_{\mathrm{Kal}}|_{S(F)}(s)=1,
\end{equation}
\begin{equation}
\epsilon_{\mathrm{HM}}(s)=1,
\end{equation}
\begin{equation}
\zeta_{T(E)}|_{S(F)}(s)=\prod\limits_{\substack{\alpha\in \Phi_{\mathrm{sym,ram}}/\Gamma_{E},\\ \alpha^{\mathrm{op}} \in \Phi_{\mathrm{sym,ram}}/\Gamma_F }}\omega_{E_\alpha/F_\alpha}(\iota_{\alpha}\alpha(s)).
\end{equation}
In this case, by \Cref{table5}, both $\alpha$ and $\alpha^{\mathrm{op}}$ are symmetric ramified over $F$. The above computation implies \Cref{productofcharacters} is true when $E/F$ is unramified.
\end{proof}

\subsection{Examples}

In this section, we apply our method to various examples and give a new proof of Prasad's conjecture for regular supercuspidal representations of these groups. Although most of these results are not new, our method is still somehow illuminating.

\begin{theorem}\label{prasadexample}
Prasad's \Cref{conjecture1} holds for regular supercuspidal representations of $\mathrm{SL}_2$, $\mathrm{GL}_2$, $\mathrm{GL}_n$ ($n$ odd) and $\mathrm{U}_{n,E/F}$ ($n$ odd).
\end{theorem}

\subsubsection{Example of $\mathrm{SL}_2$}
In the case of $\mathrm{SL}_2$, Prasad's character is trivial since $\mathrm{SL}_2$ is simply connected and $\mathrm{SL}_2(F)$ perfect. Hakim's character is also trivial since $\alpha(s)=(2\epsilon)(s)=\epsilon(s)^2$ is always a square. In fact, we can list all the possible cases of regular toral supercuspidal representations here. We also provide a comparison between our new interpretation of Hakim's character in \Cref{hakimfactorization} and the original one given by taking the determinant of the adjoint action.

\begin{example}[$\epsilon_{\mathrm{HM}}^{+}$ for $(\mathrm{SL}_2(E),\mathrm{SL}_2(F))$]\label{examplesl2}
Set $E=F(\sqrt{\alpha})$ and $K=E(\sqrt{\beta})=F(\sqrt{\alpha},\sqrt{\beta})$ with $\alpha,\beta, \alpha\beta\in F^{\times}-(F^{\times})^{2}$. Let $E_1$ denote $F(\sqrt{\beta})$ and let $E_2$ denote $F(\sqrt{\alpha\beta})$. We have the following biquadratic extension
$$
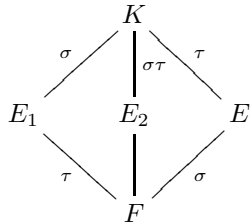
\xymatrix{& K \ar@{-}[dl]_{\sigma} \ar@{-}[d]^{\sigma\tau}
\ar@{-}[dr]^{\tau} \\ E_1 \ar@{-}[dr]_{\tau} & E_2
\ar@{-}[d] & E \ar@{-}[dl]^{\sigma} \\ & F}$$

By choosing ordered basis $(1,\sqrt{\beta})$ and $(1,\sqrt{\alpha\beta})$, $K^{\times}$ can be embedded into $\mathrm{GL}_2(E)$ in the following two ways :
\begin{enumerate}[(i)]
\item $a+b\sqrt{\beta}\longmapsto\left(\begin{matrix} a&b\beta\\ b&a\end{matrix}\right)$,
\item $a+b\sqrt{\beta}\longmapsto\left(\begin{matrix} a&b\sqrt{\alpha}\beta\\ \sqrt{\alpha}^{-1}b&a\end{matrix}\right)$.
\end{enumerate}
These two embeddings are $\mathrm{GL}_2(E)$-equivalent. More precisely, $\mathrm{Ad}\left(\begin{matrix} 1&0\\ 0& \sqrt{\alpha}^{-1}\end{matrix}\right)$ maps $(i)$ to $(ii)$. However, the intersection with $\mathrm{GL}_2(F)$ of these two elliptic maximal tori are different:
$$T_1(E)\cap \mathrm{GL}_2(F)=E_{1}^{\times},\quad T_2(E)\cap \mathrm{GL}_2(F)=E_{2}^{\times}.$$
Hence it indeed happens that there are different $F$-structures of $gTg^{-1}$, which means that the sum over abstract torus with different $F$-structure in our formula is not trivial a priori.

Let $T_0$ be the split maximal torus of $\mathrm{SL}_2$ over $E$ with $T_0(E)\cong E^{\times}$ consisting of diagonal matrices. Let $T$ be the anisotropic maximal torus of $\mathrm{SL}_2$ over $E$ with $T(E)\cong K^{1}_{K/E}$, together with an embedding given by
\begin{equation*}
\begin{aligned}
K^1&\hookrightarrow \mathrm{SL}_2(E)\\
a+b\sqrt{\beta}&\mapsto \{\left(\begin{matrix} a&b\beta\\ b&a\end{matrix}\right)|\,a^2-b^2\beta=1\}.
\end{aligned}
\end{equation*}
There are indeed two $\mathrm{SL}_2(E)$ rational conjugacy classes of admissible $E$-embeddings of $T$ into $\mathrm{SL}_2$ since $$\mathrm{ker}\left(H^1(E,T)\rightarrow H^1(E,\mathrm{SL}_2)\right)=H^1(E,T)\cong E^\times/\mathrm{Nm}_{K/E}K^\times\cong \{\pm1\}.$$
We simply take the above embedding as an example.

Let $g$ denote $\left(\begin{matrix} 1&\sqrt{\beta}\\ 1&-\sqrt{\beta}\end{matrix}\right)\in \mathrm{GL}_2(K)$ (In fact we can even choose an element $g'\in \mathrm{SL}_2(K)$. Here we use $g$ for simplicity). Then we have
$$g\left(\begin{matrix} a&b\beta\\ b&a\end{matrix}\right)g^{-1}=\left(\begin{matrix} a+b\sqrt{\beta}&0\\ 0&a-b\sqrt{\beta}\end{matrix}\right),$$
that is, $gT(K)g^{-1}=T_0(K)$. Let $\mathcal{A}(G,T_0,E)$ be the apartment associated to $T_0$ of the building $\mathcal{B}(G,E)$. There are identifications
\begin{equation*}
\begin{aligned}
\mathcal{A}(G,T_0,K)\cong X_{*}(T_{0,K})\otimes\mathbb{R}=\mathbb{R}\langle \check{\alpha}\rangle,\\
\mathcal{A}(G,T_0,E)\cong X_{*}(T_0)\otimes\mathbb{R}=\mathbb{R}\langle \check{\alpha}\rangle.
\end{aligned}
\end{equation*}

Since $T_0$ is split, $\mathrm{Gal}(K/E)$ acts trivially on $X_{*}(T_{0,K})$ and
$$\mathcal{A}(G,T_0,E)=\mathcal{A}(G,T_0,K)^{\mathrm{Gal}(K/E)}.$$
Notice that $T_0(K)$ and $T(K)$ are related by $gT(K)g^{-1}=T_0(K)$. By the $G(K)$-equivariant property of the Bruhat--Tits building $\mathcal{B}(G,K)$, we know
$$\mathcal{A}(G,T,K)=\mathcal{A}(G,g^{-1}T_{0}g,K)=g^{-1}\cdot\mathcal{A}(G,T_{0},K),$$
$$\mathcal{A}(G,T,E)=\mathcal{A}(G,T,K)^{\mathrm{Gal}(K/E)}.$$
By \cite[Section 2.2]{HL11}, condition $\tau(g^{-1}x)=g^{-1}x$ implies that $x$ is the origin $0$ of $\mathcal{A}(G,T_{0},K)$. Hence we get $\mathcal{A}(G,T,E)=\{g^{-1} \cdot0\}:=\{y\}$, which is a point in $\mathcal{B}(G,E)$.

Now we restrict ourselves to the toral case (every tame supercuspidal representation of $\mathrm{SL}_2$ is either of depth zero or toral), that is, the twisted Levi sequence is $(T(E),G(E))$. We will describe the group $J$ and $J^{+}$ and compute Hakim's character explicitly in this case. Such computations have essentially been done by Hakim and Lansky in \cite{HL11}.

Let $-r$ be the depth of the generic element $X^{*}\in \mathfrak{t}_{-r}^{*}-\mathfrak{t}_{-r+}^{*}$ over $E$, which corresponds to $X\in \mathfrak{t}_{-r}-\mathfrak{t}_{-r+}$ via $\langle X^{*},Y\rangle=\mathrm{tr}(XY)$. We fix a valuation $v_E$ on $E$ such that $v_E(E^{\times})=\mathbb{Z}$. Let $s=r/2$. The valuation of $\mathrm{tr}(X d\alpha^{\vee}(1))$ forces the condition
$$r\in \frac{1}{e(K/E)}+\mathbb{Z}.$$
By \cite[A.9]{HL11}, we have
$$\mathfrak{J}=(\mathfrak{t},\mathfrak{g})_{y,(r,s)}=\mathfrak{t}_{r}+\{\left(\begin{matrix} a&b\beta\\ -b&-a\end{matrix}\right)|a\in E_s,b\in E_{s'}\},$$
$$\mathfrak{J}^{+}=(\mathfrak{t},\mathfrak{g})_{y,(r,s+)}=\mathfrak{t}_{r}+\{\left(\begin{matrix} a&b\beta\\ -b&-a\end{matrix}\right)|a\in E_{s+},b\in E_{s'+}\},$$
where $s'=s-\frac{e(K/E)-1}{2}$.
Let $\mathfrak{W}=J/J^{+}\cong \mathfrak{J}/\mathfrak{J}^{+}$ denote the symplectic vector space over $k_F$, where the Moy--Prasad isomorphism between $\mathfrak{J}/\mathfrak{J}^{+}$ and $J/J^{+}$ is realized by Cayley transform \cite[A.6]{HL11}. Moreover, $\mathfrak{W}$ has the following explicit description.

\begin{enumerate}\label{sl2}
\item If $K/E$ is ramified, then $\mathfrak{W}=0$. \\
\item If $K/E$ is unramified and $r$ is odd, then $\mathfrak{W}=0$.\\
\item If $K/E$ is unramified, $r$ is even, then $s=s'\in \mathbb{Z}$, we know that
$$\mathfrak{J}=(\mathfrak{t},\mathfrak{g})_{y,(r,s)}=\mathfrak{t}_{r}+\{\left(\begin{matrix} a&b\beta\\ -b&-a\end{matrix}\right)|a\in E_s,b\in E_{s}\},$$
$$\mathfrak{J}^{+}=(\mathfrak{t},\mathfrak{g})_{y,(r,s+)}=\mathfrak{t}_{r}+\{\left(\begin{matrix} a&b\beta\\ -b&-a\end{matrix}\right)|a\in E_{s+1},b\in E_{s+1}\}.$$
In this case, $\mathfrak{W}\cong E_{s,s+}\times E_{s,s+}$ is endowed with a natural symplectic structure associated to an additive character such that $\mathfrak{W}^+\cong \mathfrak{J}^{\sigma}/{\mathfrak{J}^{+}}^{\sigma}\cong F_{s,s+}\times F_{s,s+}$.

\end{enumerate}

\begin{remark}
In the first two cases, the symplectic vector space $\mathfrak{W}$ degenerates since $\frac{r}{2}$ is not a jump in the relevant filtration. Indeed, we can compute the action of $T(E)$ on $\mathfrak{W}$ explicitly (although we actually restrict the action to $S(F)$ on $\mathfrak{W}^{\sigma}$). If we choose the isomorphism $\{\left(\begin{matrix} x\\ y\end{matrix}\right)|x\in E_{s,s+},y\in E_{s,s+}\}\leftrightarrow\{\left(\begin{matrix} x&y\beta\\ -y&-x\end{matrix}\right)|x\in E_s,y\in E_{s}\}$, the action of $\mathrm{Ad}(\left(\begin{matrix} a&b\beta\\ b&a\end{matrix}\right))$ is given by

$$\left(\begin{matrix} x&y\beta\\ -y&-x\end{matrix}\right) \mapsto \left(\begin{matrix} (a^2+b^2\beta)x-2ab\beta y&((a^2+b^2\beta)y-2abx)\beta\\ 2abx-(a^2+b^2\beta)y& -(a^2+b^2\beta)x-2ab\beta y)\end{matrix}\right).$$
One can see that $\det (\left(\begin{matrix} a^2+b^2\beta&-2ab\beta\\ -2ab&a^2+b^2\beta\end{matrix}\right)(\bmod\mathfrak{J}^{+}))$ is $1$ because $(a^2-b^2\beta)^2=1$.
\end{remark}

Now we use $\mathrm{SL}_2$ as an example to show how to compute Hakim's character using our reinterpretation in \ref{newinterpretation}. In this case, $\alpha$ is symmetric both over $F$ and over $E$. In particular, we have $E_{\alpha}=K$, $E_{\pm\alpha}=E$, $F_{\alpha}=E_1$ or $E_2$ depending on the $F$-embedding of the anisotropic $F$-torus, and $F_{\pm\alpha}=F$. Furthermore, we also have
$$(\bigoplus_{\alpha\in\mathcal{O}}\mathfrak{g}_{\alpha})(E)\cong E_{\alpha}\cong K,$$
together with isomorphisms
\begin{equation*}
\begin{aligned}
\mathfrak{W}\cong \mathfrak{J}/\mathfrak{J}^{+}&\cong k_K,\\
\{\left(\begin{matrix} x&y\beta\\ -y&-x\end{matrix}\right)|x\in E_s,y\in E_{s}\}&\mapsto x-y\sqrt{\beta}.
\end{aligned}
\end{equation*}
One can check directly that the adjoint action of $s=\left(\begin{matrix} a&b\beta\\ b&a\end{matrix}\right)\in S(F)=F_{\alpha}^{1}$ is simply given by multiplication by $\alpha(s)=\frac{a+b\sqrt{\beta}}{a-b\sqrt{\beta}}=(a+b\sqrt{\beta})^2$ under the above isomorphism, and taking the determinant of the adjoint action mod $\mathfrak{J}^+$ is simply given by
$$\mathrm{Nm}_{k_{F_{\alpha}}/k}(\alpha(s)\bmod(1+\varpi_{F_{\alpha}}))=1,$$
which implies the triviality of $\epsilon_{\mathrm{HM}}^{+}$.
\end{example}

In the case of $\mathrm{SL}_2/E$, $\alpha$ is always symmetric over $E$. By our interpretation \ref{toralkottwitz} or \cite[Lemma 4.5.3]{Kal19b}, the toral invariant of $\mathrm{SL}_2$ is always trivial. Hence Kaletha's character is trivial by the triviality of the toral invariant in the symmetric ramified case, together with the fact $\alpha=2\epsilon$ in the symmetric unramified case. Hence the only remaining part is the computation of character of $\zeta$-data. Actually, the computation of $\zeta|_{i(S)}$ is exactly the same as the case of $\mathrm{GL}_2$. By \Cref{zeta1,zeta2,zeta3}, we obtain $\zeta|_{i(S)(F)}=\omega_{K/E_1}|_{E_{1,E_1/F}^{1}}=\omega_{E/F}\circ\mathrm{Nm}_{E_1/F}|_{E_{1,E_1/F}^{1}}=\mathbbm{1}$ or $\zeta|_{i(S)(F)}=\mathbbm{1}|_{E_{1,E_1/F}^{1}}$, both of which are trivial. Hence, Conjecture \eqref{productofcharacters} is true for $\mathrm{SL}_2$, which gives a proof of Prasad's conjecture for regular supercuspidal representations of $\mathrm{SL}_2$.

\subsubsection{Example of $\mathrm{GL}_2$}
The difference in treating the case of $\mathrm{SL}_2$ and $\mathrm{GL}_2$ shows different complexity of Prasad's conjecture. For the case of $\mathrm{SL}_2$, the double coset part, which is related to the set of certain embeddings of tori, is more complicated, while the character part is relatively easy. However, for the case of $\mathrm{GL}_2$, the non-trivial quadratic character associated to $E/F$ appears, while the embedding part is easy since there are not so many equivalence classes of rational embeddings.
Let $T$ be an elliptic maximal torus of $\mathrm{GL}_2$ over $E$. Then $T(E)=K^{\times}$ for a quadratic extension $K/E$. A necessary condition for the existence of an elliptic maximal torus of $\mathrm{GL}_2(F)$ whose base change equals to $T$ is that $K/F$ is a biquadratic extension. Let $E_1$ and $E_2$ be the other $2$ intermediate fields between $K$ and $F$. Let $\tau\in\mathrm{Gal}(E_1/F)$ denote the non-trivial involution. Then $\mathrm{Gal}(K/F)\cong \{1,\sigma,\tau,\sigma\tau|\sigma^2=\tau^2=1,\sigma\tau=\tau\sigma\}.$ For simplicity, let $S$ be a torus over $F$, whose base change to $E$ is $T$ with $S(F)=E_1^{\times}$.

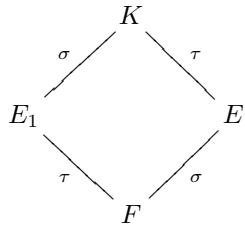
$$\xymatrix{& K \ar@{-}[dl]_{\sigma} \ar@{-}[d]^{\sigma\tau}
 \ar@{-}[dr]^{\tau} \\ E_1 \ar@{-}[dr]_{\tau} & E_2
 \ar@{-}[d] & E \ar@{-}[dl]^{\sigma} \\ & F}$$

In the case of $\mathrm{GL}_2$, Prasad's quadratic character is given by
\begin{equation*}
\begin{aligned}
\omega_{\mathrm{GL}_{2}(F),E}:\mathrm{GL}_2(F)&\rightarrow \{\pm 1\},\\
g&\mapsto \omega_{E/F}(\det(g)),\\
\omega_{\mathrm{GL}_{2}(F),E}|_{S(F)}:E_{1}^{\times}&\rightarrow \{\pm 1\}\\
s&\mapsto \omega_{E/F}(\mathrm{Nm}_{E_1/F}(s))=\omega_{K/E_{1}}(s).
\end{aligned}
\end{equation*}

We will present an explicit computation in the case of $\mathrm{GL}_2$. Before doing this, we first review some results on the descriptions of supercuspidal representations of $\mathrm{GL}_{2}(E)$, which arises as base change of irreducible representations of $\mathrm{U}_{1,1,E/F}(F)$. These results are essentially due to \cite{Bla10}.

Assume $E/F$ is ramified, and $K/E$ is unramified. Then $E_1/F$ could be unramified or ramified. Let $\theta:T(E)=K^{\times}\rightarrow \mathbb{C}^{\times}$ be a character such that $\pi$ is a regular supercuspidal representation of $\mathrm{GL}_2(E)$ associated to $(T,\theta)$. Let $\mathrm{level}(\theta)$ denote the level of $\theta$ (always an integer), which is related to $\mathrm{depth}(\theta)$ by $\mathrm{level}(\theta)=e(K/E)\mathrm{depth}(\theta)$. Since $K/E$ is unramified, we always have $\mathrm{level}(\theta)=\mathrm{depth}(\theta)$. Assume $\phi_{\pi}$ is a base change from a Langlands parameter of $\mathrm{U}_{1,1,E/F}(F)$. Then $T$ is a base change of a torus $S^{\mathrm{op}}\subset \mathrm{U}_{1,1,E/F}$ with $S^{\mathrm{op}}(F)=K_{K/E_1}^{1}$, and $\theta$ is a base change of $\mu:S^{\mathrm{op}}(F)=K_{K/E_1}^{1}\rightarrow\mathbb{C}^{\times}$, i.e., $\theta(t)=\mu(\frac{t}{\sigma(t)}),\forall t\in K^\times$.

\begin{lemma}\label{gl2depth}
In the above setting, $\mathrm{depth}(\theta)=\mathrm{level}(\theta)$ is odd if and only $E_1/F$ is unramified.
$$\xymatrix{& K \ar@{-}[dl]_{\mathrm{r}}
\ar@{-}[dr]^{\mathrm{ur}} \\ E_1 \ar@{-}[dr]_{\mathrm{ur}} &
 & E. \ar@{-}[dl]^{\mathrm{r}} \\ & F}$$
\end{lemma}

\begin{proof}
This result has been mentioned in \cite{Bla10}. We write down a proof for completeness.

$\Leftarrow$. This direction is relatively easy. Let $r$ denote the depth of $\theta$. Let $X\in K_{-r}$ be the generic element associated to $\theta=\theta_0\cdot \theta_1$, where $\theta_1$ denotes the positive-depth part. More precisely,
$$\theta_{1}(1+Y)=\psi_E(\langle X,Y\rangle).$$
Since $\theta(t)=\mu(\frac{t}{\sigma(t)})$, we know that $\theta^{\sigma}=\theta^{-1}$, in particular, $\theta_{1}^{\sigma}=\theta_{1}^{-1}$. This is equivalent to the fact $\sigma(X)=-X\bmod{\varpi_K}^{-r+1}$. Since $K/E_1$ is ramified, we can choose a uniformizer $\varpi_K$ of $K$ such that $X=a{\varpi_K}^{-r}$ and $\sigma(\varpi_K)=-{\varpi_K}$. The condition $\sigma(X)=-X$ forces $r$ to be odd.

$\Rightarrow$. The direction is suggested by Blasco. Since $K/E$ is unramified, we can choose $\varpi_K$ to be a uniformizer $\varpi_E$ of $E$. Since $E/F$ is ramified, we can choose $\varpi_E$ to satisfy $\varpi_E=-\sigma(\varpi_E)$. Since $X$ corresponds to a level $r$ character, ${\varpi_E}^{r}X$ is a unit. On the other hand, since $r$ is odd, we know $\sigma({\varpi_E}^{r}X)=(-1)^{r}{\varpi_E}^{r}(-X)={\varpi_E}^{r}X$, which means that ${\varpi_E}^{r}X\in E_1$. However ${\varpi_E}^{r}X$ does not lie in $E$, hence does not belong to $F$, which means ${\varpi_E}^{r}X$ generates $E_1$ over $F$. This implies $E_1/F$ is unramified, hence $K/E_1$ is ramified.
\end{proof}

Now we give a comparison of these four characters in the case of $\mathrm{GL}_2$.

\begin{enumerate}
\item  The odd case (this terminology is from \cite{HM02}), i.e., $E/F$ and $E_1/F$ are ramified.\\

$\bullet$ Representation side:\\
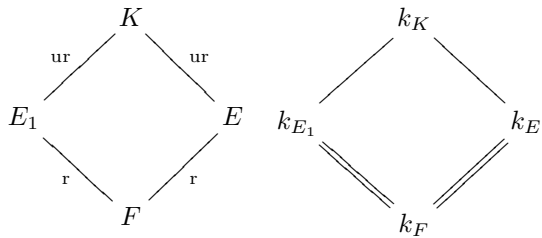
$$\xymatrix{& K \ar@{-}[dl]_{\mathrm{ur}}
\ar@{-}[dr]^{\mathrm{ur}} \\ E_1 \ar@{-}[dr]_{\mathrm{r}} &
 & E, \ar@{-}[dl]^{\mathrm{r}} \\ & F}  \quad\xymatrix{& k_K \ar@{-}[dl]
\ar@{-}[dr] \\ k_{E_1} \ar@{=}[dr] &
 & k_E \ar@{=}[dl] \\ & k_F}$$

In this case, by \Cref{gl2depth}, we know that $\mathrm{depth}\,\theta$ is even, which means that there is a non-degenerate Heisenberg quotient and $\alpha$ belongs to $\Phi_{\frac{r}{2}}$, that is, for $s\in S(F)=E_1^\times$,
\begin{equation*}
\begin{aligned}
\epsilon_{\mathrm{Kal}}|_{E_{1}^{\times}}(s)&=\mathrm{sgn}_{k_{K/E}^{1}}(\alpha(s))=(\frac{s}{\tau(s)})^{\frac{q_{E}+1}{2}},\\
\epsilon_{\mathrm{HM}}(s)&=\mathrm{sgn}_{k_{E_{1}}^{\times}}(\alpha(s))=(\frac{s}{\tau(s)})^{\frac{q_{E_1}-1}{2}}.
\end{aligned}
\end{equation*}

Observe that both of these two characters are trivial on $\mathcal{O}_{E_1}^{\times}$ since $\frac{s}{\tau(s)}=1(\bmod 1+p_{E_1})$ for $t\in \mathcal{O}_{E_1}^{\times}$. On $ \varpi_{E_1}$, the product of these two characters is given by
$$\epsilon_{\mathrm{Kal}}|_{E_{1}^{\times}}(\varpi_{E_1})\cdot\epsilon_{\mathrm{HM}}(\varpi_{E_1})=(-1)^{\frac{q_{E_1}-1}{2}}(-1)^{\frac{q_E+1}{2}}=-1,$$
which means
$$\epsilon_{\mathrm{Kal}}|_{E_{1}^{\times}}(s)\cdot\epsilon_{\mathrm{HM}}(s)=\omega_{\mathrm{Pra}}|_{E_{1}^{\times}}(s)$$
since these two non-trivial quadratic characters coincide on $\mathcal{O}_{E_1}^{\times}\times {\varpi}_{E_1}^{\mathbb{Z}}$.\\

$\bullet$ Parameter side:\\

In this case, we have $S^{\mathrm{op}}\cong \mathrm{Res}_{E_{1}/F}(\mathrm{U}_{1,K/E_1})\subset \mathrm{U}_{1,1,E/F}=G^{\mathrm{op}}$, with
$$S^{\mathrm{op}}(F)=\mathrm{U}_{1,K/E_{1}}(E_1)=\{x\in K|\,x\sigma(x)=1\}.$$

Moreover, we also have $F_{\alpha^{\mathrm{op}}}=E_2$ and $F_{\pm\alpha}=F$. The commutative diagram
$$\xymatrix{
  S^{\mathrm{op}}(F) \ar@{^{(}->}[d] \ar[r]^{\alpha}
                & {F_{\alpha^{\mathrm{op}}}}^{\times} \ar@{^{(}->}[d]  \\
  T(E)  \ar[r]^{\alpha}
                & E_{\alpha}^{\times}}$$
becomes
$$\xymatrix{
  K^{1}_{K/E_1} \ar@{^{(}->}[d] \ar[r]^{\alpha}
                & E_{2}^{\times} \ar@{^{(}->}[d]  \\
  K^{\times}  \ar[r]^{\alpha}
                & K^{\times},}$$
and $\alpha$ is given by $\alpha(t)=\frac{t}{\tau(t)},$ for any $t\in K^{\times}$ or $t\in K^{1}_{K/E_1}$.
$$\xymatrix{& K \ar@{-}[dl]_{\mathrm{r}}
\ar@{-}[dr]^{\mathrm{ur}} \\ E_2 \ar@{-}[dr]_{\mathrm{ur}} &
 & E, \ar@{-}[dl]^{\mathrm{r}} \\ & F}  \quad\xymatrix{& k_K \ar@{=}[dl]
\ar@{-}[dr] \\ k_{E_2} \ar@{-}[dr] &
 & k_E \ar@{=}[dl] \\ & k_F}$$
In this case, $\chi_{\mu\circ \mathrm{Nm}_{S(E)/S(F)},\alpha}$ is the unramified quadratic character of $K^{\times}$ and $\chi_{\mathrm{BC},\alpha}=\chi_{\mu}\circ \mathrm{Nm}_{K/E_2}$, where $\chi_{\mu}$ is the unramified quadratic character of $E_{2}^{\times}$. Hence we have $\chi_{\mu\circ \mathrm{Nm}_{S(E)/S(F)},\alpha}=\chi_{\mathrm{BC},\alpha}$, and we also have
$$\zeta_{K^{\times}}(t)=\zeta_{\alpha}(\iota_{E_\alpha}\alpha(t))=\mathbbm{1}(t)=1.$$
This implies the equality
\begin{equation}\label{zeta1}
\epsilon_{\mathrm{Kal}}|_{E_{1}^{\times}}\cdot\epsilon_{\mathrm{HM}}\cdot\omega_{\mathrm{Pra}}|_{E_{1}^{\times}}=\zeta_{K^{\times}}|_{E_{1}^{\times}}=\mathbbm{1}.
\end{equation}
\item The even case. There are two cases:
\begin{enumerate}

\item $E/F$ is ramified and $E_1/F$ is unramified.\\

$\bullet$ Representation side:\\
$$\xymatrix{& K \ar@{-}[dl]_{\mathrm{r}}
\ar@{-}[dr]^{\mathrm{ur}} \\ E_1 \ar@{-}[dr]_{\mathrm{ur}} &
 & E, \ar@{-}[dl]^{\mathrm{r}} \\ & F}  \quad\xymatrix{& k_K \ar@{=}[dl]
\ar@{-}[dr] \\ k_{E_1} \ar@{-}[dr] &
 & k_E \ar@{=}[dl] \\ & k_F}.$$
A priori, if $\alpha\in \Phi_{\frac{r}{2}}$, we get
\begin{equation*}
\begin{aligned}
\epsilon_{\mathrm{Kal}}|_{E_{1}^{\times}}(s)&=\mathrm{sgn}_{k_{K/E}^{1}}(\alpha(s))=\mathrm{sgn}_{k_{K/E}^{1}}(\frac{s}{\tau(s)}),\\
\epsilon_{\mathrm{HM}}(s)&=\mathrm{sgn}_{k_{F}^{\times}}(\mathrm{Nm}_{k_{E_{1}}/k_F}\alpha(s))=(\frac{s}{\tau(s)}\cdot \frac{\tau(s)}{s})^{\frac{q_{E_1}-1}{2}}=1.
\end{aligned}
\end{equation*}

It is clear that $\epsilon_{\mathrm{Kal}}|_{E_{1}^{\times}}(s)\cdot\epsilon_{\mathrm{HM}}(s)$ is trivial on $\varpi_{E_1}$ and given by $(\frac{s}{\tau(s)})^{\frac{q_E+1}{2}}$ for $s\in \mathcal{O}_{E_{1}}^\times$. Indeed $\omega_{K/E_{1}}$ satisfies $$\omega_{K/E_{1}}(\varpi_{E_1})=\omega_{K/E_{1}}(-1)=(-1)^{\frac{q_{E_{1}}-1}{2}}=1$$
since $q_{E_{1}}=q_{F}^{2}= 1\bmod 4$. Moreover, $\omega_{K/E_1}$ is given by $[s]^{\frac{q_{E_{1}}-1}{2}}$ for $s\in \mathcal{O}_{E_{1}}^\times$. Furthermore, we also have the following identity for $s\in \mathcal{O}_{E_{1}}^\times$
$$([\frac{s}{\tau(s)}])^{\frac{q_E+1}{2}}=(\frac{[s]}{[s]^{q_{F}}})^{\frac{q_E+1}{2}}=[s]^{\frac{1-q_{E_{1}}}{2}}=[s]^{\frac{q_{E_{1}}-1}{2}}.$$
This further means that if $\alpha\in \Phi_{\frac{r}{2}}$, then $\epsilon_{\mathrm{Kal}}|_{E_{1}^{\times}}(s)\epsilon_{\mathrm{HM}}(s)=\omega_{K/E_1}(s)$.

However, in this case, $\mathrm{depth}\,\theta$ is odd so that $\alpha\notin \Phi_{\frac{r}{2}}$ and the Heisenberg quotient is degenerate. Hence both Kaletha's character and Hakim's character are automatically trivial in this case.\\

$\bullet$ Parameter side:\\
$$\xymatrix{& K \ar@{-}[dl]_{\mathrm{ur}}
\ar@{-}[dr]^{\mathrm{ur}} \\ E_2 \ar@{-}[dr]_{\mathrm{r}} &
 & E, \ar@{-}[dl]^{\mathrm{r}} \\ & F}  \quad\xymatrix{& k_K \ar@{-}[dl]
\ar@{-}[dr] \\ k_{E_2} \ar@{=}[dr] &
 & k_E \ar@{=}[dl] \\ & k_F}$$

Following \Cref{symur'}, $\chi_{\mathrm{BC},\alpha}=\chi_{\mu,\alpha}\circ \mathrm{Nm}_{K/E_2}$, where $\chi_{\mu,\alpha}$ is the quadratic character of $E_{1}^{\times}$ determined by mod-$a$-data, and $\chi_{\mu\circ \mathrm{Nm}_{S(E)/S^{\mathrm{op}}(F)},\alpha}$ is the unramified quadratic character $\omega_{K'/K}$ of $K^{\times}$. On the one hand, we know that $\chi_{\mu\circ \mathrm{Nm}_{S(E)/S^{\mathrm{op}}(F)},\alpha}|_{E_{1}^{\times}}=\omega_{K'/K}|_{E_{1}^{\times}}=\mathbbm{1}$. On the other hand, by \Cref{missing}, we know that
$$(\chi_{\mu,\alpha}\circ \mathrm{Nm}_{K/E_2})|_{E_{1}^{\times}}=\omega_{K/E_1}.$$


Hence we also have the equality
\begin{equation}\label{zeta2}
\epsilon_{\mathrm{Kal}}|_{E_{1}^{\times}}\cdot\epsilon_{\mathrm{HM}}\cdot\omega_{\mathrm{Pra}}|_{E_{1}^{\times}}=\zeta_{K^{\times}}|_{E_{1}^{\times}}=\omega_{K/E_1}.
\end{equation}

\item $E/F$ is unramified and $E_1/F$ is ramified.\\

$\bullet$ Representation side:\\
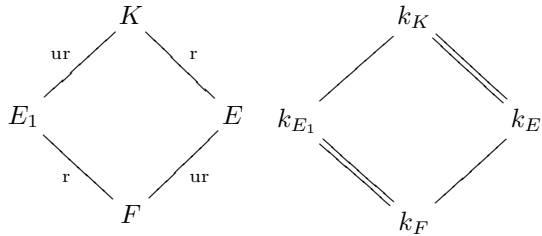
$$\xymatrix{& K \ar@{-}[dl]_{\mathrm{ur}}
\ar@{-}[dr]^{\mathrm{r}} \\ E_1 \ar@{-}[dr]_{\mathrm{r}} &
 & E, \ar@{-}[dl]^{\mathrm{ur}} \\ & F} \quad\xymatrix{& k_K \ar@{-}[dl]
\ar@{=}[dr] \\ k_{E_1} \ar@{=}[dr] &
 & k_E \ar@{-}[dl] \\ & k_F}$$
A priori, if $\alpha\in\Phi_{\frac{r}{2}}$, then we get
\begin{equation*}
\begin{aligned}
\epsilon_{\mathrm{Kal}}|_{E_{1}^{\times}}(s)&=f_{(G,T)}(\alpha)=1,\\
\epsilon_{\mathrm{HM}}(s)&=\mathrm{sgn}_{k_{E_{1}}^{\times}}(\alpha(s))=(\frac{s}{\tau(s)})^{\frac{q_{E_1}-1}{2}}.
\end{aligned}
\end{equation*}
However, since $\mathfrak{W}$ is trivial by \Cref{sl2} in this case, both Kaletha's character and Hakim's character are trivial.\\

$\bullet$ Parameter side:\\
$$\xymatrix{& K \ar@{-}[dl]_{\mathrm{ur}}
\ar@{-}[dr]^{\mathrm{r}} \\ E_2 \ar@{-}[dr]_{\mathrm{r}} &
 & E, \ar@{-}[dl]^{\mathrm{ur}} \\ & F}  \quad\xymatrix{& k_K \ar@{-}[dl]
\ar@{=}[dr] \\ k_{E_2} \ar@{=}[dr] &
 & k_E \ar@{-}[dl] \\ & k_F}$$
By \Cref{symram}, we know $\zeta_{K^{\times}}|_{E_{1}^{\times}}=\omega_{K/E_1}$. Hence we have
\begin{equation}\label{zeta3}
\epsilon_{\mathrm{Kal}}|_{E_{1}^{\times}}\cdot\epsilon_{\mathrm{HM}}\cdot\omega_{\mathrm{Pra}}|_{E_{1}^{\times}}=\zeta_{K^{\times}}|_{E_{1}^{\times}}=\omega_{K/E_1}.
\end{equation}
\end{enumerate}
\end{enumerate}

\begin{corollary}\label{gl2prasad}
Prasad's \Cref{conjecture1} holds for regular supercuspidal representations of $\mathrm{GL}_2$.
\end{corollary}
It seems that our approach is quite different from the proof existing in the literature. These characters may not be seen explicitly in the trace formula approach or in the local Rankin--Selberg method even in the simple case of $\mathrm{GL}_2$.

\subsubsection{Example of $\mathrm{GL}_n$, $n$ is odd}
Let $S$ be the elliptic maximal torus $\mathrm{Res}_{E_1/F}\mathbb{G}_m$ of $\mathrm{GL}_n$, where $E_1$ is a degree $n$ extension of $F$. Since we are considering the case when $S_E$ is also elliptic, $E$ and $E_1$ are disjoint. Let $K:=EE_1$ denote the composite of $E$ and $E_1$. We fix a set of representatives $\{g_1,\cdots,g_n\}$ of $\Gamma_F/\Gamma_{E_1}$. One can diagonalize the element $t\in E_{1}^{\times}\in \mathrm{GL}_n(F)$ over $\overline{F}$, such that $t=(t_{g_1},\cdots,t_{g_n})$, where $t_{g_i}$ lies in ${\overline{F}}^{\times}$. Then each root in $\Phi(G,S)$ is of the form $t\mapsto \frac{t_{g_1}}{t_{g_j}}$. When $E_{1}/F$ is cyclic generated by $\tau$, then $e_i-e_j$ is simply $t\mapsto \frac{\tau^{i}(t)}{\tau^{j}(t)}$. Since $n$ is odd, we know that all the roots are asymmetric both over $E$ and over $F$. In this case, we have $E_{\alpha}=E_{\pm\alpha}=K$, $F_{\alpha}=F_{\pm\alpha}=E_1$, and $S^{\mathrm{op}}=\mathrm{Res}_{E_1/F}\mathrm{U}_{1,K/E_1}\subset G^{\mathrm{op}}=\mathrm{U}_{n,E/F}$. All the roots in $\Phi(G^{\mathrm{op}},S^{\mathrm{op}})$ are symmetric over $F$ but asymmetric over $E$. In this case, we have $F_{\alpha^{\mathrm{op}}}=K$ and $F_{\pm\alpha^{\mathrm{op}}}=E_1$.
$$
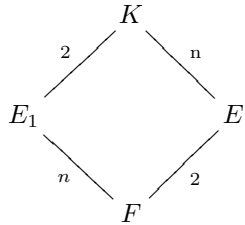
\xymatrix{& K \ar@{-}[dl]_{\mathrm{2}}
\ar@{-}[dr]^{n} \\ E_1 \ar@{-}[dr]_{n} &
 & E \ar@{-}[dl]^{2} \\ & F}$$
\paragraph{Prasad's character.}
Prasad's character is trivial, since $n$ is odd. The fact that the restriction of Prasad's character to an elliptic maximal torus is trivial can also be seen from \Cref{examplegln}. Since all the roots are asymmetric over $F$, when $n$ is odd.

\paragraph{Kaletha's character.}
A priori, it seems that Kaletha's character is not necessarily trivial from its definition. Notice that all the roots are also asymmetric over $E$. Hence from its definition, the character $\epsilon_{\mathrm{Kal}}:T(E)\rightarrow \{\pm1\}$ is given by
$$\prod_{\alpha\in (\Phi_{\frac{r_{i_\alpha}}{2}})_{\mathrm{asy}}/\Gamma_{E}\times\{\pm 1\}}\mathrm{sgn}_{k_{E_{\alpha}}^{\times}}\alpha(t).$$

\begin{lemma}\label{oddtrivial}
Let $T$ be a maximal elliptic torus of $G_E$. The character $\mathrm{sgn}_{k_{E_{\alpha}^{\times}}}(\alpha (t))$ is trivial, if $E_{\alpha}/E$ is of odd ramification degree.
\end{lemma}

\begin{proof}
Observe that $\sum\limits_{\alpha'\in\mathcal{O}_{\alpha}}\alpha'$ is a character of $T$ defined over $E$, where $\mathcal{O}_{\alpha}\cong \Gamma_E/\Gamma_{E_{\alpha}}$ is the $\Gamma_E$-orbit of $\alpha$. It must be trivial due to the ellipticity of $T$.
Hence we get
\begin{equation*}
\begin{aligned}
1&=\mathrm{sgn}_{k_{E^{\times}}}(\sum_{\alpha'\in\mathcal{O}_{\alpha}}\alpha')(t),\\
&=\mathrm{sgn}_{k_{E^{\times}}}(\mathrm{Nm}_{E_{\alpha}/E}\alpha(t)),\\
&=\mathrm{sgn}_{k_{E^{\times}}}(\mathrm{Nm}_{k_{E_{\alpha}}/k_E}\alpha(t))^{e(E_{\alpha}/E)},\\
&=\mathrm{sgn}_{k_{E_{\alpha}^{\times}}}(\alpha(t))^{e(E_{\alpha}/E)},
\end{aligned}
\end{equation*}
Hence if $e(E_{\alpha}/E)$ is odd, we know $\mathrm{sgn}_{k_{E_{\alpha}^{\times}}}(\alpha(t))=1$ for any $t\in T(E)$.
\end{proof}
Notice that all the elliptic tori of $\mathrm{GL}_{n,E}$ are of the form $\mathrm{Res}_{K/E}\mathbb{G}_m$, for a degree $n$ extension $K/E$ such that $E_\alpha=K$ for any absolute root $\alpha$. Hence by the above lemma, Kaletha's character $\epsilon_{\mathrm{Kal}}$ is trivial when $n$ is odd, and so is $\epsilon_{\mathrm{Kal}}|_{S(F)}$.

\paragraph{Hakim's character.}
Since all the roots are asymmetric over $F$ when $n$ is odd, Hakim's character $\epsilon:S(F)\rightarrow \{\pm1\}$ is given by 
\[\prod\limits_{\substack{\alpha\in \Phi_{\mathrm{asy}}/\Gamma_{F},\\ \frac{r_{i_\alpha}}{2}\in\mathrm{ord}_x{\alpha}/E}}\mathrm{sgn}_{k_{F_{\alpha}}^{\times}}\alpha(s)\]
by \Cref{hakimfactorization}. Since $e(F_{\alpha}/F)=e(E_1/F)$ is odd when $n$ is odd, by the same reason as \Cref{oddtrivial}, we know that $\mathrm{sgn}_{k_{F_{\alpha}^{\times}}}(\alpha(s))=1$. This implies that Hakim's character is trivial.

\paragraph{Character from $\zeta$-data.}
Since all the roots are asymmetric over $E$, by \eqref{zetacharacter}, the character is
$$\prod_{\alpha\in \Phi_{\mathrm{asy}}/\Gamma_E\times\{\pm1\}}\zeta_{\alpha}(\alpha(t))=\prod_{\alpha\in \Phi_{\mathrm{asy}}/\Gamma_E\times\{\pm1\}}\frac{\chi_{\alpha^{\mathrm{op}}}\circ \mathrm{id}}{\chi_{\alpha,E}}(\alpha(t)).$$

Notice that $\chi_{\alpha,E}=\mathbbm{1}$ since $\alpha$ is asymmetric over $E$. We only need to compute $\chi_{\alpha^{\mathrm{op}}}(\alpha(t))$. We can choose $\alpha$ in its orbit such that $\alpha(t)=\frac{t}{\tau(t)}$ for some $\tau\in \mathrm{Gal}(E_1/F)$. Then we come to the following cases.
\begin{enumerate}
\item If $K/E_1$ is unramified, then $\zeta_{T(E),\alpha}:K^{\times}\rightarrow \mathbb{C}^{\times}$ is given by $\omega_{K'/K}(\alpha(t))=\omega_{K'/K}(\frac{t}{\tau(t)})$. Notice that $\omega_{K'/K}|_{E_{1}^{\times}}=\omega_{K/E_1}$ since both $K'/K$ and $K/E_1$ are unramified. Hence \[\zeta_{T(E),\alpha}(s)=\omega_{K/E_1}(\frac{s}{\tau(s)})=\omega_{E/F}\circ \mathrm{Nm}_{E_1/F}(\frac{s}{\tau(s)})=\omega_{E/F}(1)=1,\quad \forall s\in S(F)=E_{1}^{\times}.\]
\item If $K/E_1$ is ramified, then $\zeta_{T(E),\alpha}$ is given by $\chi_{\alpha^{\mathrm{op}}}(\alpha (t))$, where $\chi_{\alpha^{\mathrm{op}}}$ is the tamely ramified character which is the unique non-trivial character on $k_{K}^{\times}$ and determined by the value at some mod-$a$-data. In this case, for any $s\in E_{1}^{\times}=S(F)$, we have $\alpha(s)\in {E_{1}}^{1}_{E_1/F}\subset O_{E_1}^{\times}$, which means that 
\[\zeta_{T(E),\alpha}|_{S(F)}=\mathrm{sgn}_{k_{K}^\times}\circ\alpha|_{S(F)}=\mathrm{sgn}_{k_{E_1}^\times}\circ\alpha|_{S(F)}\]
since $K/E_1$ is ramified. By the same computation as in \Cref{oddtrivial}, we know that 
\[\mathrm{sgn}_{k_{E_1}^\times}(\alpha(s))^{e(E_1/F)}=\mathrm{sgn}_{k_{E_{1}}^{\times}}((\sum_{\alpha'\in\mathcal{O}_{\alpha}}\alpha')(s))=1,\quad\forall s\in S(F)\]
due to the ellipticity of $S$ over $F$. Since $n$ is odd, we know that $e(E_1/F)$ is also odd so that \[\zeta_{T(E),\alpha}(s)=\mathrm{sgn}_{k_{E_1}^\times}(\alpha(s))=1,\quad\forall s\in S(F).\]

\end{enumerate}

The above computation shows that all these four characters are trivial for $G=\mathrm{GL}_n$. Hence \Cref{productofcharacters} is true for $G=\mathrm{GL}_n$, when $n$ is odd. This also gives a simple proof of Prasad's \Cref{conjecture1} for regular supercuspidal representations of $\mathrm{GL}_n$ when $n$ is odd.

\subsubsection{Example of $\mathrm{U}_{n,E/F},$ $n$ is odd}

Since $T\cong S_E$ is an elliptic maximal torus of $G_E=\mathrm{GL}_{n,E}$, $T$ is of the form $\mathrm{Res}_{K/E}\mathbb{G}_{m,K}$ for a degree $n$ extension $K/E$. Hence we have $S\cong\mathrm{Res}_{E_1/F}\mathrm{U}_{1,K/E_1}\subset G=\mathrm{U}_{n,E/F}$ for a degree $n$ extension $E_1/F$ such that $K=EE_1$. Moreover, we also have $E_{\alpha}=E_{\pm\alpha}=K$, $F_{\alpha}=K,F_{\pm\alpha}=E_1$, which means all the roots are asymmetric over $E$ but symmetric over $F$. We also know that $S^{\mathrm{op}}\cong \mathrm{Res}_{E_1/F}\mathbb{G}_m\subset G^{\mathrm{op}}=\mathrm{GL}_{n,F}$. Moreover, we get $F_{\alpha^{\mathrm{op}}}=F_{\pm\alpha^{\mathrm{op}}}=E_1$, which means all the roots in $\Phi(G^{\mathrm{op}},S^{\mathrm{op}})$ are asymmetric over $F$.

\paragraph{Prasad's character.}
In this case, all the roots are asymmetric over $E$, since $\mathrm{U}_{n,E/F}\times_{F} E\cong \mathrm{GL}_{n,E}$. Hence we have $E_{\alpha}=E_{\pm\alpha}=K$. Prasad's character is trivial by \eqref{unitary}. The fact that the restriction of Prasad's character to an elliptic maximal torus is trivial can also be seen from \Cref{rootcharacter}, since $E_{\alpha}=F_{\alpha}$.
\paragraph{Hakim's character.}
The only possibility for Hakim's character to be non-trivial is that $K/E_1$ is ramified quadratic by \Cref{Hakimcharacter}. In this case Hakim's character is given by $$\mathrm{sgn}_{k_{K}^{\times}}(\alpha(s)),$$
for $s\in S(F)=K^1_{K/E_1}$. Since $\alpha$ is symmetric ramified over $F$, $\alpha(s)\in K^{1}_{K/E_1}$. Moreover, since $K/E_1$ is ramified, the reduction mod $1+\varpi_{K}$ of $K^{1}_{K/E_1}$ is $\mathrm{O}_{1}(k_E)\cong \{\pm 1\}$. Hence the reduction of $\alpha(s)$ is $[\alpha(s)]=[\frac{s}{\tau(s)}]=\frac{[s]}{[\tau(s)]}=1$, which means Hakim's character is trivial.
\paragraph{Kaletha's character.}
Notice that $\mathrm{U}_{n,E/F}\times_{F} E\cong \mathrm{GL}_{n,E}$. By the argument in \Cref{oddtrivial}, Kaletha's character $\epsilon_{\mathrm{Kal}}$ is trivial when $n$ is odd. So is $\epsilon_{\mathrm{Kal}}|_{S(F)}$.
\paragraph{Character from $\zeta$-data}
Since all the roots $\alpha^{\mathrm{op}}$ are asymmetric both over $F$ and over $E$, by \eqref{zetacharacter}, the character is
\begin{equation*}
\begin{aligned}
\prod_{\alpha\in \Phi_{\mathrm{asy}}/\Gamma_E\times\{\pm1\}}\zeta_{\alpha}(\alpha(t))&=\prod_{\alpha\in \Phi_{\mathrm{asy}}/\Gamma_E\times\{\pm1\}}\frac{\chi_{\alpha^{\mathrm{op}}}\circ \mathrm{Nm}_{K/E}}{\chi_{\alpha,E}}(\alpha(t))\\
&=\prod_{\alpha\in \Phi_{\mathrm{asy}}/\Gamma_E\times\{\pm1\}}\frac{\mathbbm{1}\circ \mathrm{Nm}_{K/E}}{\mathbbm{1}}(\alpha(t))=1,\quad\forall t\in T(E).
\end{aligned}
\end{equation*}

The above computation shows that \Cref{productofcharacters} is true for $G=\mathrm{U}_{n,E/F}$, when $n$ is odd. This also gives a simple proof of Prasad's \Cref{conjecture1} for regular supercuspidal representations of $\mathrm{U}_{n,E/F}$ when $n$ is odd.

\appendix
\section{Whittaker model for distinguished regular supercuspidal representations}\label{appendix}


In this appendix, we collect some facts about the characterization of Whittaker models of regular supercuspidal representations in terms of geometry. This originates from the work of Debacker and Reeder \cite{DR10} on description of Whittaker models for generic very cuspidal representations. Basically, one can describe all the characters $\psi$ of $N(F)$ which occur in $\pi_{(S,\mu)}$ in terms of certain intersection property between the Kostant sections of $\psi$ and the $G(F)$-orbit of the generic element associated to $\pi_{(S,\mu)}$. Some of these works are generalized to regular supercuspidal representations in \cite{Kal19b} and \cite{FKS23}. We will also mention some results about Whittaker models of a $G(F)$-distinguished representation of $G(E)$. These results interact naturally when we study the distinction problem of a regular supercuspidal representation of $G(E)$. Throughout this appendix, we assume that $F$ has characteristic zero so that there will be no issue on the exponential map.

\subsection{Local character expansions}
We first review some basic facts about the character expansion and $p$-adic analogue of Rossmann's character formula, which is known as the Murnaghan--Kirillov formula \cite{Mur96}. Later, this work has been extended to a more general setting by many people, such as Jeffrey Adler, Stephen Debacker, David Kazhdan, Loren Spice \cite{AD04, DK11, AS09} etc. Fix a non-degenerate $G(F)$-invariant bilinear form $\langle,\rangle$ on $\mathfrak{g}(F)$, which identifies $\mathfrak{g}(F)$ with $\mathfrak{g}^{*}(F)$. For any regular semisimple element $X\in \mathfrak{g}(F)$, let $\mu_{X}$ denote the semisimple orbital integral
$$\mu_{X}(f):=\int_{G(F)/T(F)}f(^{g}X)d\dot{g},$$
where $T:=C_{G}(X)$ denotes the centralizer of $X$ and $d\dot{g}$ is a left invariant Haar measure on $\mathcal{O}_X\cong G(F)/T(F)$ defined by Ranga Rao \cite{RR72}, depending on a choice of Haar measures on $G(F)$ and $T(F)$. We regard $\mu_X(f)$ as a distribution if we vary $f\in C_{c}^{\infty}(\mathfrak{g})$, while we sometimes also regard $\mu_X(f)$ as a function on $\mathfrak{g}^{\mathrm{rs}}$ if we fix $f$ and vary $X\in \mathfrak{g}^{\mathrm{rs}}$. We use $\mu_{X}(f)$ and $\mu_f(X)$ to distinguish them. Since the orbital integral only depends on the $G(F)$-orbit $\mathcal{O}_X$ of $X$, we also adopt the notation $\mu_{\mathcal{O}_X}$ for $\mu_X$.

Fix an additive character $\psi:F\rightarrow \mathbb{C}^{\times}$. For any $G(F)$-orbit $\mathcal{O}$ in $\mathfrak{g}(F)$, one can define the Fourier transform of arbitrary orbital integral by
$$\widehat{\mu_{\mathcal{O}}}(f):=\mu_{\mathcal{O}}(\widehat{f}),$$
where $\widehat{f}\in C_{c}^{\infty}(\mathfrak{g}(F))$ is defined by
$$\widehat{f}(Y):=\int_{\mathfrak{g}(F)}f(Z)\psi(\langle Y,Z\rangle)dZ.$$
The Fourier transform of a nilpotent orbital integral $\widehat{\mu_{\mathcal{O}}}$ as a distribution could also be represented as a locally constant function on an open dense subset of $\mathfrak{g}(F)$, which we also denote by $\widehat{\mu_{\mathcal{O}}}$.

Let $\Theta_{\pi}$ denote the character distribution associated to an irreducible admissible representation $(\pi,V)$. More precisely, fix a Haar measure $dg$ on $G(F)$. For any $f\in C_{c}^{\infty}(G(F))$, one can define
$$\Theta_{\pi}(f):=\mathrm{tr}\,\pi(f),$$
where $\pi(f):V\rightarrow V$ is the operator defined by
$$\pi(f)v:=\int_{G(F)}f(g)\pi(g)vdg.$$
It is known that $\Theta_{\pi}$ is represented by some locally constant function on an open dense subset of $G(F)$, which we still denote by $\Theta_{\pi}$. Moreover, there is a well-known local expansion formula of $\Theta_{\pi}$, which is essentially due to Harish-Chandra and Howe \cite[Theorem 16.2]{HC99}.
\begin{theorem}[Harish-Chandra, Howe]

There exist constants $c_{\mathcal{O}}(\pi,\gamma_{0})$ and a neighbourhood $U$ of $0$ in $\mathfrak{m}(F)$ such that, for any $Y\in U\cap \mathfrak{m}(F)^{\mathrm{reg}}$
$$\Theta_{\pi}(\gamma_{0}\exp(Y))=\sum\limits_{\mathcal{O}} c_{\mathcal{O}}(\pi,\gamma_{0})\widehat{\mu_{\mathcal{O}}}(Y),$$
where $\gamma_0$ is a semisimple element of $G(F)$, $M=C_{G}(\gamma_{0})$, and $\mathcal{O}$ runs over all the nilpotent $M(F)$-orbits in $\mathfrak{m}(F)$.
\end{theorem}

\begin{remark}

For $\gamma_{0}=1$, this formula is the usual Harish-Chandra character expansion
$$\Theta_{\pi}(\exp(Y))=\sum\limits_{\mathcal{O}\in \mathcal{O}(0)} c_{\mathcal{O}}(\pi)\widehat{\mu_{\mathcal{O}}}(Y),$$
as an equality of locally constant functions in a small neighbourhood of $0$. Here we use the notation $\mathcal{O}(0)$ for the set of nilpotent $G(F)$-orbits in $\mathfrak{g}(F)$.
\end{remark}

In fact, semisimple orbital integrals and nilpotent orbital integrals both play an important role in the space of invariant distribution on $G(F)$. Moreover, they are related by the following germ expansion result of regular semisimple orbital integral by Shalika \cite{Sha72}.
\begin{theorem}[\cites{Sha72}]
For any $f\in C_{c}^{\infty}(\mathfrak{g}(F))$, there exist locally constant functions $c_{\mathcal{O}}$ on $\mathfrak{t}(F)^{\mathrm{rs}}$ and a neighbourhood $U_f$ of $0$ in $\mathfrak{t}(F)$ such that for any $X\in U_f\cap\mathfrak{t}(F)^{\mathrm{rs}}$,
$$\mu_{\mathcal{O}_X}(f)=\mu_{f}(X)=\sum\limits_{\mathcal{O}\in \mathcal{O}(0)}\mu_{\mathcal{O}}(f)c_{\mathcal{O}}(X).$$
\end{theorem}

\begin{remark}
For a fixed $X$ sufficient close to $0$, this identity is an identity of distributions on $\mathfrak{g}$:
$$\mu_{\mathcal{O}_X}=\sum\limits_{\mathcal{O}\in \mathcal{O}(0)}c_{\mathcal{O}}(X)\mu_{\mathcal{O}}.$$
A priori, this identity has nothing to do with representations. In fact, we could expect more for the relationship between character expansions and Shalika germs expansions. For instance, if we admit the philosophy of orbit method in the spirit of Kirillov--Rossmann \cite{Ros78} for real reductive groups, namely, a discrete series representation $\pi$ should correspond to a semisimple elliptic element $X_{\pi}\in \mathfrak{g}^*(F)$ (indeed a coadjoint orbit in $\mathfrak{g}^*/G \cong \mathfrak{g}/G \cong \mathfrak{h}/W$), one can actually relate $\Theta_\pi$ and $\widehat{\mu_{\mathcal{O}_{X_\pi}}}$.
\end{remark}

For $G=\mathrm{GL}_n$, Murnaghan \cite{Mur96} first constructs a semisimple element $X_\pi\in \mathfrak{g}(F)$ from a tamely ramified supercuspidal representation $\pi$. Moreover, she studies the relation between $\Theta_\pi$ and $\widehat{\mu_{\mathcal{O}_{X_\pi}}}$ and finds that $$c_{\mathcal{O}}(\pi)=\deg(\pi)c_{\mathcal{O}}(X_{\pi}),$$
for arbitrary nilpotent orbit $\mathcal{O}\subset \mathfrak{g}(F)$, which leads to a formula of Kirillov--Rossmann type.
\begin{theorem}[\cites{Mur96}]
For $G=\mathrm{GL}_n$, one has the following equality
$$\Theta_{\pi}(\exp Y)=\deg(\pi)\widehat{\mu_{X_{\pi}}}(Y),$$
where $X_{\pi}\in \mathfrak{g}(F)\cong \mathfrak{g}^*(F)$ is a regular elliptic element associated to the supercuspidal representation $\pi$ as in \cite[Section 3]{Mur96}.
\end{theorem}
Later, this phenomenon is explored in much more generality in \cites{KM03,KM06}.

\subsection{Kostant section}
Let $\mathcal{O}\subset\mathfrak{g}(F)$ be a regular nilpotent $G(F)$-orbit and let $f\in \mathcal{O}$ be a nilpotent element. By Jacobson--Morozov theorem, there exists an $\mathfrak{sl}_2$-triple $\{h,e,f\}$ such that
$$[h,e]=2e,\quad[h,f]=-2f,\quad[e,f]=h.$$
\begin{definition}
A Kostant section of $\mathcal{O}$ at $f$ is an affine subspace $\mathbb{V}(\overline{F})\subset \mathfrak{g}(\overline{F})$ defined by
$$\mathbb{V}(\overline{F}):=f+C_{\mathfrak{g}(\overline{F})}(e),$$
where $\{[e,f],e,f\}$ is an $\mathfrak{sl}_2$ triple.
\end{definition}
Since $e$ and $f$ are $F$-rational, the Kostant section $\mathbb{V}$ is defined over $F$.

\begin{lemma}[\cites{Kos63}]\label{Kostant}

Every regular $G(\overline{F})$-orbit $\mathcal{O}_{-}\subset \mathfrak{g}(\overline{F})$ meets $\mathbb{V}(\overline{F})$ in exactly one point.
\end{lemma}

Moreover, if the $G(\overline{F})$-orbit $\mathcal{O}_{-}$ is also defined over $F$, then the unique point in $\mathcal{O}_{-}\cap \mathbb{V}(\overline{F})$ is $F$-rational. Hence the Kostant section $\mathbb{V}$ determines an $F$-rational point in every regular $G(\overline{F})$-orbit.

\begin{proposition}[{\cites[Proposition 4.2]{DR10} {Kot99}{She89}}]\label{DRshe}
Let $\mathcal{O}$ be a regular nilpotent $G(F)$-orbit in $\mathfrak{g}(F)$, let $\mathbb{V}$ be a Kostant section for $\mathcal{O}$, and let $X$ be a regular semisimple element in $\mathfrak{g}(F)$. Then the coefficient $c_{\mathcal{O}}(X)$ is non zero if and only if the $G(F)$-orbit of $X$ meets $\mathbb{V}(F)$.
\end{proposition}

The above statement above nilpotent orbits can be transferred to a statement about generic characters in the following way. Let $B_0^-$ be the opposite Borel subgroup to $B_0$ and let $N_0^-$ be the unipotent radical of $B_0^-$ with corresponding Lie algebras $\mathfrak{b}_0^-$ and $\mathfrak{n}_0^-$. By \cite[Section 4.5]{DR10}, one has the assignment
\begin{equation*}
	\begin{aligned}
		\mathrm{Hom}(N_0(F),\mathbb{C}^{\times})&\rightarrow \mathfrak{n}_0^-(F)\cong{\mathfrak{n}_0(F)}^*\\
		\psi_{N_0(F)}&\mapsto f_{\psi},
	\end{aligned}
\end{equation*}
such that $f_{\psi}$ is defined by
\begin{equation}\label{liealgebrageneric}
	\begin{aligned}
		\psi_{N_0(F)}(\exp X)=\psi_F(\langle f_{\psi},X\rangle)
	\end{aligned}
\end{equation}
for any $X\in \mathfrak{n}_0(F)$. From another perspective, it is proved in \cite{MgW87} that an irreducible representation $\pi$ is $(N_0,\psi)$ generic if and only if $c_{\mathcal{O}_{f_\psi}}(\pi)$ is non-zero.

If $\pi$ is a very supercuspidal representation \cite[Section 2.6]{DR10} associated to a tame elliptic pair $(S,\mu)$ such that $S$ is unramified, let $X$ denote the good semisimple element with $C_{G}(X)=S$. The following geometric characterization of generic characters, to which $\pi$ is generic with respect, is due to Debacker and Reeder.
\begin{theorem}[{\cites[Proposition 4.10]{DR10}}]\label{DR}
Let $\pi=\pi_{(S,\mu)}=\pi(X)$ be a very supercuspidal representation of $G(F)$. Assume $\mathcal{B}_{\mathrm{red}}(S)\subset \mathcal{B}_{\mathrm{red}}(G)$ is a hyperspecial point. For any generic character $\psi:N_0(F)\rightarrow \mathbb{C}^{\times}$, $\mathrm{Hom}_{N_0(F)}(\pi,\psi)\neq 0$ if and only if the $G(F)$-orbit of $X$ meets $\mathbb{V}_{f_\psi}$.
\end{theorem}

\begin{remark}\label{characterizationremark}
The arguments in the above theorem can be applied to construct a base point in a toral regular supercuspidal $L$-packet \cite[Lemma 6.2.2]{Kal19b}. More precisely, as we mentioned in \Cref{whittakerembeddingbasepoint}, there is a unique $G(F)$-conjugacy class of embeddings $j_{\mathfrak{w}}:S\rightarrow G$ such that $\pi_{j_{\mathfrak{w}}}$ is $\mathfrak{w}$-generic, whose uniqueness is characterized by the condition that
the $G(F)$-orbit of $dj_{\mathfrak{w}}(X)$ meets the Kostant section of $f_{\psi_{N_0}}$. Indeed, Kaletha first establishes a character formula evaluating at a regular semisimple element for a toral regular supercuspidal representation \cite[Lemma 6.2.1]{Kal19b} and then gets a comparison formula between his character formula and the local character expansion \cite[Lemma 6.2.2]{Kal19b} evaluating at a strongly regular semisimple element $g\in G(F)$
\begin{equation}\label{characterequality}
\Theta_j(g)=|D_G(g)|^{-\frac{1}{2}}\widehat{\iota}_{\mathfrak{g},{}^jX*}(\log g)=|D_G({}^jX^*)|^{\frac{1}{2}}\widehat{\mu}_{\mathfrak{g},{}^jX*}(\log g).
\end{equation}
Since we are not explicitly using this formula, we will not explain this formula in detail and one can refer to his paper for these notations. Combining \Cref{characterequality}, \Cref{DRshe} and \Cref{DR}, he obtains the uniqueness of $G(E)$-conjugacy class of $j_\mathfrak{w}$. Later this argument is generalized to all nonsingular supercuspidal representations in \cite[Page 2273]{FKS23} by extending the character formula for toral supercuspidals in \cite[Lemma 6.2.1]{Kal19b} to a character formula for all non-singular supercuspidals in \cite[Proposition 4.4.1]{FKS23}. 
\end{remark}

From another perspective, if we pick up a generic representation $\pi$ of $G(E)$ which is $G(F)$-distinguished, one has the following expectations on the Whittaker model of $\pi$ from Prasad.
\begin{conjecture}[{\cites[Lemma 4.1, Proposition 4.2]{AP18}}]
Let $G$ be the special linear group $\mathrm{SL}_n$ and let $\pi$ be an irreducible generic representation of $G(E)$. If $\pi$ is $G(F)$-distinguished, then $\pi$ has a Whittaker model with respect to a character $\psi:N_{0}(E)/N_{0}(F)\rightarrow \mathbb{C}^{\times}$.
\end{conjecture}

This conjecture has been proved by Anandavardhanan and Prasad for $\mathrm{SL}_n(F)$ \cite{AP18} based on the study of the restriction of representations of $\mathrm{GL}_n(F)$ to $\mathrm{SL}_n(F)$ and the classification of representations of $\mathrm{GL}_n$. Later, these types of questions have been raised and treated in a paper of Dipendra Prasad \cite{Pra19b} with an appendix by Sakellaridis for general symmetric spaces. Later, Gourevitch and Sayag \cite{GS23} provide another proof by carefully studying the geometry of co-adjoint orbits sitting inside the image of moment map.

We have the following characterization of generic characters $\psi_{N_0(E)}:N_0(E)\rightarrow \mathbb{C}^{\times}$, such that $\psi_{N_0(E)}|_{N_0(F)}=\mathbbm{1}$. Consider the following exact sequence of $F$-groups
$$1\rightarrow [N_0,N_0]\rightarrow N_0\stackrel{\{x_{\alpha}\}_{\alpha\in \Delta}}{\longrightarrow} \prod_{\Delta}\mathbb{G}_a\rightarrow 1.$$
Since $[N_0,N_0]$ is unipotent, $H^1(F,[N_0,N_0])$ is trivial. The long exact sequence gives the isomorphism
$$N_0(F)/[N_0,N_0](F)\rightarrow F^{|\Delta|}.$$
Let $\Delta_{\Sigma}$ denote the composition
$$\Delta_{\Sigma}:N_0\stackrel{\{x_{\alpha}\}_{\alpha\in \Delta}}{\longrightarrow} \prod_{\Delta}\mathbb{G}_a\stackrel{\Sigma}{\longrightarrow}\mathbb{G}_a.$$
Notice that any generic character $\psi_{N_0(F)}:N_0(F)\rightarrow \mathbb{C}^{\times}$ is of the form $\psi_F\circ \Delta_{\Sigma}(F)$. More explicitly, $\psi_{N_0(F)}$ is defined by
\begin{equation}\label{additive}
\begin{aligned}
\psi_{N_0(F)}(\prod_{\alpha\in \Phi^{+}}x_{\alpha}(t_{\alpha})):=\psi_{F}(\sum_{\alpha\in \Delta} t_{\alpha})
\end{aligned}
\end{equation}
where $\psi_F:F\rightarrow\mathbb{C}^{\times}$ is a fixed non-trivial additive character of $F$.

\begin{lemma}
Let $\mathbb{G}_a^{\mathrm{op}}$ denote the $F$-group $\mathrm{Res}_{E/F}\mathbb{G}_{a,E}/\mathbb{G}_a$, whose $F$-rational points consist of trace zero elements in $E$
$$\mathbb{G}_a^{\mathrm{op}}(F)=\{x\in E|x+\sigma(x)=0\}=\mathrm{Im}(1-\sigma),$$
where $1-\sigma$ denotes the map $E\rightarrow E$ sending $x$ to $x-\sigma(x)$. Let $\sqrt{a}$ denote a chosen trace-zero element such that $E=F(\sqrt{a})$. There is an exact sequence
\begin{equation}
\begin{aligned}\label{ga}
1\rightarrow F^{|\Delta|}\rightarrow E^{|\Delta|}\rightarrow (\sqrt{a}F)^{|\Delta|}\rightarrow 1,
\end{aligned}
\end{equation}
where the second map is induced by the map $1-\sigma$.
\end{lemma}

\begin{proof}
Notice that there is an exact sequence of $F$-groups
\begin{equation*}
\begin{aligned}
1\rightarrow \prod_{\Delta}\mathbb{G}_a\rightarrow \prod_{\Delta}\mathrm{Res}_{E/F}\mathbb{G}_{a,E}\rightarrow \prod_{\Delta}\mathbb{G}_a^{\mathrm{op}}\rightarrow 1.
\end{aligned}
\end{equation*}
Since $H^1(F,\prod\limits_{\Delta}\mathbb{G}_a)$ is trivial, the long exact sequence implies the lemma.
\end{proof}

\begin{lemma}\label{whittakerF}
$\psi_{N_0(E)}:N_0(E)\rightarrow \mathbb{C}^{\times}$ satisfies $\psi_{N_0(E)}|_{N_0(F)}=\mathbbm{1}$ if and only if $\psi_{N_0(E)}^{\sigma}=\psi_{N_0(E)}^{-1}$.
\end{lemma}
\begin{proof}
By \Cref{additive}, we know that $\psi_{N_0(E)}|_{N_0(F)}=\mathbbm{1}$ if and only if $\psi_E\circ \Sigma|_{F^{|\Delta|}}=\mathbbm{1}$. By applying the exact functor $\mathrm{Hom}_{\mathrm{cts}}(\cdot,\mathbb{C}^{\times})$ to \ref{ga}, we know $\psi_E\circ \Sigma|_{F^{|\Delta|}}=\mathbbm{1}$ if and only if $\psi_E\circ \Sigma$ is a lift of a character of $(\sqrt{a}F)^{|\Delta|}$, that is, $\psi_E\circ \Sigma:E^{|\Delta|}\rightarrow \mathbb{C}^{\times}$ is of the form $\psi_{F}'(\frac{1}{\sqrt{a}}\Sigma\circ (1-\sigma))$ for some character $\psi_F'$ of $F$. This in turns implies that $\psi_{N_0(E)}^{\sigma}=\psi_{N_0(E)}^{-1}$. The reverse direction easily follows from the surjectivity of the trace map $\mathrm{tr}:E\rightarrow F$ sending $x$ to $x+\sigma(x)$.
\end{proof}

Combined with the identification \ref{liealgebrageneric}, this further implies that the element $f_{\psi_{N_0(E)}}\in \mathfrak{n}_0^-(E)$ associated to $\psi_{N_{0}(E)}$ satisfies the condition
\begin{equation}\label{regularunipotent}
\begin{aligned}
f_{\psi_{N_0(E)}}=-\sigma(f_{\psi_{N_0(E)}}).
\end{aligned}
\end{equation}

In fact, if $\pi$ is a regular supercuspidal representation of $G(E)$ which is $G(F)$-distinguished, the results of Prasad \cite{Pra19b} on the description of Whittaker model for $\pi$ could be regarded as a reverse direction of \Cref{embeddingoverF}.

\end{document}